\documentclass[12pt, twoside]{book}

\usepackage[T1]{fontenc}
\usepackage{charter}
\usepackage[expert]{mathdesign}
\usepackage{amsmath,amssymb,amsthm,amsfonts}
\usepackage{mathtools}
\usepackage{minitoc-hyper}
\usepackage{graphicx,wrapfig}
\graphicspath{{./figures/}}
\usepackage[a4paper,width=150mm,top=25mm,bottom=25mm,bindingoffset=6mm]{geometry}
\usepackage{emptypage}
\usepackage[cal=boondox]{mathalfa}
\usepackage{wasysym,textcomp,array,arcs,polynom,cancel,enumerate,dcolumn} 
\usepackage{float}
\usepackage{caption}
\usepackage{subcaption}
\usepackage{mathrsfs}
\usepackage{makeidx}
\makeindex 
\usepackage{afterpage}
\usepackage[all, cmtip]{xy} 
\usepackage[Lenny]{fncychap} %Options: Sonny, Lenny, Glenn, Conny, Rejne, Bjarne, Bjornstrup
\usepackage{fancyhdr}
\allowdisplaybreaks
\usepackage[usenames,dvipsnames]{xcolor}

\definecolor{aureolin}{HTML}{009EFF}
\definecolor{tealGreen}{HTML}{008731}
\usepackage[linktoc=all,colorlinks,linkcolor=teal,citecolor=aureolin]{hyperref}
\usepackage[nameinlink,noabbrev]{cleveref}
\crefname{defn}{definition}{definitions}
\Crefname{defn}{Definition}{Definitions}
\usepackage[symbols,nogroupskip,stylemods,postpunc=dot, automake]{glossaries-extra}
\usepackage{tikz-cd}
\usepackage{cases} % for equation numbering in cases
\usepackage{longtable}
\makeatletter
\newcommand{\rbb}{\ensuremath{\mathbb{R}}} 
\newcommand{\cbb}{\ensuremath{\mathbb{C}}} 
\newcommand{\zbb}{\ensuremath{\mathbb{Z}}}

\newcommand{\zbf}{\ensuremath{\mathbf{z}}} 
\newcommand{\ubf}{\ensuremath{\mathbf{u}}} 
\newcommand{\vbf}{\ensuremath{\mathbf{v}}}
\newcommand{\wbf}{\ensuremath{\mathbf{w}}}
\newcommand{\ebf}{\ensuremath{\mathbf{e}}}
\renewcommand{\bf}[1]{\ensuremath{\mathbf{#1}}}
\newcommand{\sbb}{\ensuremath{\mathbb{S}}}
\newcommand{\cali}[1]{\ensuremath{\mathcal{#1}}}

\newcommand{\blue}[1]{\textcolor{blue}{#1}}

\newcommand{\delbydel}[2]{\ensuremath{\dfrac{\partial#1}{\partial#2}}} 
\newcommand{\crf}{\ensuremath{\mathrm{Cr}(f)}} 
\newcommand{\grad}{\ensuremath{\nabla}} 
\newcommand{\comp}{\mathbin{\mathchoice {\compcent\scriptstyle}{\compcent\scriptstyle} {\compcent\scriptscriptstyle}{\compcent\scriptscriptstyle}}} 
\newcommand{\compcent}[1]{\vcenter{\hbox{$#1\circ$}}} 
 
\newcommand{\dist}{\ensuremath{\operatorname{dist}}} 
 
\newcommand{\isom}{\ensuremath{\cong}} 

\newcommand{\defeq}{\vcentcolon=}
\newcommand{\eqdef}{=\vcentcolon}

\newcommand{\directsum}{\ensuremath{\oplus}}

\newcommand{\cutn}[1][N]{\ensuremath{\mathrm{Cu}(#1)}}
\newcommand{\sen}{\ensuremath{\mathrm{Se}(N)}}
\newcommand{\innerprod}[2]{\ensuremath{\left\langle #1,#2\right\rangle}}
\newcommand{\norm}[1]{\ensuremath{\left\|#1\right\|}}

\newcommand{\paran}[1]{\ensuremath{\left( #1 \right)}}
\newcommand{\curlybracket}[1]{\ensuremath{\left\{ #1 \right\}}}
\newcommand{\squarebracket}[1]{\ensuremath{\left[ #1 \right]}}
\newcommand{\abs}[1]{\ensuremath{\left|#1\right|}}
\newcommand{\aTransInverse}{\ensuremath{\paran{A^T}^{-1}}}
\newcommand{\sqrtATransAInverse}{\ensuremath{\paran{\sqrt{A^TA}}^{-1}}}
\newcommand{\trace}[1]{\ensuremath{\mathrm{tr}\left( #1 \right)}}
\newcommand{\cu}{\ensuremath{\mathrm{Cu}(p)}}
\newcommand{\se}[1][p]{\ensuremath{\mathrm{Se}(#1)}} 
\newcommand{\co}{\ensuremath{\mathrm{Co}(x_0,\delta)~}} 
\newcommand{\costar}{\ensuremath{\mathrm{Co}^\star(x_0,\delta)~}}
\newcommand{\Ball}{\ensuremath{\overline{B(x_0,\delta)}~}}

\newcommand{\hess}{\mathrm{Hess}}
\newcommand{\R}{\mathbb{R}}
\newcommand{\C}{\mathbb{C}}
\newcommand{\CP}{\mathbb{CP}}

\newcommand{\bgd}{\begin{displaymath}}
\newcommand{\edd}{\end{displaymath}}
\newcommand{\bgc}{\begin{center}}
\newcommand{\edc}{\end{center}}
\newcommand{\hf}{\hspace*{0.5cm}}
\newcommand{\hfb}{\hspace{1cm}}
\newcommand{\lan}{\left\langle}
\newcommand{\ran}{\right\rangle}
\newcommand{\upq}{U(p,q)}
\newcommand{\ep}{\varepsilon}
\renewcommand{\epsilon}{\varepsilon}
\DeclareMathOperator{\spn}{span}
\newcommand{\bigzero}{\mbox{\normalfont\Large\bfseries 0}}
\newcommand{\ubb}{\mathcal{u}}
\newcommand{\vbb}{\mathcal{v}}
\newcommand{\rs}{0.7ex}

\newcommand{\spmat}[1]{%
  \left(\begin{smallmatrix}#1\end{smallmatrix}\right)%
}

\DeclareMathAlphabet{\mathpzc}{OT1}{pzc}{m}{it} 
\definecolor{PropColor}{HTML}{0DC8F2}
\definecolor{CorColor}{HTML}{FFC300}
\definecolor{ProbColor}{HTML}{FF1212}

\usepackage{thmtools}
\usepackage[framemethod=TikZ]{mdframed}
\mdfsetup{skipabove=1em,skipbelow=0em}
\theoremstyle{definition}
\declaretheoremstyle[
    headfont=\bfseries\sffamily\color{Orange!70!black}, bodyfont=\normalfont,
    mdframed={
        linewidth=2pt,
        rightline=false, topline=false, bottomline=false,
        linecolor=Orange, backgroundcolor=Orange!5,
    }
]{thmthmbox}

\declaretheoremstyle[
    headfont=\bfseries\sffamily\color{NavyBlue!70!black}, bodyfont=\normalfont,
    mdframed={
        linewidth=2pt,
        rightline=false, topline=false, bottomline=false,
        linecolor=NavyBlue, backgroundcolor=NavyBlue!5,
    }
]{thmdefnbox}

\declaretheoremstyle[
    headfont=\bfseries\sffamily\color{LimeGreen!70!black}, bodyfont=\normalfont,
    mdframed={
        linewidth=2pt,
        rightline=false, topline=false, bottomline=false,
        linecolor=LimeGreen, backgroundcolor=LimeGreen!5,
    }
]{thmlemmabox}

\declaretheoremstyle[
    headfont=\bfseries\sffamily\color{PropColor!70!black}, bodyfont=\normalfont,
    mdframed={
        linewidth=2pt,
        rightline=false, topline=false, bottomline=false,
        linecolor=PropColor, backgroundcolor=PropColor!5,
    }
]{thmpropbox}

\declaretheoremstyle[
    headfont=\bfseries\sffamily\color{CorColor!70!black}, bodyfont=\normalfont,
    mdframed={
        linewidth=2pt,
        rightline=false, topline=false, bottomline=false,
        linecolor=CorColor, backgroundcolor=CorColor!5,
    }
]{thmcorbox}

\declaretheoremstyle[
    headfont=\bfseries\sffamily\color{ProbColor!70!black}, bodyfont=\normalfont,
    mdframed={
        linewidth=2pt,
        rightline=false, topline=false, bottomline=false,
        linecolor=ProbColor, backgroundcolor=ProbColor!5,
    }
]{thmprobbox}

\declaretheoremstyle[
    headfont=\bfseries\sffamily, bodyfont=\normalfont
]{thmexambox}

\declaretheorem[style=thmdefnbox, name=Definition, numberwithin=section]{defn}
\declaretheorem[style=thmthmbox, name=Theorem, numberwithin=section]{thm}
\declaretheorem[style=thmlemmabox, name=Lemma, numberwithin=section]{lemma}
\declaretheorem[style=thmpropbox, name=Proposition, numberwithin=section]{prop}
\declaretheorem[style=thmprobbox, name=Conjecture, numbered=no]{conj}

\declaretheorem[style=thmcorbox, name=Corollary, numberwithin=section]{cor}
\declaretheorem[style=thmexambox, name=Example, numberwithin=section]{eg}

\theoremstyle{definition}
\newtheorem{rem}{Remark}[section]
\newtheorem{note}{Note}[section]

\declaretheorem[style=thmthmbox, name=Theorem]{mainthm}

\declaretheorem[style=thmthmbox, name=Theorem, numbered=no]{thmSec}

\usepackage{import}
\usepackage{xifthen}
\usepackage{pdfpages}
\usepackage{transparent}

\makeglossaries

\newglossaryentry{cutn}{
 name=\ensuremath{\cutn},
 description={cut locus of $N$},
 type=symbols
}

\newglossaryentry{cutp}{
 name=\ensuremath{\cutn[p]},
 description={cut locus of $p$},
 type=symbols
}
\newglossaryentry{sen}{
 name=\ensuremath{\sen},
 description={separating set of $N$},
 type=symbols
}

\newglossaryentry{sep}{
 name=\ensuremath{\se},
 description={separating set of $p$},
 type=symbols
}

\begin{document}
	\frontmatter
	\begin{titlepage}
	\centering % Centre everything on the title page
	
%	\scshape % Use small caps for all text on the title page
	
	\vspace*{\baselineskip} % White space at the top of the page
	
	%------------------------------------------------
	%	Title
	%------------------------------------------------
	
	\rule{\textwidth}{1.6pt}\vspace*{-\baselineskip}\vspace*{2pt} % Thick horizontal rule
	\rule{\textwidth}{0.4pt} % Thin horizontal rule
	
	\vspace{0.75\baselineskip} % Whitespace above the title
	\Large \textbf {Cut Locus of Submanifolds: A Geometric and Topological Viewpoint}
		\vspace{0.75\baselineskip} % Whitespace below the title
	
	\rule{\textwidth}{0.4pt}\vspace*{-\baselineskip}\vspace{3.2pt} % Thin horizontal rule
	\rule{\textwidth}{1.6pt} % Thick horizontal rule
	
	\vspace{2\baselineskip} % Whitespace after the title block

	\small \emph{A thesis submitted in partial fulfillment of the\\
		requirements for the award of the degree of}
	\vspace{.2in}
	
	{\bf Doctor of Philosophy }\\[0.2in]
	
	% Submitted by
	\emph{ Submitted by} \\[0.1cm]
	\textbf{ Sachchidanand Prasad\\[0.2mm] (17RS038)}\\[0.2in]
	
	\emph{ Under the supervision of}\\[0.1cm]
	{\textbf{Dr. Somnath Basu}}\\[0.3in]
	to the\\[0.1in]
	% Bottom of the page
	\textbf{\Large{Department of Mathematics and Statistics}}\\[2in]
	\includegraphics[width=0.12\textwidth, keepaspectratio]{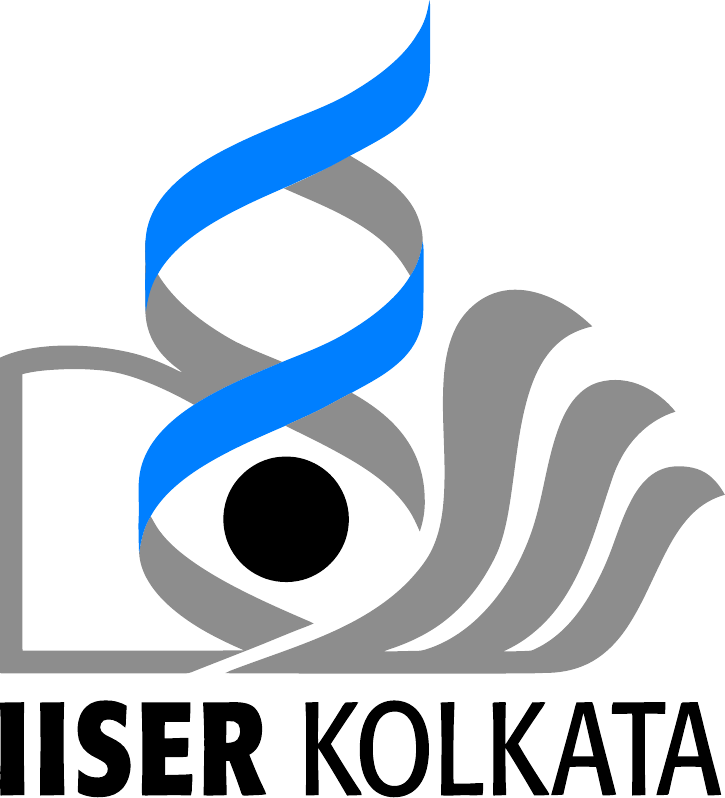}\\[0.5in]
	\normalsize
	\textsc{Indian Institute of Science Education and Research, Kolkata}
\end{titlepage}	
	
	\chapter*{}
	\vspace{4.0in}
	\begin{center}
		Copyright by \\[.2cm]
		Sachchidanand Prasad \\[0.2cm]
		2022
	\end{center}
	
	\chapter*{Declaration}
	I hereby declare that this thesis is my own work and, to the best of my knowledge, it contains no materials previously published or written by any other person, or substantial proportions of material which have been accepted for the award of any other degree or diploma at IISER Kolkata or any other educational institution, except where due acknowledgement is made in the thesis.

\vspace{3cm}

\begin{minipage}{0.45\linewidth}
    \begin{flushleft}
        \today\\[0.5ex]
        IISER Kolkata
    \end{flushleft}
\end{minipage}
\hfill
\begin{minipage}{0.45\linewidth}
    \begin{flushright}
        \includegraphics[scale=0.70]{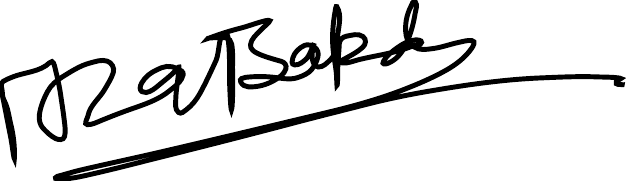}
        \rule{4.6cm}{0.15mm}\\[0.5ex]
        Sachchidanand Prasad \\[0.5ex]
        (17RS038)
    \end{flushright}
\end{minipage}

% \vspace{3cm}
% \noindent \today \hfill \includegraphics[scale=0.75]{figures/signature_mine.pdf}\\
% \mbox{} \hfill \rule{4.5cm}{0.25mm}\\
% IISER Kolkata \hfill  Sachchidanand Prasad\\[0.5ex]
% \mbox{} \hfill (17RS038)
	% \addcontentsline{toc}{chapter}{Declaration}

	\chapter*{Certificate}
	This is to certify that the thesis entitled \textit{Cut Locus of Submanifolds: A Geometric and Topological Viewpoint} is a bona fide record of work done by Sachchidanand Prasad (17RS038),
a student enrolled in the PhD Programme, under my supervision during August 2017 - October 2022, submitted in partial fulfillment of the requirements for the award of PhD Degree by the Department of Mathematics and Statistics (DMS), Indian Institute of Science Education and Research (IISER) Kolkata.
\vspace{2cm}
\begin{flushright}
    \rule{4.5cm}{0.15mm}\\
    \textit{Supervisor}\\
    Dr. Somnath Basu\\
    Associate Professor\\
    DMS, IISER Kolkata
\end{flushright}
	% \addcontentsline{toc}{chapter}{Certificate}

	\chapter{Acknowledgements}
	\hfb At first, I would like to thank my parents for their support and encouragement. Thanks to them for raising me to be what I am today, believing in me, giving me freedom and space to grow as I deemed fit. A special thanks go to my mother. Without her love and affection, it would have been impossible to finish my work. I also thank my brother and sisters, who have shown their belief in me which gave me constant confidence. Their constant support has often helped me sustain myself throughout. I also take this opportunity to thank one of my teachers Satyendra Singh who has always been helpful in many ways.

\vspace{0.3cm}
\hf At this point, I must thank the faculties at the National Institute of Technology, Rourkela, who are responsible for the academic path that has led me here. I wish to thank Dr. Debajyoti Choudhuri for his unforgettable guidance during my stay at NIT Rourkela. Thanks to Dr. Bikash Sahoo for introducing me to Prof. Swadhin Pattanayak, who deserves a special mention for sharing his philosophy of mathematics and encouragement for the research. 

\vspace{0.3cm}
\hf I must also thank the IISER Kolkata for providing a friendly environment for my academic study. I also acknowledge the staff members and security persons of IISER Kolkata for helping me in numerous ways. The library resources have been tremendously helpful during my whole stay here. I must thank the library staff for this. 

\vspace{0.3cm}
\hf  I am really grateful to the Mathematics Training and Talent Search (MTTS) program as they set the right base, which was helpful for me to build up on. A special thanks to Prof. Kumaresan and Prof. Bhabha Kumar Sharma for their wonderful teaching and ideas, which helped me to handle a research problem. The methodology given by Prof. Kumaresan was always giving me the strength to understand research articles. 

\vspace{0.3cm}
\hf This whole PhD would not have been possible without financial support. I must thank and acknowledge the CSIR-UGC Government of India for this support. 

\vspace{0.3cm}
\hf It is my privilege to express heartful thanks and sincere gratitude to the faculty members of the Department of Mathematics, IISER Kolkata. A special thanks to my RPC members Dr. Sushil Gorai and Dr. Subrata Shyam Roy, for listening to me each year and giving valuable comments. I want to extend my sincere thanks to Satyaki Sir and Shirshendu Sir for their guidance and for being a friend, which made me feel more comfortable. I am also grateful to the non-teaching staff of my department, Adrish Da and Rajesh Da, for their administrative and many other bits of help.

\vspace{0.3cm}
\hf I would like to thank the thesis reviewers, Dr.  Ritwik Mukherjee from NISER Bhubaneswar and Prof. Dr. Janko Latschev from University of Hamburg for their valuable suggestions. They have pointed out corrections and these have been invaluable in improving the mathematical exposition of the thesis.

\vspace{0.3cm} 
\hf I cannot begin to express my thanks to my friends, who played an important role during my stay. Prahllad Da's discussions were beneficial in the early days of my PhD. His experience and knowledge helped me a lot to understand basic tools in algebraic and differential topology. Many thanks to Sandip and Ramyak for the fruitful discussions. When I got stuck on some proofs, conversations with them often solved many problems. I gladly thank Golam, Subrata, Ashish Da and Mrinmoy Da for their availability for any help. Many thanks to Manish, Sanjoy, Jiten, Samiron, Avishek, Sugata. I thank Mukil, Saikat, Sandip, and Anant for their helpful discussions and make valuable comments after reading the thesis. I thank Saikat Panja from IISER Pune for many suggestions and daily discussions. Finally, I would also like to extend my deepest gratitude to Gaurav, Raksha and Sonu for their friendship. This journey would have been tough without you guys.

\vspace{0.3cm}
\hf I thank Anada Marg and my father for teaching me meditation which always gave me positive energy. Thanks to Gurudev for always helping me, showing me the right path and for constant blessings.

\vspace{0.3cm} 
\hf I would also like to thank the creators of beautiful and open source software like \LaTeX, \textsc{Inkscape}, \textsc{GeoGebra} and many more. Also, I should not forget to thank the webpages like Wikipedia, math stack exchange, math overflow and TeX stack exchange. 

\vspace{0.3cm} 
\hf Finally, it comes to the director of the thesis, my supervisor Dr. Somnath Basu. He is the most important person without whom this was absolutely not possible. His way of doing mathematics is totally different. He has crystal clear concepts, geometric intuition and beautiful imagination. He has been very patient and attentive, simultaneously providing guidance and sharing his mathematical insights. I am thankful to him for his guidance and friendly encouragement throughout my work over the last five years. I feel lucky to have him as an advisor who has taken a keen interest in my progress. Apart from mathematics, his contagious enthusiasm and work ethic makes him a great teacher to work under. I could not imagine someone better to learn mathematics from! I also appreciate his patience and efficiency with regular and extensive online discussions during the lockdown period due to the COVID pandemic. I feel fortunate to know him; he will always be an inspiration to me. Thank you so much Sir, for providing a proper guidance, suggestions and feedback throughout my PhD tenure. It was impossible for me to complete this thesis without your support and supervision.

\vspace{0.3cm}
\hf Finally, thanks to all who are not mentioned here but are associated with me.

	%\addcontentsline{toc}{chapter}{Acknowledgements}

	\chapter*{}
	\begin{center}
    \vspace*{\fill}
    \textit{To my family.} 
    \vspace*{\fill}
\end{center}
	% \addcontentsline{toc}{chapter}{Dedication}

	\chapter*{Publications related to the thesis}
	\begin{enumerate}
    \item Basu, S. and Prasad, S. (2021) \textit{A connection between cut locus, Thom space and Morse-Bott functions}, available at \url{https://arxiv.
    org/abs/2011.02972}, to appear in Algebraic \&  Geometric Topology.
\end{enumerate}

	\chapter{Abstract}
	Associated to every closed, embedded submanifold $N$ of a connected Riemannian manifold $M$, there is the distance function $d_N$ which measures the distance of a point in $M$ from $N$. We analyze the square of this function and show that it is Morse-Bott on the complement of the cut locus $\cutn$ of $N$, provided $M$ is complete. Moreover, the gradient flow lines provide a deformation retraction of $M-\cutn$ to $N$. If $M$ is a closed manifold, then we prove that the Thom space of the normal bundle of $N$ is homeomorphic to $M/\cutn$. We also discuss several interesting results which are either applications of these or related observations regarding the theory of cut locus. These results include, but are not limited to, a computation of the local homology of singular matrices, a classification of the homotopy type of the cut locus of a homology sphere inside a sphere, a deformation of the indefinite unitary group $U(p,q)$ to $U(p)\times U(q)$ and a geometric deformation of $GL(n,\R)$ to $O(n,\R)$ which is different from the Gram-Schmidt retraction. 

\bigskip
\noindent If a compact Lie group $G$ acts on a Riemannian manifold $M$ freely then $M/G$ is a manifold. In addition, if the action is isometric, then the metric of $M$ induces a metric on $M/G$.  We show that if $N$ is a $G$-invariant submanifold of $M$, then the cut locus $\cutn$ is $G$-invariant, and $\cutn/G = \cutn[N/G]$ in $M/G$. An application of this result to complex projective hypersurfaces has been provided.

	\dominitoc
	\tableofcontents
	% \printglossary[type=symbols,style=long,title={List of Symbols}]

	\listoffigures
	\addcontentsline{toc}{chapter}{List of figures}
	
	\chapter{Notations}
	\begin{longtable}{r l}
    $A\cup B:$ & union of sets $A$ and $B$ \\[\rs]
    $A\cap B:$ & intersection of sets $A$ and $B$ \\[\rs]
    $A\times B:$ & Cartesian product of $A$ and $B$ \\[\rs]
    $A\setminus B:$ & set of elements in $A$ but not in $B$ \\[\rs] 
    $A \subset B:$ & $A$ is a subset of $B$, not necessarily proper \\[\rs] 
    $A \sqcup B:$ & disjoint union of $A$ and $B$  \\[\rs] 
    $A \directsum B:$ & direct sum of $A$ and $B$  \\[\rs] 
    $\mathbb{Z}: $ & the set of all integers\\[\rs]
    $\mathbb{R}: $ & the set of all real numbers\\[\rs]
    $\mathbb{C}: $ & the set of all complex numbers\\[\rs]
    $\mathbb{Z}_d: $ & the set of all integers modulo $d$, where $d$ is a positive integer \\[\rs]
    $\mathbb{R}^n: $ & the $n$-dimensional Euclidean plane, where $n$ is a positive integer\\[\rs]
    $\mathbb{C}^n: $ & the $n$-dimensional complex plane, where $n$ is a positive integer\\[\rs]
    $\mathbb{S}^n: $ & the unit sphere in $\mathbb{R}^{n+1}$ \\[\rs]
    $\mathbb{D}^n: $ & the unit disk in $\mathbb{R}^{n}$ \\[\rs]
    $\bar{X}: $ & the closure of the space $X$  \\[\rs]
    $M(n,\mathbb{R}): $ & the set of all $n\times n$ matrices \\[\rs]
    $GL(n,\mathbb{R}): $ & the set of all $n\times n$ invertible matrices \\[\rs]
    $O(n,\mathbb{R}): $ & the set of all $n\times n$ orthogonal matrices \\[\rs]
    $SO(n,\mathbb{R}): $ & the set of all $n\times n$ orthogonal matrices with determinant $1$ \\[\rs]
    $U(n): $ & the set of all $n\times n$ unitary matrices \\[\rs]
    $I_n: $ & identity matrix of size $n$ \\[\rs]
    $\trace{A}:$ & trace of a matrix $A$ \\[\rs]
    $A^T:$ & transpose of a matrix $A$ \\[\rs]
    $A^*:$ & conjugate transpose of a matrix $A$ \\[\rs]
    $T_pM:$ & tangent space of $M$ at $p\in M$ \\[\rs]
    $(T_pN)^\perp:$ & orthogonal complement of $T_pN$, where $N$ is a submanifold of $M$ and $p\in N$ \\[\rs]
    $TM:$ & tangent bundle of $M$  \\[\rs]
    $\nu :$ & normal bundle of $N$, where $N$ is a submanifold of $M$  \\[\rs]
    $S(\nu) :$ & unit sphere bundle of $\nu$ \\[\rs]
    $D(\nu) :$ & unit disk bundle of $\nu$  \\[\rs]
    $\exp_p:$ & Riemannian exponential map at $p$  \\[\rs]
    $\exp_\nu:$ & normal exponential map (see \eqref{eq:normalExponentialMap}) \\[\rs]
    $d(p,q):$ & the distance between points $p$ and $q$ \\[\rs]
    $d(A,B):$ & the distance between sets $A$ and $B$ \\[\rs]
    $\nabla f(\mathbf{\mathbf{x}}):$ & gradient of the function $f$ at $\mathbf{x}$   \\[\rs]
    $\pi_n(X): $ & $n^{\text{th}}$ homotopy group of the space $X$ \\[\rs]
    $\pi_n(X,A): $ & $n^{\text{th}}$ relative homotopy group of the pair of spaces $(X,A)$, where $A\subset X$ \\[\rs]
    $H_n(X):$ & $n^{\text{th}}$ homology of the space $X$ \\[\rs]
    $\tilde{H}_n(X):$ & $n^{\text{th}}$ reduced homology of the space $X$ \\[\rs]
    $H_n(X,A):$ & $n^{\text{th}}$ relative homology of the pair of spaces $(X,A)$, where $A\subset X$ \\[\rs]
    $H^n(X):$ & $n^{\text{th}}$ cohomology of the space $X$ \\[\rs]
    $\cutn[p]:$ & cut locus of the point $p$ (see \Cref{defn:cutLocusOfPoint}) \\[\rs]
    $\cutn[N]:$ & cut locus of the set $N$ (see \Cref{cutlocus1}) \\[\rs]
    $\mathrm{Se}(p):$ & separating set of a point $p$ (see \eqref{eq:SeSetofPoint}) \\[\rs]
    $\sen:$ & separating set of the set $N$ (see \Cref{defn:SeparatingSet}) \\[\rs]
    $A\star B:$ & topological join of $A$ and $B$ \\[\rs]
    $df_p:$ & the derivative of $f$ at $p$ \\[\rs]
    $\hess_p(f):$ & the Hessian of $f$ at $p$ (see \Cref{defn:Hessian}) \\[\rs]
\end{longtable}
	% Main matter
	\mainmatter

	\chapter{Introduction}\label{ch:introduction}

\hfb On a Riemannian manifold $M$, the distance function $d_N(\cdot) \defeq d(N,\cdot)$\index{$d_N$} from a closed subset $N$ is fundamental in the study of variational problems. For instance, the viscosity solution of the Hamilton-Jacobi equation is given by the flow of the gradient vector of the distance function $d_N$, when $N$ is the smooth boundary of a relatively compact domain in manifolds; see  \cite{LiNi05,MaMe03}. Although the distance function $d_N$ is not differentiable at $N$, squaring the function removes this issue. Associated to $N$ and the distance function $d_N$ is a set $\cutn$,\index{\cutn} the cut locus \index{cut locus} of $N$ in $M$. The cut locus of a point (submanifold) consists of all points such that a distance minimal geodesic (see \Cref{defn:cutLocusOfPoint,cutlocus1}) starting at the point (submanifold) fails its distance minimality property. The aim of the thesis is to explore the topological and geometric properties of cut locus of a submanifold.

\section{A survey of the cut locus}
\hfb This section is devoted to the literature survey and a discussion of some known results.

\vspace{0.2cm}
\hf Cut locus of a point, a notion initiated by Henri Poincar\'{e} \cite{Poin05}, has been extensively studied (see \cite{Kob67} for a survey as well as \cite{Buc77},  \cite{Mye35}, \cite{Sak96}, and \cite{Wol79}). Prior to Poincar\'{e} it had appeared implicitly in a paper \cite{Man81}. Other articles \cite{Whi35} and \cite{Mye35,Mye36} describe topological behavior of the cut locus. Due to its topological properties, it became an important tool in the field of Riemannian geometry or Finsler geometry. We list a few references like \cite{Kli59},  \cite{Rau59}, and \cite[Chapter 5]{ChEb75} for a detailed study of cut locus of a point. We also mention the work around the Blaschke conjecture which uses the geometry of the cut locus of a point, see \cite{Bes78,McK15}. A great source of reference for articles related to cut loci is \cite[\S 4]{Sak84}. Further, articles \cite{Sak77,Sak78,Sak79} and \cite{Tak78,Tak79} discussed cut loci in symmetric spaces. For questions on the triangulability of cut loci and differential topological aspects, see \cite{Buc77,GlSi76,GlSi78,Wall77}.

\vspace{0.2cm}
\hf Cut locus of submanifolds was first studied by Ren\'{e} Thom \cite{Thom72}. We mention some references for cut locus of submanifolds where it has been analyzed via the Eikonal equations and Hamilton-Jacobi equation, for example, see \cite{AnGu11,MaMe03} as well as analyzed via topological methods, for example, see \cite{Fla65,Ozo74,Singh87A,Singh87B,Singh88}.

\section{Overview of results}
\hfb Suitable simple examples indicate that $M-\cutn$ topologically deforms to $N$. One of our main results is the following (\Cref{thm: Morse-Bott}).
\begin{mainthm}\label{thm:ThmA}
    \textit{Let $N$ be a closed embedded submanifold of a complete Riemannian manifold $M$ and $d:M\to \mathbb{R}$ denote the distance function with respect to $N$. If $f=d^2$, then its restriction to $M-\mathrm{Cu}(N)$ is a Morse-Bott function, with $N$ as the critical submanifold. Moreover, $M-\mathrm{Cu}(N)$ deforms to $N$ via the gradient flow of $f$.}
\end{mainthm}

\vspace{0.2cm}
\noindent It is observed that this deformation takes infinite time. To obtain a strong deformation retract, one reparameterizes the flow lines to be defined over $[0,1]$. It can be shown (\Cref{defretM-N}) that the cut locus $\mathrm{Cu}(N)$ is a strong deformation retract of $M-N$. A primary motivation for \Cref{thm:ThmA} came from understanding the cut locus of $N=O(n,\rbb)$ inside $M= M(n,\rbb)$, equipped with the Euclidean metric. We show in \Cref{Sec:IlluminatingExample} that the cut locus is the set $\mathrm{Sing}$ of singular matrices and the deformation of its complement is not the Gram-Schmidt deformation but rather the deformation obtained from the polar decomposition, i.e., $A\in GL(n,\R)$ deforms to $A\big(\sqrt{A^T A}\,\big)^{-1}$. Combining with a result of J. J. Hebda \cite[Theorem 1.4]{Heb83} we are able to compute the local homology of $\mathrm{Sing}$ (cf \Cref{link-sing} and \Cref{locsinghom}).
\begin{mainthm}\label{thm:ThmB}
    \textit{For $A\in M(n,\R)$
    \begin{displaymath}
        H_{n^2-1-i}(\mathrm{Sing},\mathrm{Sing} -A;G)\cong \widetilde{H}^i(O(n-k,\R);G)
    \end{displaymath}
    where $A\in \mathrm{Sing}$ has rank $k<n$ and $G$ is any abelian group. }
\end{mainthm}

\vspace{0.2cm}
\noindent When the cut locus is empty, we deduce that $M$ is diffeomorphic to the normal bundle $\nu$ of $N$ in $M$. In particular, $M$ deforms to $N$. Among applications, we discuss two families of examples. We reprove the known fact that $GL(n,\R)$ deforms to $O(n,\R)$ for any choice of left-invariant metric on $GL(n,\R)$ which is right-$O(n,\R)$-invariant. However, this deformation is not obtained topologically but by Morse-Bott flows. For a natural choice of such a metric, this deformation \eqref{GLdefOver2} is not the Gram-Schmidt deformation, but one obtained from the polar decomposition. We also consider $U(p,q)$, the group preserving the indefinite form of signature $(p,q)$ on $\C^n$. We show (\Cref{mainthm}) that $U(p,q)$ deforms to $U(p)\times U(q)$ for the left-invariant metric given by $\left\langle X,Y\right\rangle:=\textup{tr}(X^\ast Y)$. In particular, we show that the exponential map is surjective for $U(p,q)$ (\Cref{expsurj}). To our knowledge, this method is different from the standard proof.

\vspace{0.3cm}
\hf For a Riemannian manifold we have the exponential map at $p\in M$, $\exp_p:T_pM\to M$. Let $\nu$ denote the normal bundle of $N$ in $M$. We will modify the exponential map (see \S \ref{Sec: Thom}) to define the \emph{rescaled exponential} $\widetilde{\exp}:D(\nu)\to M$, the domain of which is the unit disk bundle of $\nu$. The main result (\Cref{Thomsp}) here is the observation that there is a connection between the cut locus $\mathrm{Cu}(N)$ and Thom space $\mathrm{Th}(\nu):=D(\nu)/S(\nu)$ of $\nu$.

\begin{mainthm}\label{thm:ThmC}
    \textit{ Let $N$ be an embedded submanifold inside a closed, connected Riemannian manifold $M$. If $\nu$ denotes the normal bundle of $N$ in $M$, then there is a homeomorphism
    \begin{displaymath}
        \widetilde{\exp}:D(\nu)/S(\nu) \xrightarrow{\cong}M/\mathrm{Cu}(N).
    \end{displaymath}}
\end{mainthm}

\vspace{0.3cm}
\noindent This immediately leads to a long exact sequence in homology (see \eqref{lesThom})
\begin{displaymath}
	\cdots\to  H_j(\mathrm{Cu}(N)) \stackrel{i_*}{\longrightarrow}H_j(M)\stackrel{q}{\longrightarrow} \widetilde{H}_j(\mathrm{Th}(\nu))\stackrel{\partial}{\longrightarrow} H_{j-1}(\mathrm{Cu}(N))\to \cdots.
\end{displaymath}
This is a useful tool in characterizing the homotopy type of the cut locus. We list a few applications and related results.

\begin{mainthm}\label{thm:ThmD}
    \textit{Let $N$ be a homology $k$-sphere embedded in a Riemannian manifold $M^d$ homeomorphic to $S^d$.}
    \begin{enumerate}
        \item \textit{If $d\ge k+3$, then $\mathrm{Cu}(N)$ is homotopy equivalent to $S^{d-k-1}$. Moreover, if $M,N$ are real analytic and the embedding is real analytic, then $\cutn$ is a simplicial complex of dimension at most $d-1$.}
        \item \textit{If $d=k+2$, then $\cutn$ has the homology of $S^1$. There exists homology $3$-spheres in $S^5$ for which $\cutn\simeq S^1$. However, for non-trivial knots $K$ in $S^3$, the cut locus is not homotopy equivalent to $S^1$. }
    \end{enumerate}
\end{mainthm}

\vspace{0.3cm}
\noindent The above results are a combination of Theorem \ref{homsph}, Theorem \ref{Buchner} and Example \ref{codim2}. In general, the structure of the cut locus may be wild (see \cite{GlSi78}, \cite{ItSa16}, and \cite{ItVi15}). S. B. Myers \cite{Mye35} had shown that if $M$ is a real analytic sphere, then $\mathrm{Cu}(p)$ is a finite tree each of whose edge is an analytic curve with finite length. Buchner \cite{Buc77} later generalized this result to cut locus of a point in higher dimensional manifolds. \Cref{Buchner}, which states that the cut locus of an analytic submanifold (in an analytic manifold) is a simplicial complex, is a natural generalization of Buchner's result (and its proof). We attribute it to Buchner, although it is not present in the original paper. This analyticity assumption also helps us to compute the homotopy type of the cut locus of a finite set of points in any closed, orientable, real analytic surface of genus $g$ (\Cref{cutlocus-surface}). In \Cref{codim2} we make some observations about the cut locus of embedded homology spheres of codimension $2$. This includes the case of real analytic knots in the round sphere $\mathbb{S}^3$.

\vspace{0.3cm}
\hf Let $M$ be a closed Riemannian manifold and $G$ be any compact Lie group acting on $M$ freely. Then it is known that $M/G$ is a manifold. Further, if the action is isometric, then the metric on $M$ induces a metric on $M/G$. If $N$ is any $G$-invariant submanifold of $M$, then $N/G$ is a submanifold of $M/G$. If the action is isometric, then we provide an equality between $\cutn/G$ and $\cutn[N/G]$ (\Cref{thm:equivariant-cut-locus}).

\begin{mainthm}\label{thm:ThmE}
    Let $M$ be a closed and connected Riemannian manifold  and $G$ be any compact Lie group which acts on $M$ freely and isometrically. Let $N$ be any $G$-invariant closed submanifold of $M$, then we have an equality
    \begin{displaymath}
        \mathrm{Cu}(N)/G  = \mathrm{Cu}(N/G).
    \end{displaymath}
\end{mainthm}

\section{Outline of Chapter 2}
\hfb The majority of this chapter is an overview of recalling some basic results in Riemannian geometry and differential topology. This chapter also deals with some known results for cut locus of a point. Although this chapter may be interesting to read and help clarify the concepts, the experts can skip the details.

\subsection*{\S \ref{Sec:FermiCoordinates} Fermi coordinates}
\hfb Fermi coordinates \index{Fermi coordinates} are important for studying the geometry of submanifolds. In this coordinate system the metric is rectangular and the derivative of metric vanishes at each point of a curve. It makes the calculations much simpler. This section is devoted to recalling the construction of Fermi coordinates in a tubular neighborhood of a submanifold of a Riemannian manifold. This requires us to define the exponential map restricted to the normal bundle. We have recollected some results which will be used to study the distance squared function from a submanifold. For example, it is shown that the distance squared function from a submanifold is sum of squares of Fermi coordinates in a tubular neighborhood of the submanifold.

\subsection*{\S \ref{Sec:MorseBottFunctions} Morse-Bott theory}
\hfb In order to study the space via critical points of some real valued function on that space, Morse theory plays an important role. If non-degenerate critical points are replaced by non-degenerate critical submanifolds (see \Cref{defn:nonDegenerateCriticalSubmanifolds}), then a generalization of Morse theory comes into the picture -- Morse-Bott theory. In this section, we have recalled the definition of a Morse function and some examples of Morse functions. In \S\ref{subsec:MorseBottFunctions} we have discussed Morse-Bott theory motivated by an example.

\subsection*{\S \ref{Sec:cutLocusOfPoint} Cut locus and conjugate locus}
\hfb In a Riemannian manifold $M$ a geodesic $\gamma$ joining $p,q\in M$ is said to be \textit{distance minimal} if $l(\gamma)=d(p,q)$, where $d$ is the Riemannian distance. Cut locus of a point captures all points in $M$ beyond which geodesics fail to be distance minimal. In \S\ref{subsec:CutLocusOfAPoint} we have discussed numerous example of cut locus of a point. Characterizations of cut locus has been discussed in terms of conjugate points (points $p$ and $q$ are said to be conjugate along a geodesic $\gamma$ if there exists a non-vanishing Jacobi field vanishes at $p$ and $q$) and number of geodesics joining the two points (\Cref{thm:CharacterizationOfCutLocusInTermsOfConjugatePoint}). In particular, it says that a cut point is either the first conjugate point or there exists more than one geodesic joining the point and the cut point. We also have a characterization which shows the existence of a closed geodesic (\Cref{thm:ExistenseOfClosedGeodesic-2}). One of the result \cite[Theorem 1]{Wol79} is very important to find the cut points, which says that the cut locus of a point is the closure of points which can be joined by more than one geodesic (\Cref{thm:ClosureOfSeisCup}).

\section{Outline of Chapter 3}
\hfb This chapter serves as a motivation for the results of the subsequent chapters. It includes a detailed discussion of cut locus of submanifolds with numerous examples.

\subsection*{\S\ref{sec:cutLocusOfSubmanifolds} Cut locus of submanifolds}
\hfb To define cut locus of subset of a Riemannian manifold, one needs to define distance minimal geodesic starting from the subset. This section starts with defining the same (\Cref{distmin}) and then the cut locus of a subset is similarly. \Cref{eg:cutLocusOfkPoints} shows that the cut locus need not be a manifold. \Cref{join} shows that the topological join of $\mathbb{S}^k$ and $\mathbb{S}^{n-k-1}$ is induced from cut locus by showing that $\cutn[\mathbb{S}^k_i] = \mathbb{S}^{n-k-1}_l$, where $\mathbb{S}_i^k \hookrightarrow \mathbb{S}^n$ denote the embedding of the $k$-sphere in the first $k+1$ coordinates and $\mathbb{S}^{n-k-1}_l$ denote the embedding of the $(n-k-1)$-sphere in the last $n-k$ coordinates. In \S \ref{subsec:separatingSet} we have defined the separating set of a subset which consists of all points which have more than one distance minimal geodesic joining the subset. In \Cref{eg:CutLocusOfEllipse} we have shown that the cut locus is strictly bigger than the separating set. 

\subsection*{\S \ref{Sec:IlluminatingExample} An illuminating example}
\hfb The main aim of this section is to find the cut locus of the set of all $n\times n$ orthogonal matrices. We have shown that the cut locus is the set of all singular matrices by showing that it is the separating set. We also analyzed the regularity of distance squared function on the singular set and outside the singular set, set of all invertible matrices. In fact, we have shown that the distance squared function is differentiable at $A$ if and only if $A\in GL(n,\mathbb{R})$ In this section we have also shown that $GL(n,\mathbb{R})$ deforms to the set of all orthogonal matrices, but we noted that this deformation is  different from one we obtained via Gram-Schmidt. We will also prove \Cref{thm:ThmB}.

\section{Outline of Chapter 4}
\hfb This chapter is based on joint work with Basu \cite{BaPr21}. Here we have explored some topological properties (relation with the Thom space (\Cref{Thomsp}), homology and homotopy groups of cut locus) and geometric properties (regularity of the distance squared function \S\ref{sec:RegularityOfDistanceSquaredFunction}, complement of cut locus deforms to the submanifold (\Cref{thm: Morse-Bott})).

\subsection*{\S \ref{sec:RegularityOfDistanceSquaredFunction} Regularity of distance squared function}
\hfb This section is motivated by the example of cut locus of $O(n,\mathbb{R})$ in $M(n,\mathbb{R})$ (\S \ref{Sec:IlluminatingExample}). We proved that the distance squared function is not differentiable on the separating set (\Cref{Lmm: singdsq}). We have also shown by an example that the distance squared function can be differentiable on points which are cut points but not separating points (\Cref{eg:CutLocusOfEllipse-2}).

\subsection*{\S\ref{sec:characterizationOfCutLocus} Characterizations of \texorpdfstring{\cutn}{Cu(N)}}
\hfb We have discussed two characterizations of cut locus. One in terms of first focal points (\Cref{defn:focalPoint}) and number of geodesics joining the submanifold to the cut points (\Cref{thm:CharacterizationOfCutLocusInTermsOfFocalPoint}) and other is in terms of separating set (\Cref{thm:SeClosureIsCutLocus}). The latter one is important for computation viewpoint. Let $\nu$\index{$\nu$} denotes the normal bundle of $N$ and $S(\nu)$\index{$S(\nu)$} be the unit sphere bundle. Consider a map
\begin{gather*}
    \rho:S(\nu) \to [0,\infty),\\
    v\mapsto \sup\{t\in[0,\infty):\gamma_v|_{[0,t]} \text{ is a distance minimal geodesic from $N$}\}\index{$\rho$}
\end{gather*}
where $\gamma_v$ \index{$\gamma_v$} means $\gamma'(0)=v$ (also see \eqref{snu}).

\begin{thmSec}
    Let $u\in S(\nu)$. A positive real number $T$ is $\rho(u)$ if and only if $\gamma_u:[0,T]$ is a distance minimal geodesic from $N$ and at least one of the following holds:
    \begin{enumerate}[(i)]
        \item $\gamma_u(T)$ is the first focal point of $N$ along $\gamma_u$,
        \item there exists $v\in S(\nu)$ with $v\neq u$ such that $\gamma_v(T)=\gamma_u(T)$.
    \end{enumerate}
\end{thmSec}

\vspace{0.3cm}
\begin{thmSec}
    Let $\cutn$\index{\cutn} be the cut locus of a compact submanifold $N$  of a  complete Riemannian manifold $M$. The subset $\sen$,\index{\sen} the set of all points in $M$ which can be joined by at least two distance minimal geodesic starting from $N$,  of $\cutn$ is dense in $\cutn$.
\end{thmSec}

\subsection*{\S\ref{sec:topologicalProperties} Topological properties}
\hfb In this section we start by showing that the cut locus is a simplicial complex for an analytic pair (following Buchner \cite{Buc77}). In \S \ref{Sec: Thom} we prove \Cref{thm:ThmC} and discuss some applications including \Cref{thm:ThmD}. We end this section by proving one of the main theorem \Cref{thm:ThmA}.

\section{Outline of Chapter 5}
\hfb We apply our study of gradient of distance squared function to two families of Lie groups - $GL(n,\R)$ and $U(p,q)$. With a particular choice of left-invariant Riemannian metric which is right-invariant with respect to a maximally compact subgroup $K$, we analyze the geodesics and the cut locus of $K$. In both cases, we obtain that $G$ deforms to $K$ via Morse-Bott flow (\Cref{CartanGLn} and \Cref{mainthm}). Although these results are deducible from classical results of Cartan and Iwasawa, our method is geometric and specific to suitable choices of Riemannian metrics. It also makes very little use of structure theory of Lie algebras. 

\section{Outline of Chapter 6}
\hfb Consider a Riemannian manifold $M$ on which a compact Lie group $G$ acts freely. It is well known that the quotient $M/G$ is a manifold. This chapter is devoted to the study of cut locus of a $G$-invariant submanifold $N$ inside $M$. We will prove \Cref{thm:ThmE}. As an application of \Cref{thm:ThmE}, we have shown some examples of cut locus in orbit space. We also discuss an application to complex hypersurfaces. Let $\pi:\mathbb{S}^{2n+1}\to \mathbb{CP}^n$ be the quotient map. If 
\begin{displaymath}
    X(d)=\Bigg\{[z_0:z_1:\cdots:z_n]\in \mathbb{CP}^n:\sum_{i=0}^n z_i^d=0\Bigg\}\index{$X(d)$}
\end{displaymath}
and $\tilde{X}(d)\defeq \pi^{-1}(X(d))$, then we make the following conjecture. 
\begin{conj}\label{thm:cut-locus-of_X(d)}
	The cut locus of $\tilde{X}(d)\subseteq \sbb^{2n+1}$ is $\zbb_d^{\star(n+1)}\times_{\zbb_d}\sbb^1$, where $\times_{\mathbb{Z}_d}$ is the diagonal action of $\mathbb{Z}_d$ and $\star$ denotes the topological join of spaces.  
\end{conj}

\vspace{0.3cm}
\noindent We prove the above conjecture for two families:  $d=2, n$ arbitrary (\Cref{thm:cut-locus-of_Xn_2}) and $n=1, d$ arbitrary (\Cref{thm:cut-locus-for_X1_d}). 
	\chapter[Preliminaries]{Preliminaries} \label{ch:preliminaries}
\setcounter{mtc}{6}
\minitoc
\section{Fermi coordinates}\label{Sec:FermiCoordinates}\index{Fermi coordinates}
\hfb In this section we give a brief overview of the Fermi coordinates which are generalizations of normal coordinates in Riemannian geometry. To study the distance squared function from a submanifold $N$ of a Riemannian manifold $M$, it is essential to analyze the local geometry of $M$ around $N$. For this the Fermi coordinates are the most convenient tool. In 1922, Enrico Fermi \cite{Fer22} came up  with a coordinate system in which the Christoffel symbols vanish along geodesics which makes the metric simpler. For an extensive reading we refer to the book \cite[Chapter 2]{Gr04} and an article \cite{MaMi63}.

\subsection{Normal exponential map} \index{normal exponential map}
\hfb Let $N$ be an embedded submanifold of a Riemannian manifold $M$. We define the \textit{normal bundle},\index{normal bundle} denoted by $\nu$,
\begin{displaymath}
    \nu \defeq \left\{(p,v):p\in N\text{ and } v\in \left(T_pN\right)^\perp\right\}, \index{$\nu$}
\end{displaymath}
where $(T_pN)^\perp$ is the orthogonal complement of $T_pN$. Indeed, $\nu$ is a subbundle of the restriction of the tangent bundle $TM$ to $N$. We can restrict the usual exponential map of the Riemannian manifold to the normal bundle to define the exponential map of the normal bundle. We define the \textit{exponential map of the normal bundle} as follows:\index{$\exp_{\nu}$}
\begin{equation}\label{eq:normalExponentialMap}
    \exp_\nu:\nu\to M,~(p,v)\mapsto \exp_p(v),
\end{equation} 
where $\exp_p:T_pM\to M$ is the exponential map of $M$. We may write $\exp_\nu(v)$ in short and call this the \textit{normal exponential map}. Note that we can identify $N$ as the zero section of the normal bundle and hence $N$ can be assumed to be submanifold of $\nu$. 
\begin{lemma}\cite[Lemma 2.3]{Gr04}
    Let $M$ be a Riemannian manifold and $N$ be any embedded submanifold. Then the normal exponential map $\exp_\nu:\nu\to M$ is a diffeomorphism from a neighbourhood of $N\subseteq \nu$ onto a neighbourhood of $N\subseteq M$. 
\end{lemma}
\bigskip
Using the above lemma, let $\mathcal{U}_N$ be the largest open neighbourhood of $N\subseteq \nu$ for which $\exp_\nu$ is a diffeomorphism. We shall later be able to describe this neighbourhood in terms of a function $\rho$ \eqref{snu}. We now ready to define the Fermi coordinates. 
\subsection{Fermi coordinate system} \index{Fermi coordinates}
\hfb To define a system of Fermi coordinates, we need an arbitrary system of coordinates $\left(y_1,\cdots,y_k\right)$ defined in a neighborhood $\mathcal{O}\subseteq N$ of $p\in N$ together with orthogonal sections $\mathcal{E}_{k+1},\cdots,\mathcal{E}_n$ of the restriction on $\nu$ to $\mathcal{O}$.
\begin{defn}[Fermi coordinates]\label{Defn:FermiCoordinates}
    The \textit{Fermi coordinates} $\left(x_1,\cdots,x_n\right)$ of $N\subseteq M$ centered at $p$ (relative to a given coordinate $\left(y_1,\cdots,y_k\right)$ on $N$ and given orthogonal sections $\mathcal{E}_{k+1},\cdots,\mathcal{E}_n$ of $\nu$) are defined by
    \begin{align*}
        & x_l \left(\exp_{\nu} \bigg(\sum_{i=k+1}^n \tau_i \mathcal{E}_i \left(p'\right)\bigg)\right) = y_l \left(p'\right),~ l = 1,\cdots,k 
        \\[1ex]
        & x_m \left(\exp_{\nu} \bigg(\sum_{i=k+1}^n \tau_i \mathcal{E}_i \left(p'\right)\bigg)\right) = \tau_m,~ m = k+1,\cdots,n
    \end{align*}
    for $p'\in \mathcal{O}$ provided the numbers $\tau_{k+1},\cdots,\tau_n$ are small enough so that $\tau_{k+1}\mathcal{E}_{k+1}\left(p'\right)+\cdots+\tau_{n}\mathcal{E}_{n}\left(p'\right)\in \mathcal{U}_N$. 
\end{defn}

\bigskip 

\hf As the normal exponential map is a diffeomorphism on the set $~\mathcal{U}_N$, $\left(x_1,\cdots,x_k,x_{k+1}\right.$ $\left.,\cdots, x_n\right)$ defines a coordinate system near $p$. In fact, the restrictions to $N$ of coordinate vector fields $ \allowbreak\partial/\partial x_{k+1},\ldots, \partial/ \partial x_n$ are orthonormal.
\begin{lemma} \label{Lemma: geodesic in fermi}
    Let $\gamma$ be  a unit speed geodesic normal to $N$ with $\gamma(0) = p\in N$. If $v = \gamma'(0)$, then there exists a system of Fermi coordinates $(x_1,\cdots,x_n)$ such that whenever $(p,tv)\in \mathcal{U}_N$, we have
    \begin{align*}
        & \kern 1cm\left.\delbydel{}{x_{k+1}}\right|_{\gamma(t)} = \gamma'(t),\\
        & \left.\delbydel{}{x_l}\right|_p\in T_pN, \text{ and } \left.\delbydel{}{x_i}\right|_p \in (T_p N)^\perp
    \end{align*}
    for $1\le l\le k$ and $k+1\le i\le n.$ Furthermore, for $1\le j \le n$
    \begin{displaymath}
        (x_j \comp \gamma)(t) = t\delta_{j (k+1)}.
    \end{displaymath}
\end{lemma}

\bigskip 

\hf The following object will be useful while studying the distance squared function from a submanifold $N$.
\begin{defn}\label{defn:DeltaMap}
    Let $N$ be a submanifold of a Riemannian manifold and let $(x_1,\cdots,x_n)$ be a system of Fermi coordinates for $N$. We define $\Delta(x_1,\cdots,x_n)$ to be the non-negative number satisfying
    \begin{displaymath}
        \Delta^2 = \sum_{i=k+1}^n x_i^2.
    \end{displaymath}
\end{defn}

% \bigskip 

% \hf The function $\Delta$ is independent of the choice of Fermi coordinate system.
\begin{lemma}
    Let $p\in N$. The $\Delta$ is independent of the choice of Fermi coordinates at $p$.
\end{lemma}
% \vspace{-0.7cm}
\begin{proof}
    Let $(x_1',\cdots,x_n')$ be another system of Fermi coordinates at $p$, and let $\curlybracket{\mathcal{E}_{k+1}', \cdots, \mathcal{E}_n'}$ be the orthonormal sections of $\nu$ that give rise to it. We can write
    \begin{displaymath}
        \mathcal{E}_j' = \sum_{i=k+1}^na_{ij}\mathcal{E}_i
    \end{displaymath}
    where $(a_{ij})$ is a matrix of functions in the orthogonal group $O(n-k)$ with each $a_{ji}$ being a smooth function on $N.$ Now,
    \begin{align*}
        x_m\paran{\exp_\nu\paran{\sum_{j=k+1}^n \tau_j'\textcolor{blue}{\mathcal{E}_j'}}} & = x_m \paran{\exp_\nu\paran{ \sum_{j=k+1}^n\tau_j' \textcolor{blue}{\sum_{i=k+1}^na_{ij}\mathcal{E}_i}}} 
        \\[1ex]
        & = x_m\paran{\exp_\nu \bigg(\sum_{i=k+1}^n \bigg(\sum_{j=k+1}^na_{ij}\tau_j'\bigg) \mathcal{E}_i\bigg)}
        \\[1ex]
        & = \sum_{l=k+1}^n a_{ml}\textcolor[HTML]{06a3b3}{\tau_l'} 
        \\[1ex]
        & = \sum_{l=k+1}^n a_{ml}\textcolor[HTML]{06a3b3}{x_l' \left(\exp_\nu \bigg(\sum_{j=k+1}^n\tau_j'\mathcal{E}_j'\bigg)\right)}.
    \end{align*}
    Therefore, we have 
    \begin{equation}
        x_m = \sum_{l=k+1}^na_{ml}x_l',~m=k+1,\cdots,n.
    \end{equation}
    Now consider,
    \begin{align*}
        \sum_{m=k+1}^n x_m^2 & = \sum_{m=k+1}^n\paran{\sum_{l=k+1}^na_{ml}x_l'}^2 
        \\[1ex]
        & = \sum_{m=k+1}^n\paran{\sum_{l=k+1}^n\sum_{j=k+1}^n (a_{ml}x_l') (a_{mj}x_j')} 
        \\[1ex]
        & = \sum_{l=k+1}^n\sum_{j=k+1}^n \left(\sum_{m=k+1}^na_{ml}a_{mj}\right)x_l'x_j'\\[1ex]
        & = \sum_{l=k+1}^n\sum_{j=k+1}^n\delta_{lj}x_l'x_j' \\[1ex]
        & = \sum_{m=k+1}^n \left(x_m'\right)^2.
    \end{align*}
\end{proof}

\section{Morse-Bott theory} \label{Sec:MorseBottFunctions}\index{Morse-Bott functions}
\hfb This section will be devoted to a generalization of Morse function in which we study the space by looking at the critical points of a smooth real valued function. We will briefly recall Morse functions with a couple of examples, and then we will define Morse-Bott functions. The reference for this section will be the original article by Raoul Bott \cite{Bot54} and the book \cite[Section 3.5]{BaHu04}.

\subsection{Morse functions}
\hfb Broadly the ``functions'' and ``spaces'' are objects of study in analysis and geometry respectively. However, these two objects are related to each other. For example, on a line we can have functions like $f(x)=x,~g(x)=x^2$ which takes arbitrarily large values, whereas on the circle there does not exist any function which takes arbitrarily large value. In this way, we are able to differentiate circles with lines by seeing functions on them. Morse theory studies relations between shape of space and function defined on this space. We study the critical points of a function defined on spaces to find out information on the space. More specifically, in Morse theory we study the topology of smooth manifolds by analyzing the critical point of a smooth real valued function. If $f:M\to\mathbb{R}$ is a smooth function on a smooth manifold $M$, then using Morse theory we can find a CW-complex which is homotopy equivalent to $M$ and the CW-complex has one cell for each critical point of $f$. For a detailed study of Morse theory we refer to the book \cite{Mil63} by John Milnor.
\begin{defn}[Critical Points] \label{defn:CriticalPoints}\index{critical points}
    Let $M$ and $N$ be two smooth manifolds of dimension $m$ and $n$ respectively. A point $p\in M$ is said to be \emph{critical point} of a smooth function $f:M\to N$ if the differential map
    \begin{displaymath}
        df_p:T_pM\to T_{f(p)}N\index{$df_p$}
    \end{displaymath}
    does not have full rank.
\end{defn}

\bigskip 

\noindent We confine our study to real-valued functions. In this case the above is equivalent to $df_p\equiv 0$. In a coordinate neighborhood $(\phi=(x_1,x_2,\ldots,x_n),U)$ around $p$, we have 
\begin{equation}\label{eq:LocalExpressionOfCriticalPoint}
    \delbydel{(f\comp \phi^{-1})}{x_j}(\phi(p))=0,~j=1,\cdots,n.
\end{equation} At a critical point of $f:M\to \mathbb{R}$, we define the Hessian which is similar to the second derivative of the function. 
\begin{defn}[Hessian of $f$ at $p$]\label{defn:Hessian}\index{Hessian}
    Let $f:M\to \mathbb{R}$ be any smooth real valued function and $p$ be any critical point of $f$. The \emph{Hessian of $f$ at $p$} is the map
    \begin{equation}\label{eq:LocalExpressionOfHessianMap}
        \hess_p(f):T_pM \times T_pM\to \mathbb{R},~\hess_p(f)(V,W)=\tilde{V}\cdot \left(\tilde{W}\cdot f\right)(p),\index{$\hess_p(f)$}
    \end{equation}
    where $\tilde{V}$ and $\tilde{W}$ are any extensions of $V$ and $W$ respectively.
\end{defn}

\bigskip

\noindent Note that the Hessian is a bilinear form of $V$ and $W$. Consider
\begin{align*}
    V\cdot \left(\tilde{W}\cdot f\right)(p) - W\cdot \left(\tilde{V}\cdot f\right)(p)  & = \left[\tilde{V}, \tilde{W}\right]_p\cdot f \\
    & = df_p \left(\left[\tilde{V},\tilde{W}\right]_p\right) \\
    & = 0.
\end{align*}
Thus, Hessian is a symmetric bilinear form on $T_pM\times T_pM$. The above computation, in particular, also proves that the definition is well defined, that is, it is  independent of the choice of extension. 

\bigskip 

\hf Any critical point is categorized by looking at the value of Hessian at that point.

\begin{defn}\label{defn:nonDegenerateCriticalPoints} \index{non-degenerate critical point}
    A critical point $p\in M$ of a smooth function $f:M\to \mathbb{R}$ is said to be \emph{non-degenerate} if the Hessian is non-degenerate. Otherwise, we call $p$ to be a \emph{degenerate} critical point. \index{degenerate critical point} The \emph{index}\index{index} of a non-degenerate critical point $p$ is the dimension of the subspace of the maximum dimension on which $\hess_pf$ is negative definite.
\end{defn}

\bigskip

\noindent For example, the function $f:\mathbb{R}\to \mathbb{R} ~x\mapsto x^2$ has $0$ a critical point which is non-degenerate but $0$ is the degenerate critical point of the function $f(x)=x^3$.

\begin{defn}\label{defn:MorseFunction}\index{Morse function}
    A smooth function $f:M\to \mathbb{R}$ is said to be a \textit{Morse function}  if all its critical points are non-degenerate.
\end{defn}

\begin{eg}
    The function 
    \begin{displaymath}
        f:\mathbb{R}^2\to \mathbb{R},~(x,y) \mapsto x^2-3xy^2
    \end{displaymath}
    is not a Morse function, as the critical point $(0,0)$ is not non-degenerate.
\end{eg}

\begin{eg}[Height function on sphere]\label{eg:HeightFunctionOnSphere}\index{height function of sphere}
    The height function on the $n$-sphere is a Morse function with critical points $N=(0,0,\cdots, 1)$ and $S=(0,0,\cdots,-1)$. The index of $N$ and $S$ is $n$ and $0$ respectively. 
    \begin{figure}[H]
        \centering
    \def\svgwidth{0.5\columnwidth}
    %% Creator: Inkscape 1.2 (1:1.2.1+202207142221+cd75a1ee6d), www.inkscape.org
%% Accompanies image file 'sphereMorseFunction.pdf' (pdf, eps, ps)
\begingroup%
  \makeatletter%
  \providecommand\color[2][]{%
    \errmessage{(Inkscape) Color is used for the text in Inkscape, but the package 'color.sty' is not loaded}%
    \renewcommand\color[2][]{}%
  }%
  \providecommand\transparent[1]{%
    \errmessage{(Inkscape) Transparency is used (non-zero) for the text in Inkscape, but the package 'transparent.sty' is not loaded}%
    \renewcommand\transparent[1]{}%
  }%
  \providecommand\rotatebox[2]{#2}%
  \newcommand*\fsize{\dimexpr\f@size pt\relax}%
  \newcommand*\lineheight[1]{\fontsize{\fsize}{#1\fsize}\selectfont}%
  \ifx\svgwidth\undefined%
    \setlength{\unitlength}{499.69775273bp}%
    \ifx\svgscale\undefined%
      \relax%
    \else%
      \setlength{\unitlength}{\unitlength * \real{\svgscale}}%
    \fi%
  \else%
    \setlength{\unitlength}{\svgwidth}%
  \fi%
  \global\let\svgwidth\undefined%
  \global\let\svgscale\undefined%
  \makeatother%
  \begin{picture}(1,0.68948658)%
    \lineheight{1}%
    \setlength\tabcolsep{0pt}%
    \put(0,0){\includegraphics[width=\unitlength,page=1]{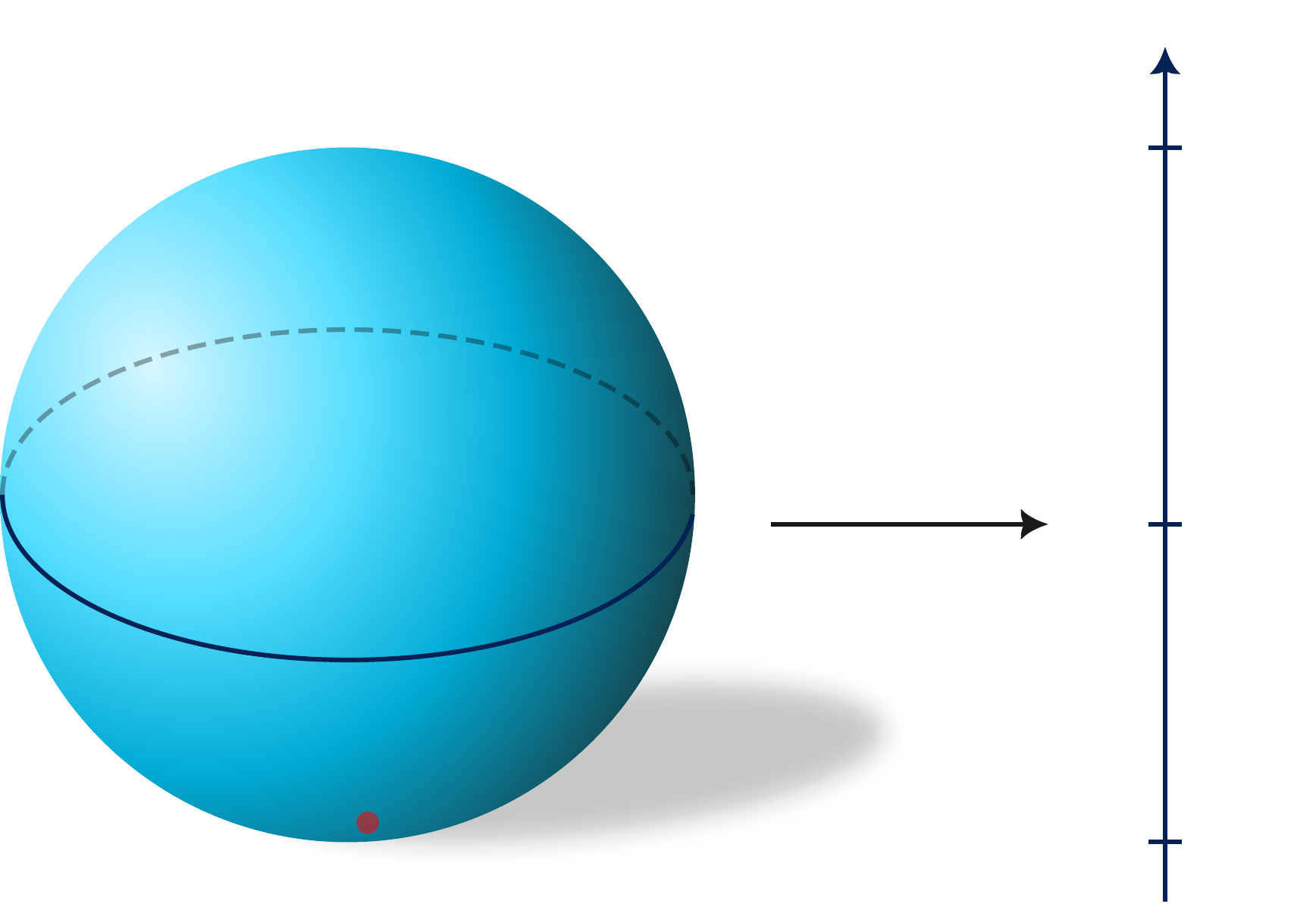}}%
    \put(0.68196421,0.31063587){\makebox(0,0)[lt]{\lineheight{1.25}\smash{\begin{tabular}[t]{l}$f$\end{tabular}}}}%
    \put(0.9062871,0.03642934){\makebox(0,0)[lt]{\lineheight{1.25}\smash{\begin{tabular}[t]{l}$-1$\end{tabular}}}}%
    \put(0.9062871,0.27630601){\makebox(0,0)[lt]{\lineheight{1.25}\smash{\begin{tabular}[t]{l}$0$\end{tabular}}}}%
    \put(0.9062871,0.56291761){\makebox(0,0)[lt]{\lineheight{1.25}\smash{\begin{tabular}[t]{l}$1$\end{tabular}}}}%
    \put(0.89416194,0.65412145){\makebox(0,0)[lt]{\lineheight{1.25}\smash{\begin{tabular}[t]{l}$z$\end{tabular}}}}%
    \put(0,0){\includegraphics[width=\unitlength,page=2]{sphereMorseFunction.pdf}}%
    \put(0.25200771,0.58212091){\makebox(0,0)[lt]{\lineheight{1.25}\smash{\begin{tabular}[t]{l}$N$\end{tabular}}}}%
    \put(0.25350986,0.00651957){\makebox(0,0)[lt]{\lineheight{1.25}\smash{\begin{tabular}[t]{l}$S$\end{tabular}}}}%
  \end{picture}%
\endgroup%

        \caption{Height function on the $2$-sphere is a Morse function with two non-degenerate critical points with index $2$ and $0$.}
        \label{fig:HeightFunctionOnSphere}
    \end{figure}
    
    For this, let
    \begin{align*}
        & \phi_1:\mathbb{S}^n
        \setminus\{N\}\to \mathbb{R}^n,~ \left(x_1,\cdots,x_{n+1}\right)\mapsto \left(\dfrac{x_1 }{1-x_{n+1}},\cdots,\dfrac{x_n}{1-x_{n+1}}\right),\text{ and } \\[1ex]
        & \phi_2:\mathbb{S}^n\setminus\{S\}\to \mathbb{R}^n,~ \left(x_1,\cdots,x_{n+1}\right)\mapsto \left(\dfrac{x_1 }{1+x_{n+1}},\cdots,\dfrac{x_n}{1+x_{n+1}}\right)
    \end{align*} 
    be two charts of $\mathbb{S}^n$. The inverse is given by 
    \begin{align*}
        & \phi_1^{-1}(\mathbf{y}) = \left(\dfrac{2y_1}{\left\|\mathbf{y}\right\|^2+1}, \cdots, \dfrac{2y_n}{\left\|\mathbf{y}\right\|^2+1},\dfrac{\left\|\mathbf{y}\right\|^2-1}{\left\|\mathbf{y}\right\|^2+1}\right) \\[1ex]
        & \phi_1^{-1}(\mathbf{y}) = \left(\dfrac{2y_1}{\left\|\mathbf{y}\right\|^2+1}, \cdots, \dfrac{2y_n}{\left\|\mathbf{y}\right\|^2+1},-\dfrac{\left\|\mathbf{y}\right\|^2-1}{\left\|\mathbf{y}\right\|^2+1}\right).
    \end{align*} 
    From \cref{eq:LocalExpressionOfCriticalPoint}, the critical points of $f$ will be the critical points of $\psi_i=f\circ \phi_i^{-1}:\mathbb{R}^n\to \mathbb{R},~i=1,2$. Note that 
    \begin{align*}
        & \psi_1(\mathbf{x}) = \dfrac{\left\|\mathbf{x}\right\|^2-1}{\left\|\mathbf{x}\right\|^2+1},~\text{ and } \psi_2(\mathbf{x}) = -\dfrac{\left\|\mathbf{x}\right\|^2-1}{\left\|\mathbf{x}\right\|^2+1},~\mathbf{x}\in \mathbb{R}^n. \\[1ex]
        \implies & \left(d\psi_1\right)_\mathbf{x} = \frac{4 \mathbf{x}}{\left(\left\|\mathbf{x}\right\|+1\right)^2},\text{ and } \left(d\psi_2\right)_\mathbf{x} = -\frac{4 \mathbf{x}}{\left(\left\|\mathbf{x}\right\|+1\right)^2}.
    \end{align*}
    Therefore, the critical points are $\phi_1^{-1}(\mathbf{0})=S$ in $\mathbb{S}^n\setminus \{N\}$  and $\phi_2^{-1}(\mathbf{0})=N$ in $\mathbb{S}^n\setminus \{N\}$. Note that
    \begin{align*}
        & \hess_S(f) = \left(\delbydel{^2 \psi_1}{x_i\partial x_j}(\mathbf{0})\right)_{1\le i,j\le n} = 4 I_{n\times n}, \text{ and } \\[1ex]
        & \hess_N(f) = \left(\delbydel{^2 \psi_1}{x_i\partial x_j}(\mathbf{0})\right)_{1\le i,j\le n} = -4 I_{n\times n}.
    \end{align*}
    Hence, both critical points are non-degenerate and index of $N$ and $S$ is $n$ and $0$ respectively.
\end{eg}

\begin{eg}[Height function on torus]\label{eg:HeightFunctionOnTorus}\index{height function of torus}
    If $a$ and $b$ be two positive real numbers with $0<b<a$, then the torus is
    \begin{displaymath}
        \mathbb{T} \eqdef \left\{(x,y,z):x^2+\left(\sqrt{y^2+z^2}-a\right)^2=b^2\right\}.
    \end{displaymath} 
    The function 
    \begin{displaymath}
        f:\mathbb{T}\to \mathbb{R},~(x,y,z)\mapsto z
    \end{displaymath}
    is a Morse function with critical points $(0,0,\pm(a+b))$ and $(0,0,\pm(a-b))$.
    \begin{figure}[H]
        \centering
    \def\svgwidth{0.45\columnwidth}
    %% Creator: Inkscape 1.2 (1:1.2.1+202207142221+cd75a1ee6d), www.inkscape.org
%% Accompanies image file 'torusMorseFunction.pdf' (pdf, eps, ps)
\begingroup%
  \makeatletter%
  \providecommand\color[2][]{%
    \errmessage{(Inkscape) Color is used for the text in Inkscape, but the package 'color.sty' is not loaded}%
    \renewcommand\color[2][]{}%
  }%
  \providecommand\transparent[1]{%
    \errmessage{(Inkscape) Transparency is used (non-zero) for the text in Inkscape, but the package 'transparent.sty' is not loaded}%
    \renewcommand\transparent[1]{}%
  }%
  \providecommand\rotatebox[2]{#2}%
  \newcommand*\fsize{\dimexpr\f@size pt\relax}%
  \newcommand*\lineheight[1]{\fontsize{\fsize}{#1\fsize}\selectfont}%
  \ifx\svgwidth\undefined%
    \setlength{\unitlength}{237.5006464bp}%
    \ifx\svgscale\undefined%
      \relax%
    \else%
      \setlength{\unitlength}{\unitlength * \real{\svgscale}}%
    \fi%
  \else%
    \setlength{\unitlength}{\svgwidth}%
  \fi%
  \global\let\svgwidth\undefined%
  \global\let\svgscale\undefined%
  \makeatother%
  \begin{picture}(1,0.91657516)%
    \lineheight{1}%
    \setlength\tabcolsep{0pt}%
    \put(0,0){\includegraphics[width=\unitlength,page=1]{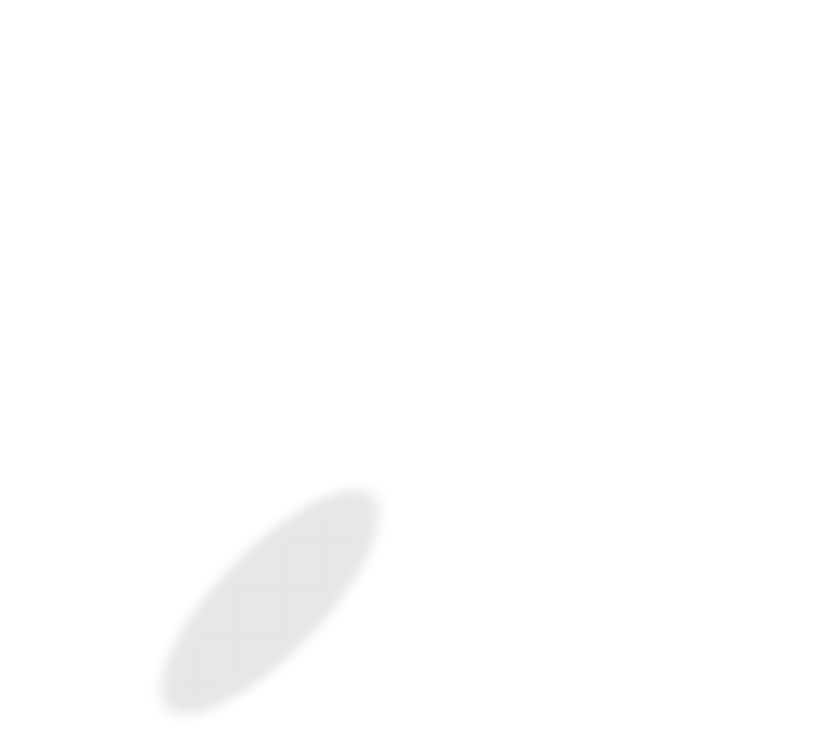}}%
    \put(0.62708109,0.40454946){\color[rgb]{0,0,0}\makebox(0,0)[lt]{\lineheight{1.25}\smash{\begin{tabular}[t]{l}$f$\end{tabular}}}}%
    \put(0,0){\includegraphics[width=\unitlength,page=2]{torusMorseFunction.pdf}}%
    \put(0.82137569,0.84216746){\makebox(0,0)[lt]{\lineheight{1.25}\smash{\begin{tabular}[t]{l}$z$\end{tabular}}}}%
    \put(0,0){\includegraphics[width=\unitlength,page=3]{torusMorseFunction.pdf}}%
  \end{picture}%
\endgroup%

        \caption{Height function on torus is a Morse function with four non-degenerate critical points}
        \label{fig:HeightFunctionOnTorus}
    \end{figure}
\end{eg}

\subsection{Morse-Bott functions}\label{subsec:MorseBottFunctions}
\hfb Morse-Bott functions are generalizations of Morse functions where we are allowed to have critical set need not be isolated but may form a submanifold. For example, let a torus be kept horizontally (a donut is kept in a plate). If $f$ is the height function on the torus, then there are two critical submanifolds, the top and bottom circles.

\vspace{0.1cm}
\hf Let $M$ be a Riemannian manifold and $f$ be any real valued smooth function on $M$. Let $\crf$ denotes the set of all critical points of $f$ and $N$ be any submanifold of $M$  which is contained in $\crf$. For any point $p\in M$ we have the following decomposition:
\begin{displaymath}
    T_pM = T_pN \directsum \nu_pN,
\end{displaymath}
where $\nu_pN$ is the normal bundle at $p$. Note that if $p\in N$ then for any $V\in T_pN$ and $W\in T_pM$ the Hessian vanishes, i.e., $\hess_p(f)(V,W)=0$. Therefore, $\hess_p(f)$ induces a symmetric bilinear form on $\nu_pN$. Now we can define non-degenerate critical submanifold similar to the non-degenerate critical points. 
\begin{defn}[Non-degenerate critical submanifold] \label{defn:nonDegenerateCriticalSubmanifolds}
    Let $N\subset M$ be a submanifold of a Riemannian manifold $M$. Then $N$ is said to be \emph{non-degenerate critical submanifold} of $f$ if $N\subseteq\crf$ and for any $p\in N$ the Hessian, $\hess_p(f)$ is non-degenerate in the direction normal to $N$ at $p$.
\end{defn}

\vspace{0.1cm}
\hf In the above definition, by $\hess_p(f)$ is non-degenerate in the direction normal to $N$ at $p$ we mean that for any $V\in \nu_pN$ there exists $W\in \nu_pN$ such that $\hess_p(f)(V,W)\neq 0$.

\begin{defn}[Morse-Bott functions] \label{defn:MorseBottFunction} 
    The function $f:M\to \mathbb{R}$ is said to be \emph{Morse-Bott} if the connected components of $\crf$ are non-degenerate critical submanifolds.
\end{defn}

% \subsection*{Examples of Morse-Bott functions}
\begin{eg}
    Let $f:M\to \mathbb{R}$ be a Morse function. Then the critical submanifolds are zero-dimensional and hence the Hessian $\hess_p(f)$  at any critical point $p$ is non-degenerate in every direction as all the directions are normal. So $f$ is Morse-Bott with critical submanifolds as critical points.
\end{eg}

\begin{eg}
    Any constant function defined on a smooth manifold $M$ is a Morse-Bott function with critical submanifold $M$.
\end{eg}

\begin{eg}
    Let $M=\mathbb{R}^2$. Define
    \begin{displaymath}
        f:M\to \mathbb{R},~(x,y)\mapsto x^4.
    \end{displaymath}
    Then the derivative map 
    \begin{displaymath}
        df_{(x,y)} = \left(4x^3,0\right)=(0,0) \implies x=0.
    \end{displaymath}
    Thus, the critical set is $\{(x,y):x=0\}$ which is $y$-axis. Now the Hessian at $(0,y)$ will be
    \begin{align*}
        \hess_{(0,y)}(f) = 
        \begin{pmatrix}
            0 & 0 \\ 
            0 & 0
        \end{pmatrix},
    \end{align*}
    which is degenerate in every direction and hence it is not a Morse-Bott function.
\end{eg}

\begin{eg}\label{eg:MorseBottDistanceSquaredFromLine}
    Let $M=\mathbb{R}^2$ with the Euclidean distance $\dist$  and $N=\{(x,x):x\in \mathbb{R}\}$. Consider the function 
    \begin{displaymath}
        f:M\to \mathbb{R},~(x,y)\mapsto \dist^2((x,y),N)= \dfrac{(x-y)^2}{2}.
    \end{displaymath} 
    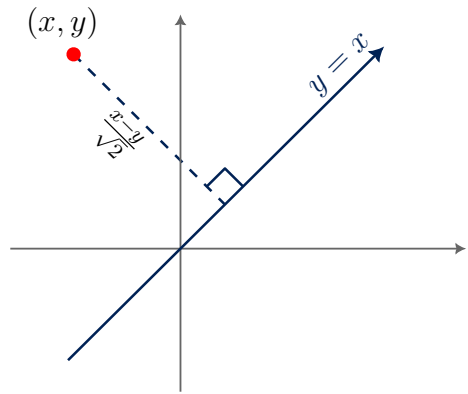
\begin{figure}[H]
        \centering
    \def\svgwidth{0.4\columnwidth}
    %% Creator: Inkscape 1.2 (1:1.2.1+202207142221+cd75a1ee6d), www.inkscape.org
%% Accompanies image file 'Example-MorseBott-DistanceSquaredFromX-Axis.pdf' (pdf, eps, ps)
\begingroup%
  \makeatletter%
  \providecommand\color[2][]{%
    \errmessage{(Inkscape) Color is used for the text in Inkscape, but the package 'color.sty' is not loaded}%
    \renewcommand\color[2][]{}%
  }%
  \providecommand\transparent[1]{%
    \errmessage{(Inkscape) Transparency is used (non-zero) for the text in Inkscape, but the package 'transparent.sty' is not loaded}%
    \renewcommand\transparent[1]{}%
  }%
  \providecommand\rotatebox[2]{#2}%
  \newcommand*\fsize{\dimexpr\f@size pt\relax}%
  \newcommand*\lineheight[1]{\fontsize{\fsize}{#1\fsize}\selectfont}%
  \ifx\svgwidth\undefined%
    \setlength{\unitlength}{315.13986585bp}%
    \ifx\svgscale\undefined%
      \relax%
    \else%
      \setlength{\unitlength}{\unitlength * \real{\svgscale}}%
    \fi%
  \else%
    \setlength{\unitlength}{\svgwidth}%
  \fi%
  \global\let\svgwidth\undefined%
  \global\let\svgscale\undefined%
  \makeatother%
  \begin{picture}(1,0.8533331)%
    \lineheight{1}%
    \setlength\tabcolsep{0pt}%
    \put(0.03456115,0.79710804){\makebox(0,0)[lt]{\lineheight{1.25}\smash{\begin{tabular}[t]{l}$(x,y)$\end{tabular}}}}%
    \put(0,0){\includegraphics[width=\unitlength,page=1]{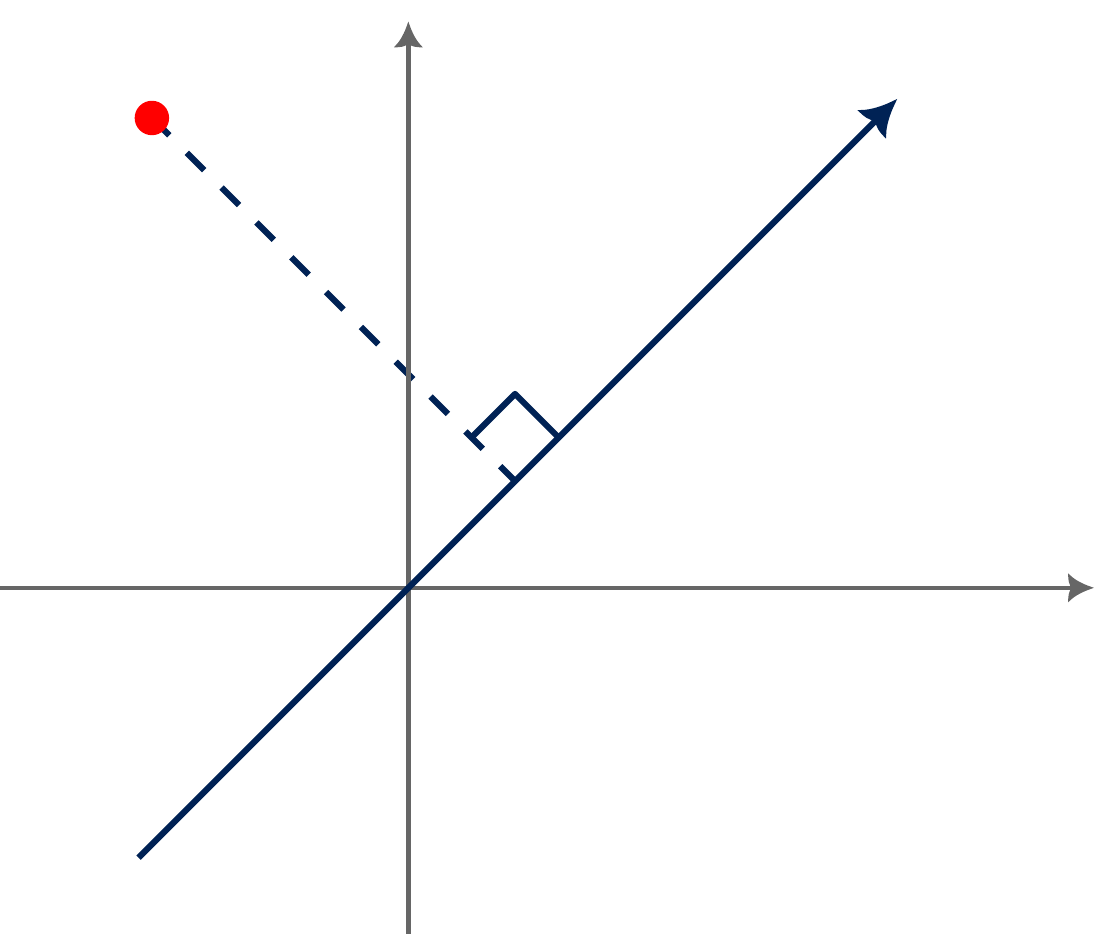}}%
    \put(0.18899145,0.59438195){\rotatebox{-45}{\makebox(0,0)[lt]{\lineheight{1.25}\smash{\begin{tabular}[t]{l}$\frac{x-y}{\sqrt{2}}$\end{tabular}}}}}%
    \put(0.67321396,0.6555711){\color[rgb]{0,0.13333333,0.33333333}\rotatebox{45}{\makebox(0,0)[lt]{\lineheight{1.25}\smash{\begin{tabular}[t]{l}$y=x$\end{tabular}}}}}%
  \end{picture}%
\endgroup%

        \caption{Distance of $(x,y)$ from the line $y=x$.}
        \label{fig:MorseBottDistanceSquaredFromLine}
    \end{figure}
    \noindent So we have
    \begin{align*}
        df_{(x,y)} = \left(x-y,y-x\right) = 0 \implies x=y.
    \end{align*}
    Thus, the critical submanifold is $N$. Now to see whether it is non-degenerate or not in the normal direction, we need to compute the Hessian. Let $(p,p)\in N$ be any critical point. 
    \begin{align*}
        \hess_{(p,p)}(f) & = 
        \begin{pmatrix}
            1 & -1 \\
            -1 & 1
        \end{pmatrix}.
    \end{align*}
    Note that for any $\mathbf{v}=(a,-a)\in \left(T_{(p,p)}N\right)^\perp$ with $\mathbf{v}\neq 0$, we have
    \begin{displaymath}
        \hess_{(p,p)}(f)(\mathbf{v},\mathbf{v})= \mathbf{v}^T \hess_{(p,p)}(f) \mathbf{v} = 4a^2\neq 0.
    \end{displaymath} 
    Thus, the given function is Morse-Bott.
\end{eg}
\begin{eg} \label{eg:MorseBottDistanceSquaredFromSphere}
    Let $M=\mathbb{R}^{n+1}$ with the Euclidean metric $d$. If $N=\mathbb{S}^n$ be the unit sphere, then the distance between a point $\bf{p}\in \rbb^{n+1}$ and $N$ is given by
    \begin{displaymath}
        \dist(\bf{p},N) \defeq \inf_{\bf{q}\in N} \dist(\bf{p},\bf{q}).
    \end{displaymath} 
    We shall denote by $d^2$ the square of the distance. Now consider the function
\begin{displaymath}
    f:M\to \rbb,~~\bf{x}\mapsto \dist^2(\bf{x},N)= \left(\|\bf{x}\|-1\right)^2.
\end{displaymath}
\begin{figure}[H]
    \centering
    \begin{subfigure}{.5\textwidth}
      \centering
    \def\svgwidth{1\columnwidth}
    %% Creator: Inkscape 1.2 (1:1.2.1+202207142221+cd75a1ee6d), www.inkscape.org
%% Accompanies image file 'Example-MorseBott-DistanceSquaredFromCircleOutside.pdf' (pdf, eps, ps)
\begingroup%
  \makeatletter%
  \providecommand\color[2][]{%
    \errmessage{(Inkscape) Color is used for the text in Inkscape, but the package 'color.sty' is not loaded}%
    \renewcommand\color[2][]{}%
  }%
  \providecommand\transparent[1]{%
    \errmessage{(Inkscape) Transparency is used (non-zero) for the text in Inkscape, but the package 'transparent.sty' is not loaded}%
    \renewcommand\transparent[1]{}%
  }%
  \providecommand\rotatebox[2]{#2}%
  \newcommand*\fsize{\dimexpr\f@size pt\relax}%
  \newcommand*\lineheight[1]{\fontsize{\fsize}{#1\fsize}\selectfont}%
  \ifx\svgwidth\undefined%
    \setlength{\unitlength}{469.26559205bp}%
    \ifx\svgscale\undefined%
      \relax%
    \else%
      \setlength{\unitlength}{\unitlength * \real{\svgscale}}%
    \fi%
  \else%
    \setlength{\unitlength}{\svgwidth}%
  \fi%
  \global\let\svgwidth\undefined%
  \global\let\svgscale\undefined%
  \makeatother%
  \begin{picture}(1,0.75899863)%
    \lineheight{1}%
    \setlength\tabcolsep{0pt}%
    \put(0,0){\includegraphics[width=\unitlength,page=1]{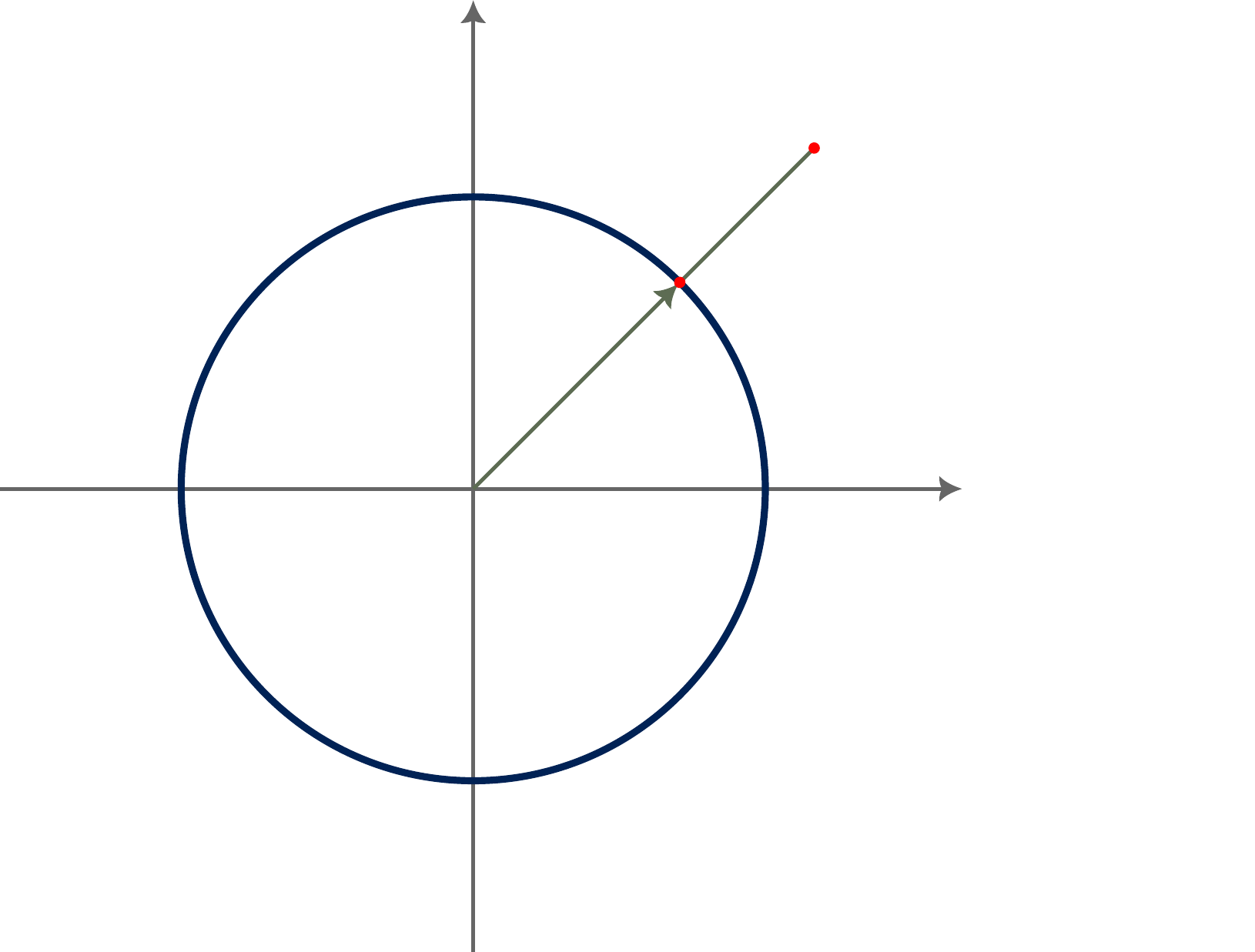}}%
    \put(0.42635609,0.45729261){\makebox(0,0)[lt]{\lineheight{1.25}\smash{\begin{tabular}[t]{l}$1$\end{tabular}}}}%
    \put(0.60650733,0.64634465){\makebox(0,0)[lt]{\lineheight{1.25}\smash{\begin{tabular}[t]{l}$\bf{x}$\end{tabular}}}}%
    \put(0.63824776,0.53625436){\makebox(0,0)[lt]{\lineheight{1.25}\smash{\begin{tabular}[t]{l}$\|\bf{x}\|-1$\end{tabular}}}}%
    \put(0,0){\includegraphics[width=\unitlength,page=2]{Example-MorseBott-DistanceSquaredFromCircleOutside.pdf}}%
  \end{picture}%
\endgroup%

    \end{subfigure}%
    \begin{subfigure}{.5\textwidth}
      \centering
    \def\svgwidth{1\columnwidth}
    %% Creator: Inkscape 1.2 (1:1.2.1+202207142221+cd75a1ee6d), www.inkscape.org
%% Accompanies image file 'Example-MorseBott-DistanceSquaredFromCircleInside.pdf' (pdf, eps, ps)
\begingroup%
  \makeatletter%
  \providecommand\color[2][]{%
    \errmessage{(Inkscape) Color is used for the text in Inkscape, but the package 'color.sty' is not loaded}%
    \renewcommand\color[2][]{}%
  }%
  \providecommand\transparent[1]{%
    \errmessage{(Inkscape) Transparency is used (non-zero) for the text in Inkscape, but the package 'transparent.sty' is not loaded}%
    \renewcommand\transparent[1]{}%
  }%
  \providecommand\rotatebox[2]{#2}%
  \newcommand*\fsize{\dimexpr\f@size pt\relax}%
  \newcommand*\lineheight[1]{\fontsize{\fsize}{#1\fsize}\selectfont}%
  \ifx\svgwidth\undefined%
    \setlength{\unitlength}{426.2005699bp}%
    \ifx\svgscale\undefined%
      \relax%
    \else%
      \setlength{\unitlength}{\unitlength * \real{\svgscale}}%
    \fi%
  \else%
    \setlength{\unitlength}{\svgwidth}%
  \fi%
  \global\let\svgwidth\undefined%
  \global\let\svgscale\undefined%
  \makeatother%
  \begin{picture}(1,0.8356909)%
    \lineheight{1}%
    \setlength\tabcolsep{0pt}%
    \put(0,0){\includegraphics[width=\unitlength,page=1]{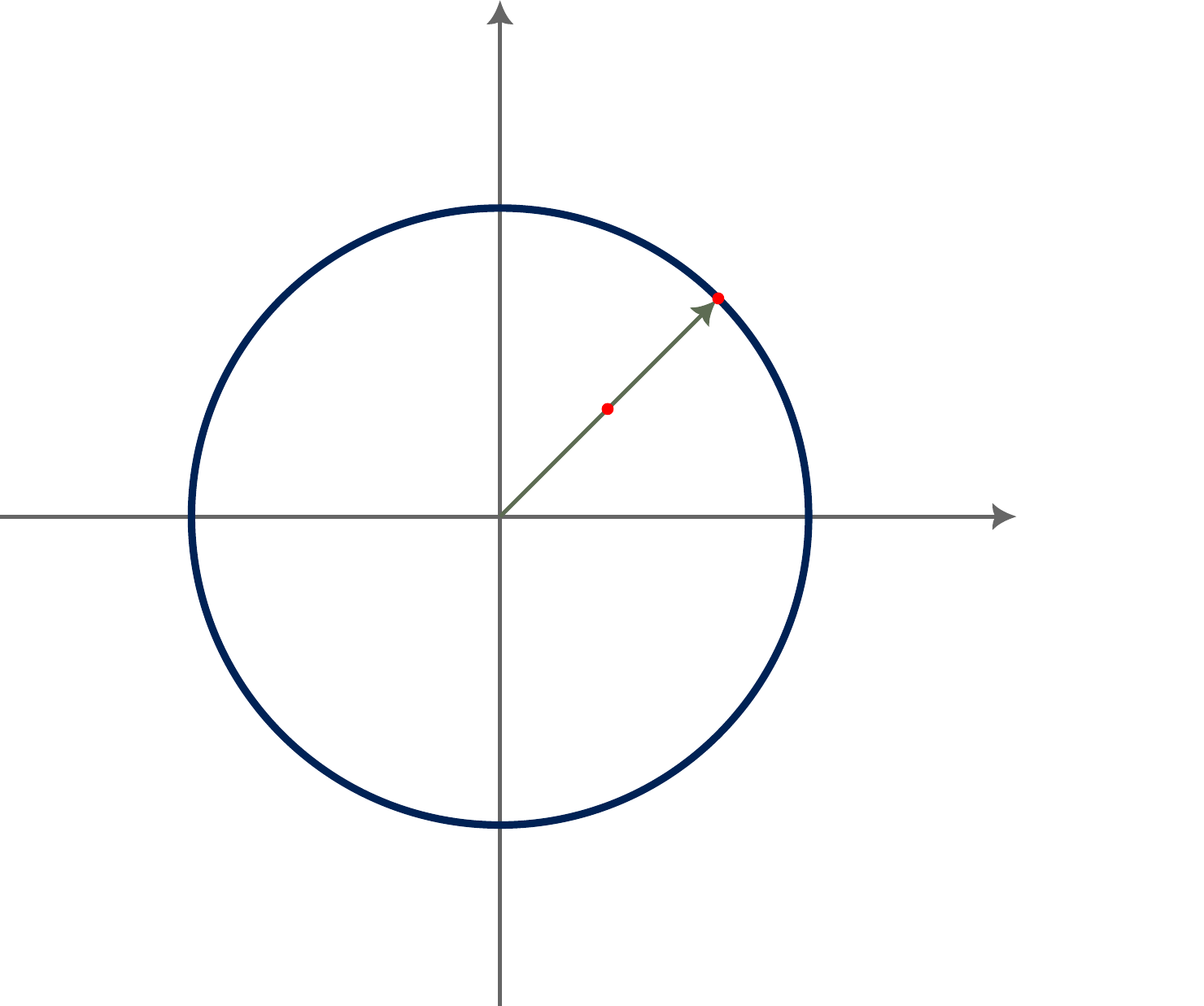}}%
    \put(0.46960291,0.50181748){\makebox(0,0)[lt]{\lineheight{1.25}\smash{\begin{tabular}[t]{l}$\bf{x}$\end{tabular}}}}%
    \put(0.60169486,0.50342247){\makebox(0,0)[lt]{\lineheight{1.25}\smash{\begin{tabular}[t]{l}$1-\|\bf{x}\|$\end{tabular}}}}%
    \put(0,0){\includegraphics[width=\unitlength,page=2]{Example-MorseBott-DistanceSquaredFromCircleInside.pdf}}%
  \end{picture}%
\endgroup%

    \end{subfigure}
    \caption{Distance of $\mathbf{x}$ from the unit circle.}
        \label{fig:MorseBottDistanceSquaredFromSphere}
\end{figure}
\noindent The function $f:M-\{\mathbf{0}\}$ is a Morse-Bott function with $N=\sbb^n$ as the critical submanifold. We will see a general version of this example in \Cref{ch:GeometricViewpointOfCutLocus}.
\end{eg}

\begin{eg}\label{eg:MorseBottSphereHeightSquared}
    Consider the function 
    \begin{displaymath}
        f:\mathbb{S}^2\to \mathbb{R},~(x,y,z)\mapsto z^2.
    \end{displaymath}
    It is square of the height function discussed in the \Cref{eg:MorseBottDistanceSquaredFromSphere}. We claim that $f$ is a Morse-Bott function with critical set as $N=(0,0,1),~S=(0,0,-1)$ and the equator $E=\left\{(x,y,0):x^2+y^2=1\right\}$.  
    \begin{figure}[H]
        \centering
    \def\svgwidth{0.5\columnwidth}
    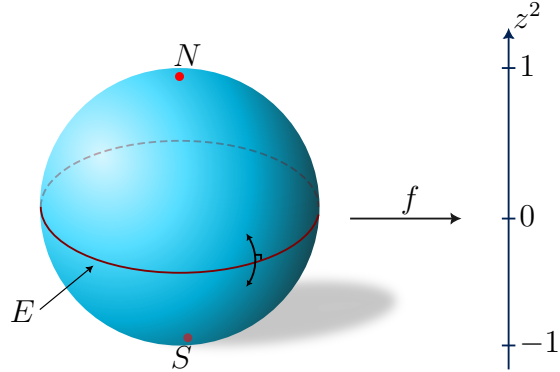

        \caption{Square of the height function on sphere has three critical submanifolds; north pole, south pole and the equator circle.\label{fig:MorseBottSphereHeightSquared}}
    \end{figure}
    \noindent We take the charts on $\mathbb{S}^2$ as given in \Cref{eg:HeightFunctionOnSphere}. So we have
    \begin{align*}
        & \psi_1:\mathbb{R}^2\to \mathbb{R},~\mathbf{p}=(x,y)\mapsto \left(\dfrac{\left\|\mathbf{p}\right\|^2-1}{1+\left\|\mathbf{p}\right\|^2} \right)^2,~\text{ and} \\[1ex]
        & \psi_2:\mathbb{R}^2\to \mathbb{R},~\mathbf{p}=(x,y)\mapsto \left(\dfrac{1-\left\|\mathbf{p}\right\|^2}{1+\left\|\mathbf{p}\right\|^2} \right)^2.
    \end{align*} 
    The critical points are 
    \begin{align*}
        d\psi_{1_{\mathbf{p}}} = (0,0) & \implies \dfrac{8 \left(\left\|\mathbf{p}\right\|^2-1\right)}{\left(1+\left\|\mathbf{p}\right\|^2\right)^3}\mathbf{p} = \mathbf{0} \\[1ex]
        & \implies \left\|\mathbf{p}\right\| = 1 ~\text{ or}~~ \mathbf{p} = 0.
    \end{align*}
    Similarly, $\psi_2$ gives the same condition and hence the critical set will be $N,S$, and $E$ (\Cref{fig:MorseBottSphereHeightSquared}). It is clear that the two submanifolds $\{N\}$ and $\{S\}$ are non-degenerate. To show that $E$ is non-degenerate, we calculate the Hessian matrix at any point of $E$, say $\mathbf{p}=(x,y,0)$. The Hessian with respect to the above charts is given by
    \begin{displaymath}
        \hess_\mathbf{p}(f) = 
        \begin{pmatrix}
            2x^2 & 2xy \\
            2xy & 2y^2
        \end{pmatrix}.
    \end{displaymath}     
    Note that for any $\mathbf{p}\in E$ the normal space $\left(T_\mathbf{p}E\right)^\perp$ is spanned by $(0,0,1)$. Since we have 
    \begin{displaymath}
        T_\mathbf{p}\mathbb{S}^2 =T_\mathbf{p}E\directsum \left(T_\mathbf{p}E\right)^\perp = \spn\{(-y,x,0)\} \directsum \spn\{(0,0,1)\}.
    \end{displaymath}
    We need to show that for any $\mathbf{v}(0,0,\alpha)$ with $\alpha\neq 0$, there exists $\mathbf{w}(0,0,\beta)$ such that $\hess_\mathbf{p}(f)(\mathbf{v},\mathbf{w})\neq 0$. For that we will identify $T_\mathbf{p}\mathbb{S}^2$ with $\mathbb{R}^2$ Consider two curves $\gamma$ and $\eta$ passing through $p$.
    \begin{align*}
        & \gamma(t) = \left\{(x\cos t-y\sin t,y\cos t+x\sin t,0):0\le t\le 2\pi\right\}\\
        & \eta(t) = \left\{(x\cos t,y\cos t,\sin t):0\le t\le 2\pi\right\}.
    \end{align*}
    Note that $\gamma'(0)=(-y,x,0)=\mathbf{v}_1$ and $\eta'(0)=(0,0,1)=\mathbf{v}_2$. So we have
    \begin{align*}
        & d\phi_{1_\mathbf{p}} \left(\mathbf{v}_1\right)=\dfrac{\mathrm{d}}{\mathrm{d}t}\left(\phi_1\circ \gamma\right)(t)\big|_{t=0} =(-y,x) \\[1ex]
        & d\phi_{1_\mathbf{p}} \left(\mathbf{v}_2\right)=\dfrac{\mathrm{d}}{\mathrm{d}t}\left(\phi_1\circ \eta\right)(t)\big|_{t=0} =(-x,-y).
    \end{align*}
    \begin{figure}[H]
        \centering
    \def\svgwidth{0.5\columnwidth}
    %% Creator: Inkscape 1.2 (1:1.2.1+202207142221+cd75a1ee6d), www.inkscape.org
%% Accompanies image file 'curvesOnCircle.pdf' (pdf, eps, ps)
\begingroup%
  \makeatletter%
  \providecommand\color[2][]{%
    \errmessage{(Inkscape) Color is used for the text in Inkscape, but the package 'color.sty' is not loaded}%
    \renewcommand\color[2][]{}%
  }%
  \providecommand\transparent[1]{%
    \errmessage{(Inkscape) Transparency is used (non-zero) for the text in Inkscape, but the package 'transparent.sty' is not loaded}%
    \renewcommand\transparent[1]{}%
  }%
  \providecommand\rotatebox[2]{#2}%
  \newcommand*\fsize{\dimexpr\f@size pt\relax}%
  \newcommand*\lineheight[1]{\fontsize{\fsize}{#1\fsize}\selectfont}%
  \ifx\svgwidth\undefined%
    \setlength{\unitlength}{276.70651939bp}%
    \ifx\svgscale\undefined%
      \relax%
    \else%
      \setlength{\unitlength}{\unitlength * \real{\svgscale}}%
    \fi%
  \else%
    \setlength{\unitlength}{\svgwidth}%
  \fi%
  \global\let\svgwidth\undefined%
  \global\let\svgscale\undefined%
  \makeatother%
  \begin{picture}(1,0.69357904)%
    \lineheight{1}%
    \setlength\tabcolsep{0pt}%
    \put(0,0){\includegraphics[width=\unitlength,page=1]{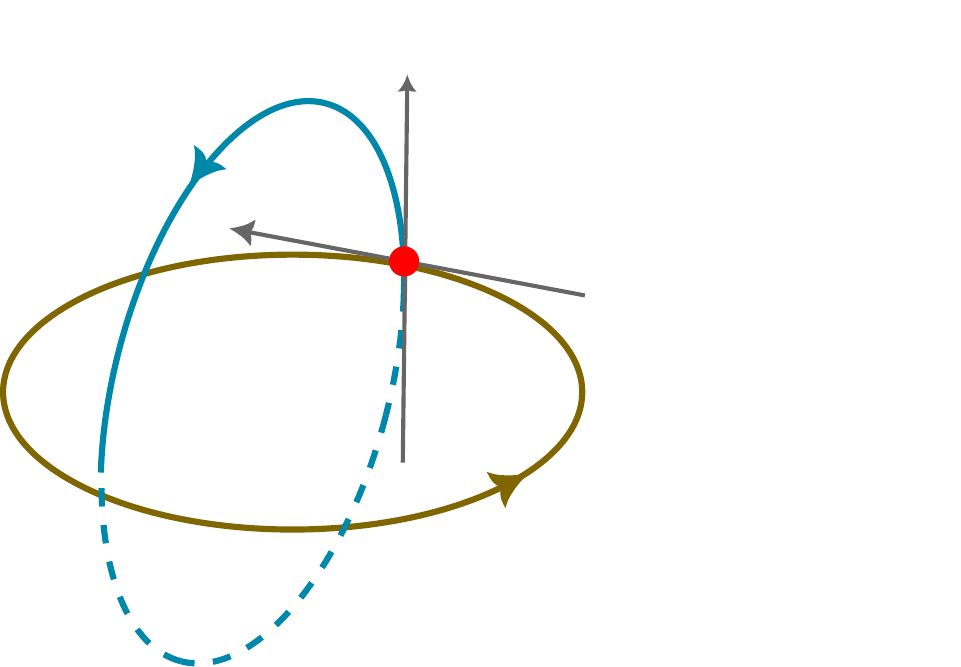}}%
    \put(0.55802661,0.15910496){\color[rgb]{0.50196078,0.4,0}\makebox(0,0)[lt]{\lineheight{1.25}\smash{\begin{tabular}[t]{l}$\gamma$\end{tabular}}}}%
    \put(0.20845421,0.57603658){\color[rgb]{0,0.53333333,0.66666667}\rotatebox{-14.615651}{\makebox(0,0)[lt]{\lineheight{1.25}\smash{\begin{tabular}[t]{l}$\eta$\end{tabular}}}}}%
    \put(0.36600518,0.3777595){\makebox(0,0)[lt]{\lineheight{1.25}\smash{\begin{tabular}[t]{l}$\bf{p}$\end{tabular}}}}%
    \put(0.18600895,0.45693283){\makebox(0,0)[lt]{\lineheight{1.25}\smash{\begin{tabular}[t]{l}$\bf{v}_1$\end{tabular}}}}%
    \put(0.39155229,0.62937516){\makebox(0,0)[lt]{\lineheight{1.25}\smash{\begin{tabular}[t]{l}$\bf{v}_2$\end{tabular}}}}%
    \put(0,0){\includegraphics[width=\unitlength,page=2]{curvesOnCircle.pdf}}%
  \end{picture}%
\endgroup%

        \caption{The curves $\gamma$ and $\eta$ passing through $\mathbf{p}$ \label{fig:curvesOnCircle}}
    \end{figure}
    \noindent So, we can define an isomorphism between
    \begin{displaymath}
        T_\mathbf{p}\mathbb{S}^2\to \mathbb{R}^2,~(-y,x,0)\mapsto (-y,x)~\text{ and } (0,0,1)\mapsto (-x,-y).
    \end{displaymath}
    Now note that
    \begin{align*}
        \hess_{\mathbf{p}}((-x,-y),(-x,-y)) & = (-x,-y) \begin{pmatrix}
            2x^2 & 2xy \\ 2xy & 2y^2
        \end{pmatrix}
        \begin{pmatrix}
            -x\\-y
        \end{pmatrix} \\[1ex]
        & = (-x,-y) 
        \begin{pmatrix}
            -2x^3 & -2xy^2 \\ -2x^2y & -2y^3    
        \end{pmatrix} \\[1ex]
        & = 2x^4+2x^2y^2+2x^2y^2+2y^4 \\
        & = 2 \left(x^2+y^2\right)^2 = 2.
    \end{align*}
    Thus, Hessian is non-degenerate in the normal direction and hence it is a Morse-Bott functions.
\end{eg}
\begin{rem}
    The above example, in particular, shows that the critical submanifolds may have different dimensions.
\end{rem}

\begin{eg}
    Let $f:M\to \mathbb{R}$ be a Morse-Bott function. If $\pi:X\to M$ is any smooth fiber bundle, then the composition $\pi\circ f:X\to M$ is a Morse-Bott function. 
\end{eg}
\bigskip
\noindent The trace function on $SO(n,\rbb), U(n,\cbb)$ and $Sp(n,\cbb)$ is a Morse-Bott function (cf \cite[page~90, Exercise~22]{BaHu04}).

\section{Cut locus and conjugate locus}\label{Sec:cutLocusOfPoint}
\hfb Let $M$ be a complete Riemannian manifold and $p\in M$. Let $\gamma$ be a geodesic such that $\gamma(0)=p$. A \textit{cut point}\index{cut point} of $p$ along the geodesic $\gamma$ is the first point $q$ on $\gamma$ such that for any point $\tilde{q}$ on $\gamma$ beyond $q$, there exists a geodesic $\tilde{\gamma}$ joining $p$ to $\tilde{q}$ such that $l \left(\tilde{\gamma}\right)<l(\gamma)$, where $l(\gamma)$ is the length of $\gamma$. In simple words, $q$ is the first point beyond which $\gamma$ stops to minimize the distance. In this section we will recall the definition of cut locus of a point with some examples. We will also mention some important results which will be generalized in the upcoming chapters. The main references for this section are books \cite[Chapter 3, Section 4]{Sak96} and \cite[Chapter 5]{ChEb75}.

\subsection{Cut locus of a point}\label{subsec:CutLocusOfAPoint}
\hfb Let $M$ be a Riemannian manifold and $p,q\in M$ be two points. If there exists a piecewise differentiable curve joining them, then using the Riemannian metric we can measure the length of the curve. We now consider all possible curves joining these points. Then the distance between $p$ and $q$ is the infimum of the length of all (piecewise differentiable) curves joining $p$ and $q$. This distance induces a metric. We call $M$ to be \textit{complete Riemannian manifold} if $(M,d)$ is a complete metric space. From now onwards, we always consider $M$ to be a complete Riemannian manifold. A geodesic $\gamma(t),~t\in[a,b]$ is said to be \emph{extendable} if it can be extended to a geodesic $\gamma(t),~t\in[c,d]\supsetneq [a,b]$. A Riemannian manifold is said to be \textit{geodesically complete}\index{geodesically complete} if any geodesic can be extendable for all $t\in \mathbb{R}$. Then the Hopf-Rinow Theorem \cite{HoRi31} says that these two notions of completeness are equivalent. If a manifold is not complete, then we can not always extend a geodesic. For example, $\mathbb{R}^2\setminus\{\bf{0}\}$ is not complete and the geodesic $\gamma(t)=t,~t>0$ is not extendable in the negative $x$-axis. This problem does not arise if the manifold is complete. The more is true which says that $M$ is complete if and only if every geodesic can be extended for infinite time. The completeness of $M$ also guarantees that any two points can be joined by a distance minimal geodesic which is defined as follows.

\begin{defn}[Distance minimal geodesic]\index{distance minimal geodesic}
    A geodesic joining $p$ and $q$ is said to be \textit{distance minimal} if the length of the geodesic is equal to the distance between these points, i.e., $l(\gamma)=d(p,q)$.
\end{defn}

\vspace{0.1cm}
\hf We shall now define the cut locus, $\cutn[p]$ of a point $p$ in a complete Riemannian manifold $M$. The notion of cut locus was first introduced for convex surfaces by Henri Poincar\'{e} \cite{Poin05} in 1905 under the name \textit{la ligne de partage} meaning \textit{the dividing line}.

\begin{defn}[Cut locus of a point]\label{defn:cutLocusOfPoint}\index{cut locus of a point}
    Let $M$ be a complete Riemannian manifold and $p\in M$. If $\cu$ denotes the \emph{cut locus} of $p$, then a point $q\in \cu$ if there exists a minimal geodesic joining $p$ to $q$ any extension of which beyond $q$ is not minimal. 
\end{defn}

\vspace{0.1cm}
\noindent Consider the set
\begin{displaymath}
    S=\{s>0:\gamma(t),~0\le t\le s \text{ is a distance minimal geodesic}\}.
\end{displaymath}
\noindent If $S=(0,t_0)$, then $\gamma(t_0)$ is the cut point of $p$ along $\gamma$, and if $S=(0,\infty)$, then the point $p$ does not have a cut locus along $\gamma(t)$. 

\vspace{0.1cm}
\noindent Note that if $q_0$ is a point on the geodesic $\gamma(t)$ which comes after the cut point, i.e., $q=\gamma(t_0)$ and $q_0=\gamma(t),~t>t_0$, then there  is a geodesic $\eta(t)$ joining $p$ to $q_0$ such that $l(\eta)<l(\gamma)$ (see \Cref{fig:shorterGeodesicExists}).
\begin{figure}[!htb]
    \centering
    \def\svgwidth{0.5\columnwidth}
    %% Creator: Inkscape 1.2 (1:1.2.1+202207142221+cd75a1ee6d), www.inkscape.org
%% Accompanies image file 'shorterGeodesicExists.pdf' (pdf, eps, ps)
\begingroup%
  \makeatletter%
  \providecommand\color[2][]{%
    \errmessage{(Inkscape) Color is used for the text in Inkscape, but the package 'color.sty' is not loaded}%
    \renewcommand\color[2][]{}%
  }%
  \providecommand\transparent[1]{%
    \errmessage{(Inkscape) Transparency is used (non-zero) for the text in Inkscape, but the package 'transparent.sty' is not loaded}%
    \renewcommand\transparent[1]{}%
  }%
  \providecommand\rotatebox[2]{#2}%
  \newcommand*\fsize{\dimexpr\f@size pt\relax}%
  \newcommand*\lineheight[1]{\fontsize{\fsize}{#1\fsize}\selectfont}%
  \ifx\svgwidth\undefined%
    \setlength{\unitlength}{409.4569457bp}%
    \ifx\svgscale\undefined%
      \relax%
    \else%
      \setlength{\unitlength}{\unitlength * \real{\svgscale}}%
    \fi%
  \else%
    \setlength{\unitlength}{\svgwidth}%
  \fi%
  \global\let\svgwidth\undefined%
  \global\let\svgscale\undefined%
  \makeatother%
  \begin{picture}(1,0.55134443)%
    \lineheight{1}%
    \setlength\tabcolsep{0pt}%
    \put(0,0){\includegraphics[width=\unitlength,page=1]{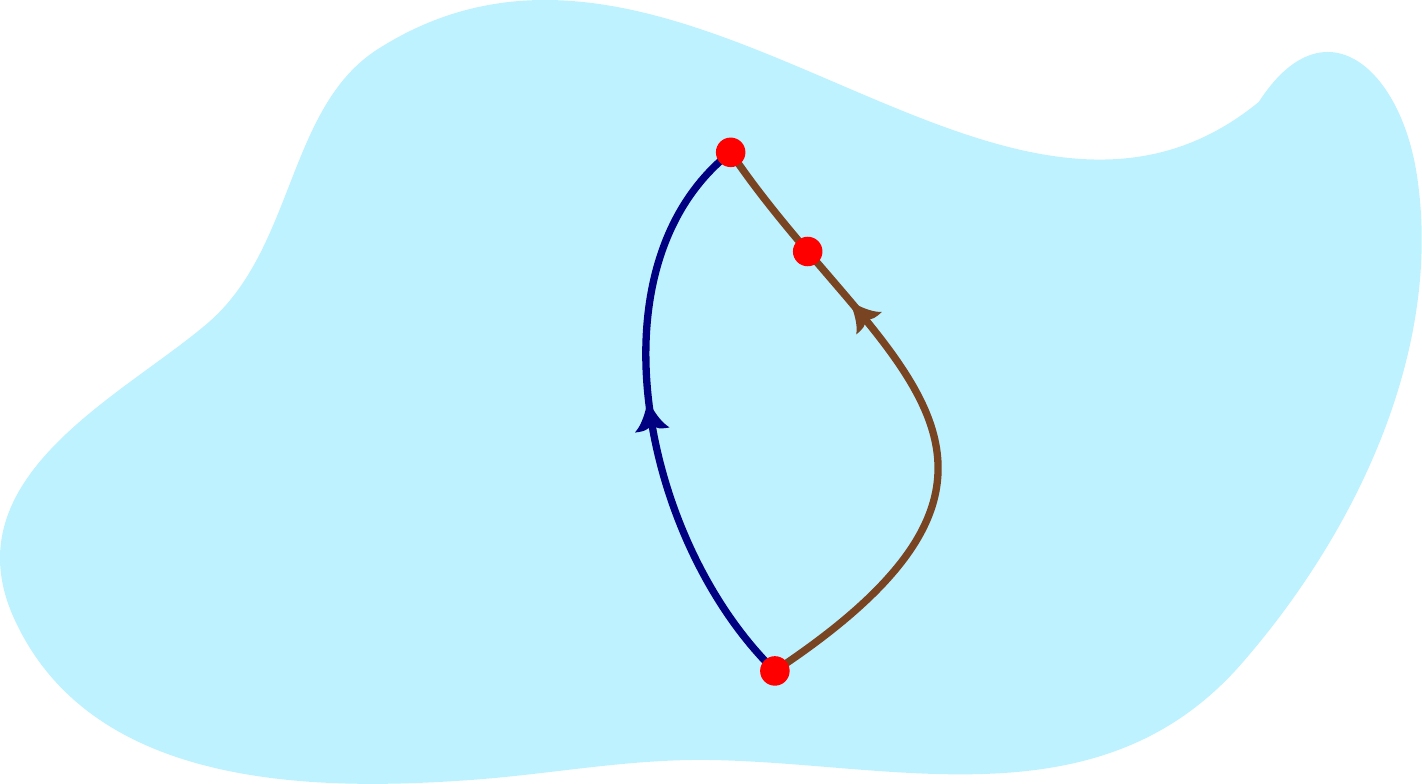}}%
    \put(0.55541775,0.04646027){\makebox(0,0)[lt]{\lineheight{1.25}\smash{\begin{tabular}[t]{l}$p$\end{tabular}}}}%
    \put(0.58338749,0.37870601){\makebox(0,0)[lt]{\lineheight{1.25}\smash{\begin{tabular}[t]{l}$q$\end{tabular}}}}%
    \put(0.49561576,0.47351329){\makebox(0,0)[lt]{\lineheight{1.25}\smash{\begin{tabular}[t]{l}$q_0$\end{tabular}}}}%
    \put(0.66148918,0.25980228){\color[rgb]{0.47058824,0.26666667,0.12941176}\makebox(0,0)[lt]{\lineheight{1.25}\smash{\begin{tabular}[t]{l}$\gamma$\end{tabular}}}}%
    \put(0.40003338,0.26457015){\color[rgb]{0,0,0.50196078}\makebox(0,0)[lt]{\lineheight{1.25}\smash{\begin{tabular}[t]{l}$\eta$\end{tabular}}}}%
    \put(0.06500667,0.08497852){\makebox(0,0)[lt]{\lineheight{1.25}\smash{\begin{tabular}[t]{l}$M$\end{tabular}}}}%
    \put(0.69662899,0.3715044){\makebox(0,0)[lt]{\lineheight{1.25}\smash{\begin{tabular}[t]{l}$l(\eta)<l(\gamma)$\end{tabular}}}}%
  \end{picture}%
\endgroup%

    \caption{A point which is beyond cut point can be joined by a shorter geodesic\label{fig:shorterGeodesicExists}}
\end{figure}
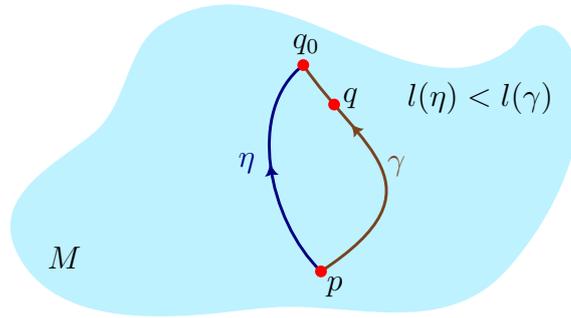

\vspace{0.3cm}
\noindent If $q_0$ comes before the cut point $q$, then we can not find any geodesic shorter than $\gamma$ joining $p$ to $q_0$. Moreover, we even can not find another geodesic $\eta$ joining $p$ to $q_0$ such that $l(\gamma)=l(\eta)$. So we can say that if $q_0$ is coming before cut point, then $\gamma$ is the only minimal geodesic joining $p$ to $q_0$. To prove this fact, we assume that if $\eta$ is another geodesic joining $p$ to $q_0$ such that $l(\gamma)=l(\eta)$ then 
\begin{displaymath}
    \delta(t)= 
    \begin{cases}
        \eta(t),~0\le t\le t_1 \\
        \gamma(t),~0\le t_1\le t_0
    \end{cases}
\end{displaymath}
is a curve such that $l(\delta)=d(p,q)=l(\gamma)$. Choose two points $q_1$ and $q_2$ sufficiently close to $q_0$ as shown in \Cref{fig:shorterGeodesicDoesNotExists}.
\begin{figure}[!htpb]
    \centering
    \def\svgwidth{0.35\columnwidth}
    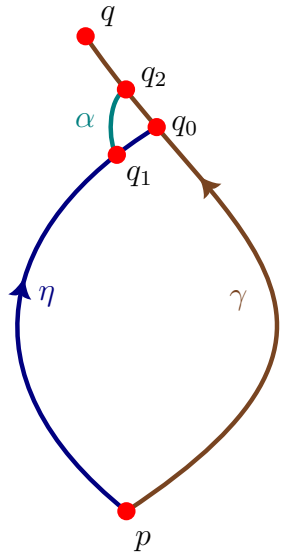

    \caption{If a point appears before the cut point along the geodesic, then it can not be joined by two or more minimal geodesics\label{fig:shorterGeodesicDoesNotExists}}
\end{figure}
\noindent Then $\alpha$ is a distance minimal geodesic joining points $q_1$ and $q_2$ and hence we got a curve $\zeta$ which is $\eta$ from $p$ to $q_1$, $\alpha$ from $q_1$ to $q_2$ and $\gamma$ from $q_2$ to $q$. Note that $l(\zeta)$ is less than the distance between $p$ and $q$, which is a contradiction.

\vspace{0.1cm}
\noindent We now will discuss some examples.
\begin{eg}
    Let $M=\mathbb{R}^n$ be the $n$-Euclidean plane equipped with the Euclidean metric. The cut locus of any point is a null set because any geodesic never fails to satisfy its distance minimizing property.
    \begin{figure}[!htb]
        \centering
    \def\svgwidth{0.6\columnwidth}
    %% Creator: Inkscape 1.2 (1:1.2.1+202207142221+cd75a1ee6d), www.inkscape.org
%% Accompanies image file 'Example-CutLocus-EuclideanSpace.pdf' (pdf, eps, ps)
\begingroup%
  \makeatletter%
  \providecommand\color[2][]{%
    \errmessage{(Inkscape) Color is used for the text in Inkscape, but the package 'color.sty' is not loaded}%
    \renewcommand\color[2][]{}%
  }%
  \providecommand\transparent[1]{%
    \errmessage{(Inkscape) Transparency is used (non-zero) for the text in Inkscape, but the package 'transparent.sty' is not loaded}%
    \renewcommand\transparent[1]{}%
  }%
  \providecommand\rotatebox[2]{#2}%
  \newcommand*\fsize{\dimexpr\f@size pt\relax}%
  \newcommand*\lineheight[1]{\fontsize{\fsize}{#1\fsize}\selectfont}%
  \ifx\svgwidth\undefined%
    \setlength{\unitlength}{210.53334075bp}%
    \ifx\svgscale\undefined%
      \relax%
    \else%
      \setlength{\unitlength}{\unitlength * \real{\svgscale}}%
    \fi%
  \else%
    \setlength{\unitlength}{\svgwidth}%
  \fi%
  \global\let\svgwidth\undefined%
  \global\let\svgscale\undefined%
  \makeatother%
  \begin{picture}(1,0.79105822)%
    \lineheight{1}%
    \setlength\tabcolsep{0pt}%
    \put(0,0){\includegraphics[width=\unitlength,page=1]{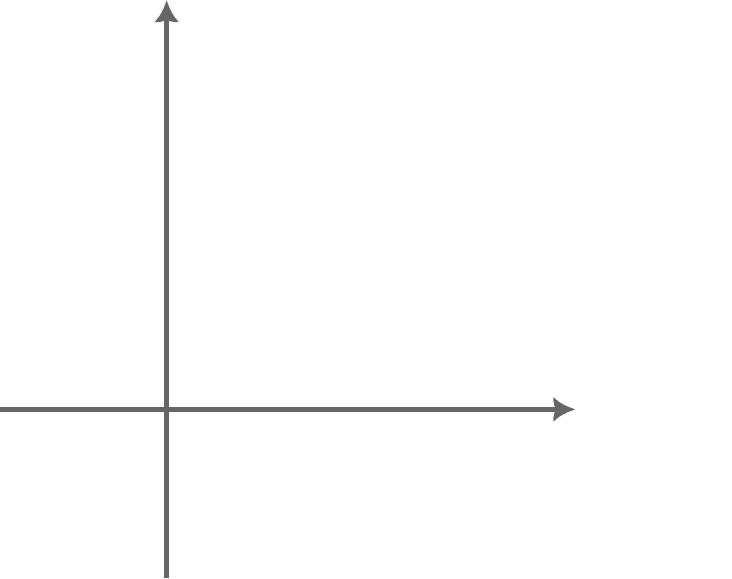}}%
    \put(0.17429767,0.18280737){\makebox(0,0)[lt]{\lineheight{1.25}\smash{\begin{tabular}[t]{l}$\bf{0}$\end{tabular}}}}%
    \put(0,0){\includegraphics[width=\unitlength,page=2]{Example-CutLocus-EuclideanSpace.pdf}}%
    \put(0.60782445,0.39006307){\color[rgb]{0,0,0}\makebox(0,0)[lt]{\lineheight{1.25}\smash{\begin{tabular}[t]{l}$\bf{q}_1$\end{tabular}}}}%
    \put(0.72438883,0.46954432){\color[rgb]{0,0,0}\makebox(0,0)[lt]{\lineheight{1.25}\smash{\begin{tabular}[t]{l}$\bf{q}_2$\end{tabular}}}}%
    \put(0,0){\includegraphics[width=\unitlength,page=3]{Example-CutLocus-EuclideanSpace.pdf}}%
    \put(0.26907706,0.67136166){\makebox(0,0)[lt]{\lineheight{1.25}\smash{\begin{tabular}[t]{l}Geodesic never fails\\to be distance minimal\end{tabular}}}}%
  \end{picture}%
\endgroup%

        \caption{Cut locus of $(0,0)$ in $\mathbb{R}^2$ \label{fig:Example-CutLocus-EuclideanSpace}}
    \end{figure}
     
\end{eg}

\begin{eg}[Cut locus of a point in $n$-sphere]
    Let $M=\mathbb{S}^n$ be the $n$-sphere with the round metric. The geodesics are great circles. The cut locus of the south pole is the north pole. 
    \begin{figure}[H]
        \centering
        \begin{subfigure}{.30\textwidth}
          \centering
    \def\svgwidth{0.9\columnwidth}
    %% Creator: Inkscape 1.2 (1:1.2.1+202207142221+cd75a1ee6d), www.inkscape.org
%% Accompanies image file 'Example-Sphere-CutLocusPoint-1.pdf' (pdf, eps, ps)
\begingroup%
  \makeatletter%
  \providecommand\color[2][]{%
    \errmessage{(Inkscape) Color is used for the text in Inkscape, but the package 'color.sty' is not loaded}%
    \renewcommand\color[2][]{}%
  }%
  \providecommand\transparent[1]{%
    \errmessage{(Inkscape) Transparency is used (non-zero) for the text in Inkscape, but the package 'transparent.sty' is not loaded}%
    \renewcommand\transparent[1]{}%
  }%
  \providecommand\rotatebox[2]{#2}%
  \newcommand*\fsize{\dimexpr\f@size pt\relax}%
  \newcommand*\lineheight[1]{\fontsize{\fsize}{#1\fsize}\selectfont}%
  \ifx\svgwidth\undefined%
    \setlength{\unitlength}{347.31249952bp}%
    \ifx\svgscale\undefined%
      \relax%
    \else%
      \setlength{\unitlength}{\unitlength * \real{\svgscale}}%
    \fi%
  \else%
    \setlength{\unitlength}{\svgwidth}%
  \fi%
  \global\let\svgwidth\undefined%
  \global\let\svgscale\undefined%
  \makeatother%
  \begin{picture}(1,1.08101537)%
    \lineheight{1}%
    \setlength\tabcolsep{0pt}%
    \put(0,0){\includegraphics[width=\unitlength,page=1]{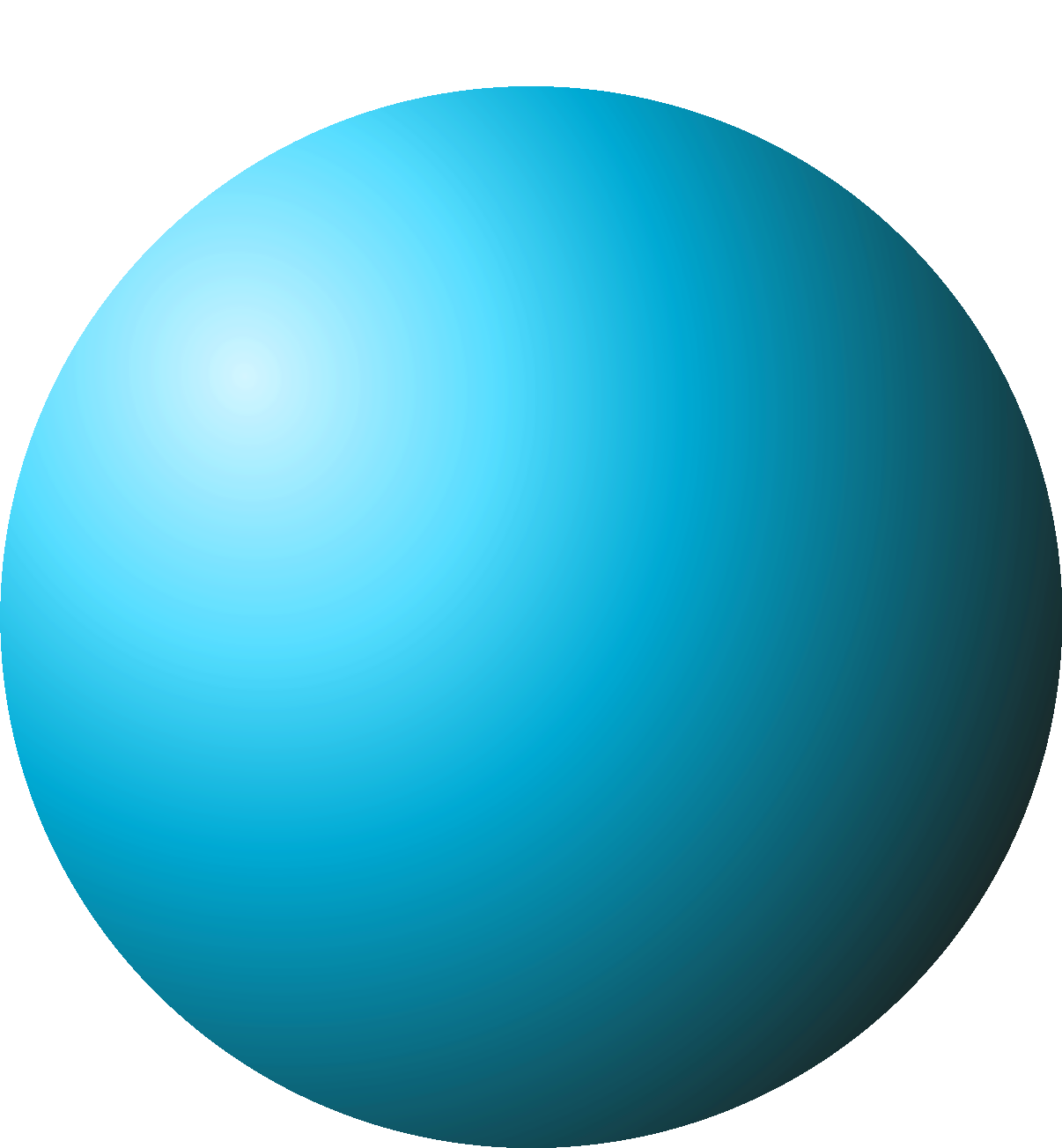}}%
    \put(0.53175319,0.96200861){\makebox(0,0)[lt]{\lineheight{1.25}\smash{\begin{tabular}[t]{l}$N$\end{tabular}}}}%
    \put(0,0){\includegraphics[width=\unitlength,page=2]{Example-Sphere-CutLocusPoint-1.pdf}}%
    \put(0.48223441,0.07563445){\color[rgb]{0.94901961,0.94901961,0.94901961}\makebox(0,0)[lt]{\lineheight{1.25}\smash{\begin{tabular}[t]{l}$S$\end{tabular}}}}%
    \put(0,0){\includegraphics[width=\unitlength,page=3]{Example-Sphere-CutLocusPoint-1.pdf}}%
  \end{picture}%
\endgroup%

          \caption{}
          \label{fig:Sphere-CutLocusPoint-01}
        \end{subfigure}%
        \begin{subfigure}{.30\textwidth}
          \centering
    \def\svgwidth{.9\columnwidth}
    %% Creator: Inkscape 1.2 (1:1.2.1+202207142221+cd75a1ee6d), www.inkscape.org
%% Accompanies image file 'Example-Sphere-CutLocusPoint-2.pdf' (pdf, eps, ps)
\begingroup%
  \makeatletter%
  \providecommand\color[2][]{%
    \errmessage{(Inkscape) Color is used for the text in Inkscape, but the package 'color.sty' is not loaded}%
    \renewcommand\color[2][]{}%
  }%
  \providecommand\transparent[1]{%
    \errmessage{(Inkscape) Transparency is used (non-zero) for the text in Inkscape, but the package 'transparent.sty' is not loaded}%
    \renewcommand\transparent[1]{}%
  }%
  \providecommand\rotatebox[2]{#2}%
  \newcommand*\fsize{\dimexpr\f@size pt\relax}%
  \newcommand*\lineheight[1]{\fontsize{\fsize}{#1\fsize}\selectfont}%
  \ifx\svgwidth\undefined%
    \setlength{\unitlength}{347.31249952bp}%
    \ifx\svgscale\undefined%
      \relax%
    \else%
      \setlength{\unitlength}{\unitlength * \real{\svgscale}}%
    \fi%
  \else%
    \setlength{\unitlength}{\svgwidth}%
  \fi%
  \global\let\svgwidth\undefined%
  \global\let\svgscale\undefined%
  \makeatother%
  \begin{picture}(1,1.08101537)%
    \lineheight{1}%
    \setlength\tabcolsep{0pt}%
    \put(0,0){\includegraphics[width=\unitlength,page=1]{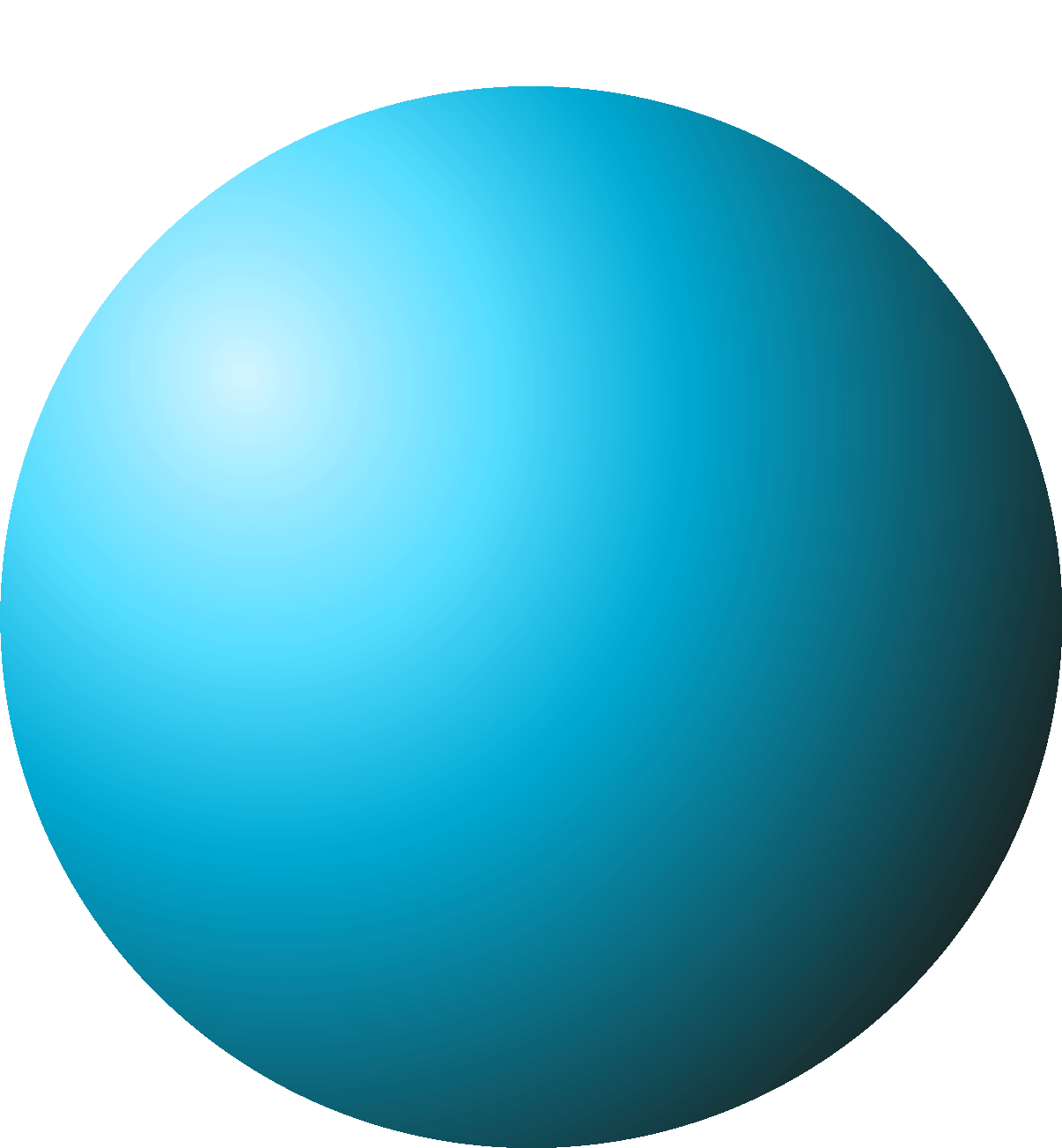}}%
    \put(0.53175319,0.96200861){\makebox(0,0)[lt]{\lineheight{1.25}\smash{\begin{tabular}[t]{l}$N$\end{tabular}}}}%
    \put(0,0){\includegraphics[width=\unitlength,page=2]{Example-Sphere-CutLocusPoint-2.pdf}}%
    \put(0.48223441,0.07563445){\color[rgb]{0.94901961,0.94901961,0.94901961}\makebox(0,0)[lt]{\lineheight{1.25}\smash{\begin{tabular}[t]{l}$S$\end{tabular}}}}%
    \put(0,0){\includegraphics[width=\unitlength,page=3]{Example-Sphere-CutLocusPoint-2.pdf}}%
    \put(0.27587484,0.77203511){\makebox(0,0)[lt]{\lineheight{1.25}\smash{\begin{tabular}[t]{l}$P$\end{tabular}}}}%
    \put(0,0){\includegraphics[width=\unitlength,page=4]{Example-Sphere-CutLocusPoint-2.pdf}}%
  \end{picture}%
\endgroup%

          \caption{}
          \label{fig:Sphere-CutLocusPoint-02}
        \end{subfigure}
        \begin{subfigure}{.30\textwidth}
            \centering
    \def\svgwidth{.9\columnwidth}
    %% Creator: Inkscape 1.2 (1:1.2.1+202207142221+cd75a1ee6d), www.inkscape.org
%% Accompanies image file 'Example-Sphere-CutLocusPoint-3.pdf' (pdf, eps, ps)
\begingroup%
  \makeatletter%
  \providecommand\color[2][]{%
    \errmessage{(Inkscape) Color is used for the text in Inkscape, but the package 'color.sty' is not loaded}%
    \renewcommand\color[2][]{}%
  }%
  \providecommand\transparent[1]{%
    \errmessage{(Inkscape) Transparency is used (non-zero) for the text in Inkscape, but the package 'transparent.sty' is not loaded}%
    \renewcommand\transparent[1]{}%
  }%
  \providecommand\rotatebox[2]{#2}%
  \newcommand*\fsize{\dimexpr\f@size pt\relax}%
  \newcommand*\lineheight[1]{\fontsize{\fsize}{#1\fsize}\selectfont}%
  \ifx\svgwidth\undefined%
    \setlength{\unitlength}{347.31249952bp}%
    \ifx\svgscale\undefined%
      \relax%
    \else%
      \setlength{\unitlength}{\unitlength * \real{\svgscale}}%
    \fi%
  \else%
    \setlength{\unitlength}{\svgwidth}%
  \fi%
  \global\let\svgwidth\undefined%
  \global\let\svgscale\undefined%
  \makeatother%
  \begin{picture}(1,1.08101537)%
    \lineheight{1}%
    \setlength\tabcolsep{0pt}%
    \put(0,0){\includegraphics[width=\unitlength,page=1]{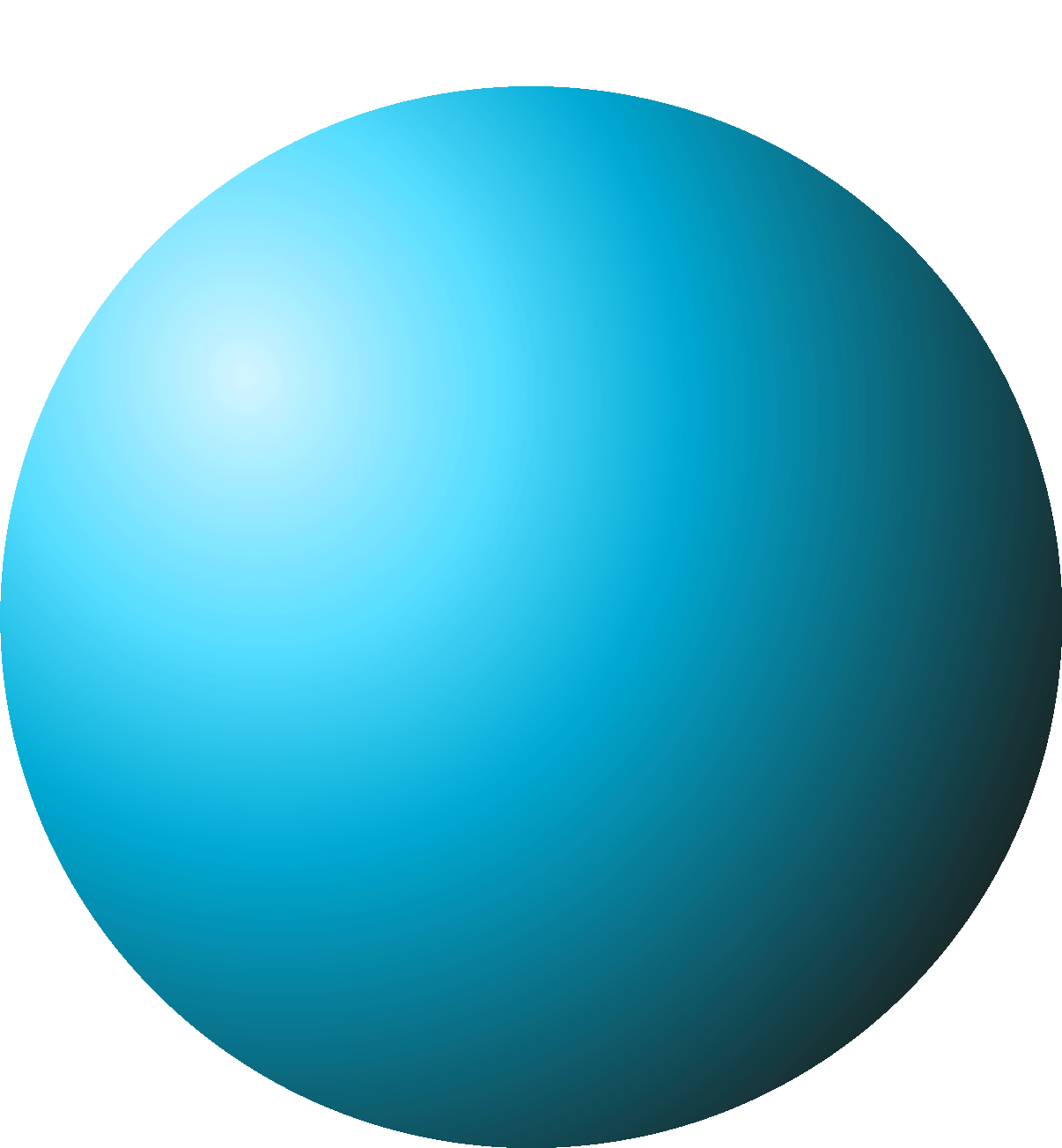}}%
    \put(0.53175319,0.96200861){\makebox(0,0)[lt]{\lineheight{1.25}\smash{\begin{tabular}[t]{l}$N$\end{tabular}}}}%
    \put(0,0){\includegraphics[width=\unitlength,page=2]{Example-Sphere-CutLocusPoint-3.pdf}}%
    \put(0.48223441,0.07563445){\color[rgb]{0.94901961,0.94901961,0.94901961}\makebox(0,0)[lt]{\lineheight{1.25}\smash{\begin{tabular}[t]{l}$S$\end{tabular}}}}%
    \put(0,0){\includegraphics[width=\unitlength,page=3]{Example-Sphere-CutLocusPoint-3.pdf}}%
    \put(0.27587484,0.77203511){\makebox(0,0)[lt]{\lineheight{1.25}\smash{\begin{tabular}[t]{l}$P$\end{tabular}}}}%
    \put(0,0){\includegraphics[width=\unitlength,page=4]{Example-Sphere-CutLocusPoint-3.pdf}}%
  \end{picture}%
\endgroup%

            \caption{}
            \label{fig:Sphere-CutLocusPoint-03}
          \end{subfigure}
        \caption{Cut locus of south pole in $\mathbb{S}^2$}
            \label{fig:Sphere-CutLocusPoint}
    \end{figure}
    \noindent In \Cref{fig:Sphere-CutLocusPoint} we have proven the claim. If $\gamma$ is a geodesic from south pole $S$ to north pole $N$, then the length of $\gamma$ is $\pi$ which is also the distance between these two points. Extending this geodesic beyond $N$ (\Cref{fig:Sphere-CutLocusPoint-02}) makes its length more than $\pi$, whereas the distance between $S$ to $P$ is less than $\pi$ (\Cref{fig:Sphere-CutLocusPoint-03}).
\end{eg}

\begin{eg}[Flat torus]
    Consider $[0,1]\times [0,1]\subseteq \mathbb{R}^2$. We identify $(x,0)$ with $(x,1)$ and $(0,y)$ with $(1,y)$ where $x,y\in [0,1]$. The obtained quotient space is the flat torus. The metric is naturally induced from the Euclidean metric and hence the geodesics are straight lines. If $p$ be the center $\left(\frac{1}{2},\frac{1}{2}\right)$, then the cut locus is the wedge of two circles.
    \begin{figure}[!htpb]
        \centering
        \begin{subfigure}{.45\textwidth}
          \centering
    \def\svgwidth{0.5\columnwidth}
    %% Creator: Inkscape 1.2 (1:1.2.1+202207142221+cd75a1ee6d), www.inkscape.org
%% Accompanies image file 'Example-Torus-CutLocusPoint-03.pdf' (pdf, eps, ps)
\begingroup%
  \makeatletter%
  \providecommand\color[2][]{%
    \errmessage{(Inkscape) Color is used for the text in Inkscape, but the package 'color.sty' is not loaded}%
    \renewcommand\color[2][]{}%
  }%
  \providecommand\transparent[1]{%
    \errmessage{(Inkscape) Transparency is used (non-zero) for the text in Inkscape, but the package 'transparent.sty' is not loaded}%
    \renewcommand\transparent[1]{}%
  }%
  \providecommand\rotatebox[2]{#2}%
  \newcommand*\fsize{\dimexpr\f@size pt\relax}%
  \newcommand*\lineheight[1]{\fontsize{\fsize}{#1\fsize}\selectfont}%
  \ifx\svgwidth\undefined%
    \setlength{\unitlength}{266.42330446bp}%
    \ifx\svgscale\undefined%
      \relax%
    \else%
      \setlength{\unitlength}{\unitlength * \real{\svgscale}}%
    \fi%
  \else%
    \setlength{\unitlength}{\svgwidth}%
  \fi%
  \global\let\svgwidth\undefined%
  \global\let\svgscale\undefined%
  \makeatother%
  \begin{picture}(1,0.86854653)%
    \lineheight{1}%
    \setlength\tabcolsep{0pt}%
    \put(0,0){\includegraphics[width=\unitlength,page=1]{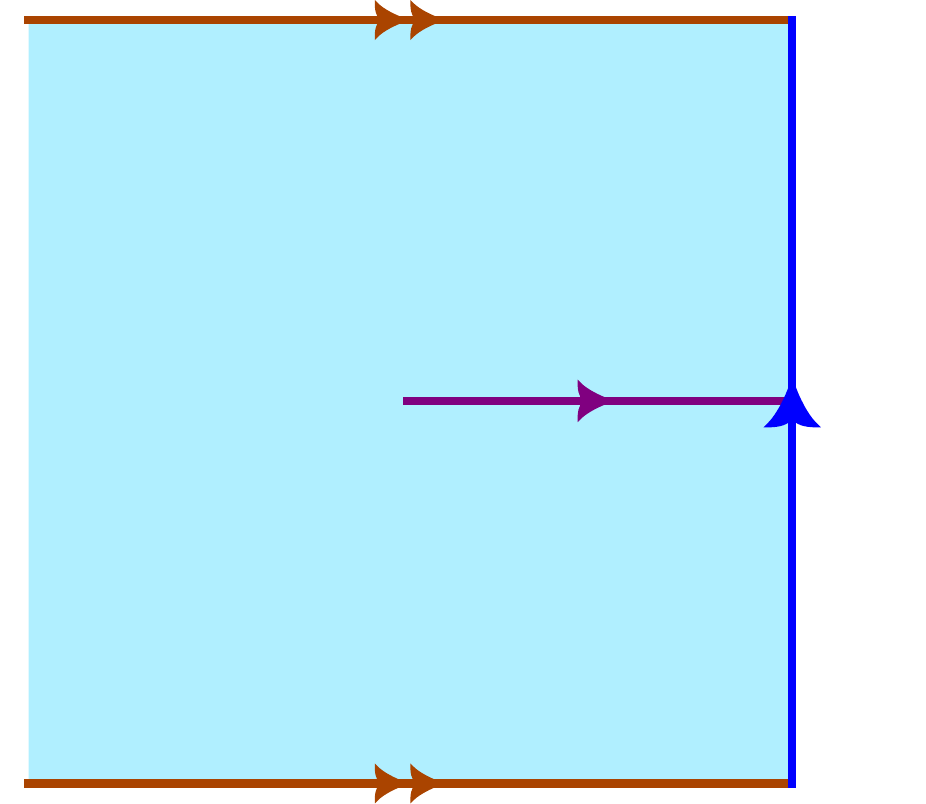}}%
    \put(0.45262449,0.48478503){\makebox(0,0)[lt]{\lineheight{1.25}\smash{\begin{tabular}[t]{l}$p$\end{tabular}}}}%
    \put(0,0){\includegraphics[width=\unitlength,page=2]{Example-Torus-CutLocusPoint-03.pdf}}%
    \put(0.54096764,0.3627715){\color[rgb]{0.50196078,0,0.50196078}\makebox(0,0)[lt]{\lineheight{1.25}\smash{\begin{tabular}[t]{l}$\gamma$\end{tabular}}}}%
    \put(0.30642161,0.48953582){\color[rgb]{0.50196078,0.4,0}\makebox(0,0)[lt]{\lineheight{1.25}\smash{\begin{tabular}[t]{l}$\eta$\end{tabular}}}}%
  \end{picture}%
\endgroup%

          \caption{Extending $\gamma$ beyond the blue line fails to be distance minimal}
        \end{subfigure}%
        \begin{subfigure}{.45\textwidth}
            \centering
    \def\svgwidth{0.43\columnwidth}
    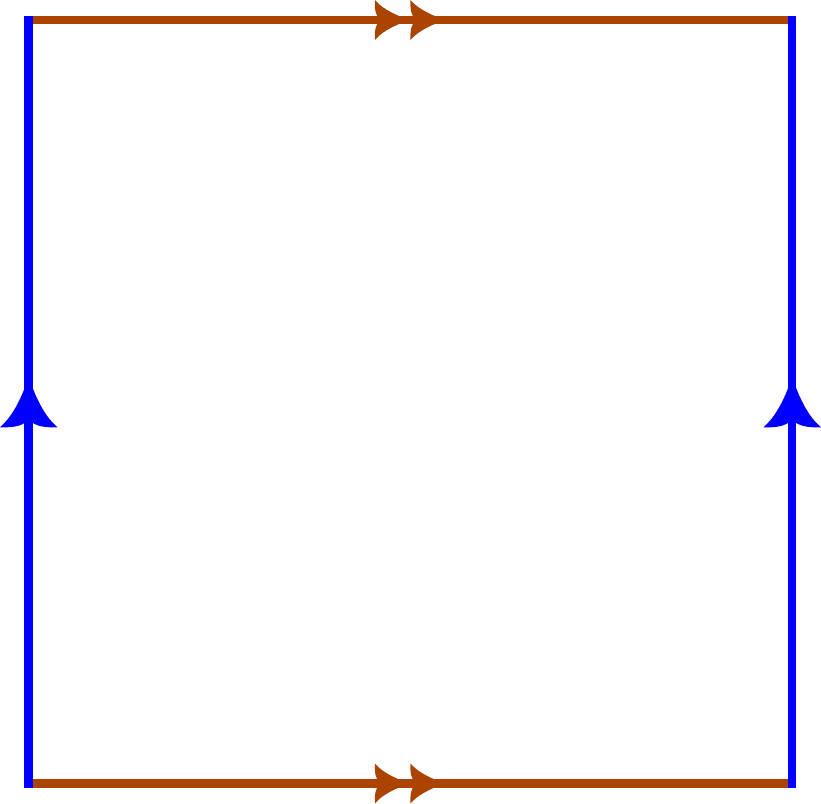

            \caption{cut locus of $p$}
          \end{subfigure}
        \caption{Cut locus of a point in a flat torus}
            \label{fig:Torus-CutLocusPoint}
    \end{figure}
\end{eg}

% \begin{eg}[Torus with product metric]
%     Let $M=\mathbb{S}^1\times \mathbb{S}^1$ with the product metric. Then the cut locus a point is $\mathbb{S}^1\vee \mathbb{S}^1$.
%     \begin{figure}[H]
%         \centering
%         \begin{subfigure}{.45\textwidth}
%           \centering
%           \incfig[0.6]{Example-Torus-CutLocusPointProductMetric-01}
%           \caption{Extending $\gamma$ beyond the point $q$ fails to be distance minimal}
%         \end{subfigure}%
%         \begin{subfigure}{.45\textwidth}
%             \centering
%             \incfig[0.5]{Example-Torus-CutLocusPointProductMetric}
%             \caption{cut locus of $p$}
%           \end{subfigure}
%         \caption{Cut locus of a point in a flat torus}
%             \label{fig:Torus-CutLocusPointProductMetric}
%     \end{figure}
% \end{eg}

\begin{eg}[Real projective planes]\label{eg:CutLocusofPointRPn}
    We obtain the real projective plane $\mathbb{RP}^n$ by identifying the antipodal points of the round sphere $\mathbb{S}^n$. The metric on $\mathbb{RP}^n$ is induced from the metric on $\mathbb{S}^n$. If $\pi:\mathbb{S}^n\to \mathbb{RP}^n,~p,-p\mapsto [p]$, then 
    \begin{displaymath}
         \left\langle X,Y \right\rangle_{[p]}\eqdef \left\langle \left(d\pi_p\right)^{-1}(X),\left(d\pi_p\right)^{-1}(Y) \right\rangle_p
    \end{displaymath}
    is a metric on $\mathbb{RP}^n$. Since the antipodal map is an isometry of $\mathbb{S}^n$, the map $\pi$ is a local isometry. Let $[p]\in \mathbb{RP}^n$ such that it is the image of north and south pole under the map $\pi$. Then the image of the equator of $\mathbb{S}^n$ under the quotient map $\pi$, $\mathbb{RP}^{n-1}$, is the cut locus of $[p]$. We will see a generalization of similar result in \Cref{ch:equivariantCutLocusTheorem}.
\end{eg}

\begin{eg}[Cut locus a point in cylinder]
    Note that for a given point $p$ if more than one distance minimal geodesic joining $p$ and $q$ exists, then $q$ is a cut point. Using this, we observed that the cut locus of a point in cylinder is a line (shown in \Cref{fig:Example-Cylinder-CutLocusPoint}). We also note that the point $-p$ is the closest point and there exists a closed geodesic passing through $-p$ starting and ending at $p$. This fact is more generally true (see \Cref{thm:ExistenseOfClosedGeodesic-2}).  Generalizing this example, the cut locus of a point $(\mathbf{p},\mathbf{v})\in\mathbb{S}^n\times \mathbb{R}^m$ with the product metric is $\{-\mathbf{p}\}\times \mathbb{R}^m$.
    \begin{figure}[H]
        \centering
    \def\svgwidth{0.3\columnwidth}
    %% Creator: Inkscape 1.2 (1:1.2.1+202207142221+cd75a1ee6d), www.inkscape.org
%% Accompanies image file 'Example-Cylinder-CutLocusPoint.pdf' (pdf, eps, ps)
\begingroup%
  \makeatletter%
  \providecommand\color[2][]{%
    \errmessage{(Inkscape) Color is used for the text in Inkscape, but the package 'color.sty' is not loaded}%
    \renewcommand\color[2][]{}%
  }%
  \providecommand\transparent[1]{%
    \errmessage{(Inkscape) Transparency is used (non-zero) for the text in Inkscape, but the package 'transparent.sty' is not loaded}%
    \renewcommand\transparent[1]{}%
  }%
  \providecommand\rotatebox[2]{#2}%
  \newcommand*\fsize{\dimexpr\f@size pt\relax}%
  \newcommand*\lineheight[1]{\fontsize{\fsize}{#1\fsize}\selectfont}%
  \ifx\svgwidth\undefined%
    \setlength{\unitlength}{633.82758065bp}%
    \ifx\svgscale\undefined%
      \relax%
    \else%
      \setlength{\unitlength}{\unitlength * \real{\svgscale}}%
    \fi%
  \else%
    \setlength{\unitlength}{\svgwidth}%
  \fi%
  \global\let\svgwidth\undefined%
  \global\let\svgscale\undefined%
  \makeatother%
  \begin{picture}(1,1.14453702)%
    \lineheight{1}%
    \setlength\tabcolsep{0pt}%
    \put(0,0){\includegraphics[width=\unitlength,page=1]{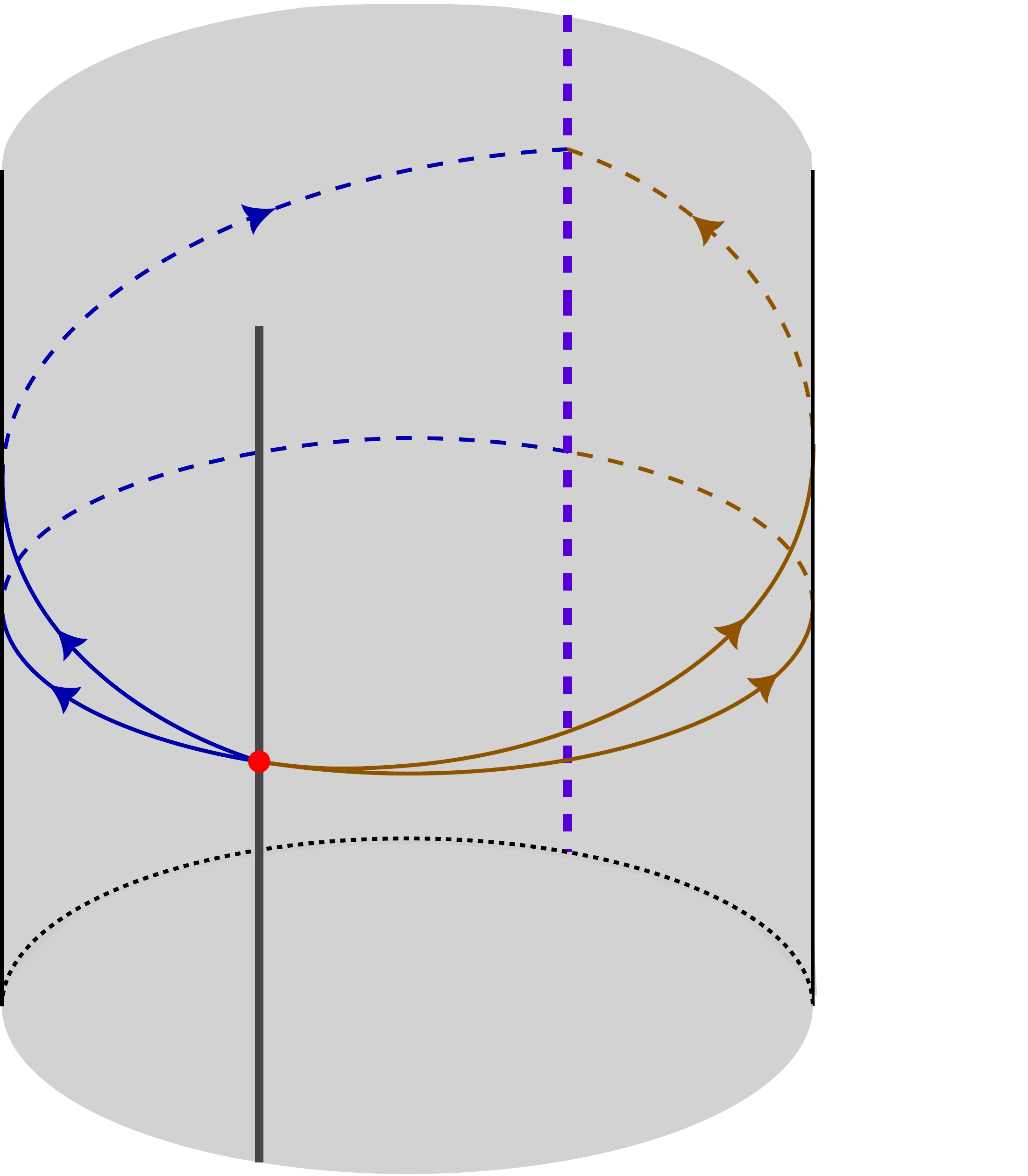}}%
    \put(0.26741269,0.4311601){\makebox(0,0)[lt]{\lineheight{1.25}\smash{\begin{tabular}[t]{l}$p$\end{tabular}}}}%
    \put(0,0){\includegraphics[width=\unitlength,page=2]{Example-Cylinder-CutLocusPoint.pdf}}%
    \put(0.81207922,0.5223416){\makebox(0,0)[lt]{\lineheight{1.25}\smash{\begin{tabular}[t]{l}$\cutn[p]$\end{tabular}}}}%
  \end{picture}%
\endgroup%

        \caption{Cut locus of a point in a cylinder\label{fig:Example-Cylinder-CutLocusPoint}}
    \end{figure}
    
\end{eg}

\subsection{Conjugate locus of a point}\index{conjugate locus}
\hfb Let $M$ be a Riemannian manifold and $\gamma$ be any curve defined on $[a,b]$. A \textit{variation of $\gamma$}\index{variation of a curve} is a function $\Gamma:[a,b]\times (-\varepsilon,\varepsilon)\to M$ such that $\Gamma(t,0)=\gamma(t)$. So $\Gamma$ is a one-parameter family of curves $\gamma_s(t)\defeq \Gamma(t,s)$. If each of $\gamma_s$ is a geodesic, then we call it is a \textit{geodesic variation} \index{geodesic variation}. In this section by variation we mean the geodesic variation. The variation field $\delbydel{\Gamma}{s}(t,0)$ is called a \textit{Jacobi field} and we will denote it by $J(t)$.

\bigskip
\hf Let $\gamma$ be a geodesic. We say that a point $q\in M$ on $\gamma$ is \textit{conjugate to $p\in M$} if we can find a variation $\gamma_s$ of $\gamma$ such that $\gamma_s(0)=p$ for $s\in (-\varepsilon,\varepsilon)$ and each of the geodesic $\gamma_s$ meet \textit{infinitesimally} at $q$. That is, if $\gamma(t_0)=q$, then
\begin{displaymath}
    \left.\delbydel{\gamma_s}{t}\right|_{(t,s)=(0,0)}=0 = \left.\delbydel{\gamma_s}{t}\right|_{(t,s)=(t_0,0)}.
\end{displaymath}
The conjugate points can be defined in two more equivalent ways. One of them uses the Jacobi field and other uses the exponential map. Recall that a Jacobi field along a geodesic $\gamma$ satisfies 
\begin{displaymath}
    \nabla_{\dot{\gamma}}\nabla_{\dot{\gamma}}J + R \left(\dot{\gamma},J\right)\dot{\gamma}=0,
\end{displaymath}
where $R$ is the Riemann curvature tensor.
\begin{defn}[Conjugate points in terms of Jacobi fields] \label{defn:ConjugatePointsInTermsOfJacobiFields}\index{conjugate points via Jacobi fields}
    A point $p$ is said to be \textit{conjugate to q along a geodesic $\gamma$} if there exists a non-vanishing Jacobi field along $\gamma$ which vanishes at $p$ and $q$.
\end{defn}
\begin{figure}[H]
    \centering
    \def\svgwidth{0.5\columnwidth}
    %% Creator: Inkscape 1.2 (1:1.2.1+202207142221+cd75a1ee6d), www.inkscape.org
%% Accompanies image file 'ConjugateLocusIllustration.pdf' (pdf, eps, ps)
\begingroup%
  \makeatletter%
  \providecommand\color[2][]{%
    \errmessage{(Inkscape) Color is used for the text in Inkscape, but the package 'color.sty' is not loaded}%
    \renewcommand\color[2][]{}%
  }%
  \providecommand\transparent[1]{%
    \errmessage{(Inkscape) Transparency is used (non-zero) for the text in Inkscape, but the package 'transparent.sty' is not loaded}%
    \renewcommand\transparent[1]{}%
  }%
  \providecommand\rotatebox[2]{#2}%
  \newcommand*\fsize{\dimexpr\f@size pt\relax}%
  \newcommand*\lineheight[1]{\fontsize{\fsize}{#1\fsize}\selectfont}%
  \ifx\svgwidth\undefined%
    \setlength{\unitlength}{514.98244588bp}%
    \ifx\svgscale\undefined%
      \relax%
    \else%
      \setlength{\unitlength}{\unitlength * \real{\svgscale}}%
    \fi%
  \else%
    \setlength{\unitlength}{\svgwidth}%
  \fi%
  \global\let\svgwidth\undefined%
  \global\let\svgscale\undefined%
  \makeatother%
  \begin{picture}(1,0.27667711)%
    \lineheight{1}%
    \setlength\tabcolsep{0pt}%
    \put(0,0){\includegraphics[width=\unitlength,page=1]{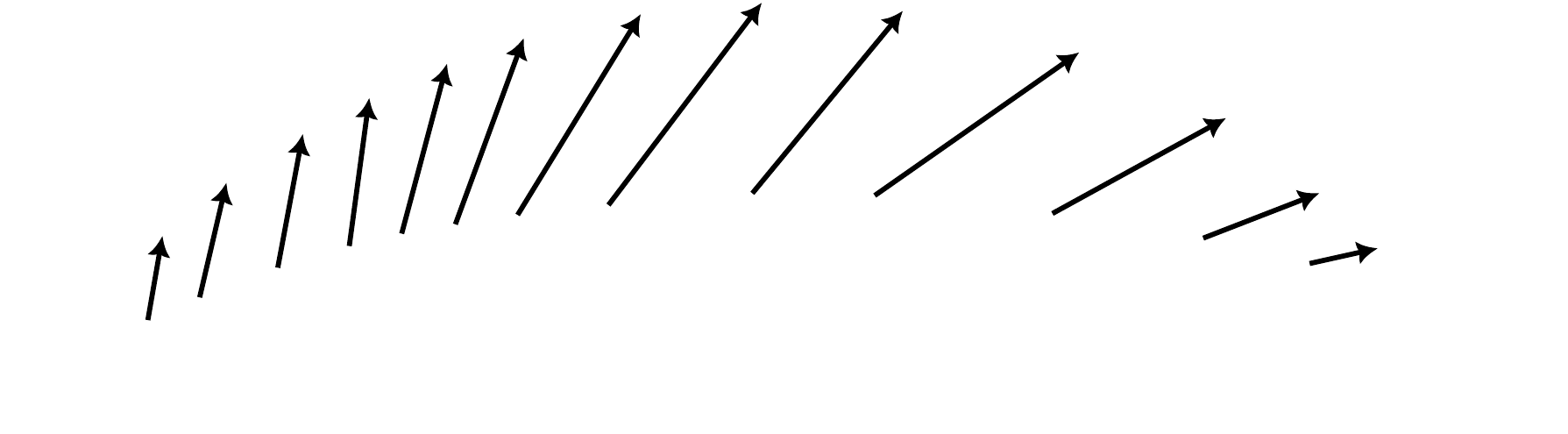}}%
    \put(-0.00332232,0.00876092){\makebox(0,0)[lt]{\lineheight{1.25}\smash{\begin{tabular}[t]{l}$p$\end{tabular}}}}%
    \put(0.92458781,0.06690133){\makebox(0,0)[lt]{\lineheight{1.25}\smash{\begin{tabular}[t]{l}$q$\end{tabular}}}}%
    \put(0,0){\includegraphics[width=\unitlength,page=2]{ConjugateLocusIllustration.pdf}}%
    \put(0.44555607,0.10645926){\color[rgb]{0,0.4,0.50196078}\makebox(0,0)[lt]{\lineheight{1.25}\smash{\begin{tabular}[t]{l}$\gamma$\end{tabular}}}}%
    \put(0.45793306,0.20902614){\makebox(0,0)[lt]{\lineheight{1.25}\smash{\begin{tabular}[t]{l}$J$\end{tabular}}}}%
    \put(0,0){\includegraphics[width=\unitlength,page=3]{ConjugateLocusIllustration.pdf}}%
  \end{picture}%
\endgroup%

    \caption{$p$ and $q$ are conjugate to each other along $\gamma$\label{fig:sample}}
\end{figure}

\begin{defn}[Conjugate points in terms of exponential map] \label{defn:ConjugatePointsInTermsOfExponentialMap}\index{conjugate points via exponential map}
    For a point $p\in M$ we say that $\mathbf{v}\in T_pM$ is a \textit{tangent conjugate point of $p$} if the derivative of the exponential map is singular at $\mathbf{v}$ i.e., $\det \left(d \left(\exp_p\right)_\mathbf{v}\right)=0$. The point $q\defeq \exp_p(\mathbf{v})$ is said to be \textit{conjugate point of $p$} along the geodesic $\gamma(t)=\exp_p(t \mathbf{v})$. 
\end{defn}
\vspace{0.3cm}
For a proof of equivalence of these three definitions we refer the reader to books on Riemannian geometry, for example see \cite{Car92}. 

\vspace{0.3cm} 
\hf The \textit{multiplicity of a conjugate point}\index{multiplicity of a conjugate point} is defined to be the nullity of $d(\exp_p)_\mathbf{v}$. If nullity is one, then we say it is first conjugate point.\index{first conjugate point} 
\vspace{0.3cm}
\begin{eg}
    Let $M=\mathbb{R}^n$ with the Euclidean metric, then there are no conjugate points along any geodesic. 
\end{eg}

\begin{eg}
    If $M=\mathbb{S}^n$ with the round metric, then any antipodal points are conjugate to each other along any great circles. In particular, south pole and north poles are conjugate to each other.
\end{eg}

\begin{eg}\label{eg:ConugatePointsonTorus}
    If $M$ is a flat torus, then there are no conjugate points along any geodesic. Recall that metric is flat if and only if the Riemann curvature tensor $R$ vanishes, which implies the Jacobi field is affine and vanishes at two points forced it to be zero everywhere. This example, in particular, proves that if $M$ is flat, then there are no conjugate points along any geodesic.
\end{eg}
\begin{eg}
    Let $M$ be the real projective plane, $\mathbb{RP}^n$, with the metric induced from $\mathbb{S}^n$. Here any point $p$ is conjugate to itself along any geodesic.   
\end{eg}

\subsection{Some results involving cut and conjugate locus}
\hfb We will present some results related to the two concepts. As most of the results are standard, we will not provide proofs. Instead, we will mention references for each.

\vspace{0.3cm}
\hf The following result is one of the most important characterization of cut locus in terms of first conjugate point.
\begin{thm}\cite[Chapter 3, Proposition 4.1]{Sak96}\label{thm:CharacterizationOfCutLocusInTermsOfConjugatePoint}
    Let $\gamma$ be a unit speed geodesic. Then $q=\gamma(t_0)$ is a cut point of $p=\gamma(0)$ along $\gamma$ if either of the following holds.
    \begin{enumerate}[(i)]
        \item The point $q=\gamma(t_0)$  is the first conjugate point of $p$ along $\gamma$.
        \item There exists at least two distance minimal geodesic joining $p$ to $q$.
    \end{enumerate}
\end{thm}
\vspace{0.3cm}
\noindent The next result is also a relation between the two loci. In particular, it states that the cut point of $p$ always comes before (if not the same) the conjugate point.

\begin{thm}\cite[Theorem 4.1]{Kob67}
    Let $\gamma$ be a unit speed geodesic starting at $p$. Let $q=\gamma(t_0)$ be the first conjugate point along $\gamma$. Then $\gamma$ is not a distance minimal geodesic beyond $q$.
\end{thm}

\vspace{0.3cm}
\noindent There is one more characterization of the cut locus in terms of number of geodesics joining the point to the cut point. For $p\in M$ we define the set $\mathrm{Se}(p)$ as 
\begin{equation}\label{eq:SeSetofPoint}
    \mathrm{Se}(p) \defeq
    \left\{q\in M ~\middle|~
        \begin{aligned}
            & \text{ there exists at least two distance }\\
            & \text{minimal geodesics joining $p$ to $q$}
        \end{aligned}
    \right\}
\end{equation}
Note that $\mathrm{Se}(p)\subseteq \cutn[p]$. Franz-Erich Wolter in 1979 showed that the closure of $\mathrm{Se}(p)$ is the cut locus. 
\begin{thm}\cite[Theorem 1]{Wol79}\label{thm:ClosureOfSeisCup}
    Let $M$ be a complete Riemannian manifold and $p$ be any point in $M$. Then 
    \begin{displaymath}
        \overline{\mathrm{Se}(p)} = \cutn[p].
    \end{displaymath}
\end{thm} 
\bigskip
The above theorem, in particular, shows that the cut locus of a point is a closed set. He also proved that the distance squared function from the point $p$ is not differentiable on the set $\mathrm{Se}$.
\begin{thm}\cite[Lemma 1]{Wol79}
    Let $d^2(p,\cdot)$ denotes the square of the distance from the point $p$. Let $q\in \mathrm{Se}(p)$ and let $\gamma_1$ and $\gamma_2$ be two distance minimal geodesics joining $p$ to $q$. Then the directional derivative of $d^2(p,\cdot)$ does not exist at $q$ in the direction of $\gamma_i,~i=1,2$.
\end{thm}
\bigskip 
\noindent For some special point of the cut locus of $p$ we can improve \Cref{thm:CharacterizationOfCutLocusInTermsOfConjugatePoint}.
\begin{thm}\cite[Theorem 4.4]{Kob67}\label{thm:ExistenseOfClosedGeodesic-1}
    Let $q$ be a cut point of $p$ and we assume that it is the closest point of $p$. Then $q$ is either conjugate to $p$ along a minimal geodesic joining these two points, or $q$ is the mid-point of a closed geodesic starting and ending at $p$.
\end{thm}
\bigskip 
We can even make the above theorem sharper if we provide an additional condition. 

\vspace{0.3cm}
\begin{thm}\cite[Theorem 4.5]{Kob67}\label{thm:ExistenseOfClosedGeodesic-2}
    Let $p$ be any point in $M$ such that $d(p,\cutn[p])$ is the smallest and $q$ be any cut point closest to $p$. Then either $q$ is conjugate to $p$  with respect to a distance minimal geodesic joining $p$ and $q$ or $q$ is the mid-point of a closed geodesic starting and ending smoothly at $p$.
\end{thm}

	\chapter[Cut locus of a submanifold]{Cut locus of a submanifold}
\minitoc

\hf The cut locus of a point plays an important role in analyzing the local structure of a Riemannian manifold $M$. In the last chapter we had studied cut locus of a point, conjugate locus and some of their properties. Similarly, one can ask about the notion of cut locus for a non-empty subset of $M$. In order to give a similar definition, we need to first define distance minimal geodesics joining a point $p\in M$ to a subset $N\subseteq M$. If $q\in N$ such that the geodesic is distance minimal geodesic joining $p$ and $q$, and it minimizes the distance between the set $N$ and the point $p$ then we call such a geodesic a distance minimal geodesic. Now the cut locus of $N$ consists of all points $q\in M$ such that there exists a distance minimal geodesic which fails to be minimal beyond $q$. Conjugate locus is termed as focal locus if we replace point with submanifold. In this chapter we will study the cut locus of a subset, in particular, a submanifold. We will motivate the results based on some examples and the proofs will be discussed in the subsequent chapters.

\section{Cut locus of submanifolds}\label{sec:cutLocusOfSubmanifolds}
\hfb In order to have a definition of the cut locus for a submanifold (or a subset), we need to generalize the notion of a minimal geodesic.
\begin{defn}\label{distmin}\index{distance minimal geodesic}\index{$N$-geodesic}
    A geodesic $\gamma $ is called a \emph{distance minimal geodesic} joining $N$ to $p$ if there exists $q\in N$ such that $\gamma$ is a minimal geodesic joining $q$ to $p$ and $l(\gamma)= d(p,N) $. We will refer to such geodesics as \textit{$N$-geodesics}.\index{$N$-geodesic}
\end{defn}
\vspace{0.3cm}
\noindent If $N$ is an embedded submanifold, then an $N$-geodesic is necessarily orthogonal to $N$. This follows from the first variational principle. We are ready to define the cut locus for $N\subset M$. 

\begin{defn}[Cut locus of a subset]\index{cut locus of a subset}\label{cutlocus1}
    Let $M$ be a Riemannian manifold and $N$ be any non-empty subset of $M.$ If $\cutn$ denotes the \emph{cut locus of $N$}, then we say that $q\in \cutn $ if and only if there exists a distance minimal geodesic joining $N$ to $q$ such that any extension of it beyond $q$ is not a distance minimal geodesic.
\end{defn}
\bigskip
\begin{eg}\label{eg:Example-cutLocusOfx-Axis}
    Let $M=\mathbb{R}^2$ with the Euclidean metric and $N$ be the $x$-axis. Then the cut locus of $N$ will be empty. If we shoot any geodesic, which are straight lines, perpendicular to $x$-axis, these will never fail to be distance minimal and hence the cut locus will be empty.
    \begin{figure}[!htb]
        \centering
    \def\svgwidth{0.5\columnwidth}
    %% Creator: Inkscape 1.2 (1:1.2.1+202207142221+cd75a1ee6d), www.inkscape.org
%% Accompanies image file 'Example-CutLocusOfx-Axis.pdf' (pdf, eps, ps)
\begingroup%
  \makeatletter%
  \providecommand\color[2][]{%
    \errmessage{(Inkscape) Color is used for the text in Inkscape, but the package 'color.sty' is not loaded}%
    \renewcommand\color[2][]{}%
  }%
  \providecommand\transparent[1]{%
    \errmessage{(Inkscape) Transparency is used (non-zero) for the text in Inkscape, but the package 'transparent.sty' is not loaded}%
    \renewcommand\transparent[1]{}%
  }%
  \providecommand\rotatebox[2]{#2}%
  \newcommand*\fsize{\dimexpr\f@size pt\relax}%
  \newcommand*\lineheight[1]{\fontsize{\fsize}{#1\fsize}\selectfont}%
  \ifx\svgwidth\undefined%
    \setlength{\unitlength}{422.340909bp}%
    \ifx\svgscale\undefined%
      \relax%
    \else%
      \setlength{\unitlength}{\unitlength * \real{\svgscale}}%
    \fi%
  \else%
    \setlength{\unitlength}{\svgwidth}%
  \fi%
  \global\let\svgwidth\undefined%
  \global\let\svgscale\undefined%
  \makeatother%
  \begin{picture}(1,0.44860352)%
    \lineheight{1}%
    \setlength\tabcolsep{0pt}%
    \put(0,0){\includegraphics[width=\unitlength,page=1]{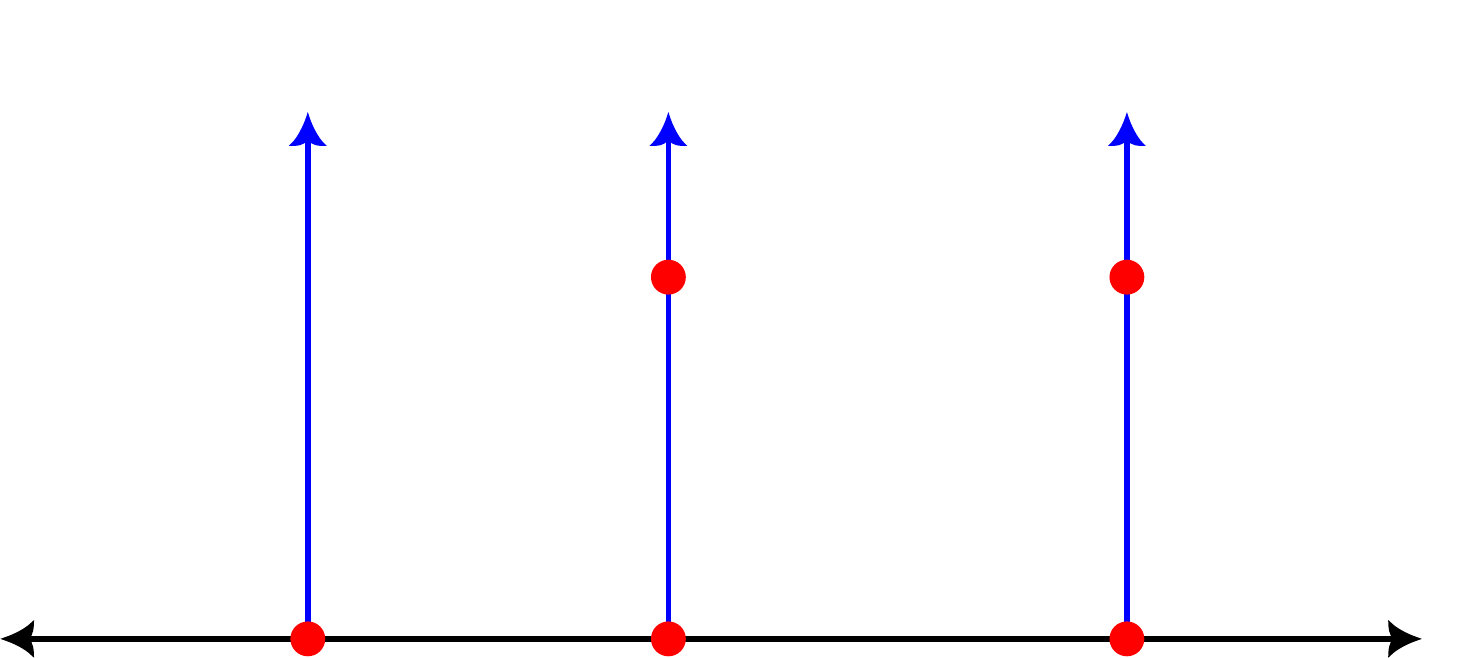}}%
    \put(0.89333997,0.03362074){\makebox(0,0)[lt]{\lineheight{1.25}\smash{\begin{tabular}[t]{l}$N$\end{tabular}}}}%
    \put(0,0){\includegraphics[width=\unitlength,page=2]{Example-CutLocusOfx-Axis.pdf}}%
    \put(0.09519414,0.40653885){\makebox(0,0)[lt]{\lineheight{1.25}\smash{\begin{tabular}[t]{l}Geodesics never fails to minimize\end{tabular}}}}%
  \end{picture}%
\endgroup%

        \caption{Cut locus of $x$-axis in $\mathbb{R}^2$ }
    \end{figure}
\end{eg}
\vspace{0.3cm}

\begin{eg}\label{eg:Example-cutLocusOfCircle}
    Let $M=\mathbb{R}^{n+1}$ with the Euclidean metric and $N=\mathbb{S}^n$. Then the cut locus will be the center of the sphere. Note that if we shoot a minimal geodesic from any point of $\mathbb{S}^n$, then it fails its minimizing property beyond the origin (see \Cref{fig:Example-CutLocusCircle}). Hence, $\bf{0}$ is a cut point. To see this is the only cut point, we start with any point $\mathbf{a}$ other than origin. Consider a distance minimal geodesic $\gamma$ starting at $\mathbb{S}^n$ to $\mathbf{a}$.  
    \begin{equation*}\label{eq:lineJoiningTwoPoints}
        \gamma(t) = (1-t) \frac{\mathbf{a}}{\left\|\mathbf{a}\right\|}+ t \mathbf{a},~0\le t\le 1.
    \end{equation*}
    Note that 
    \begin{displaymath}
        l(\gamma)=d(P,Q) = \left|1-\left\|\mathbf{a}\right\|\right|.
    \end{displaymath}
    Let $\epsilon=\frac{\left\|\mathbf{a}\right\|}{2}$. Consider the point $R=\gamma(1+\epsilon)$. We have
    \begin{align*}
        l(\gamma)\big|_{[0,1+\epsilon]} = d(P,R).
    \end{align*}
    \begin{figure}[!htbp]
        \centering
        \begin{subfigure}{.45\textwidth}
    \def\svgwidth{0.9\columnwidth}
    %% Creator: Inkscape 1.2 (1:1.2.1+202207142221+cd75a1ee6d), www.inkscape.org
%% Accompanies image file 'Example-CutLocusCircle.pdf' (pdf, eps, ps)
\begingroup%
  \makeatletter%
  \providecommand\color[2][]{%
    \errmessage{(Inkscape) Color is used for the text in Inkscape, but the package 'color.sty' is not loaded}%
    \renewcommand\color[2][]{}%
  }%
  \providecommand\transparent[1]{%
    \errmessage{(Inkscape) Transparency is used (non-zero) for the text in Inkscape, but the package 'transparent.sty' is not loaded}%
    \renewcommand\transparent[1]{}%
  }%
  \providecommand\rotatebox[2]{#2}%
  \newcommand*\fsize{\dimexpr\f@size pt\relax}%
  \newcommand*\lineheight[1]{\fontsize{\fsize}{#1\fsize}\selectfont}%
  \ifx\svgwidth\undefined%
    \setlength{\unitlength}{270.62126917bp}%
    \ifx\svgscale\undefined%
      \relax%
    \else%
      \setlength{\unitlength}{\unitlength * \real{\svgscale}}%
    \fi%
  \else%
    \setlength{\unitlength}{\svgwidth}%
  \fi%
  \global\let\svgwidth\undefined%
  \global\let\svgscale\undefined%
  \makeatother%
  \begin{picture}(1,0.79078738)%
    \lineheight{1}%
    \setlength\tabcolsep{0pt}%
    \put(0,0){\includegraphics[width=\unitlength,page=1]{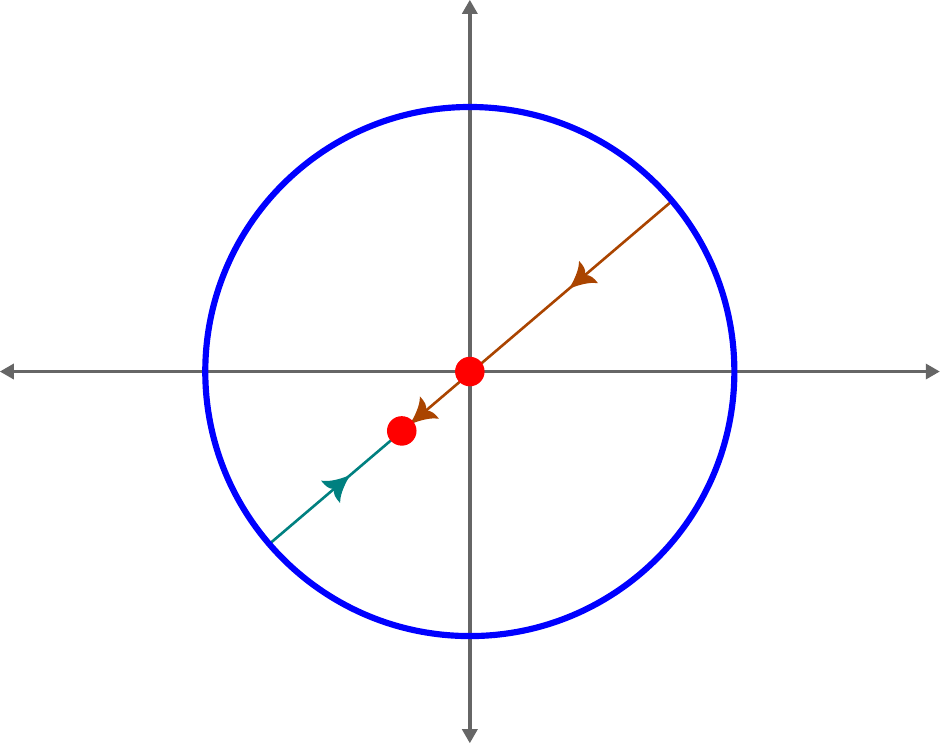}}%
    \put(0.53162493,0.33206525){\makebox(0,0)[lt]{\lineheight{1.25}\smash{\begin{tabular}[t]{l}$(0,0)$\end{tabular}}}}%
  \end{picture}%
\endgroup%

            \caption{$(0,0)$ is a cut point}
            \label{fig:Example-CutLocusCircle}
        \end{subfigure}
        \begin{subfigure}{.45\textwidth}
    \def\svgwidth{0.9\columnwidth}
    %% Creator: Inkscape 1.2 (1:1.2.1+202207142221+cd75a1ee6d), www.inkscape.org
%% Accompanies image file 'Example-CutLocusCircle-02.pdf' (pdf, eps, ps)
\begingroup%
  \makeatletter%
  \providecommand\color[2][]{%
    \errmessage{(Inkscape) Color is used for the text in Inkscape, but the package 'color.sty' is not loaded}%
    \renewcommand\color[2][]{}%
  }%
  \providecommand\transparent[1]{%
    \errmessage{(Inkscape) Transparency is used (non-zero) for the text in Inkscape, but the package 'transparent.sty' is not loaded}%
    \renewcommand\transparent[1]{}%
  }%
  \providecommand\rotatebox[2]{#2}%
  \newcommand*\fsize{\dimexpr\f@size pt\relax}%
  \newcommand*\lineheight[1]{\fontsize{\fsize}{#1\fsize}\selectfont}%
  \ifx\svgwidth\undefined%
    \setlength{\unitlength}{270.62126917bp}%
    \ifx\svgscale\undefined%
      \relax%
    \else%
      \setlength{\unitlength}{\unitlength * \real{\svgscale}}%
    \fi%
  \else%
    \setlength{\unitlength}{\svgwidth}%
  \fi%
  \global\let\svgwidth\undefined%
  \global\let\svgscale\undefined%
  \makeatother%
  \begin{picture}(1,0.79078738)%
    \lineheight{1}%
    \setlength\tabcolsep{0pt}%
    \put(0,0){\includegraphics[width=\unitlength,page=1]{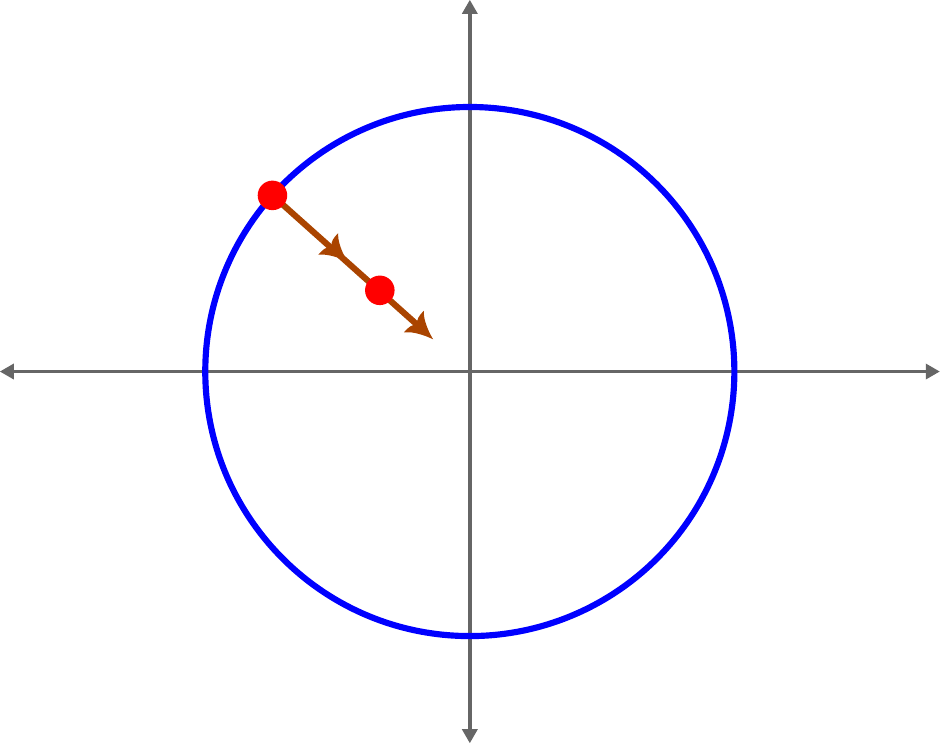}}%
    \put(0.26202462,0.60877853){\makebox(0,0)[lt]{\lineheight{1.25}\smash{\begin{tabular}[t]{l}$P$\end{tabular}}}}%
    \put(0.4019754,0.50614482){\makebox(0,0)[lt]{\lineheight{1.25}\smash{\begin{tabular}[t]{l}$\bf{a}$\end{tabular}}}}%
    \put(0.31436806,0.47474264){\color[rgb]{0.66666667,0.26666667,0}\makebox(0,0)[lt]{\lineheight{1.25}\smash{\begin{tabular}[t]{l}$\gamma$\end{tabular}}}}%
    \put(0,0){\includegraphics[width=\unitlength,page=2]{Example-CutLocusCircle-02.pdf}}%
    \put(0.47292787,0.41736657){\makebox(0,0)[lt]{\lineheight{1.25}\smash{\begin{tabular}[t]{l}$R$\end{tabular}}}}%
  \end{picture}%
\endgroup%

            \caption{No other points are cut points}
        \end{subfigure}
        \caption{Cut locus of $\mathbb{S}^1$ in $\mathbb{R}^2$}
    \end{figure}
    \noindent Now we will show that $d(\mathbb{S}^n,R)$ is same as the length of $\gamma$. Note that $R=(1+\epsilon)\mathbf{a}-\frac{\epsilon \mathbf{a}}{\left\|\mathbf{a}\right\|}$ which simplifies to $\frac{\mathbf{a}}{2}(1-\left\|\mathbf{a}\right\|) $. Note that 
    \begin{align*}
        d^2(\mathbb{S}^n,R) & = \inf_{\mathbf{x}\in \mathbb{S}^n} d^2 \left(\mathbf{x},\frac{\mathbf{a}}{2}(1-\left\|\mathbf{a}\right\|)\right) 
        \\[1ex]
        & = \min \left\{\left\|\mathbf{x}-{\mathbf{v}}\right\|^2:\|\mathbf{x}\|^2=1\right\},~\mathbf{v}=\frac{\mathbf{a}}{2}(1-\left\|\mathbf{a}\right\|).
    \end{align*}
    Set
    \begin{align*}
        f(\mathbf{x})=\left\|\mathbf{x}-\mathbf{v}\right\|^2,~g(\mathbf{x})= \left\|\mathbf{x}\right\|^2-1.
    \end{align*}
    So we want to minimize $f$ such that $g(\mathbf{x})=0$. 
    \begin{align*}
       \nabla f(\mathbf{x}) - \lambda \nabla g(\mathbf{x})= 0 & \implies (\mathbf{x}-\mathbf{v}) - \lambda \mathbf{x} = 0  \implies \mathbf{x} = \dfrac{\mathbf{v}}{1-\lambda}.
    \end{align*}
    Note that  the above quantity is well defined as $\mathbf{v}\neq 0$. Now we will use the given constrain,
    \begin{align*}
        \left\|\mathbf{x}\right\|=1 & \implies \left\|\mathbf{v}\right\|=|1-\lambda| \\
        & \implies \lambda = 1 \pm \left\|\mathbf{v}\right\| \\
        & \implies \mathbf{x} = \pm \frac{\mathbf{v}}{\left\|\mathbf{v}\right\|} = \pm \frac{\mathbf{a}}{\left\|\mathbf{a}\right\|}.
    \end{align*}
    The point $\frac{\mathbf{a}}{\left\|\mathbf{a}\right\|}$ corresponds to the minima and note that $P=\frac{\mathbf{a}}{\left\|\mathbf{a}\right\|}$. This proves the claim.
\end{eg}

\begin{eg}\label{eg:cutLocusOfS0}
    Let $M=\mathbb{S}^2$ with the round metric and $N=\mathbb{S}^0=\{\mathbf{e}_3,-\mathbf{e}_3\}$, where $\mathbf{e}_3=(0,0,1)$. We claim that the cut locus is the equator, $\{(x,y,0):x^2+y^2=1\}$. Note that if $\gamma$ is an $N$-geodesic (great circle) starting at the North Pole, then it remains distance minimal until it hits the equator circle. As soon as it goes beyond the circle, we can find another geodesic $\eta$ from the South Pole which is shorter and hence $\gamma$ is no longer distance minimal, see \Cref{fig:Example-CutLocus-SphereS0}. 
    \begin{figure}[!htb]
        \centering
    \def\svgwidth{0.45\columnwidth}
    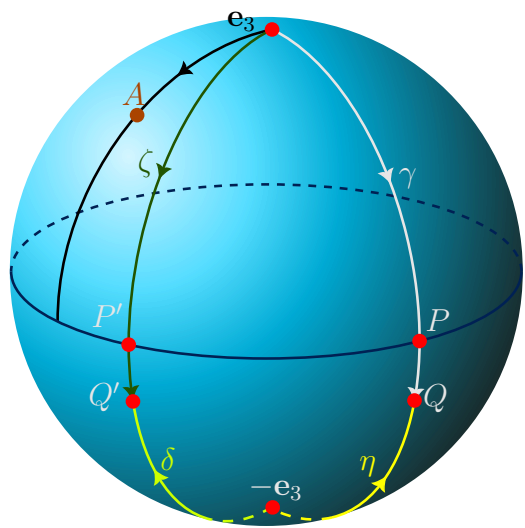

        \caption{Cut locus of $\mathbb{S}^0$ in $2$-sphere \label{fig:Example-CutLocus-SphereS0}}
    \end{figure}

    \noindent Therefore, the equator circle is in the cut locus. We also note that any other point not on the equator is either on the top or bottom hemisphere. In either of the cases, a minimal geodesic does not fail its distance minimal property beyond the point. So the cut locus is precisely the circle with $z=0$. By the same argument it follows that the cut locus of $\mathbb{S}^0$ in $\mathbb{S}^n$ is
    \begin{displaymath}
        \mathbb{S}^{n-1}= \left\{(x_0,\cdots,x_{n-1},0): x_0^2+\cdots+x_{n-1}^2=1\right\}.
    \end{displaymath}
\end{eg}

\begin{eg}[Cut locus of equator in $2$-sphere]
    Let $M=\mathbb{S}^2$ with the round metric and $N=\mathbb{S}^1=\{(x,y,0):x^2+y^2=1\}$. The cut locus of $N$ is $\mathbb{S}^0=\left\{\mathbf{e}_3,-\mathbf{e}_3\right\}$. The argument is similar as above.
\end{eg}

% \textcolor{red}{Finding the cut locus of the equator. We claim that the cut locus is the two poles. It is clear that both points are cut points. In order to show that no other points are cut points, take $\mathbf{v} \in \mathbb{S}^2$  such that $\mathbf{v}\notin \mathbb{S}^1$, which is the equator. Note that any distance minimal geodesic, a great circle, must pass through two poles. If we shoot a distance minimal geodesic from $\mathbb{S}^1$ to $\mathbf{v}$, then it must pass through the two poles and hence the circle will be unique. Note that the geodesic $\gamma$ is distance minimal beyond $\mathbf{v}$ as it remains distance minimal till the north pole, see \Cref{fig:Example-CutLocus-SphereS1}.}
%     \begin{figure}[!htpb]
%         \centering
%         \incfig[0.65]{Example-CutLocus-SphereS1}
%         \caption{Cut locus of the equator in $\mathbb{S}^2$ \label{fig:Example-CutLocus-SphereS1}}
%     \end{figure}

\begin{eg}[Join induced by cut locus]\label{join}
    Let $\mathbb{S}_i^k \hookrightarrow \mathbb{S}^n$ denote the embedding of the $k$-sphere in the first $k+1$ coordinates while $\mathbb{S}^{n-k-1}_l$ denote the embedding of the $(n-k-1)$-sphere in the last $n-k$ coordinates. It can be seen that $\textup{Cu}(\mathbb{S}_i^k)=\mathbb{S}_l^{n-k-1}$. In fact, starting at a point $p\in \mathbb{S}^k_i$ and travelling along a unit speed geodesic in a direction normal to $T_p\mathbb{S}^k_i$, we obtain a cut point at a distance $\pi/2$ from $\mathbb{S}^k_i$.
    \begin{figure}[!htpb]
        \centering
    \def\svgwidth{0.45\columnwidth}
    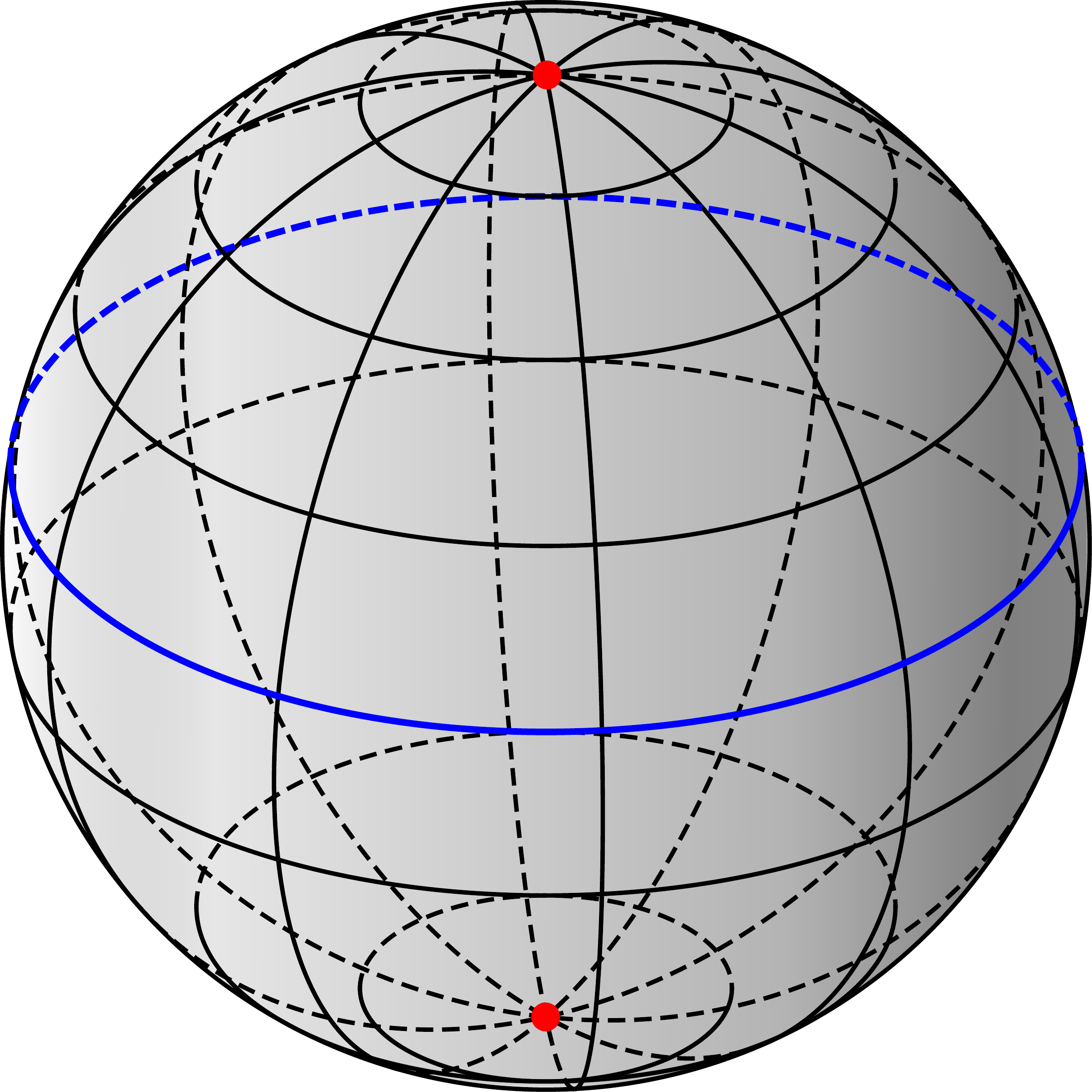

        \caption{The cut locus of the equator in $\mathbb{S}^2$\label{fig:S2_equator_cutlocus}}
    \end{figure}
    
    \hf Moreover, in this case $\textup{Cu}(\mathbb{S}_l^{n-k-1})=\mathbb{S}_i^{k}$ and the $n$-sphere $\mathbb{S}^n$ can be expressed as the union of geodesic segments joining $\mathbb{S}_i^k$ to $\mathbb{S}_l^{n-k-1}$. This is a geometric variant of the fact that the $n$-sphere is the (topological) join of $\mathbb{S}^k$ and $\mathbb{S}^{n-k-1}$. We also observe that $\mathbb{S}^n - \mathbb{S}_l^{n-k-1}$ deforms to $\mathbb{S}_i^k$ while $\mathbb{S}^n - \mathbb{S}_i^{k}$ deforms to $\mathbb{S}_l^{n-k-1}$. 
    
    \hf In our example, let $\nu_i^{n-k}$ and $\nu_l^{k+1}$ denote the normal bundles of $\mathbb{S}_i^k$ and $\mathbb{S}_l^{n-k-1}$ respectively. We may express $\mathbb{S}^n$ as the union of normal disk bundles $D(\nu_i)$ and $D(\nu_l)$. These disk bundles are trivial and are glued along their common boundary $\mathbb{S}^k_i\times \mathbb{S}_l^{n-k-1}$ to produce $\mathbb{S}^n$. Moreover, $\mathbb{S}^k_i$ is an analytic submanifold of the real analytic Riemannian manifold $\mathbb{S}^n$ with the round metric. There is a generalization of this phenomenon  \cite[Lemmas 1.3-1.5, Theorem 3.1]{Omo68}.
    \begin{thm}[Omori 1968]
        Let $M$ be a compact, connected, real analytic Riemannian manifold which has an analytic submanifold $N$ such that the cut  point of $N$ with respect to every geodesic, which starts from $N$ and whose initial  direction is orthogonal to $N$ has a constant distance $\pi$ from $N$. Then $N'=\cutn$ is an analytic submanifold and $M$ has a decomposition $M = DN\cup_{\varphi} DN'$, where $DN,  DN'$ are normal disk bundles of $N, N'$ respectively.
    \end{thm}
\end{eg}

\subsection{Separating set}\label{subsec:separatingSet}
\hfb In all the examples in the previous section, we observed the cut locus of any submanifold $N$ is same as the set of all points which has at least two minimal geodesics joining $N$ to the point. This leads to the following definition.

\begin{defn}[Separating set]\label{defn:SeparatingSet}\index{separating set}
    Let $N$ be a subset of a Riemannian manifold $M$. The set $\sen$, called the \textit{separating set}, consists of all points $q\in M$ such that at least two distance minimal geodesics from $N$ to $q$ exist.
\end{defn}
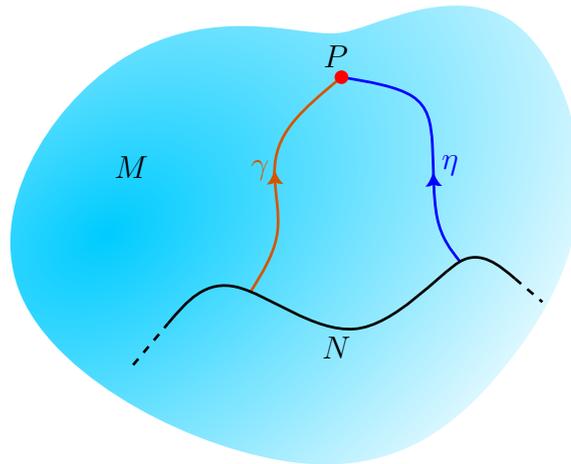
\begin{figure}[!htpb]
    \centering
    \def\svgwidth{0.5\columnwidth}
    %% Creator: Inkscape 1.2 (1:1.2.1+202207142221+cd75a1ee6d), www.inkscape.org
%% Accompanies image file 'SeparatingSetDescription.pdf' (pdf, eps, ps)
\begingroup%
  \makeatletter%
  \providecommand\color[2][]{%
    \errmessage{(Inkscape) Color is used for the text in Inkscape, but the package 'color.sty' is not loaded}%
    \renewcommand\color[2][]{}%
  }%
  \providecommand\transparent[1]{%
    \errmessage{(Inkscape) Transparency is used (non-zero) for the text in Inkscape, but the package 'transparent.sty' is not loaded}%
    \renewcommand\transparent[1]{}%
  }%
  \providecommand\rotatebox[2]{#2}%
  \newcommand*\fsize{\dimexpr\f@size pt\relax}%
  \newcommand*\lineheight[1]{\fontsize{\fsize}{#1\fsize}\selectfont}%
  \ifx\svgwidth\undefined%
    \setlength{\unitlength}{347.61040635bp}%
    \ifx\svgscale\undefined%
      \relax%
    \else%
      \setlength{\unitlength}{\unitlength * \real{\svgscale}}%
    \fi%
  \else%
    \setlength{\unitlength}{\svgwidth}%
  \fi%
  \global\let\svgwidth\undefined%
  \global\let\svgscale\undefined%
  \makeatother%
  \begin{picture}(1,0.81136593)%
    \lineheight{1}%
    \setlength\tabcolsep{0pt}%
    \put(0,0){\includegraphics[width=\unitlength,page=1]{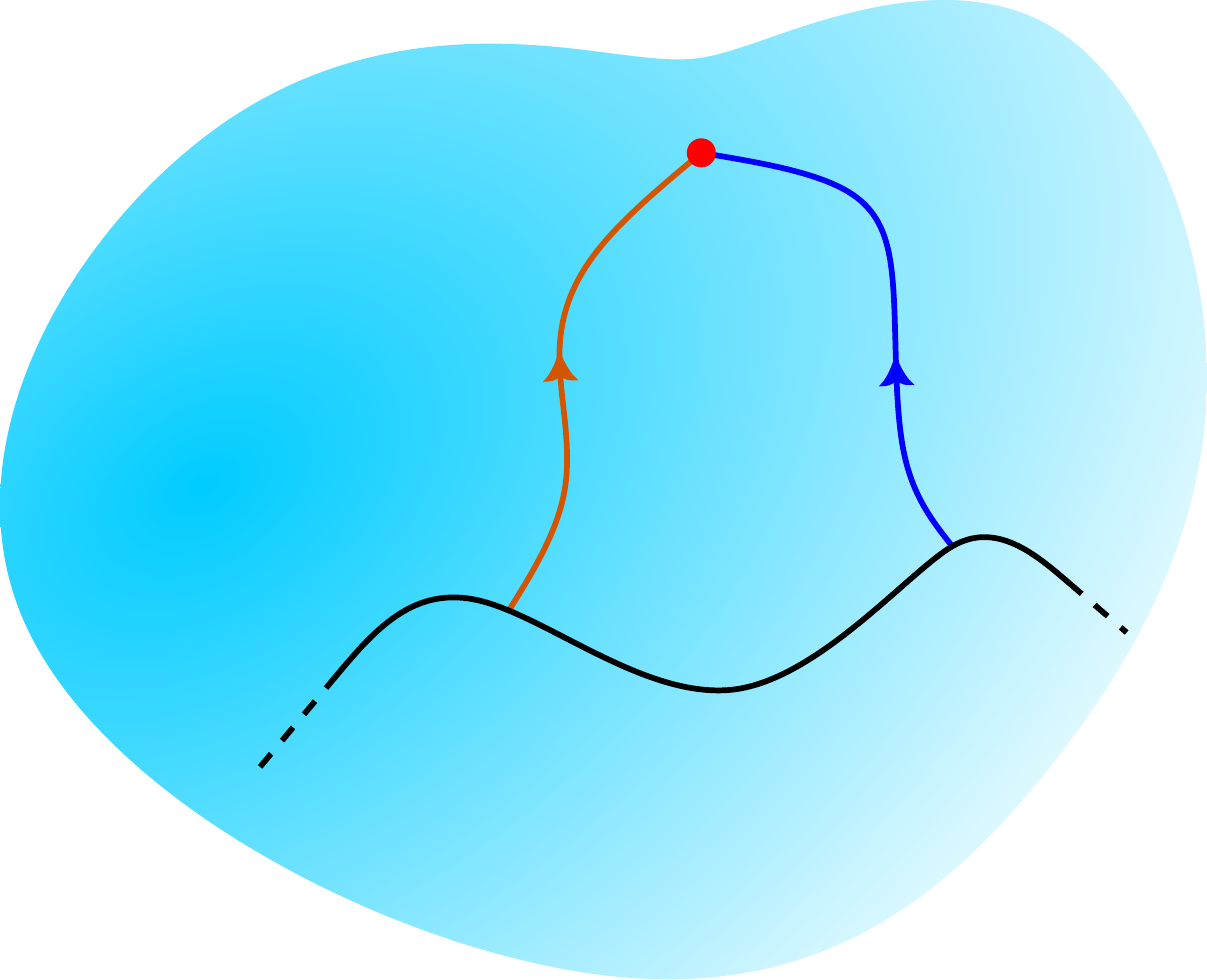}}%
    \put(0.18256893,0.50747148){\makebox(0,0)[lt]{\lineheight{1.25}\smash{\begin{tabular}[t]{l}$M$\end{tabular}}}}%
    \put(0.54743335,0.18709238){\makebox(0,0)[lt]{\lineheight{1.25}\smash{\begin{tabular}[t]{l}$N$\end{tabular}}}}%
    \put(0.42248357,0.51284121){\color[rgb]{0.83137255,0.33333333,0}\makebox(0,0)[lt]{\lineheight{1.25}\smash{\begin{tabular}[t]{l}$\gamma$\end{tabular}}}}%
    \put(0.75686376,0.52296249){\color[rgb]{0,0,1}\makebox(0,0)[lt]{\lineheight{1.25}\smash{\begin{tabular}[t]{l}$\eta$\end{tabular}}}}%
    \put(0.54960043,0.70292147){\makebox(0,0)[lt]{\lineheight{1.25}\smash{\begin{tabular}[t]{l}$P$\end{tabular}}}}%
  \end{picture}%
\endgroup%

    \caption{Separating set of $N$ \label{fig:SeparatingSet}}
\end{figure}

\vspace{0.3cm}
\hf The following example shows that for a given submanifold $N\subseteq M$, the separating set need not be same as the cut locus.

\begin{eg}[Cut locus of ellipse]\label{eg:CutLocusOfEllipse}
    Let $M=\mathbb{R}^2$ with the Euclidean metric and $N= \left\{(x,y):\frac{x^2}{a^2}+\frac{y^2}{b^2}=1\right\}$ for some non-zero real numbers $a$ and $b$ with $a\neq b$.
    \begin{figure}[!htb]
        \centering
    \def\svgwidth{0.5\columnwidth}
    %% Creator: Inkscape 1.2 (1:1.2.1+202207142221+cd75a1ee6d), www.inkscape.org
%% Accompanies image file 'Example-CutLocusEllipse.pdf' (pdf, eps, ps)
\begingroup%
  \makeatletter%
  \providecommand\color[2][]{%
    \errmessage{(Inkscape) Color is used for the text in Inkscape, but the package 'color.sty' is not loaded}%
    \renewcommand\color[2][]{}%
  }%
  \providecommand\transparent[1]{%
    \errmessage{(Inkscape) Transparency is used (non-zero) for the text in Inkscape, but the package 'transparent.sty' is not loaded}%
    \renewcommand\transparent[1]{}%
  }%
  \providecommand\rotatebox[2]{#2}%
  \newcommand*\fsize{\dimexpr\f@size pt\relax}%
  \newcommand*\lineheight[1]{\fontsize{\fsize}{#1\fsize}\selectfont}%
  \ifx\svgwidth\undefined%
    \setlength{\unitlength}{482.82491037bp}%
    \ifx\svgscale\undefined%
      \relax%
    \else%
      \setlength{\unitlength}{\unitlength * \real{\svgscale}}%
    \fi%
  \else%
    \setlength{\unitlength}{\svgwidth}%
  \fi%
  \global\let\svgwidth\undefined%
  \global\let\svgscale\undefined%
  \makeatother%
  \begin{picture}(1,0.69454453)%
    \lineheight{1}%
    \setlength\tabcolsep{0pt}%
    \put(0,0){\includegraphics[width=\unitlength,page=1]{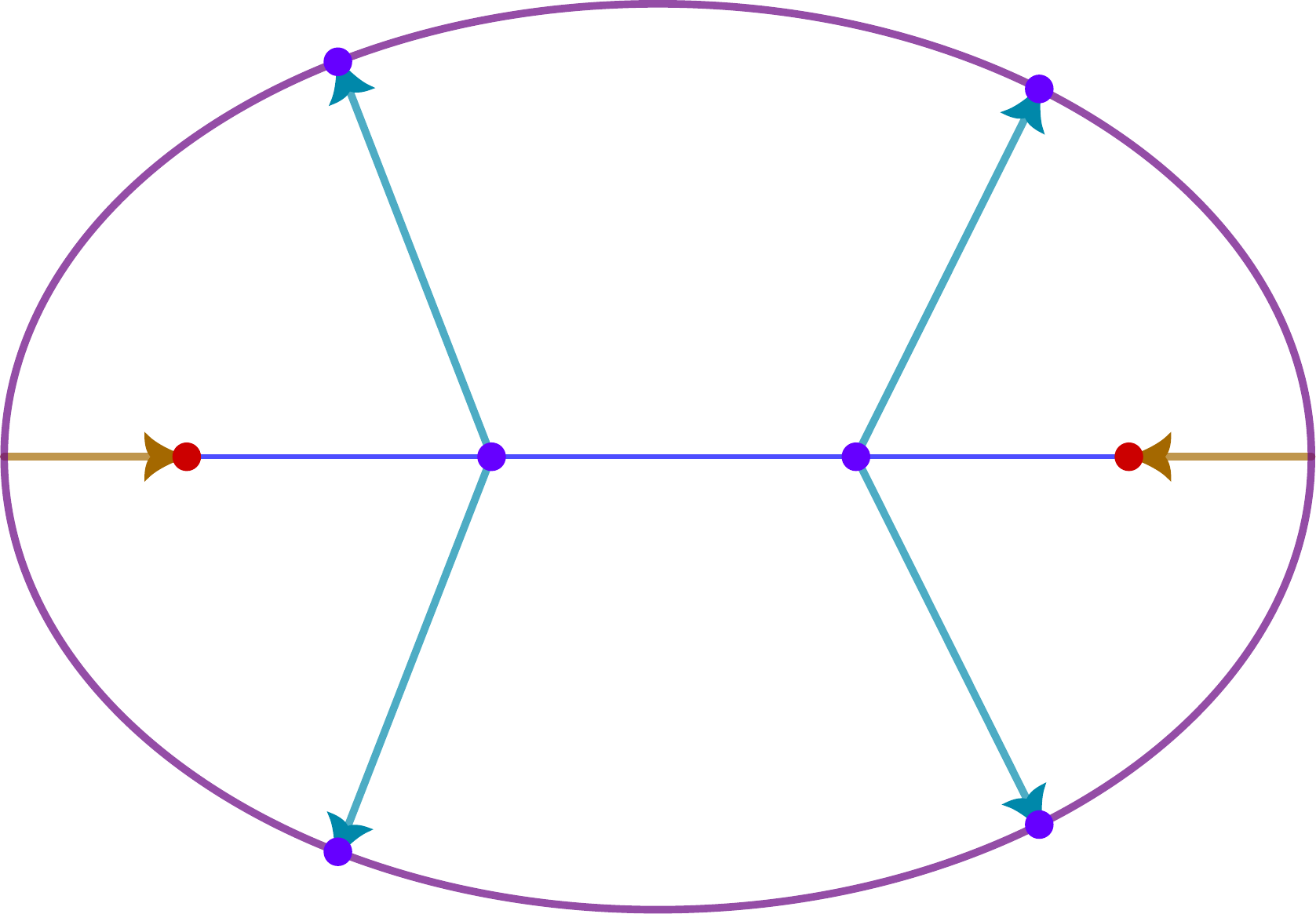}}%
    \put(0.11885568,0.2921624){\makebox(0,0)[lt]{\lineheight{1.25}\smash{\begin{tabular}[t]{l}$A$\end{tabular}}}}%
    \put(0.84077071,0.28535783){\makebox(0,0)[lt]{\lineheight{1.25}\smash{\begin{tabular}[t]{l}$B$\end{tabular}}}}%
    \put(0.36148978,0.2921624){\makebox(0,0)[lt]{\lineheight{1.25}\smash{\begin{tabular}[t]{l}$C$\end{tabular}}}}%
  \end{picture}%
\endgroup%

        \caption{$\cutn$ and $\sen$ \label{fig:Example-CutLocusEllipse}}
    \end{figure}
    Let $A=(-a,0)$ and $B=(a,0)$ be two foci of the ellipse. Note that for any point $C=(x,0)$ with $x\in (-a,a)$, we have two $N$-geodesics joining $N$ to $C$. Hence, all the points are separating point (see \Cref{fig:Example-CutLocusEllipse}). However, the two foci are not separating points, but they are in the cut locus. So $\sen \neq \cutn$. 
\end{eg}

\vspace{0.3cm}
\noindent Note that $\sen\subset \cutn$. Although the sets $\sen$ and $\cutn$ are not same, in general, we can ask whether including the limit points of $\sen$ make them equal. In the next chapter, we will see that indeed this is the case, and we have $\overline{\sen}=\cutn$ (\Cref{thm:SeClosureIsCutLocus}). \label{Page:SeClosureIsCu} This, in particular, proves that cut locus is a closed set. In general, the cut locus of a subset need not be closed, as illustrated by the following example \cite{TaSa16}.
\begin{eg}[Sabau-Tanaka 2016]
    Consider $\rbb^2$ with the Euclidean inner product. Let $\curlybracket{\theta_n}$, with $\theta_1\in (0,\pi)$, be a decreasing sequence converging to $0$. Let $\overline{B(\mathbf{0},1)}$ be the closed unit ball centered at $(0,0)$. Suppose $B_n:=B(q_n,1)$ is the open ball with radius $1$ and centered at $q_n$. We have chosen $q_n$ such that it does not belong to $\overline{B(\mathbf{0},1)}$ and denotes the center of the circle passing through $p_n=(\cos\theta_n,\sin\theta_n)$ and $p_{n+1} = (\cos\theta_{n+1},\sin\theta_{n+1})$. Define $ N\subset \rbb^2$ by
    \begin{displaymath}
        N \defeq \overline{B(\mathbf{0},1)}\setminus \cup_{n=1}^\infty B(q_n,1).
    \end{displaymath}
    \begin{figure}[!htpb]
        \centering
        \includegraphics[scale=0.7]{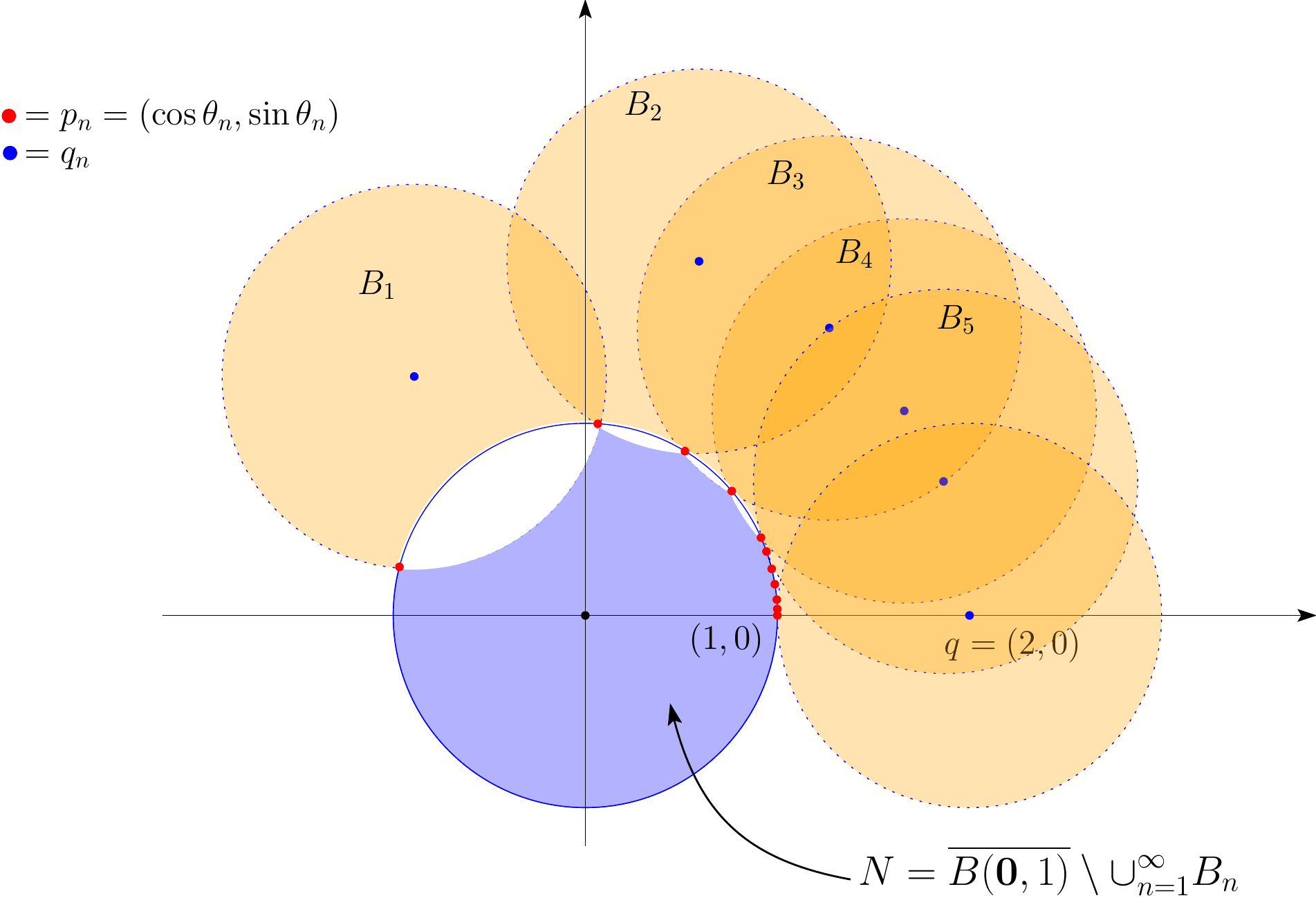}
        \caption{Cut locus need not be closed}
        \label{fig:CutLocusNotClosed}
    \end{figure}
    
    \noindent Note that $N$ is a closed set and the sequence $\curlybracket{q_n}$ of cut points of $N$ converges to the point $(2,0)$. However, $(2,0)$ is not a cut point of $N$.
\end{eg}

\hf Using the characterization of cut locus in terms of the separating set, we will list some more examples. Most of the justification is provided by the help of pictures. 

\begin{eg}[Cut locus of $k$ points on the unit circle]
    Let $M=\mathbb{S}^1$ with the round metric and $N=\{A_1,A_2,A_3\}$. Then the cut locus will be $\{B_{12},B_{23},B_{31}\}$, see \Cref{fig:Example-CutLocus-CircleThreePoints}.
    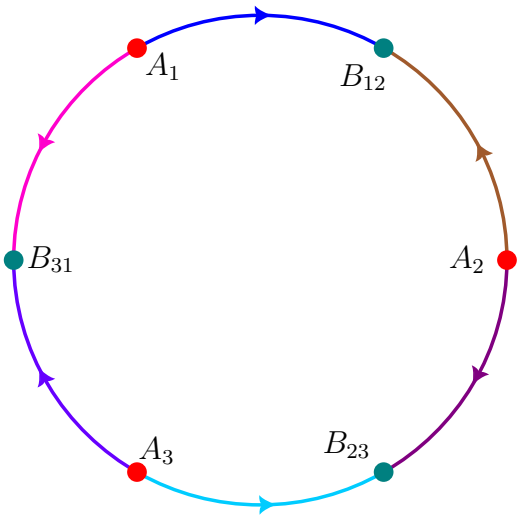
\begin{figure}[!htpb]
        \centering
    \def\svgwidth{0.45\columnwidth}
    %% Creator: Inkscape 1.2 (1:1.2.1+202207142221+cd75a1ee6d), www.inkscape.org
%% Accompanies image file 'Example-CutLocus-CircleThreePoints.pdf' (pdf, eps, ps)
\begingroup%
  \makeatletter%
  \providecommand\color[2][]{%
    \errmessage{(Inkscape) Color is used for the text in Inkscape, but the package 'color.sty' is not loaded}%
    \renewcommand\color[2][]{}%
  }%
  \providecommand\transparent[1]{%
    \errmessage{(Inkscape) Transparency is used (non-zero) for the text in Inkscape, but the package 'transparent.sty' is not loaded}%
    \renewcommand\transparent[1]{}%
  }%
  \providecommand\rotatebox[2]{#2}%
  \newcommand*\fsize{\dimexpr\f@size pt\relax}%
  \newcommand*\lineheight[1]{\fontsize{\fsize}{#1\fsize}\selectfont}%
  \ifx\svgwidth\undefined%
    \setlength{\unitlength}{258.31026573bp}%
    \ifx\svgscale\undefined%
      \relax%
    \else%
      \setlength{\unitlength}{\unitlength * \real{\svgscale}}%
    \fi%
  \else%
    \setlength{\unitlength}{\svgwidth}%
  \fi%
  \global\let\svgwidth\undefined%
  \global\let\svgscale\undefined%
  \makeatother%
  \begin{picture}(1,1.00020321)%
    \lineheight{1}%
    \setlength\tabcolsep{0pt}%
    \put(0,0){\includegraphics[width=\unitlength,page=1]{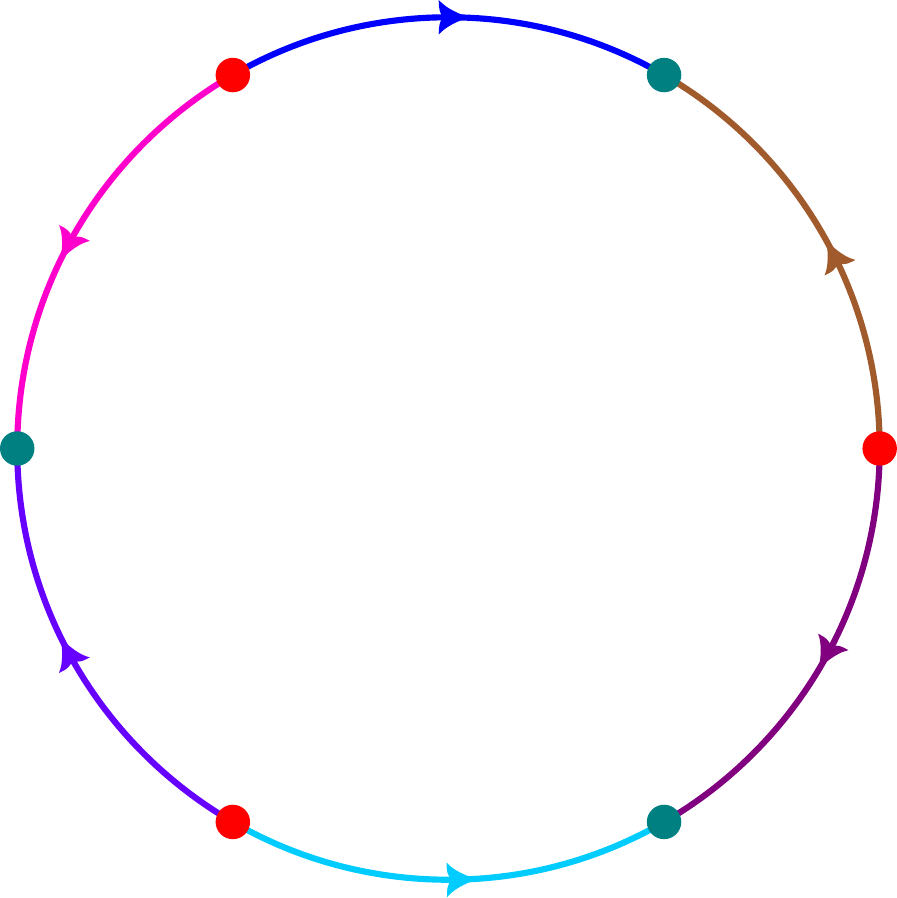}}%
    \put(0.27549432,0.86308076){\makebox(0,0)[lt]{\lineheight{1.25}\smash{\begin{tabular}[t]{l}$A_1$\end{tabular}}}}%
    \put(0.86805031,0.48510767){\makebox(0,0)[lt]{\lineheight{1.25}\smash{\begin{tabular}[t]{l}$A_2$\end{tabular}}}}%
    \put(0.26252413,0.10923888){\makebox(0,0)[lt]{\lineheight{1.25}\smash{\begin{tabular}[t]{l}$A_3$\end{tabular}}}}%
    \put(0.65356311,0.84282437){\makebox(0,0)[lt]{\lineheight{1.25}\smash{\begin{tabular}[t]{l}$B_{12}$\end{tabular}}}}%
    \put(0.62184548,0.12089804){\makebox(0,0)[lt]{\lineheight{1.25}\smash{\begin{tabular}[t]{l}$B_{23}$\end{tabular}}}}%
    \put(0.04491605,0.48562759){\makebox(0,0)[lt]{\lineheight{1.25}\smash{\begin{tabular}[t]{l}$B_{31}$\end{tabular}}}}%
  \end{picture}%
\endgroup%

        \caption{Cut locus of three points on unit circle \label{fig:Example-CutLocus-CircleThreePoints}}
    \end{figure}
    The above example can be generalized for any $k$-points on $\mathbb{S}^1$. The cut locus of $\{A_1,A_2,\cdots,A_k\}$ will be $\{B_{12},B_{23},\cdots, B_{k1}\}$ where $B_{i i+1}$ is the mid-point of $A_i$  and $A_{i+1}$.
\end{eg}

\begin{eg}[Cut locus of $k$ points on $\mathbb{S}^2$]\label{eg:cutLocusOfkPoints}
    Let $A_1,A_2$ and $A_3$ be three points on the equator. The cut locus will be half great circles passing through the mid-points $B_{12},B_{23}$ and $B_{31}$, see \Cref{fig:Example-CutLocus-SphereThreePoints}. In fact, all these semicircles are the separating set of $\{A_1,A_2,A_3\}$, being closed the closure is itself. So, the cut locus is homotopic to wedge of two circles. 
    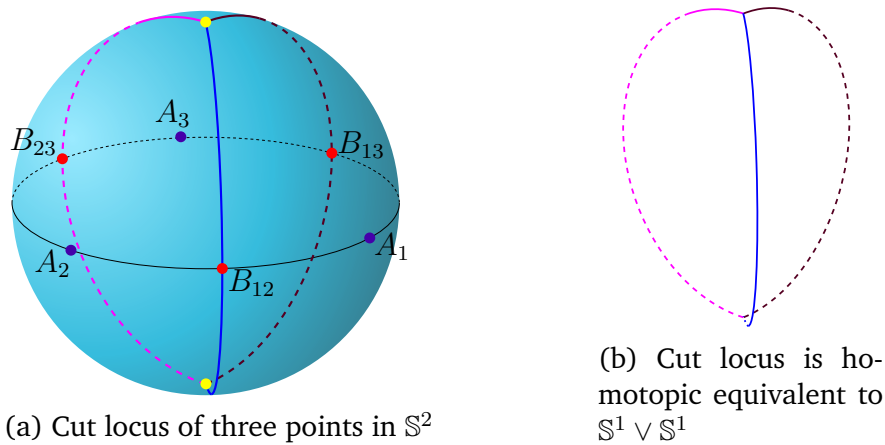
\begin{figure}[!htpb]
        \centering
        \begin{subfigure}{0.65\textwidth}
            \centering
    \def\svgwidth{0.55\columnwidth}
    %% Creator: Inkscape 1.2 (1:1.2.1+202207142221+cd75a1ee6d), www.inkscape.org
%% Accompanies image file 'Example-CutLocus-SphereThreePoints.pdf' (pdf, eps, ps)
\begingroup%
  \makeatletter%
  \providecommand\color[2][]{%
    \errmessage{(Inkscape) Color is used for the text in Inkscape, but the package 'color.sty' is not loaded}%
    \renewcommand\color[2][]{}%
  }%
  \providecommand\transparent[1]{%
    \errmessage{(Inkscape) Transparency is used (non-zero) for the text in Inkscape, but the package 'transparent.sty' is not loaded}%
    \renewcommand\transparent[1]{}%
  }%
  \providecommand\rotatebox[2]{#2}%
  \newcommand*\fsize{\dimexpr\f@size pt\relax}%
  \newcommand*\lineheight[1]{\fontsize{\fsize}{#1\fsize}\selectfont}%
  \ifx\svgwidth\undefined%
    \setlength{\unitlength}{315.65975742bp}%
    \ifx\svgscale\undefined%
      \relax%
    \else%
      \setlength{\unitlength}{\unitlength * \real{\svgscale}}%
    \fi%
  \else%
    \setlength{\unitlength}{\svgwidth}%
  \fi%
  \global\let\svgwidth\undefined%
  \global\let\svgscale\undefined%
  \makeatother%
  \begin{picture}(1,0.95039041)%
    \lineheight{1}%
    \setlength\tabcolsep{0pt}%
    \put(0,0){\includegraphics[width=\unitlength,page=1]{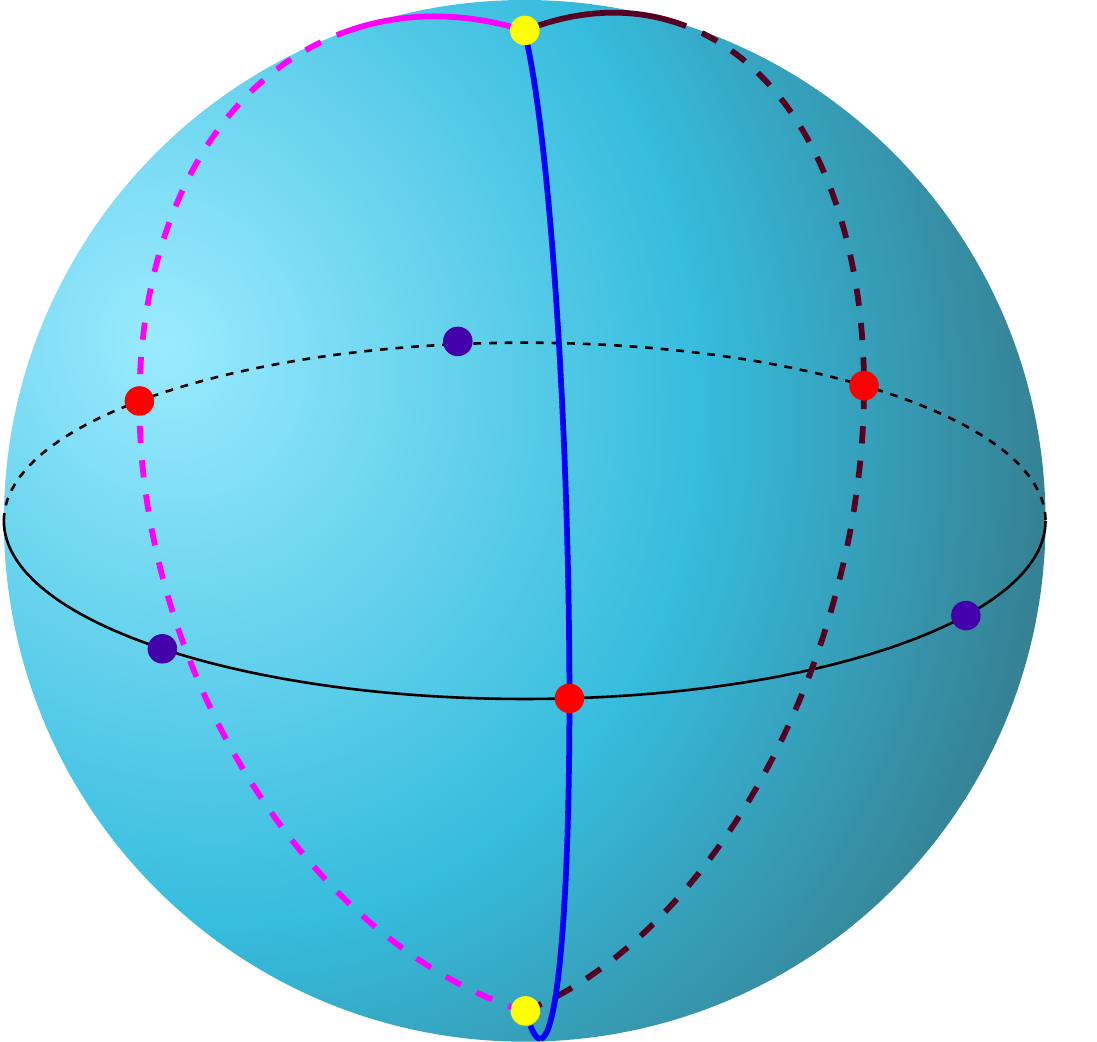}}%
    \put(0.05906534,0.30494103){\makebox(0,0)[lt]{\lineheight{1.25}\smash{\begin{tabular}[t]{l}$A_2$\end{tabular}}}}%
    \put(0.89202312,0.34727297){\makebox(0,0)[lt]{\lineheight{1.25}\smash{\begin{tabular}[t]{l}$A_1$\end{tabular}}}}%
    \put(0.35363588,0.67614496){\makebox(0,0)[lt]{\lineheight{1.25}\smash{\begin{tabular}[t]{l}$A_3$\end{tabular}}}}%
    \put(0.80176122,0.59851656){\makebox(0,0)[lt]{\lineheight{1.25}\smash{\begin{tabular}[t]{l}$B_{13}$\end{tabular}}}}%
    \put(-0.0027101,0.60610041){\makebox(0,0)[lt]{\lineheight{1.25}\smash{\begin{tabular}[t]{l}$B_{23}$\end{tabular}}}}%
    \put(0.53301482,0.25668591){\makebox(0,0)[lt]{\lineheight{1.25}\smash{\begin{tabular}[t]{l}$B_{12}$\end{tabular}}}}%
  \end{picture}%
\endgroup%

            \caption{Cut locus of three points in $\mathbb{S}^2$}
            \label{fig:Example-CutLocus-SphereThreePoints}
        \end{subfigure}
        \begin{subfigure}{0.25\textwidth}
            \centering
    \def\svgwidth{0.8\columnwidth}
    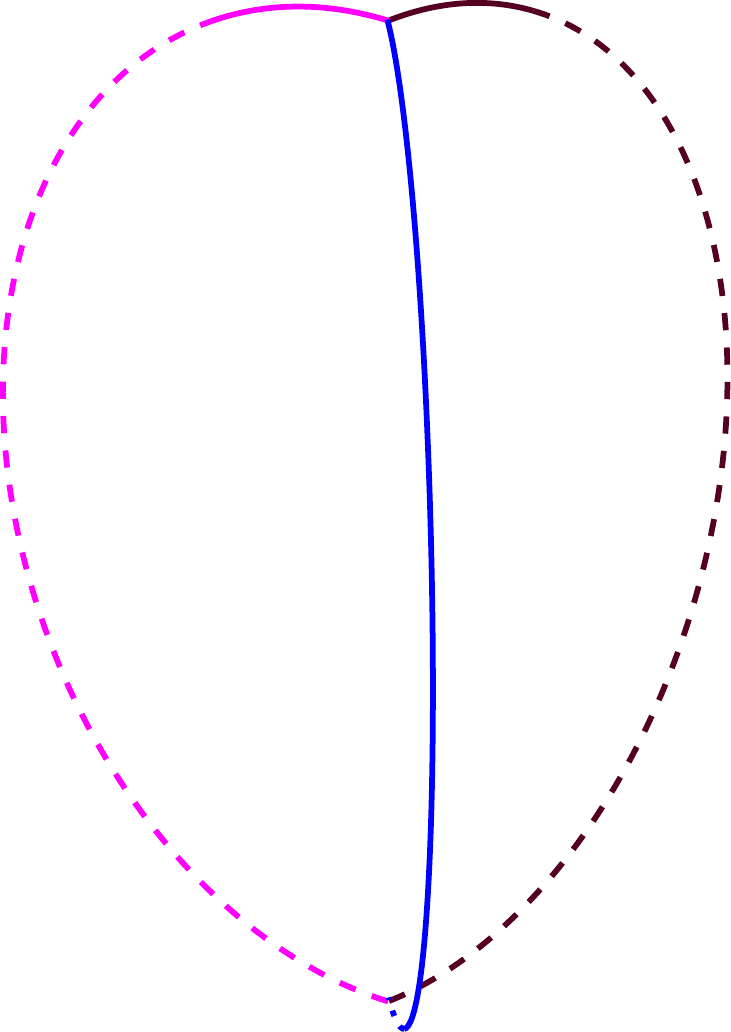

            \caption{Cut locus is homotopic equivalent to $\mathbb{S}^1\vee \mathbb{S}^1$ }
            \label{fig:Example-CutLocus-SphereThreePoints-02}
        \end{subfigure}
        \caption{Cut locus of three points in $\mathbb{S}^2$}
    \end{figure}
    The same can be generalized for $k$-points on the equator of $\mathbb{S}^2$ to conclude that the cut locus is homotopic to $\vee_{k-1}\mathbb{S}^1$.  Similarly, one can show that cut locus of $k$-points in $\mathbb{S}^n$ is homotopic to $\vee_{k-1}\mathbb{S}^{n-1}$. In this example also, the separating set is same as the cut locus as the separating set is closed. 
\end{eg}

\hf The above example, in particular, shows that cut locus need not be a manifold.

\begin{eg}
    Let $M=\mathbb{R}^2$ with the Euclidean metric and $N$ be the wedge of two circles. The cut locus of $N$ consists of centers of these two circles and the $y$-axis with origin removed, see \Cref{fig:Example-CutLocus-Figure8}.
    \begin{figure}[!htb]
        \centering
    \def\svgwidth{0.6\columnwidth}
    %% Creator: Inkscape 1.2 (1:1.2.1+202207142221+cd75a1ee6d), www.inkscape.org
%% Accompanies image file 'Example-CutLocus-Figure8.pdf' (pdf, eps, ps)
\begingroup%
  \makeatletter%
  \providecommand\color[2][]{%
    \errmessage{(Inkscape) Color is used for the text in Inkscape, but the package 'color.sty' is not loaded}%
    \renewcommand\color[2][]{}%
  }%
  \providecommand\transparent[1]{%
    \errmessage{(Inkscape) Transparency is used (non-zero) for the text in Inkscape, but the package 'transparent.sty' is not loaded}%
    \renewcommand\transparent[1]{}%
  }%
  \providecommand\rotatebox[2]{#2}%
  \newcommand*\fsize{\dimexpr\f@size pt\relax}%
  \newcommand*\lineheight[1]{\fontsize{\fsize}{#1\fsize}\selectfont}%
  \ifx\svgwidth\undefined%
    \setlength{\unitlength}{291.56121963bp}%
    \ifx\svgscale\undefined%
      \relax%
    \else%
      \setlength{\unitlength}{\unitlength * \real{\svgscale}}%
    \fi%
  \else%
    \setlength{\unitlength}{\svgwidth}%
  \fi%
  \global\let\svgwidth\undefined%
  \global\let\svgscale\undefined%
  \makeatother%
  \begin{picture}(1,0.79351313)%
    \lineheight{1}%
    \setlength\tabcolsep{0pt}%
    \put(0,0){\includegraphics[width=\unitlength,page=1]{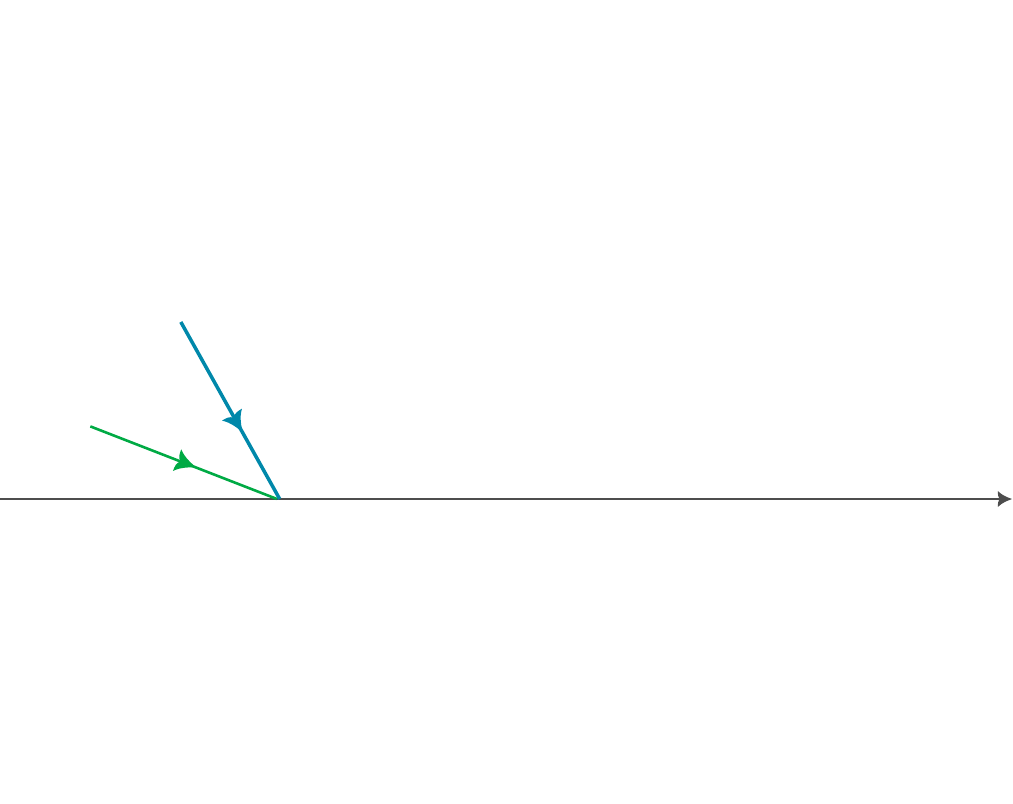}}%
    \put(0.23565501,0.23654977){\makebox(0,0)[lt]{\lineheight{1.25}\smash{\begin{tabular}[t]{l}$C_1$\end{tabular}}}}%
    \put(0.65298758,0.23654977){\makebox(0,0)[lt]{\lineheight{1.25}\smash{\begin{tabular}[t]{l}$C_2$\end{tabular}}}}%
    \put(0,0){\includegraphics[width=\unitlength,page=2]{Example-CutLocus-Figure8.pdf}}%
    \put(0.22461246,0.40040641){\color[rgb]{0,0.53333333,0.66666667}\makebox(0,0)[lt]{\lineheight{1.25}\smash{\begin{tabular}[t]{l}$\gamma$\end{tabular}}}}%
    \put(0,0){\includegraphics[width=\unitlength,page=3]{Example-CutLocus-Figure8.pdf}}%
    \put(0.10745245,0.3231126){\color[rgb]{0,0.66666667,0.26666667}\makebox(0,0)[lt]{\lineheight{1.25}\smash{\begin{tabular}[t]{l}$\eta$\end{tabular}}}}%
    \put(0,0){\includegraphics[width=\unitlength,page=4]{Example-CutLocus-Figure8.pdf}}%
  \end{picture}%
\endgroup%

        \caption{Cut locus of wedge of two circles in $\mathbb{R}^2$}
        \label{fig:Example-CutLocus-Figure8}
    \end{figure}
    In fact, if we take any other point then that is not a cut point as any geodesic, a straight line, never fails its distance minimal property. This example, also shows that the cut locus of a subset need not be a closed set.
\end{eg}

\begin{eg}
    Let $M$ be the cylinder $\mathbb{S}^1\times \mathbb{R}$ with the product metric. Let $N=\{\mathbf{v}\}\times \mathbb{R}$ for some $\mathbf{v} \in \mathbb{S}^1$. Then cut locus of $N$ is $\{-\mathbf{v}\}\times \mathbb{R}$, see \Cref{fig:Example-CutLocus-Cylinder-Submanifold}. 
    \begin{figure}[!htpb]
        \centering
    \def\svgwidth{0.35\columnwidth}
    %% Creator: Inkscape 1.2 (1:1.2.1+202207142221+cd75a1ee6d), www.inkscape.org
%% Accompanies image file 'Example-CutLocus-Cylinder-Submanifold.pdf' (pdf, eps, ps)
\begingroup%
  \makeatletter%
  \providecommand\color[2][]{%
    \errmessage{(Inkscape) Color is used for the text in Inkscape, but the package 'color.sty' is not loaded}%
    \renewcommand\color[2][]{}%
  }%
  \providecommand\transparent[1]{%
    \errmessage{(Inkscape) Transparency is used (non-zero) for the text in Inkscape, but the package 'transparent.sty' is not loaded}%
    \renewcommand\transparent[1]{}%
  }%
  \providecommand\rotatebox[2]{#2}%
  \newcommand*\fsize{\dimexpr\f@size pt\relax}%
  \newcommand*\lineheight[1]{\fontsize{\fsize}{#1\fsize}\selectfont}%
  \ifx\svgwidth\undefined%
    \setlength{\unitlength}{561.39963163bp}%
    \ifx\svgscale\undefined%
      \relax%
    \else%
      \setlength{\unitlength}{\unitlength * \real{\svgscale}}%
    \fi%
  \else%
    \setlength{\unitlength}{\svgwidth}%
  \fi%
  \global\let\svgwidth\undefined%
  \global\let\svgscale\undefined%
  \makeatother%
  \begin{picture}(1,1.30820871)%
    \lineheight{1}%
    \setlength\tabcolsep{0pt}%
    \put(0,0){\includegraphics[width=\unitlength,page=1]{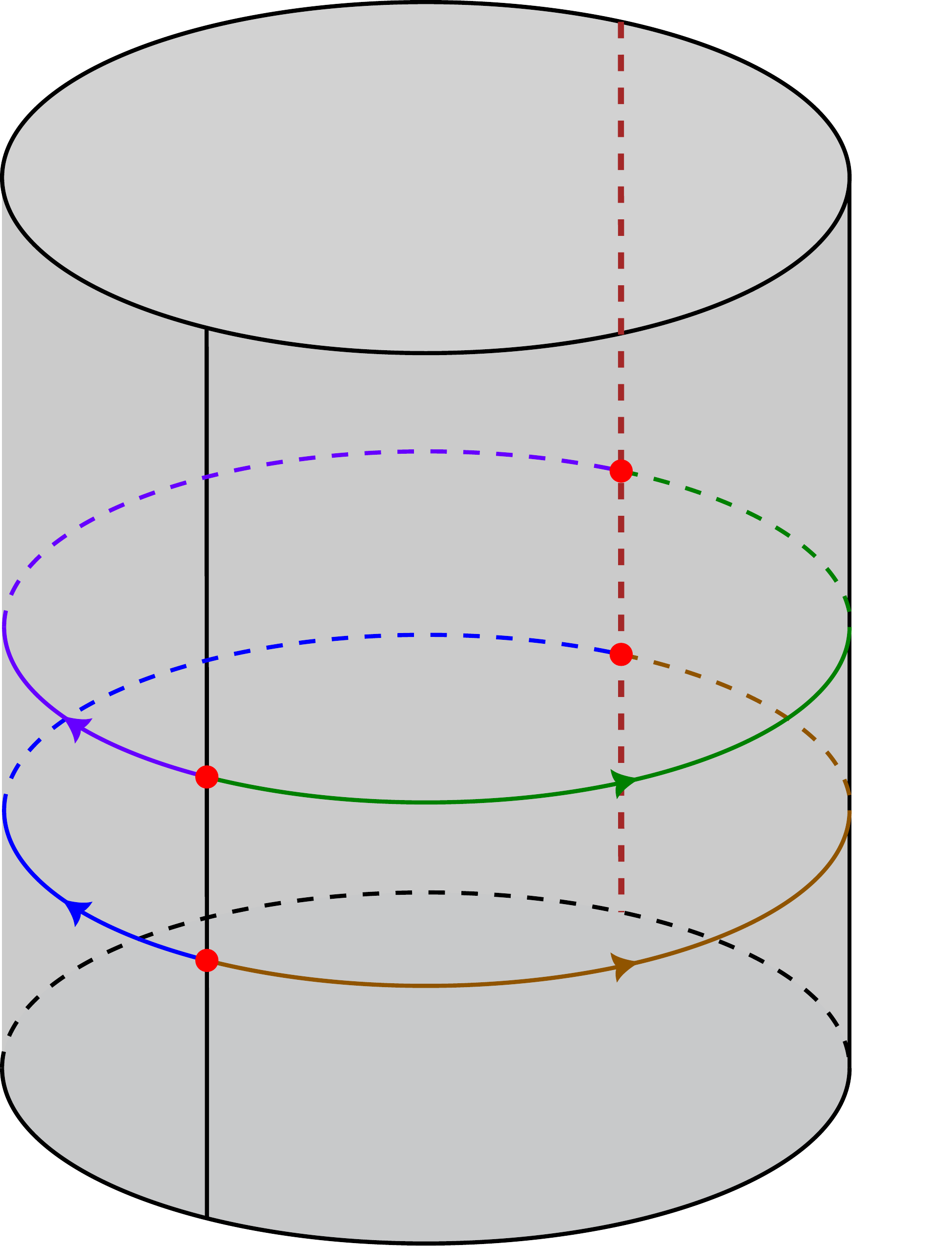}}%
    \put(0.11116984,0.87103526){\makebox(0,0)[lt]{\lineheight{1.25}\smash{\begin{tabular}[t]{l}$N$\end{tabular}}}}%
    \put(0,0){\includegraphics[width=\unitlength,page=2]{Example-CutLocus-Cylinder-Submanifold.pdf}}%
    \put(0.92042767,0.84920362){\color[rgb]{0.64313725,0.15686275,0.15686275}\makebox(0,0)[lt]{\lineheight{1.25}\smash{\begin{tabular}[t]{l}$\cutn$\end{tabular}}}}%
    \put(0.12206094,0.25548188){\makebox(0,0)[lt]{\lineheight{1.25}\smash{\begin{tabular}[t]{l}$\bf{v}$\end{tabular}}}}%
  \end{picture}%
\endgroup%

        \caption{Cut locus of a line on cylinder \label{fig:Example-CutLocus-Cylinder-Submanifold}}
    \end{figure}
    
\end{eg}

\section{An illuminating example}\label{Sec:IlluminatingExample}
\hfb Let $ M = M(n,\rbb) $ be the set of $n\times n$ matrices, and $ N=O(n,\rbb) $ be the set of all orthogonal $n\times n$ matrices. Let $A,B\in M(n,\rbb)$. We fix the standard flat Euclidean metric on $M(n,\rbb)$ by identifying it with $\mathbb{R}^{n^2}$. This induces a distance function given by
\begin{equation*}\label{Example 1 Eq: 1}
	d(A,B) \defeq \sqrt{\trace{(A-B)^T(A-B)}}
\end{equation*}
Consider the distance squared function
\begin{displaymath}
    f: GL(n,\rbb)\to \rbb,~~ A\mapsto d^2(A, O(n,\rbb)).
\end{displaymath}
In order to study this function, we want a closed formula for it.
\begin{lemma}
    The function $f$ can be explicitly expressed as 
    \begin{equation}\label{eq:d2On}
        f(A) = n + \trace{A^TA} - 2\trace{\sqrt{A^TA}}.
    \end{equation} 
\end{lemma}

\begin{proof}
 Let $A\in GL(n,\rbb)$ be any invertible matrix. Then,
\begin{align}
    f(A) & =  \inf_{B\in O(n,\rbb)}d^2(A,B)\nonumber
    \\
         & =  \inf_{B\in O(n,\rbb)} \|A-B\|^2\nonumber
         \\
         & =  \inf_{B\in O(n,\rbb)} \trace{(A-B)^T(A-B)}\nonumber
         \\
         & =  \inf_{B\in O(n,\rbb)} \trace{A^TA-A^TB-B^TA+B^TB}\nonumber
         \\
         & =  \inf_{B\in O(n,\rbb)} \big[\trace{A^TA}-\trace{A^TB}-\trace{B^TA}+\trace{B^TB} \big]\nonumber
         \\
         & =  \trace{A^TA} + \inf_{B\in O(n,\rbb)}\big[- 2\trace{A^TB} \big] + n \nonumber
         \\
         & =  \trace{A^TA} -2\sup_{B\in O(n,\rbb)} \trace{A^TB} + n.  \label{maxdist}
\end{align}
The problem of computing $f(A)$ is equivalent to maximizing the function
\begin{displaymath}
    h_A: O(n,\rbb)\to \rbb, B\mapsto \trace{A^TB}.
\end{displaymath} 
\paragraph{\blue{Case I:}} $A$ is a diagonal matrix with positive entries. Then, 
\begin{align*}
    \left|h_A(B)\right|  = \left|\trace{A^TB}\right|
     = \left|\sum_{i=1}^n a_{ii} b_{ii} \right| 
     \le \sum_{i=1}^n \abs{a_{ii}b_{ii}}
      & \le \sum_{i=1}^n a_{ii}
    = \trace{A^T} 
    = h_A(I).
\end{align*}
Thus, one of the maximizer is $B=I.$
\paragraph{\blue{Case II:}}  For any non-singular matrix $A$, we will use the \textit{singular value decomposition} (SVD). \index{singular value decomposition} Write $A=UDV^T$, where $U$ and $V$ are $n\times n$ orthogonal matrices and $D$ is a diagonal matrix with positive entries. For any $B\in O(n,\rbb)$ using the cyclic property of the trace we have
\begin{align}
    \trace{A^TB}  & = \trace{VDU^TB} = \trace{D(U^TBV)}.
\end{align}
Since $U^TBV$ is an orthogonal matrix, maximizing over $B$ reduces to the earlier observation that $B$ will be a maximizer if $U^TBV = I$, which implies $B = UV^T$. 

\vspace{0.3cm}
\hf Since $A$ is invertible, by the polar decomposition, there exists an orthogonal matrix $Q$ and a symmetric positive definite matrix $S=\sqrt{A^TA}$ such that $A=QS$. As $S$ is symmetric matrix we can diagonalize it, that is, $S = P\tilde{D}P^T$, where $P\in O(n,\rbb)$ and $\tilde{D}$ is a diagonal matrix. Thus,
\begin{displaymath}
    A = QS = Q P\tilde{D}P^T.
\end{displaymath}
Set $U=QP$, $V=P$ to obtain the SVD of $A$. In particular, the minimizer is given by
$$B=Q=A\big(\sqrt{A^TA}\big)^{-1}.$$
Therefore, 
	\begin{equation*}
		f(A) = n+\trace{A^TA}  - 2\,\trace{\sqrt{A^TA}}
	\end{equation*}
for invertible matrices. 

\vspace{0.3cm}
\hf To find out $f(A)$ for a non-invertible matrix $A$, we note that $GL(n,\rbb)$ is dense in $M(n,\rbb)$ and that $\sqrt{A^TA}$ is well-defined for $A\in M(n,\rbb)$. The continuity of the map $A\mapsto \sqrt{A^T A}$ on $M(n,\rbb)$ implies that the same formula \eqref{eq:d2On} for $f$ applies to $A$ as well. 
\end{proof}

\noindent In order to understand the differentiability of $f$, it suffices to analyze the function $A\mapsto \trace{\sqrt{A^TA}}$. 
\begin{lemma}
	The map $g : M(n,\rbb)\to \rbb, ~A\mapsto \trace{\sqrt{A^TA}}$ is differentiable if and only if $A$ is invertible. 
\end{lemma}

\begin{proof}
    Let $A$ be an invertible matrix. We will prove that the function $g$ is differentiable at $A$. Let $\cali{P}$  be the set of all positive definite matrices which is an open subset of the set of all symmetric matrices $\cali{S}$. We will prove that the map 
    \begin{displaymath}
        r : \cali{P}\to \cali{P}, ~ A\mapsto \sqrt{A}
    \end{displaymath}
    is differentiable. Define a function 
    \begin{displaymath}
        s : \cali{P}\to \cali{P},~ A\mapsto A^2.
    \end{displaymath}
    We will show that $s$ is a diffeomorphism and from the inverse function theorem $r$ will be differentiable. In order to show that $s$ is a diffeomorphism, we claim that for $A\in \cali{P},~ds_{A}:T_{A}\cali{P}\to T_{A^2}\cali{P}$ is injective. Note that $\cali{P}$ is an open subset of a vector space $\cali{S}$ and therefore, $T_{A}\cali{P}\isom \cali{S} \isom T_{A^2}\cali{P}.$ So, take $B\in \cali{S}$ such that $ds_{A}(B) = 0$. We will show that $B = 0.$ Recall that $ds_{A}(B) = AB +BA.$ Now choose an orthonormal basis $\{\vbf_1,\vbf_2,\cdots,\vbf_n\}$ of eigenspace of $A$ and $A\vbf_{i} = \lambda_{i} \vbf_i$ ($\lambda_i>0$). Then,
    \begin{displaymath}
        A(B\vbf_i) = -BA\vbf_i = -B\lambda_i\vbf_i = -\lambda i(B\vbf_i)
    \end{displaymath} 
    which implies $Bv_{i}$ is also an eigenvector of $A$ with eigenvalue $-\lambda_{i}<0$. Hence, $Bv_{i} = 0$ which implies $B=0$.\\
    \hspace{0.7cm} For the converse, we will show that if $A$ is a singular matrix, then the map $g$ is not directional differentiable. Let $A$ be a singular matrix. Using the singular value decomposition, we write 
    \begin{displaymath}
        A = U \begin{pmatrix}
            D & 0 \\ 
            0   & 0_k
        \end{pmatrix} V^T,    
    \end{displaymath}
    where $D$ is a $(n-k)\times (n-k)$ diagonal matrix with positive entries. If 
    \begin{displaymath}
        B = U \begin{pmatrix}
            0_{n-k} & 0 \\
            0 & I_k
        \end{pmatrix}    
    \end{displaymath}
    then we claim that $g$ is not differentiable in the direction of $B$. Since 
    \begin{displaymath}
        \sqrt{(A+tB)^T(A+tB)} = V \begin{pmatrix}
            D & 0 \\
            0 & I_k|t|
        \end{pmatrix} V^T
    \end{displaymath}
    the limit 
    \begin{align*}
        \lim_{t\to 0} \dfrac{g(A+tB) - g(A)}{t} & = \lim_{t\to 0} \dfrac{\trace{V\begin{pmatrix}
            D & 0 \\
            0 & I_k|t|
        \end{pmatrix} V^T} - \trace{V \begin{pmatrix}
            D & 0 \\
            0 & 0_k
        \end{pmatrix} V^T}}{t}\\
        & = k \lim_{t\to 0} \dfrac{|t|}{t} 
    \end{align*}
    does not exist and hence the function $g$ is not differentiable.
\end{proof}

\begin{prop}\label{prop:derivativeOfSqrtMatrixMap}
    If $A$ is an invertible matrix, then 
    \begin{equation}\label{derivativeOfSquareRoot}
        dg_A(H) = \innerprod{A\paran{\sqrt{A^TA}}^{-1}}{H},
    \end{equation}
    where $H$ is a symmetric matrix of order $n$.
\end{prop}

\vspace{0.3cm}
\noindent The following lemma along with chain rule will prove the above proposition. 
\begin{lemma} \label{lem:derivativeOfSquareRoot}
    Let $A$ be a positive definite matrix and $\psi:A\mapsto \sqrt{A}$. Then 
    \begin{equation*}\label{eq: sqrtderivative}
        d\psi_A(H) = \int_0^\infty e^{-t\sqrt{A}}He^{-t\sqrt{A}}~dt,
    \end{equation*}
    for any symmetric matrix $H.$
\end{lemma}

\begin{proof}
    As $\psi(A)\cdot \psi(A) = A$, differentiating at $A$ we obtain
    \begin{equation}\label{eq: sylvester}
            d\psi_A(H)\psi(A) + \psi(A)d\psi_A(H) = H.
    \end{equation}
    We will show the following:
    \begin{enumerate}
        \item [(i)] For any positive definite matrix $X$ and for any symmetric matrix $Y$ the integral 
        \begin{equation}\label{eq: matsqrtderivative}
            \int_0^\infty e^{-tX}Ye^{-tX}~\mathrm{d}t
        \end{equation}
        converges. We note that the eigenvalues of $e^{-tX}$ are $e^{-t\lambda_j}$, where $\lambda_j$ are the eigenvalues of $X$. Since $X$ is a positive definite matrix, each of the $\lambda_j$ is positive. Without loss of generality, we assume that $\lambda=\lambda_1$ is the smallest eigenvalue of $X$. Then we have
        \begin{align*}
            e^{-t \lambda_j} \le e^{-t \lambda} & \implies \left\|e^{-tX}\right\| = e^{-t \lambda}.
        \end{align*} 
        where  $\|\cdot\|$ is the operator norm.  Therefore, the operator norm of the integrand in \eqref{eq: matsqrtderivative} is bounded by $2e^{-t \lambda}\|Y\|$, which is an integrable function. Hence, the integral given by \eqref{eq: matsqrtderivative} converges. 
    
        \item [(ii)] The matrix $d\psi_A(H)$ satisfies \eqref{eq: sylvester}. Observe that 
        \begin{align*}
                & \paran{\int_0^\infty e^{-t\sqrt{A}}\cdot H\cdot e^{-t\sqrt{A}} \mathrm{d}t } \sqrt{A} + \sqrt{A} \paran{\int_0^\infty e^{-t\sqrt{A}}\cdot H\cdot e^{-t\sqrt{A}}~\mathrm{d}t} 
            \\
            = & \int_0^\infty\paran{e^{-t\sqrt{A}}\cdot H\cdot e^{-t\sqrt{A}}  \sqrt{A} + \sqrt{A} e^{-t\sqrt{A}}\cdot H\cdot e^{-t\sqrt{A}}}~\mathrm{d}t
            \\
            = & \int_0^\infty-\paran{e^{-t\sqrt{A}} H e^{-t\sqrt{A}}}'~\mathrm{d}t = H.
        \end{align*}
    \end{enumerate}
    From (i), (ii) and the uniqueness of the derivative, the lemma is proved. 
\end{proof}

\vspace{0.3cm}
\noindent We now give a proof of \Cref{prop:derivativeOfSqrtMatrixMap}.
\begin{proof}[Proof of \Cref{prop:derivativeOfSqrtMatrixMap}]
    Note that using \Cref{lem:derivativeOfSquareRoot}, for any symmetric matrix $H$ we have
    \begin{equation}\label{eq:derivativeOfSquareRoot}
        dg_A(H) = \trace{\int_0^\infty e^{-t\sqrt{A^TA}}\paran{A^TH+H^TA}e^{-t\sqrt{A^TA}}~dt}.
    \end{equation}
    Let us simplify the above expression to get the desired result.
    \begin{align*}
        dg_A(H) & = \int_0^\infty \squarebracket{\trace{e^{-2t\sqrt{A^TA}}H^TA} + \trace{e^{-2t\sqrt{A^TA}}A^TH}}~dt
        \\[1ex] 
        & = \trace{\int_0^\infty e^{-2t \sqrt{A^TA}}H^TA~dt} + \trace{\int_0^\infty e^{-2t \sqrt{A^TA}}A^TH~dt}
        \\[1ex]
        & = \trace{\squarebracket{\int_0^\infty e^{-2t \sqrt{A^TA}}~dt}H^TA} + \trace{\squarebracket{\int_0^\infty e^{-2t \sqrt{A^TA}}~dt}A^TH}
        \\[1ex]
        \begin{split}
            & = \trace{\squarebracket{-\dfrac{\paran{\sqrt{A^TA}}^{-1}}{2} \int_0^\infty \dfrac{d}{dt}e^{-2t\sqrt{A^TA}}~dt}H^TA} 
            \\[0.5ex]
            & \kern 1cm + \trace{\squarebracket{-\dfrac{\paran{\sqrt{A^TA}}^{-1}}{2} \int_0^\infty \dfrac{d}{dt}e^{-2t\sqrt{A^TA}}~dt}A^TH}
        \end{split}
        \\[1ex]
        & = \trace{\dfrac{\paran{\sqrt{A^TA}}^{-1}}{2}H^TA}  + \trace{\dfrac{\paran{\sqrt{A^TA}}^{-1}}{2}A^TH}
        \\[1ex]
        & =  \dfrac{1}{2} \trace{\paran{\sqrt{A^TA}}^{-1}A^TH} + \dfrac{1}{2}  \trace{A^TH\paran{\sqrt{A^TA}}^{-1}}
        \\
        & =  \trace{\paran{\sqrt{A^TA}}^{-1}A^TH}
         = \innerprod{A\paran{\sqrt{A^TA}}^{-1}}{H}
    \end{align*}
    Thus, 
    \begin{equation*} \label{differentialOfSquareRootFinal}
	    dg_A(H) = \innerprod{A\paran{\sqrt{A^TA}}^{-1}}{H}.
    \end{equation*}
\end{proof}
\vspace{0.3cm}
\hspace*{0.5cm}For  any $A\in GL(n,\rbb)$
\begin{equation*}\label{gradfForGL}
    df_A  = 2A-2A\paran{\sqrt{A^TA}}^{-1}=-2A \paran{\sqrt{A^TA}^{-1}-I}.
\end{equation*}
Hence, the negative gradient of the function $f$, restricted to $GL(n,\rbb)$ is given by 
\begin{displaymath}
	-\grad f\big|_A = 2A\paran{\sqrt{A^TA}^{-1}-I}.
\end{displaymath}
The critical points are orthogonal matrices. If $\gamma(t)$ is an integral curve of $-\grad f$ initialized at $A$, then $\gamma(0)=A$ and 
\begin{equation}
\dfrac{d\gamma}{dt} =-2\gamma(t)+2\gamma(t)\paran{\sqrt{\gamma(t)^T\gamma(t)}}^{-1}=-2\gamma(t)+2\paran{\gamma(t)^T}^{-1}\sqrt{\gamma(t)^T\gamma(t)} \label{eq: 4.5}.
\end{equation}
Take the test solution of \eqref{eq: 4.5} given by
\begin{equation}\label{eq: 4.6}
\gamma(t)  = Ae^{-2t} + (1-e^{-2t})\paran{A^T}^{-1}\sqrt{A^TA} = Ae^{-2t} + (1-e^{-2t})A\paran{\sqrt{A^TA}}^{-1}.
\end{equation}
% Observe that, 
% \begin{align*}
% 	Df_A & = 2A - 2A \paran{\sqrt{A^TA}^{-1}}\\
% 		 & = 2A - 2\paran{A^T}^{-1}\sqrt{A^TA}.
% \end{align*}
% And, hence 
% \begin{align}
% 	\dfrac{d\gamma(t)}{dt} = -2\gamma(t)+ 2\paran{\gamma(t)^T}^{-1}\sqrt{\gamma(t)^T\gamma(t)}.\label{eq:4.7}
% \end{align}
\noindent In order to show that $\gamma(t)$ satisfies \eqref{eq: 4.5}, note that
\begin{align*}
    \begin{split}
		\gamma(t)^T\gamma(t) & = \squarebracket{Ae^{-2t}+(1-e^{-2t})\paran{A^T}^{-1}\sqrt{A^TA}}^T\\
        & \qquad \qquad \squarebracket{Ae^{-2t}+(1-e^{-2t})\paran{A^T}^{-1}\sqrt{A^TA}}  
    \end{split}
    \\[1ex]
    \begin{split}
        & = \squarebracket{A^Te^{-2t} +(1-e^{-2t})\sqrt{A^TA} A^{-1}} \\
        & \qquad \qquad \squarebracket{Ae^{-2t}+(1-e^{-2t})\paran{A^T}^{-1}\sqrt{A^TA}}
    \end{split}
    \\[1ex]
    \begin{split}
        & = A^TAe^{-4t} + 2e^{-2t}(1-e^{-2t})\sqrt{A^TA}\\
        & \qquad \qquad +(1-e^{-2t})^2\paran{\sqrt{A^TA} A^{-1} \aTransInverse \sqrt{A^TA}}
    \end{split}
    \\[1ex] 
    \begin{split}
        & = A^TAe^{-4t} + 2e^{-2t}(1-e^{-2t})\sqrt{A^TA}\\
        & \qquad \qquad +(1-e^{-2t})^2\paran{\sqrt{A^TA} \paran{A^TA}^{-1} \sqrt{A^TA}}
    \end{split}
    \\[1ex]
    \begin{split}
        & = A^TAe^{-4t} + 2e^{-2t}(1-e^{-2t})\sqrt{A^TA}\\
        & \qquad \qquad +(1-e^{-2t})^2\paran{\sqrt{A^TA} \sqrtATransAInverse \sqrtATransAInverse \sqrt{A^TA}}
    \end{split}
    \\[1ex] 
    & = A^TAe^{-4t} + 2e^{-2t}(1-e^{-2t})\sqrt{A^TA} + \paran{1-e^{-2t}}^2I\\
    & = \paran{\sqrt{A^TA}~e^{-2t}+\paran{1-e^{-2t}}I}^2.
\end{align*}
Thus,
\begin{displaymath}
    \gamma(t)^T\gamma(t) = \paran{\sqrt{A^TA}~e^{-2t}+\paran{1-e^{-2t}}I}^2
\end{displaymath}
and hence
\begin{displaymath}
    \paran{\sqrt{\gamma(t)^T\gamma(t)}}^T = \paran{\sqrt{A^TA}A^{-1}\gamma(t)}^T = \gamma(t)^T\aTransInverse \sqrt{A^TA}
\end{displaymath}

\noindent This implies that 
\begin{align*}
    \sqrt{\gamma(t)^T\gamma(t)} & = \sqrt{A^TA}\paran{e^{-2t}I+\sqrtATransAInverse \paran{1-e^{-2t}}}\nonumber\\
    \implies \sqrt{\gamma(t)^t\gamma(t)} & = \sqrt{A^TA}A^{-1}\gamma(t)\nonumber\\
    \implies \paran{\sqrt{\gamma(t)^T\gamma(t)}}^T & = \paran{\sqrt{A^TA}A^{-1}\gamma(t)}^T = \gamma(t)^T\aTransInverse \sqrt{A^TA}\nonumber\\
    \implies \paran{\gamma(t)^T}^{-1} \sqrt{\gamma(t)^T\gamma} & = \aTransInverse\sqrt{A^TA}.
\end{align*}
\noindent The right hand side of \eqref{eq: 4.5}, with the test solution, can be simplified to 
\begin{displaymath}
-2Ae^{-2t} + 2e^{-2t}\aTransInverse \sqrt{A^TA}
\end{displaymath}
which is the derivative of $\gamma$. Thus, $\gamma(t)$, as defined in \eqref{eq: 4.6}, is the required flow line which deforms $GL(n,\rbb)$ to $O(n,\rbb).$ In particular, $GL^+(n,\rbb)$ deforms to $SO(n,\rbb)$ and other component of $GL(n,\rbb)$ deforms to $O(n,\rbb)\setminus SO(n,\rbb)$. We note, however, that this deformation takes infinite time to perform the retraction.
\begin{rem}
A modified curve
\begin{equation}\label{GLdefOver1}
\eta(t)=A(1-t)+tA\paran{\sqrt{A^TA}}^{-1}
\end{equation}
with the same image as $\gamma$, defines an actual deformation retraction of $GL(n,\rbb)$ to $O(n,\rbb)$. Apart from its origin via the distance function, this is a geometric deformation in the following sense. Given $A\in GL(n,\rbb)$, consider its columns as an ordered basis. This deformation deforms the ordered basis according to the length of the basis vectors and mutual angles between pairs of basis vectors in a geometrically uniform manner. This is in sharp contrast with Gram-Schmidt orthogonalization, also a deformation of $GL(n,\rbb)$ to $O(n,\rbb)$, which is asymmetric as it never changes the direction of the first column, the modified second column only depends on the first two columns and so on.
\end{rem}
\hspace*{0.5cm}We now show that $f$ is Morse-Bott.\index{Morse-Bott functions} The tangent space $T_I O(n,\rbb)$ consists of skew-symmetric matrices while the normal vectors at $I_n$ are the symmetric matrices. As left translation by an orthogonal matrix is an isometry of $M(n,\rbb)$, normal vectors at $A\in O(n,\rbb)$ are of the form $AW$ for symmetric matrices $W$. Since 
\begin{displaymath}
	df_A(H) = 2\innerprod{A}{H}-2\innerprod{A\big(\sqrt{A^TA}\big)^{-1}}{H}
\end{displaymath}
the relevant Hessian is 
\begin{equation*}
	\hess(f)_A(H,H') = \lim_{t\to 0}\dfrac{df_{A+tH'}(H)-df_A(H)}{t} 
\end{equation*}
with $H=AW, H'=AW'$ and symmetric matrices $W,W'$. Solving the Hessian expression, we have
\begin{align*}
	\hess(f)_A(H,H') 
		& = \lim_{t\to 0} \left(\dfrac{2\innerprod{A+tH'}{H} - 2\innerprod{(A+tH')\sqrt{(A+tH')(A+tH')^T}^{-1}}{H}}{t}\right.
		\\
		& \qquad \qquad \qquad \left. -\dfrac{2\innerprod{A}{H}-2\innerprod{A\sqrt{A^TA}^{-1}}{H}}{t} \right)
	\\[1ex]
	\begin{split}
		& = \lim_{t\to 0} \left( \dfrac{2\innerprod{A}{H}+2t\innerprod{H'}{H}-2\innerprod{A\sqrt{(A+tH')^T(A+tH')}^{-1}}{H}}{t} \right. 
		\\
		& \kern -.5cm \left. -\dfrac{2t\innerprod{H'\sqrt{(A+tH')^T(A+tH')}^{-1}}{H} + 2\innerprod{A}{H}-2\innerprod{A\sqrt{A^TA}^{-1}}{H}}{t} \right)
	\end{split}
	\\[1ex]
	\begin{split}
		& = \lim_{t\to 0} \left( \dfrac{2t\innerprod{H'}{H}-2t\innerprod{H'\sqrt{(A+tH')^T(A+tH')}^{-1}}{H}}{t} \right. 
		\\ 
		& \qquad \qquad \left. \dfrac{-2\innerprod{A\sqrt{(A+tH')^T(A+tH')}^{-1}}{H}+2\innerprod{A\sqrt{A^TA}^{-1}}{H}}{t} \right)
	\end{split}
	\\[1ex]
	\begin{split}
		& = \lim_{t\to 0} 2\cancel{t}\paran{\dfrac{\innerprod{H'}{H}-\innerprod{H'\sqrt{(A+tH')^T(A+tH')}^{-1}}{H}}{\cancel{t}}}
		\\
		& \qquad  -2\lim_{t\to 0} \paran{\dfrac{\innerprod{A\sqrt{(A+tH')^T(A+tH')}^{-1}}{H}-\innerprod{A\sqrt{A^TA}^{-1}}{H}}{t}}
	\end{split}
	\\[1ex]
	\begin{split}
		& = 2\innerprod{H'}{H} - 2\innerprod{H'\cancelto{I}{\sqrt{A^TA}^{-1}}}{H} -2 \innerprod{A\paran{Dg^{-1}}_A(H')}{H}
	\end{split}
	\\[1ex]
	& = -2\innerprod{A\cdot \frac{A^TH'+H'^TA}{2}}{H} = \innerprod{H'}{H} + \innerprod{AH'^TA}{H} 
	\\
	& = \innerprod{H'}{H} + \innerprod{A(AW')^TA}{AW} =  \innerprod{H'}{H} + \innerprod{AW'^T}{AW}
	\\
	& = \trace{H'^TH} + \trace{W'^TW} = \trace{H'^TH} + \trace{H'^TH} = 2~\trace{H'^TH}.
\end{align*}

\noindent Thus, the Hessian is, 
\begin{equation*}
\hess(f)_A(H,H') =2\,\trace{H^TH'}=2\left\langle H, H'\right\rangle.
\end{equation*}

\noindent Therefore, the Hessian matrix \index{Hessian matrix} restricted to $(T_A O(n,\rbb))^\perp$ is $2I_{\frac{n(n+1)}{2}}$. This is a recurring feature of distance squared functions associated to embedded submanifolds (see Proposition \ref{dsq-Fermi}).

\vspace{0.3cm}
\hf There is a relationship between the local homology of cut loci and the reduced $\check{\textup{C}}$ech cohomology\index{\v{C}ech cohomology} of the \textit{link} of a point in the cut locus. This is due to Hebda \cite[Theorem 1.4 and the remark following it]{Heb83}.
\begin{defn}
    Let $N$ be an embedded submanifold of a complete smooth Riemannian manifold $M$. For each $q\in \cutn$, consider the set $\Lambda(q,N)$ of unit tangent vectors at $q$ so that the associated geodesics realize the distance between $q$ and $N$. This set is called the \textit{link} of $q$ \index{link of a point} with respect to $N$.

    \hf The set of points in $N$ obtained by the end points of the geodesics associated to $\Lambda(q,N)$ will be called the \textit{equidistant set},\index{equidistant set} denoted by $\mathrm{Eq}(q,N)$, of $q$ with respect to $N$. 
\end{defn}
\vspace{0.3cm}
\noindent Since the equidistant set \index{equidistant set} $\mathrm{Eq}(q,N)$, consisting of points which realize the distance $d(q,N)$, is obtained by exponentiating the points in $\Lambda(q,N)$, there is a natural surjection map from $\Lambda(q,N)$ to $\mathrm{Eq}(q,N)$. 
\begin{thm}[Hebda 1983]
    Let $N$ be a properly embedded submanifold of a complete Riemannian manifold $M$ of dimension $n$. If $q\in \cutn$ and $v$ is an element of $\Lambda:=\Lambda(q,N)$, then for any abelian group $G$ there is an isomorphism
    \begin{equation}\label{duality}
        \check{H}^i(\Lambda,v;G)\cong H_{n-1-i}(\cutn,\cutn -q;G).
    \end{equation}
\end{thm}
\vspace{0.3cm}
We are interested in computing $\Lambda(A,O(n,\R))$ for singular matrices $A$. Note that geodesics in $M(n,\R)$, initialized at $A$, are straight lines and any two such geodesics can never meet other than at $A$. Therefore, there is a natural identification between the link and the equidistant set of $A$. 
\begin{lemma}\cite[Lemma 2.15]{BaPr21}\label{link-sing}
    If $A\in M(n,\R)$ is singular of rank $k$, then $\mathrm{Eq}(A,O(n,\R))$ is homeomorphic to $O(n-k,\R)$.
\end{lemma}
\begin{proof}
    Using the singular value decomposition, \index{singular value decomposition} we write $A=UDV^T$, where $U,V\in O(n,\R)$ and $D$ is a diagonal matrix with entries the eigenvalues of $\sqrt{A^T A}$. If we specify that the diagonal entries of $D$ are arranged in decreasing order, then $D$ is unique. Moreover, as $A$ has rank $k<n$, the first $k$ diagonal entries of $D$ are positive while the last $n-k$ diagonal entries are zero. In order to find the matrices in $O(n,\R)$ which realize the distance $d(A,O(n,\R))$, by \eqref{maxdist}, it suffices to find $B\in O(n,\R)$ such that 
    \begin{displaymath}
        \sup_{B\in O(n,\R)} \trace{A^T B}=\sup_{B\in O(n,\R)} \trace{VDU^T B}=\sup_{B\in O(n,\R)} \trace{DU^T BV}
    \end{displaymath}
    is maximized. However, $U^TBV\in O(n,\R)$ has orthonormal rows and the specific form of $D$ implies that the maximum is attained if and only if $U^T BV$ has $e_1,\ldots,e_k$ as the first $k$ rows, in order. Therefore, $U^T BV$ is a block orthogonal matrix, with blocks of $I_k$ and $C\in O(n-k,\R)$, i.e., $B\in U(I_k \times O(n-k,\R))V^T$.
\end{proof}
\begin{cor}\label{locsinghom}
    Let $\mathrm{Sing}$ denote the space of singular matrices in $M(n,\R)$. If $A\in \mathrm{Sing}$ is of rank $k<n$, then for any abelian group $G$ there is an isomorphism
    \begin{equation}\label{locsing}
        H_{n^2-1-i}(\mathrm{Sing},\mathrm{Sing} -A;G) \cong \widetilde{H}^i(O(n-k,\R);G).
    \end{equation}
\end{cor}

\begin{proof}
    It follows from \Cref{link-sing} that $\Lambda(A,O(n,\R))\cong O(n-k,\R)$ if $A$ has rank $k$. Since $O(n-k,\R)$ is a manifold, $\check{\textup{C}}$ech and singular cohomology groups are isomorphic. The space $\mathrm{Sing}$ is a star-convex set, whence all homotopy and homology groups are that of a point. Applying \eqref{duality} in our case, we obtain an isomorphism
    \begin{equation*}
        H_{n^2-1-i}(\mathrm{Sing},\mathrm{Sing} -A)\cong\widetilde{H}^i(O(n-k,\R))
    \end{equation*}
    between reduced cohomology and local homology \index{local homology} groups. In particular, the local homology of the cut locus at $A$ detects the rank of $A$. 
\end{proof}

\begin{rem}
    Note that for a smooth manifold, the relative homology group $H_k(M,M-p)$ does not depend on the point $p$; it is in fact isomorphic to $H_k(\mathbb{R}^m,\mathbb{R}^m-\mathbf{0})$, where $m$ is the dimension of $M$. However, the above result shows that $H_{n^2-1-i}(\mathrm{Sing},\mathrm{Sing}-A;G)$ does depend on $A$ (it depends on the rank of $A$). This is happening because $\mathrm{Sing}$ is not a smooth manifold. It is the zero set of the determinant map $\det:M(n,\mathbb{R})\to \mathbb{R}$.
\end{rem}

\noindent For a computation for $\tilde{H}^i(O(n-k,\mathbb{R});\mathbb{Z})$, we refer the reader to \cite[\S 3.D]{Hat02}.

\vspace{0.3cm}
\noindent Similar computations hold for $U(n,\C)$ and singular $n\times n$ complex matrices.

\begin{thm}
    Let $M(n,\mathbb{C})$ denotes the set of all $n\times n$ complex matrices and $U(n)$ denotes the set of all $n\times n$ unitary matrices. Then we have
    \begin{enumerate}[(i)]
        \item $\cutn[U(n)] = \mathrm{Sing} = \text{set of all singular matrices in $M(n,\mathbb{C})$}$ 
        \item If $A\in M(n,\mathbb{C})$ is singular of rank $k<n$, then $\mathrm{Eq}(A,U(n))$ is homeomorphic to $U(n-k)$.
        \item If $A\in M(n,\mathbb{C})$ is singular of rank $k<n$, then for any abelian group $G$ there is an isomorphism
        \begin{displaymath}
            H_{n^2-1-i}(\mathrm{Sing},\mathrm{Sing} -A;G) \cong \widetilde{H}^i(U(n-k);G).
        \end{displaymath}  
    \end{enumerate}
\end{thm}

\vspace{0.3cm}
\hf We end this chapter by mentioning some properties of the cut locus and separating set with the help of the listed examples. In the next chapter, we will prove these results.
\begin{enumerate}[(P1)]
    \item For a submanifold $N$, the cut locus is the closure of the separating set.
    \item The distance squared function $d^2(N,\cdot)$ from a submanifold $N$ is not differentiable on the separating set $\sen$. 
    \item The distance squared function from $N$ is a Morse-Bott function with $N$ as its critical submanifold.
    \item The complement of $\cutn$ deformation retracts to $N$. Also, the complement of $N$ deforms to the cut locus of $N$.
\end{enumerate}

\chapter{Geometric and topological nature of cut locus}\label{ch:GeometricViewpointOfCutLocus}
\minitoc
\hf The objective of this chapter is to analyze the geometric and topological properties of cut locus of submanifolds. In particular, we will study relations between the distance squared function from a submanifold, the cut locus of submanifold and Thom space of the normal bundle of the submanifold. We will also prove that the distance squared function is a Morse-Bott function. Results in this chapter are based on joint work with Basu \cite{BaPr21}.

\hf A result due to Wolter \cite[Lemma 1]{Wol79} may be generalized to prove (\Cref{Lmm: singdsq}) that the distance squared function from a submanifold is not differentiable on the separating set. This result may be well known to experts, but we provide a proof, following Wolter, which is elementary.

\section{Regularity of distance squared function}\label{sec:RegularityOfDistanceSquaredFunction}
\hfb Recall \Cref{defn:DeltaMap}, where we have defined the $\Delta$ map. The following proposition describes $\Delta$ in terms of the distance function on the Riemannian manifold $M$ from the submanifold $N$.

\vspace{0.3cm}
\begin{prop}\label{dsq-Fermi}
    Let $U$ be a neighbourhood of $N$ such that each point in $U$ admits a unique unit speed $N$-geodesic. If $p\in U$, then 
    \begin{displaymath}
        \Delta(p) = \dist(N,p).
    \end{displaymath}
\end{prop}
\vspace{0.1cm}
\begin{proof}
    Since the expression of $\Delta$ is independent of the choice of the Fermi coordinates, we will make a special choice of the Fermi coordinates $(x_1,\cdots,x_n)$. For $p\in U$, choose the unique unit speed $N$-geodesic $\gamma$ joining $p$ to $N$. This geodesic meets $N$ orthogonally at $\gamma(0)=p'$. Choose $t_0$ such that $\gamma(t_0)= p$.
    \begin{figure}[!htpb]
        \centering
        \includegraphics[width=0.85\textwidth]{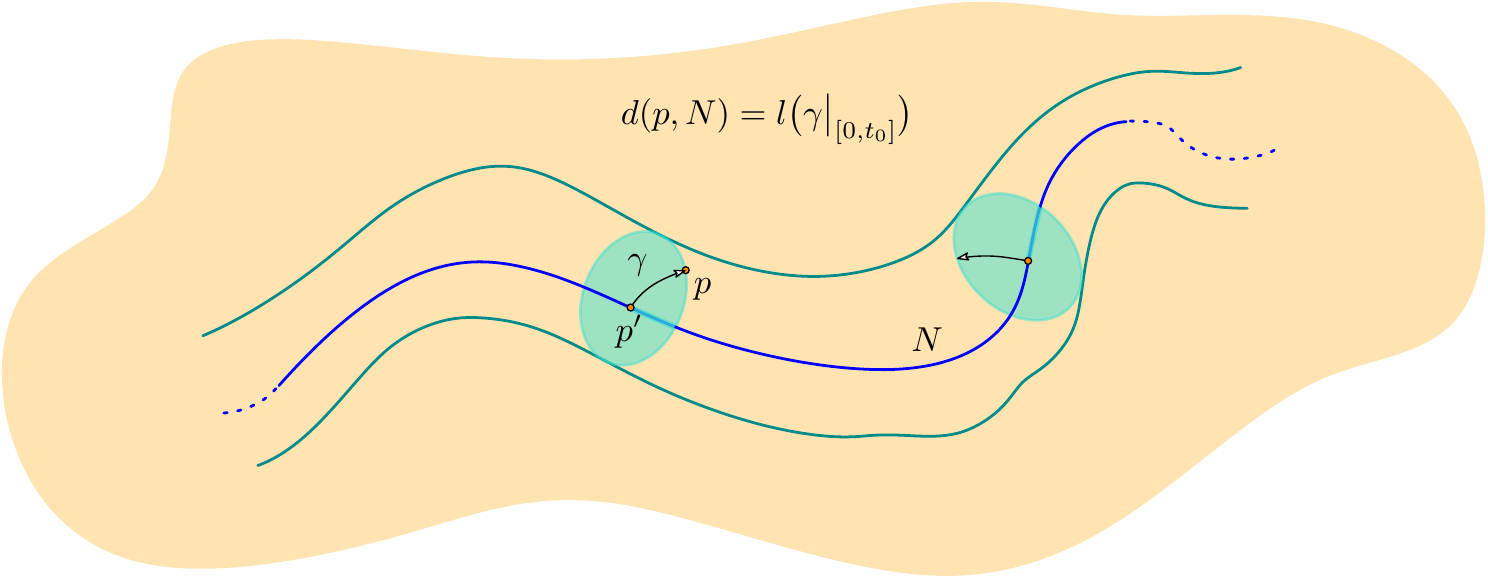}
        \caption{Distance via Fermi coordinates}
        \label{fig:distancesquared}
    \end{figure}
    \noindent According to \autoref{Lemma: geodesic in fermi}, there is a system of Fermi coordinates $(x_1,\cdots,x_n)$ centered at $p'$ such that $x_i(\gamma(t)) = t\delta_{i(k+1)}$. The sequence of equalities 
    \begin{displaymath}
        \Delta(p) = x_{k+1}(\gamma(t_0)) = t_0 = \dist(p,N)
    \end{displaymath}
    complete the proof.
\end{proof}

\begin{cor}\label{dsq-MB}
    Consider the distance squared function with respect to a submanifold $N$ in $M$. The Hessian of the distance squared function at the critical submanifold $N$ is non-degenerate in the normal direction. 
\end{cor}

\vspace{0.3cm}
\hf Towards the regularity of distance squared function, the following observation will be useful. It is a routine generalization of \cite[Lemma~1]{Wol79}.

\begin{lemma}\cite[Lemma 3.7]{BaPr21} \label{Lmm: singdsq}
    Let $M$ be a connected, complete Riemannian manifold and $N$ be an embedded submanifold of $M$. Suppose two $N$-geodesics exist joining $N$ to $q\in M$. Then $d^2(N,\cdot):M\to \rbb$ has no directional derivative at $q$ for vectors in direction of those two $N$-geodesics. 
\end{lemma}
\begin{proof}
    Let us assume that all the geodesics are arc-length parametrized. Let $\gamma_i:[0,\hat{t}]\to M,~~i=1,2$,  be two distinct geodesics with $\gamma_1(0), \gamma_2(0)\in N$ and $\gamma_1(l)=q=\gamma_2(l)$, where $l=d(N,q)$ and $0<l<\hat{t}$. Suppose that the two geodesics start at $p_1$ and $p_2$ and so $d(p_1,q)=l=d(p_2,q)$. Note that the directional derivative of $d^2$ at $q$ in the direction of $\gamma_i'(q)$ from the left is given by
    \begin{align*}
        (d^2)'_-(q) & := \lim_{\varepsilon\to 0^+} \dfrac{(d(N,\gamma_i(l)))^2-(d(N,\gamma_i(l-\varepsilon)))^2}{\varepsilon}\\
		& = \lim_{\varepsilon\to 0} \dfrac{(d(p_i,\gamma_i(l)))^2-(d(p_i,\gamma_i(l-\varepsilon)))^2}{\varepsilon}\\
        & = \lim_{\varepsilon\to 0^+} \dfrac{l^2-(l-\varepsilon)^2}{\varepsilon}\\
		& = \lim_{\varepsilon\to 0} \dfrac{l^2-l^2+2l\varepsilon-\varepsilon^2}{\varepsilon}\\
		& = \lim_{\varepsilon\to 0}\dfrac{2l\varepsilon-\varepsilon^2}{\varepsilon}\\
        & = 2l.
    \end{align*}
    Next, we claim that the derivative of the same function from the right is strictly bounded above by $2l$. Let $\omega\in(0,\pi]$ be the angle between the two geodesics $\gamma_1$ and $\gamma_2$ at $q$.  Define the function,
    \begin{equation*}
        u(\tau)\defeq d(N,\gamma_1(l-\varepsilon))+d(\gamma_1(l-\varepsilon),\gamma_2(\tau+l)).
    \end{equation*} 
    \begin{figure}[h]
        \centering
        \includegraphics[scale=1]{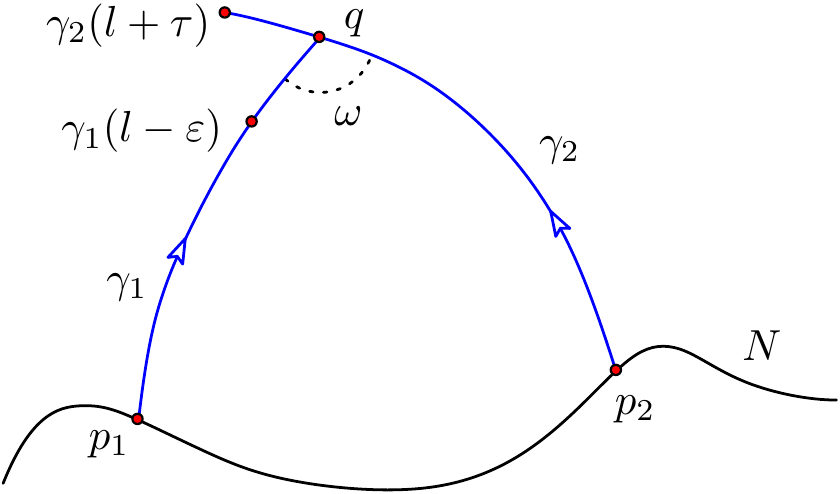}
        \caption{When two $N$-geodesics meet}
        \label{fig: notDifferentiable}
    \end{figure}
    By triangle inequality, we observe that
    \begin{equation*}
    f(\tau)\defeq (u(\tau))^2 \ge d^2(p_1,\gamma_2(\tau+l)) \ge d^2(N,\gamma_2(\tau+l)),
    \end{equation*}
    and equality holds at $\tau=0$ and $(u(0))^2=d^2(N,q)=l^2$. Thus, in order to prove the claim, it suffices to show that the derivative of $f$ from right, at $\tau=0$, is bounded below by $2l$. We need to invoke a version of the cosine law for small geodesic triangles. Although this may be well-known to experts, we will use the version that appears in \cite{Sha07} (also see \cite[Lemma 2.4]{DaDe18} for a detailed proof). In our case, this means that 
    \begin{displaymath}
        d^2(\gamma_1(l-\varepsilon),\gamma_2(\tau+l))=\ep^2+\tau^2+2\ep\tau \cos \omega +K(\tau)\ep^2\tau^2
    \end{displaymath}
    where $|K(\tau)|$ is bounded, and the side lengths are sufficiently small. Note that we are considering geodesic triangles with two vertices constant and the varying vertex being $\gamma_2(l+\tau)$. It follows from taking a square root and then expanding in powers of $\tau$ that
    \begin{displaymath}
        d(\gamma_1(l-\varepsilon),\gamma_2(\tau+l))=\sqrt{\varepsilon^2+\tau^2+2\varepsilon\tau\cos\omega}~(1+O(\tau^2)).
    \end{displaymath}
    It follows that 
    \begin{displaymath}
        u(\tau)=l-\ep+\sqrt{\varepsilon^2+\tau^2+2\varepsilon\tau\cos\omega}~(1+O(\tau^2)).
    \end{displaymath} 
    Therefore, $u'_+(0)=\cos \omega=d'_+(\gamma_1(l-\varepsilon),\gamma_2(l))$. Observe that 
    \begin{align*}
        f'_+(\tau)\Big|_{\tau=0} & = 2d(N,\gamma_1(l-\varepsilon))d'_+(\gamma_1(l-\varepsilon),\gamma_2(l)) \\
        & \kern 3cm +2d(\gamma_1(l-\varepsilon),\gamma_2(l))d'_+(\gamma_1(l-
    \varepsilon),\gamma_2(l))\\
        & = 2d(N,\gamma_1(l-\varepsilon))\cos \omega+2d(\gamma_1(l-\varepsilon),\gamma_2(l))\cos\omega\\
		& = 2\cos \omega\big[ d(N,\gamma_1(l-\varepsilon))+d(\gamma_1(l-\varepsilon),\gamma_2(l))) \big]\\
		& = 2\cos\omega \big[ d(N,\gamma_1(l-\varepsilon))+d(\gamma_1(l-\varepsilon),\gamma_1(l))) \big] \\
        & =2d(N,\gamma_1(l))\cos\omega<2l.
    \end{align*}
    Thus, we have proved the claim and subsequently the result.
\end{proof}

\hf The above lemma shows that $d^2$ is smooth away from the cut locus. The following example suggests that $d^2$ can be differentiable at points in $\cutn -\sen$ (see \Cref{defn:SeparatingSet}) but not twice differentiable.

\vspace{0.3cm}
\begin{eg}[Cut locus of an ellipse]\label{eg:CutLocusOfEllipse-2}
    We discuss the regularity of the distance squared function from an ellipse $x^2/a^2+y^2/b^2=1$ (with $a>b>0$) in $\R^2$. For a discussion of the cut locus for ellipses inside $\mathbb{S}^2$ and ellipsoids, see \cite[pages 90-91]{Heb95}. Let $(x_0,y_0)$ be a point inside the ellipse lying in the first quadrant. The point closest to $(x_0,y_0)$ and lying on the ellipse is given by 
    \begin{displaymath}
        x=\frac{a^2x_0}{t+a^2},\,\,y=\frac{b^2y_0}{t+b^2},
    \end{displaymath}
    where $t$ is the unique root of the quartic
    \begin{displaymath}
        \left(\frac{ax_0}{t+a^2}\right)^2+\left(\frac{by_0}{t+b^2}\right)^2=1
    \end{displaymath}
    in the interval $(-b^2,\infty)$. Given $(\alpha,\beta)$ with $\beta>0$, we set $P_\ep(\alpha,\beta)=(\frac{a^2-b^2}{a}+\ep\alpha,\ep\beta)$; this defines a straight line passing through $P_0(\alpha,\beta)$ in the direction of $(\alpha,\beta)$. For $\ep>0$, $P_\ep(\alpha,\beta)$ lies in the first quadrant and we denote by $t=t(\ep)$ be the unique relevant root of the quartic
    \begin{displaymath}
        \left(\frac{a(\frac{a^2-b^2}{a}+\ep\alpha)}{t+a^2}\right)^2+\left(\frac{b\ep\beta}{t+b^2}\right)^2=1.
    \end{displaymath}
    Simplifying this after dividing by $\ep$ and taking a limit $\ep\to 0^+$, we obtain
    \begin{displaymath}
        \frac{2a\alpha}{a^2-b^2}=\lim_{\ep\to 0^+}\left(\Big(\frac{2}{a^2-b^2}\Big)\frac{t+b^2}{\ep}-b^2\beta^2 \frac{\ep}{(t+b^2)^2}\right).
    \end{displaymath}
    On the other hand, the point $Q_\ep(\alpha,\beta)$ on the ellipse closest to $P_\ep(\alpha,\beta)$ is given by
    \begin{figure}[!htpb]
        \centering
        \includegraphics[scale=0.7]{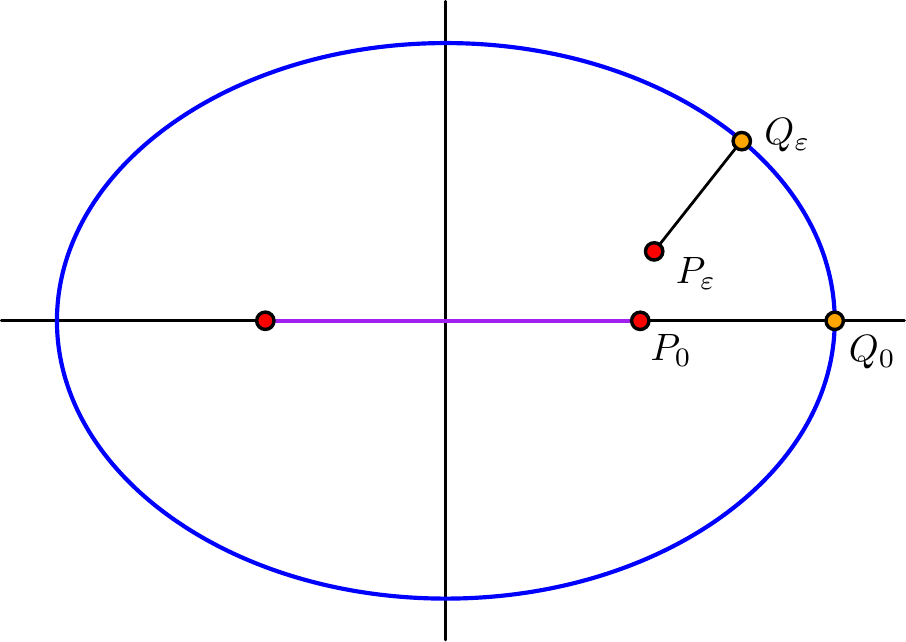}
        \caption{Cut locus of an ellipse}
    \end{figure}
    \begin{displaymath}
        x_\ep=\frac{a^2(\frac{a^2-b^2}{a}+\ep\alpha)}{t+a^2},\,\,y_\ep=\frac{b^2 \ep\beta}{t+b^2}.
    \end{displaymath}
    It follows that
    \begin{equation}
        d_\ep^2(\alpha,\beta):=d^2(P_\ep,Q_\ep)=\frac{t^2}{a^2}\left(\frac{a^2-b^2+a\ep\alpha}{t+a^2}\right)^2+\frac{t^2}{b^2}\left(\frac{b\ep\beta}{t+b^2}\right)^2\label{dsqellip}
    \end{equation}
    Using $t(0)=-b^2$ and simplifications lead us to the following
\begin{align*}
    \lim_{\ep\to 0^+}\frac{d_\ep^2-d_0^2}{\ep} & =\frac{2ab^4\alpha}{a^2(a^2-b^2)}-\lim_{\ep\to 0^+}\left(\frac{(t+b^2)(a^2b^2-a^2 t+2b^2t)}{\ep(t+a^2)^2}-\beta^2 \frac{t^2\ep}{(t+b^2)^2}\right)\\
    & = \frac{2ab^4\alpha}{a^2(a^2-b^2)}-\frac{2b^2}{a^2-b^2}\lim_{\ep\to 0}\frac{t+b^2}{\ep}+\beta^2 b^4\lim_{\ep\to 0}\frac{\ep}{(t+b^2)^2}\\
    & = \frac{2ab^4\alpha}{a^2(a^2-b^2)}-\frac{2ab^2\alpha}{a^2-b^2}=-\frac{2b^2\alpha}{a}.
\end{align*}
    
\hf On the other hand, for $\ep<0$, the point $P_\ep(\alpha,\beta)$ lies in the fourth quadrant. By symmetry, the distance between $P_\ep(\alpha,\beta)$ and $Q_\ep(\alpha,\beta)$ is the same as that between $P_{-\ep}(-\alpha,\beta)$ and $Q_{-\ep}(-\alpha,\beta)$. However, it is seen that 
    \begin{displaymath}
        d^2(P_{-\ep}(-\alpha,\beta), Q_{-\ep}(-\alpha,\beta))=d^2_{-\ep}(-\alpha,\beta)
    \end{displaymath}
    as defined in \eqref{dsqellip}. Therefore,
    \begin{align*}
        \lim_{\ep\to 0^-}\frac{d^2(P_{\ep}(\alpha,\beta), Q_{\ep}(\alpha,\beta))-d^2(P_{0}(\alpha,\beta), Q_{0}(\alpha,\beta))}{\ep} & =  \lim_{\ep\to 0^-}\frac{d^2 _{-\ep}(-\alpha,\beta)-d^2 _0(-\alpha,\beta)}{\ep}\\
        & =  -\lim_{-\ep\to 0^+}\frac{d^2 _{-\ep}(-\alpha,\beta)-d^2 _0(-\alpha,\beta)}{-\ep}\\
        & = -\frac{2b^2\alpha}{a},
    \end{align*}
    where the last equality follows from the right hand derivative of $d^2$, as computed previously.
    
    \vspace{0.1cm}
    \hf When $\beta=0$ we would like to compute $d^2_\ep(\alpha,0)$. If $\ep>0$, then 
    \begin{equation}
        \label{dsqP01}d^2_\ep(\alpha,0)=(b^2/a-\ep\alpha)^2=\frac{b^4}{a^2}-\frac{2b^2\alpha\ep}{a}+\alpha^2\ep^2.
    \end{equation}
    On the other hand, if $\ep<0$ is sufficiently small, then there are two points on the ellipse closest to $P_\ep(\alpha,0)=(\frac{a^2-b^2}{a}+\ep\alpha,0)$, with exactly one on the first quadrant, say $Q_\ep$. Since the segment $P_\ep Q_\ep$ must be orthogonal to the tangent to the ellipse at $Q_\ep$, we obtain the coordinates for $Q_\ep$:
    \begin{displaymath}
        x_\ep=\frac{a^2(\frac{a^2-b^2}{a}+\ep\alpha)}{a^2-b^2},\,\,y^2_\ep=b^2\bigg(1-\frac{x_\ep^2}{a^2}\bigg),\,\,y_\ep>0.
    \end{displaymath}
    We may compute the distance
    \begin{equation}
        \label{dsqP02}d^2_{\ep}(\alpha,0):=(d(P_\ep,Q_\ep))^2=\frac{b^4}{a^2}-\frac{2b^2\alpha\ep}{a}-\frac{b^2\alpha^2\ep^2}{a^2-b^2},
    \end{equation}
    where $\ep<0$. Combining \eqref{dsqP01} and \eqref{dsqP02} we conclude that $d^2$ is differentiable at $P_0=(\frac{a^2-b^2}{a},0)$, a point in $\cutn$ but not in $\sen$. However, comparing the quadratic part of $d^2$ in \eqref{dsqP01},\eqref{dsqP02} we conclude that $d^2$ is not twice differentiable at $P_0$. 
\end{eg}

\section{Characterizations of \texorpdfstring{\cutn}{Cu(N)}}\label{sec:characterizationOfCutLocus}
\hfb Let $(M,g)$ be a complete Riemannian manifold with distance function $d$. The exponential map \index{exponential map} at $p\in M$
\begin{displaymath}
    \exp_p:T_p M \to M
\end{displaymath}
is defined on the tangent space of $M$. Moreover, there exists minimal geodesic joining any two points in $M$. However, not all geodesics are distance realizing. Given $v\in T_p M$ with $\|v\|=1$, let $\gamma_v$ be the geodesic initialized at $p$ with velocity $v$. Let $S(TM)$ denote the unit tangent bundle and let $ [0,\infty]$ be the one point compactification of $ [0,\infty)$. Define 
\begin{displaymath}
    s:S(TM)\to [0,\infty],\,\,s(v):=\sup\{t\in [0,\infty)\,|\,\gamma_v |_{[0,t]}\,\,\textup{is minimal}\}.
\end{displaymath}

\begin{defn}[Cut Locus]\label{cutlocus2}
    Let $M$ be a complete, connected Riemannian manifold. If $s(v)<\infty$ for some $v\in S(T_pM)$, then $\textup{exp}_p(s(v)v)$ is called a \textit{cut point}. The collection of cut points is defined to be the cut locus of $p$.
\end{defn}

\vspace{0.3cm}
\hf As geodesics are locally distance realizing, $s(v)> 0$ for any $v\in S(TM)$. The following result \cite[Proposition 4.1]{Sak96} will be important for the underlying ideas in its proof.
\begin{prop}\label{scts}
    The  map $s:S(TM)\to [0,\infty], u\mapsto s(u)$ is continuous.
\end{prop}
\vspace{0.3cm}
\noindent The proof relies on a characterization of $s(v)$ provided $s(v)<\infty$, (\Cref{thm:CharacterizationOfCutLocusInTermsOfConjugatePoint}). A positive real number $T$ is $s(v)$ if and only if $\gamma_v:[0,T]$ is minimal and at least one of the following holds:\\
\hf (i) $\gamma_v(T)$ is the first conjugate point of $p$ along $\gamma_v$,\\
\hf (ii) there exists $u\in S(T_pM), u\neq v$ such that $\gamma_u(T)=\gamma_v(T)$.

\begin{rem}
    If $M$ is compact, then it has bounded diameter, which implies that $s(v)<\infty$ for any $v\in S(TM)$. The converse is also true: if $M$ is complete and connected with $s(v)<\infty$ for any $v\in S(TM)$, then $M$ has bounded diameter, whence it is compact. 
\end{rem}

\subsection{Characterization in terms of focal points}
\hfb We shall be concerned with closed Riemannian manifolds in what follows. Let $N$ be an embedded submanifold inside a 
closed, i.e., compact without boundary, manifold $M$. Let $\nu$ denote the normal bundle of $N$ in $M$ with $D(\nu)$ denoting the unit disk bundle. In the context of $S(\nu)$, the unit normal bundle and the cut locus of $N$, distance minimal geodesics or $N$-geodesics are relevant (see \Cref{distmin,cutlocus1}). We want to consider 
\begin{equation}\label{snu}
    \rho:S(\nu)\to [0,\infty),\,\,\rho(v):=\sup\{t\in [0,\infty)\,|\,\gamma_v |_{[0,t]}\,\,\textup{is an $N$-geodesic}\}.
\end{equation}
Notice that $0<\rho(v)\leq s(v)$ for any $v\in S(\nu)$. In the special case when $N=\{p\}$, $\rho$ is simply the restriction of $s$ to $T_p M$. The continuity of $\rho$ requires a result similar to \Cref{thm:CharacterizationOfCutLocusInTermsOfConjugatePoint}, which requires the definition of focal points.  
\begin{defn}[Focal point]\index{focal point}\label{defn:focalPoint}
    Let $p\in N$ and $(p,v)\in S(\nu)$. We say that $v$ is a \emph{tangent focal point}\index{tangent focal point} of $N$ if $d (\exp_\nu)_v$ is not of full rank. If $\gamma$ is a geodesic from $0$ to $v$ in $\nu_p$, then $\exp_\nu(v)$ is called a \emph{focal point of $N$ along $\exp_\nu(\gamma)$}.
\end{defn}
\vspace{0.3cm}
\begin{figure}[!htb]
    \centering
    \def\svgwidth{0.5\columnwidth}
    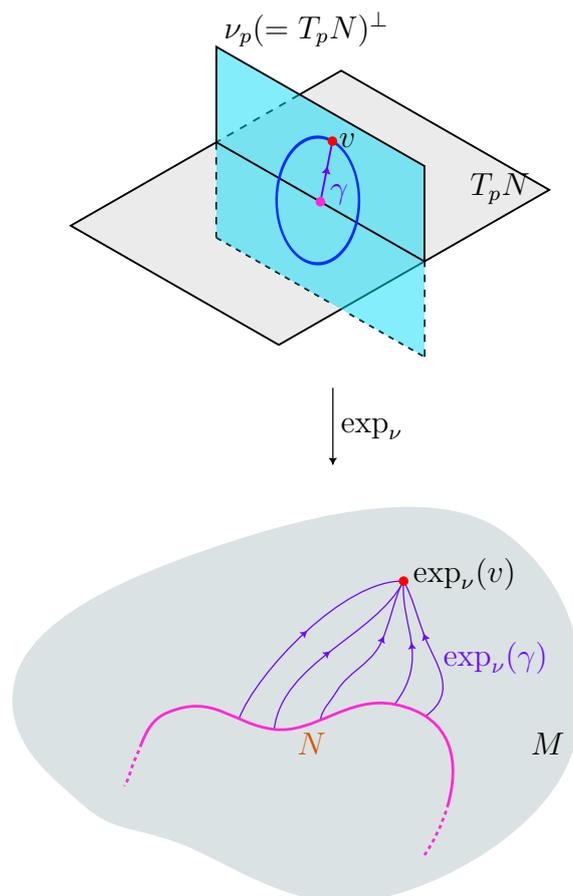

    \caption{Focal points\label{fig:FocalPoints}}
\end{figure}

\hf The nullity of $d\exp_\nu$ at $v$ is called the \textit{multiplicity} of a focal point.\index{muliplicity of focal point} If it is one, we say it the \textit{first focal point}.\index{first focal point}  Analogous to \Cref{thm:CharacterizationOfCutLocusInTermsOfConjugatePoint}, we have the following result.
\begin{thm}\label{thm:CharacterizationOfCutLocusInTermsOfFocalPoint}
    Let $u\in S_p(\nu)$. A positive real number $T$ is $\rho(u)$ if and only if $\gamma_u:[0,T]$ is an $N$-geodesic and at least one of the following holds:
    \begin{enumerate}[(i)]
        \item $\gamma_u(T)$ is the first focal point of $N$ along $\gamma_u$,
        \item there exists $v\in S(\nu)$ with $v\neq u$ such that $\gamma_v(T)=\gamma_u(T)$.
    \end{enumerate}
\end{thm}

\vspace{0.3cm}
\noindent In order to prove the above theorem, we need the following observations.

\vspace{0.3cm}
\noindent \textbf{Observation A}\,\textup{\cite[Lemma 2.11, page 96]{Sak96}}\,\,\textit{Let $N$ be a submanifold of a Riemannian manifold $M$ and $\gamma:[t_0,\infty)\to M$ a geodesic emanating perpendicularly from $N$. If $\gamma(t_1)$ is the first focal point of $N$ along $\gamma$, then for $t>t_1$, $\gamma|_{[t_0,t_1]}$ cannot be an $N$-geodesic, i.e., $L\paran{\gamma|_{[t_0,t]}}>d(N,\gamma(t))$.}\\[0.2cm]
\noindent Recall that a sequence $\{\gamma_n\}$ of geodesics, defined on closed intervals, is said to converge to a geodesic $\gamma$ if $\gamma_{n}(0)\to \gamma(0)$ and $\gamma_n'(0)\to \gamma'(0)$. It follows from the continuity of the exponential map that if $t_n\to t$, then $\gamma_n(t_n)\to \gamma(t)$.\\[0.3cm]
\textbf{Observation B}\,\,\textit{Let $\gamma_n$ be a sequence of unit speed $N$-geodesics joining $p_n=\gamma_n(0)$ to $q_n=\gamma_n(t_n)$. If $\gamma_n$ converges to a geodesic $\gamma$ and $t_n\to l$, then $\gamma$ is a unit speed $N$-geodesic joining $p=\lim_n p_n$ to $q \defeq \gamma(l)=\lim_n \gamma_n(t_n)$.}
\begin{proof}
    The unit normal bundle $S(\nu)$ is closed. Since $\gamma_n'(0)\to \gamma'(0)$, it follows that $\gamma'(0)\in S(\nu)$. Note that
        \begin{displaymath}
        d(N,q) = \lim_{n\to \infty} d(N,q_n) = \lim_{n\to \infty} d(p_n,q_n)= \lim_{n\to \infty} t_n = l = L\paran{\gamma|_{[0,l]}}
        \end{displaymath}
    implies that $\gamma$ is an $N$-geodesic.
\end{proof}

\begin{proof}[Proof of \Cref{thm:CharacterizationOfCutLocusInTermsOfFocalPoint}]
    If $\gamma_u(t)$ is the first focal point of $N$ along $\gamma_u$, then Observation A implies that $\gamma_u$ cannot be minimal beyond this value. If (ii) holds, then we need to show that for sufficiently small $\varepsilon>0~$ $\gamma_u|_{[0,T+\varepsilon]}$ is not minimal. Suppose, on the contrary, that $\gamma_u$ is minimal beyond $T$. Take a minimal geodesic $\beta$ joining $\gamma_{v}(T-\varepsilon)$ to $\gamma_{u}(T+\varepsilon)$. Observe that,
	\begin{align*}
	    2\varepsilon & = d\paran{\gamma_{u}(T+\varepsilon),\gamma_u(T)} + d\paran{\gamma_v(T),\gamma_v(T-\varepsilon)} \\
        & > d\paran{\gamma_u(T+\varepsilon),\gamma_v(T-\varepsilon)}.
	\end{align*}
    If $p,q,r\in M$ such that $d(p,q)+d(q,r) = d(p,r)$ and there exist the shortest normal geodesic $\gamma_{1}$ and $\gamma_{2}$ joining $p$ to $q$ and $q$ to $r$, respectively, then $\gamma_{1}\cup \gamma_2$ is smooth at $q$ and defines a shortest normal geodesic joining $p$ to $r$. Therefore, we have
    \begin{align*}
		L(\gamma_v|_{[0,T-\varepsilon]}\cup \beta) & = T-\varepsilon + d(\gamma_v(T-\varepsilon),\gamma_u(T+\varepsilon)) \\
        & < T + \varepsilon = L(\gamma_u|_{[0,T+\varepsilon]}).
	\end{align*}
    This contradiction establishes that $\gamma_u|_{[0,T+\varepsilon]}$ is not minimal.
    
    \vspace{0.1cm}
    \hf For the converse, set $T=\rho(u)$ and observe that $\gamma_u|_{[0,T]}$ is an $N$-geodesic. Assuming that $q\defeq \gamma_{u}(T)$ is not the first focal point of $N$ along $\gamma_{u}$, we will prove that (ii) holds. Let $p=\gamma_u(0)$ and choose a neighbourhood $\tilde{U}$ of $Tu$ in $\nu$ such that $\exp_{\nu}|_{\tilde{U}}$ is a diffeomorphism. For sufficiently large $n$, $q_{n}\defeq \gamma_{u}\paran{T+1/n}\in \exp_{\nu}(\tilde{U})$. Take $N$-geodesics $\gamma_{n}$ parametrized by arc-length joining $p_n$ to $q_{n}$ and set $u_{n}\defeq \dot{\gamma}_{n}(0)\in S((T_{p_n}N)^\perp)$. Since $S((T_{p_n}N)^\perp)$ is compact, by passing to a subsequence, we may assume that $u_{n}$ converges to $v\in S(N_{p})$. By Observation B, 
	\begin{displaymath}
	    \gamma_v(T) = \lim_{n\to \infty} \gamma_{u_n}\paran{T+\textstyle{\frac{1}{n}}} = \gamma_u(T).
	\end{displaymath}
    If $v = u$, then for sufficiently large $n$, $d(p,q_n)u_{n} \in \tilde{U}$, whence 
	\begin{displaymath}
	    \paran{T+\textstyle{\frac{1}{n}}}u = d(p,q_n)u_n.
	\end{displaymath}
    Taking absolute values on both sides imply $T+1/n>d(p,q_n)$. This contradiction implies $v\neq u$.
\end{proof}
\hf We now will prove that the map $\rho$ defined in \ref{snu} is a continuous function.
\begin{prop}\label{snucts}
    The map $\rho:S(\nu)\to [0,\infty)$, as defined in \eqref{snu}, is continuous.
\end{prop}
\begin{proof}
    We will prove that $\rho(u_n)\to \rho(u)$ whenever $(p_n,u_n)\to (p,u)$ in the unit normal bundle $S(\nu)$. Let $T$ be any accumulation point of the sequence $\{\rho(u_n)\}$ including $\infty$. By Observation B, $\gamma_{u}|_{[0,T]}$ is an $N$-geodesic and hence $T\le \rho(u)$. If $T=+\infty$, we are done. So let us assume that $T<+\infty$. From \Cref{thm:CharacterizationOfCutLocusInTermsOfFocalPoint}, at least one of the following holds for infinitely many $n$.
    \begin{enumerate}[(i)]
        \item The sequence $\rho(u_n)$ is the first focal point to $N$ along $\gamma_{u_n}$
        \item  there exist $v_{n}\in S(\nu)$,  $v_{n}\neq u_n$ with $\gamma_{u_n}\paran{\rho(u_n)} = \gamma_{v_n}\paran{\rho(u_n)}$.
    \end{enumerate}

    \hf If (i) is true for infinitely many $n$, then choose infinitely many unit vectors $\{w_n\}$ which belong to the kernel $\ker \paran{D\exp_{\nu}(\rho(u_n)u_n)}$ and are contained in a compact subset of $S(\nu)$. Choose a convergent subsequence whose limit $w$ is contained in $\ker \paran{D\exp_{\nu}{(Tu)}}$. Since $w\neq 0$, the rank of $D\exp_{\nu}{(Tu)}$ is less than $\dim M$. Thus, $\gamma_{u}(T)$ is the first focal point of $N$ along $\gamma_u$ and $T=\rho(u)$.

    \vspace{0.1cm}
    \hf If (ii) is true for infinitely many $n$, then we may assume that $v_{n}\to v\in S(\nu)$. If $v\neq u$, then \Cref{thm:CharacterizationOfCutLocusInTermsOfFocalPoint} (ii) holds for $T$, whence $T=\rho(u)$. If $v=u$, we claim that $\gamma_{u}(T)$ is the first focal point of $N$ along $\gamma_{u}$. If not, then the map $\exp_{\nu}$ is regular at $Tu\in \nu$ and hence the map 
    \begin{displaymath}
        \Phi:\nu\to M\times M,~(p,u)\mapsto (p,\exp_\nu(p,u))
    \end{displaymath}
    is regular at $Tu$. Therefore, $\Phi$ is a diffeomorphism if restricted to an open neighbourhood $\tilde{U}$ of $Tu$ in $\nu$. Since $v=u$, which implies for sufficiently large $n$,~ $(p_n,\rho(u_n)u_n)$ and $(p_n,\rho(u_n)v_n)$ belong to $\tilde{U}$ and are different. On the other hand, by assumption $\Phi(\rho(u_n)u_n) = \Phi(\rho(u_n)v_n) $, which is a contradiction. Therefore, $\gamma_u(T)$ is the first focal point and $T=\rho(u)$. 
\end{proof}

\subsection{Characterization in terms of separating set}
\hfb Recall that the separating set \index{separating set} of $N$, $\sen$, consists of all points $q\in M$ such that at least two distance minimal geodesics from $N$ to $q$ exist. If $q\in \sen$ but $q\not\in \cutn$, then we have \Cref{fig: notDifferentiable}, i.e., $\gamma_1$ is an $N$-geodesic beyond $q$ while $\gamma_2$ is another $N$-geodesic for $q$. The triangle inequality applied to $\gamma_1(0)$, $q=\gamma_1(l)$ and $\gamma_2(l+\tau)$ implies that 
\begin{displaymath}
    d(\gamma_2(l+\tau),N)< l+\tau
\end{displaymath}
while for $\tau$ small enough $d(\gamma_2(l+\tau),N)=l+\tau$ as $\gamma_2$ is an $N$-geodesic beyond $q$. This contradiction establishes the well-known fact $\sen\subseteq \cutn$. In quite a few examples, these two sets are equal (see \Cref{eg:CutLocusOfEllipse} where these two are not same). In the case of $M=\mathbb{S}^n$ with $N=\{p\}$, the set $\sen$ consists of $-p$. There is an infinite family of minimal geodesics joining $p$ to $-p$. An appropriate choice of a pair of such minimal geodesics would create a loop, which is permissible in the definition of $\sen$. According to \Cref{thm:CharacterizationOfCutLocusInTermsOfFocalPoint}, a cut point is either a first focal point of $N$ along a geodesic or it is a separating point. We will now prove our one of the observations in the last chapter, that  cut locus of a submanifold $N$ is the closure of separating set of $N$. 
\begin{thm}\label{thm: Theorem 2}\label{thm:SeClosureIsCutLocus}
    Let $\cutn$ be the cut locus of a compact submanifold $N$  of a  complete Riemannian manifold $M$. The subset $\sen$ of $\cutn$ is dense in $\cutn.$
\end{thm}
\begin{proof}
    Let $q\in \cutn $ but not in $\sen$. Choose an $N$-geodesic $\gamma$ joining $N$ to $q$ such that any extension of $\gamma$ is not an $N$-geodesic. This geodesic $\gamma$ is unique as $q\notin \sen$. We may write $\gamma(t) = \exp_\nu(tx)$, where $\gamma(0)=p\in N$ and $\gamma'(0)=x_0\in S(\nu_p)$. It follows from the definition of $\rho$ that $q=\exp_\nu\paran{\rho(x_0)x_0}$. We need to show that every neighborhood of $q$ in $\cutn$ must intersect $\sen$. Suppose it is false. Let $\delta>0$ and consider $\Ball$, the closed ball with center $x_0$ and radius $\delta$. Define the cone
    \begin{equation*}
        \co\! \defeq \curlybracket{tx:0\le t\le 1,~ x\in \Ball\cap S(\nu)}.
    \end{equation*}
    Since $\Ball\cap S(\nu)$ is homeomorphic to a closed $(n-1)$-ball for sufficiently small $\delta$, the cone will be homeomorphic to a closed Euclidean $n$-ball.
    \begin{figure}[H] 
        \centering 
        \includegraphics[angle=270, scale=0.8]{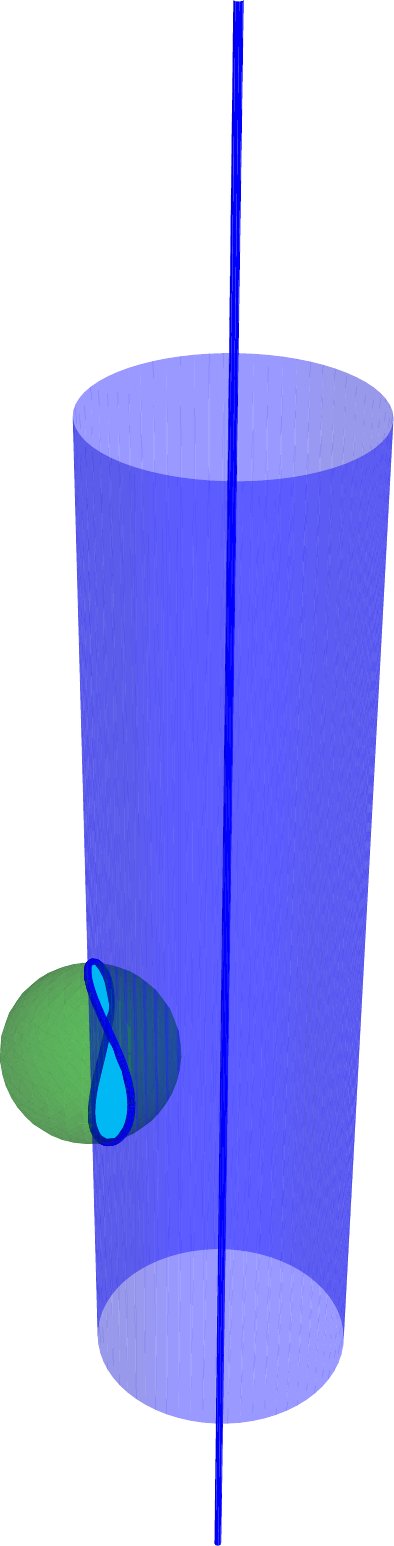} 
        \caption{\co} 
        \label{fig: cone} 
    \end{figure}
    \noindent Similarly, define another cone
    \begin{equation*}
        \costar\! \defeq \big\{\rho\paran{ x/\norm{x}}x\,|\, x\in \co\!, ~x\neq 0\big\}\cup \{0\}.
    \end{equation*} 
    Note that $\rho(x_0)$ is finite. As $\rho$ is continuous, due to \Cref{snucts}, for sufficiently small $\delta$ the term $\rho\paran{x/\norm{x}}$ is still finite, whence \costar is well defined. We claim that \costar is also homeomorphic to a closed Euclidean $(n-k)$-ball. Indeed,  a non-zero $x\in \co\!$ implies $x=\lambda \hat{x}$, for some $\lambda \in (0,1]$ and $\hat{x}\in \Ball\cap S(\nu)$. Since $\rho(\hat{x})x = \lambda \rho(\hat{x})\hat{x}$, it follows that \costar is the cone of the set
    \begin{displaymath}
        \{\rho(\hat{x})\hat{x}\,|\,\hat{x}\in \Ball\cap S(\nu)\},
    \end{displaymath}
    which is homeomorphic to $\Ball\cap S(\nu)$. Now we have a dichotomy: 
    \begin{enumerate}[(a)]
        \item For a fixed small $\delta>0$, the restriction of $\exp_\nu$ to \costar is a homeomorphism to its image because it is injective, or
        \item For any $\delta>0$, the restriction of $\exp_\nu$ to \costar is not injective.
    \end{enumerate}
    If (b) holds, choose $v_n\neq w_n\in \mathrm{Co}^\star(x_0,\frac{1}{n})$ such that these map to $q_n$ under $\exp_\nu$. Thus,  $q_n\in \sen$ and compactness of $S(\nu)$ ensures that $q_n$ converges to $q$. If (a) holds, then let $B(q,\varepsilon)$ denote the open ball in $M$ centered at $q$ with radius $\varepsilon>0$. We claim that it intersects the complement of $\exp_\nu(\costar\!\!)$ in $M$. But it is true as $\rho(x_0)x_0$ lies on the boundary of \costar and hence it has a neighborhood in \costar which is homeomorphic to a closed $n$-dimensional Euclidean half plane. Since $\exp_{\nu}$ restricted to \costar was a homeomorphism, the open ball $B(q,\varepsilon)$ must intersect the points outside the image of $\exp_{\nu}(\costar\!\!)$. 

    \vspace{0.2cm}
    \hf Now take $\varepsilon = \frac{1}{n}$. For each $n$, there exists $q_{n}\in B\paran{q,\frac{1}{n}}$ such that $q_{n}\notin \exp_{\nu}(\costar\!\!)$. Since $M$ is complete, for each point $q_{n}$ let $\gamma_{n}$ be a $N$-geodesics joining $p_n\in N$ to $q_{n}$. We may invoke the following result from Buseman's book \cite[Theorem 5.16, page 24]{Bus55}. Let $\curlybracket{\gamma_{n}}$ be a sequence of rectifiable curves in a finitely compact set $X$ and the lengths $\ell(\gamma_n)$ are bounded. If the initial points $p_{n}$ of $\gamma_{n}$ forms a bounded set, then $\curlybracket{\gamma_n}$ contains a subsequence $\gamma_{n_k}$ which converges uniformly to a rectifiable curve $\tilde{\gamma}$ in $X$ and 
    \begin{equation*}\label{Fact A eqn: 1} 
        \ell(\tilde{\gamma}) \le \lim\inf \ell\paran{\gamma_{n_k}}
    \end{equation*}
    Since $\{p_n\}$ lie in the compact set $N$, we obtain a rectifiable curve $\tilde{\gamma}$ such that 
    \begin{displaymath}
        \ell(\tilde{\gamma})\leq  \lim\inf \ell\paran{\gamma_{n_k}}=\lim_k \ell(\gamma_{n_k})=\lim_k d(q_{n_k},N)=d(q,N).
    \end{displaymath}
    Thus, $\tilde{\gamma}$ is actually an $N$-geodesic joining $p'=\lim_k p_{n_k}$ to $q$  and the unit tangent vectors $x_{n_k}=\gamma_{n_k}'(0)$ at $p_{n_k}$ converges to the unit tangent vector $\tilde{x}=\tilde{\gamma}'(0)$ at $p'$.  Since $x_{0}$ is an interior point of the set $\Ball\cap S(\nu)$, any sequence in $S(\nu)$ converging to $x_{0}$ must eventually lie in \co\!. According to our choice, $q_{n_k}\notin \exp_{\nu}(\costar\!\!)$ and $x_{n_k}$ all lie outside of \co\!. Hence $x_{0}\neq \tilde{x}$ and $\gamma \neq \tilde{\gamma}$. Thus, there are two distinct $N$-geodesics $\gamma$ and $\tilde{\gamma}$ joining $N$ to $q$, a contradiction to $q\notin \sen$. This completes the proof.
\end{proof}

\section{Topological properties}\label{sec:topologicalProperties}
\hfb In this section we will study the structure of the cut locus and the relation of cut locus to the Thom space. We will also see various applications of this relation.
\subsection{Structure of the cut locus}
\hfb On a complete Riemannian manifold it is very difficult to analyze the structure of cut locus of a point or a submanifold. The main problem is that cut locus is not $C^1$-smooth\index{$C^1$-smooth}. For example, the authors in \cite{GlSi78} showed that cut locus of a point is not always triangulable\index{triangulable}. Regarding the question of cut loci being triangulable, we recall the result \cite{Buc77} that the cut locus (of a point) of a real analytic Riemannian manifold (of dimension $d$) is a simplicial complex of dimension at most $d-1$. It follows, without much changes, that the result holds for cut loci of submanifolds as well. Hence, we attribute the following result to Buchner. 

\begin{thm}[Buchner 1977]\label{Buchner}
    Let $N$ be an analytic submanifold of a real analytic manifold $M$. If $M$ is of dimension $d$, then the cut locus $\cutn$ is a simplicial complex of dimension at most $d-1$.
\end{thm}
\vspace{0.3cm}
\noindent The obvious modifications to the proof by Buchner are the following:
\begin{enumerate}[(i)]
    \item Choose $\ep$ to be such that there is a unique geodesic from $p$ to $q$ if $d(p,q)<\ep$ and if $d(N,q)<\ep$, then there is a unique $N$-geodesic to $q$;
    \item  Consider the set $\Omega_N(t_0,t_1,\ldots,t_k)$, the space of piecewise broken geodesics starting at $N$, and define $\Omega_N(t_0,t_1,\ldots,t_k)^s$ analogously;
    \item The map
    \begin{displaymath}
        \Omega_N(t_0,t_1,\ldots,t_k)^s\to N\times M\times \cdots\times M,\,\,\omega\mapsto (\omega(t_0),\omega(t_1),\ldots,\omega(t_k))
    \end{displaymath}
    determines an analytic structure on $\Omega_N(t_0,t_1,\ldots,t_k)^s$.
\end{enumerate}
The remainder of the proof works essentially verbatim.

\begin{rem}
    As we have seen in \Cref{join}, the dimension of the cut locus of a $k$-dimensional submanifold can be of dimension $d-k-1$. However, generically, we may not expect this to be true. In fact, for real analytic knots (except the unknot) in $\mathbb{S}^3$, it is always the case that the cut locus cannot be homotopic to a (connected) $1$-dimensional simplicial complex (cf \Cref{codim2}). 
\end{rem}

\subsection{Thom space via cut locus}\label{Sec: Thom}
\hfb Recall that the \textit{Thom space} \index{Thom space} $\textup{Th}(E)$ of a real vector bundle \index{vector bundle} $E\to B$ of rank $k$ is $D(E)/S(E)$, where it is understood that we have chosen a Euclidean metric on $E$. If $B$ is compact, then the Thom space $\textup{Th}(E)$ is the one-point compactification \index{one-point compactification} of $E$. In general, we compactify the fibres and then collapse the section at infinity to a point to obtain $\textup{Th}(E)$. Thus, Thom spaces obtained via two different metrics are homeomorphic. We start with a similar exponential map which is obtained by the $\rho$ map, and we named it \textit{rescaled exponential map}. \index{rescaled exponential map}  

\begin{defn}[Rescaled Exponential]\index{rescaled exponential map}
    The \textit{rescaled exponential} or $\rho$-exponential map is defined to be 
    \begin{displaymath}
        \widetilde{\textup{exp}}:D(\nu)\to M,\,\,(p,v)\mapsto \left\{
        \begin{array}{rl}
            \textup{exp}_p(\rho(\hat{v})v) & \textup{if $v=\|v\|\hat{v}\neq 0$}\\
            p & \textup{if $v=0$.}
        \end{array}\right.
    \end{displaymath}
\end{defn}

\vspace{0.3cm}
\noindent We are now ready to prove the main result of this section.

\begin{thm}\label{Thomsp}
    Let $N$ be an embedded submanifold inside a closed, connected Riemannian manifold $M$. If $\nu$ denotes the normal bundle\index{normal bundle} of $N$ in $M$, then there is a homeomorphism 
    \begin{displaymath}
        \widetilde{\textup{exp}}:D(\nu)/S(\nu)\stackrel{\cong}{\longrightarrow} M/\cutn.
    \end{displaymath}
\end{thm}
\begin{proof}
    It follows from \Cref{snucts} that the rescaled exponential is continuous. Moreover, $\widetilde{\textup{exp}}$ is surjective and $\widetilde{\textup{exp}}(S(\nu))=\cutn$. If there exists $(p,v)\neq (q,w)\in D(\nu)$ such that 
    \begin{displaymath}
        \widetilde{\textup{exp}}(p,v)=\widetilde{\textup{exp}}(q,w)=p',
    \end{displaymath}
    then $d(p',N)$ can be computed in two ways to obtain 
    \begin{displaymath}
        d(p',N)=\rho(\hat{v})\|v\|=\rho(\hat{w})\|w\|.
    \end{displaymath}
    Thus, $T=d(p',N)$ is a number such that $\gamma_v:[0,T]$ is an $N$-geodesic and $\gamma_v(T)=\gamma_w(T)=p'$. By \Cref{thm:CharacterizationOfCutLocusInTermsOfFocalPoint}, we conclude that $T=\rho(\hat{v})=\rho(\hat{w})$, whence $\|v\|=\|w\|=1$. Therefore, $\widetilde{\textup{exp}}$ is injective on the interior of $D(\nu)$. \\
    \hfb As $\cutn$ is closed and $M$ is a compact metric space, the quotient space $M/\cutn$ is Hausdorff. As the quotient $D(\nu)/S(\nu)$ is compact, standard topological arguments imply the map induced by the rescaled exponential is a homeomorphism.
\end{proof}

\vspace{0.3cm}
\hf We will now revisit a basic property of Thom space via its connection to the cut locus. It can be seen that 
\begin{equation}\label{cutprod}
    \mathrm{Cu}(N_1\times N_2)=(\mathrm{Cu}(N_1)\times M_2)\cup (M_1\times \mathrm{Cu}(N_2))
\end{equation}
for an embedding $N_1\times N_2$ inside $M_1\times M_2$. If $\nu_j$ is the normal bundle of $N_j$ inside $M_j$, then \Cref{Thomsp} along with \eqref{cutprod} implies that
\begin{align*}
    \textup{Th}(\nu_1\oplus \nu_2) & \cong \frac{M_1\times M_2}{(M_1\times \mathrm{Cu}(N_2))\cup (\mathrm{Cu}(N_1)\times M_2)}
    \\[1ex]
    & \cong \frac{M_1/\mathrm{Cu}(N_1)\times M_2/\mathrm{Cu}(N_2)}{M_1/\mathrm{Cu}(N_1)\vee M_2/\mathrm{Cu}(N_2)}
    \\[1ex]
    & \cong \textup{Th}(\nu_1)\wedge \textup{Th}(\nu_2).
\end{align*}
Let $N=N_1\sqcup N_2$ be a disjoint union of connected manifolds of the same dimension. If $N\hookrightarrow M$, then let $\nu_j$ denote the normal bundle of $N_j$ in $M$. If $\nu$ is the normal bundle of $N$ in $M$, then 
\begin{equation}\label{Thomwedge}
    \textup{Th}(\nu)\cong \textup{Th}(\nu_1)\vee \textup{Th}(\nu_2).
\end{equation}
This implies that
\begin{displaymath}
    M/\cutn \cong M/\mathrm{Cu}(N_1)\vee M/\mathrm{Cu}(N_2).
\end{displaymath}
\begin{eg}\label{eg:CutLocusOfLink}
    Consider the two circles
    \begin{displaymath}
        N_1=\{(\cos t,\sin t,0,0)\,|\,t\in\R\},\,\,N_2=\{(0,0,\cos t,\sin t)\,|\,t\in\R\}
    \end{displaymath}
    in $\mathbb{S}^3$. The link $N:=N_1\sqcup N_2$ has linking number $1$. We claim that the cut locus
    \begin{displaymath}
        \cutn=\left\{\frac{1}{\sqrt{2}}(\cos s,\sin s, \cos t,\sin t)\,|\,s,t\in\R\right\}
    \end{displaymath}
    is a torus. We will prove the claim by showing that the above set is separating set of $N$. As it is closed, the claim will follow by using \Cref{thm:SeClosureIsCutLocus}. Let $P=(a,b,0,0)\in N$ with $a^2+b^2=1$. Note that 
    \begin{displaymath}
        T_P\mathbb{S}^3 \cong T_PN\directsum \left(T_PN\right)^\perp \cong \spn\{(-b,a,0,0)\} \directsum \spn \{\mathbf{e}_3,\mathbf{e}_4\},
    \end{displaymath}  
    where $\mathbf{e}_3=(0,0,1,0)$ and $\mathbf{e}_4=(0,0,0,1)$. We take any unit vector at $P$ which is perpendicular to $T_PN$, say $\mathbf{v}=(0,0,\cos \theta, \sin \theta)$. An $N$-geodesic starting at $P$ in the direction of $\mathbf{v}$ will be
    \begin{displaymath}
        \gamma(t)=P\cos t+\mathbf{v} \sin t,~= (a\cos t,b\cos t,\cos \theta \sin t, \sin \theta \sin t),~ ~0\le t \le \pi.
    \end{displaymath} 
    We have 
    \begin{displaymath}
        d(\gamma(t),N) = \inf_{X\in N}d(\gamma(t),X) = \min\left\{\inf_{X\in N_1}d(\gamma(t),X), \inf_{X\in N_2}d(\gamma(t),X),\right\}.
    \end{displaymath}
    
    \noindent Look at \Cref{fig:DistanceFromLink} and note that 
    \begin{displaymath}
        d(X,\gamma(t)) = \cos^{-1}(X\cdot \gamma(t)).
    \end{displaymath}
    \begin{figure}[!htb]
        \centering
    \def\svgwidth{0.3\columnwidth}
    %% Creator: Inkscape 1.2 (1:1.2.1+202207142221+cd75a1ee6d), www.inkscape.org
%% Accompanies image file 'DistanceFromLink.pdf' (pdf, eps, ps)
\begingroup%
  \makeatletter%
  \providecommand\color[2][]{%
    \errmessage{(Inkscape) Color is used for the text in Inkscape, but the package 'color.sty' is not loaded}%
    \renewcommand\color[2][]{}%
  }%
  \providecommand\transparent[1]{%
    \errmessage{(Inkscape) Transparency is used (non-zero) for the text in Inkscape, but the package 'transparent.sty' is not loaded}%
    \renewcommand\transparent[1]{}%
  }%
  \providecommand\rotatebox[2]{#2}%
  \newcommand*\fsize{\dimexpr\f@size pt\relax}%
  \newcommand*\lineheight[1]{\fontsize{\fsize}{#1\fsize}\selectfont}%
  \ifx\svgwidth\undefined%
    \setlength{\unitlength}{184.97258014bp}%
    \ifx\svgscale\undefined%
      \relax%
    \else%
      \setlength{\unitlength}{\unitlength * \real{\svgscale}}%
    \fi%
  \else%
    \setlength{\unitlength}{\svgwidth}%
  \fi%
  \global\let\svgwidth\undefined%
  \global\let\svgscale\undefined%
  \makeatother%
  \begin{picture}(1,1.68762091)%
    \lineheight{1}%
    \setlength\tabcolsep{0pt}%
    \put(0,0){\includegraphics[width=\unitlength,page=1]{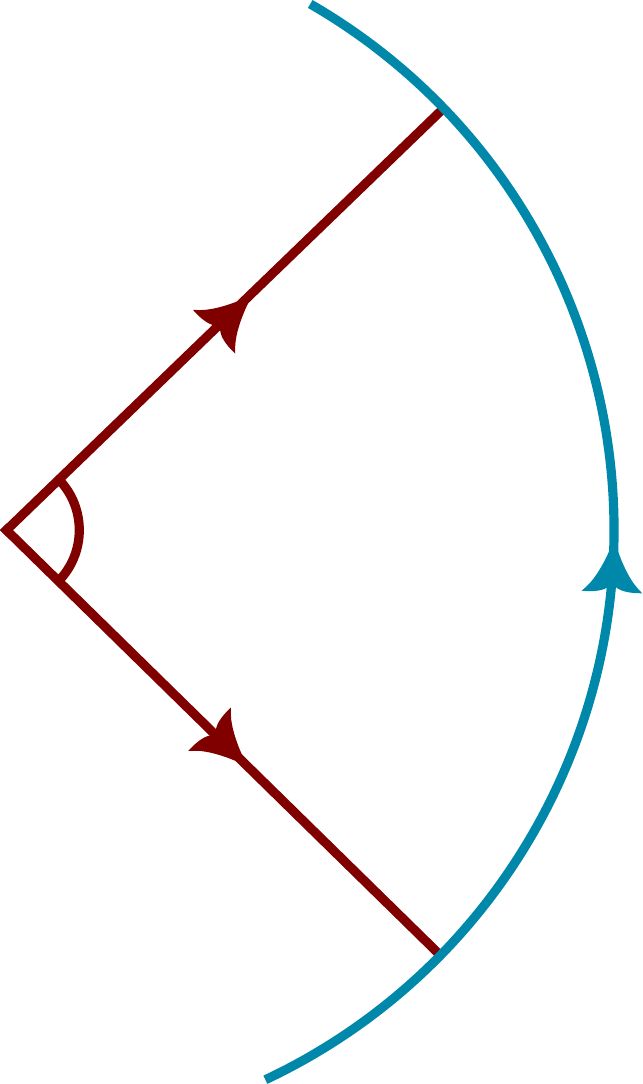}}%
    \put(0.75216063,0.18131511){\makebox(0,0)[lt]{\lineheight{1.25}\smash{\begin{tabular}[t]{l}$X$\end{tabular}}}}%
    \put(0.74529824,1.5081435){\makebox(0,0)[lt]{\lineheight{1.25}\smash{\begin{tabular}[t]{l}$\gamma(t)$\end{tabular}}}}%
    \put(0,0){\includegraphics[width=\unitlength,page=2]{DistanceFromLink.pdf}}%
    \put(0.158948,0.84229764){\makebox(0,0)[lt]{\lineheight{1.25}\smash{\begin{tabular}[t]{l}$\theta$\end{tabular}}}}%
  \end{picture}%
\endgroup%

        \caption{Distance of $X$ to $\gamma(t)$  \label{fig:DistanceFromLink}}
    \end{figure}
    \noindent Therefore, the problem of finding the distance of $N$ to $\gamma(t)$ is equivalent to maximizing the dot product $X\cdot \gamma(t)$.
    Let $X\in N_1$ and $X=(x,y,0,0),~x^2+y^2=1$. Then 
    \begin{displaymath}
        X\cdot \gamma(t) = ax\cos t+by\cos t.
    \end{displaymath} 
    Maximizing the above such that $x^2+y^2=1$ by the method of Lagrange multiplier, we have
    \begin{displaymath}
        x = \dfrac{a\cos t}{|\cos t|}, \text{ and } y = \dfrac{b\cos t}{|\cos t|}.
    \end{displaymath}
    Note that the above expression is well-defined, as if $\cos t =0$, then $X\cdot \gamma(t)=0$. The maximum value of $X\cdot \gamma(t)$ will be $|\cos t|$ and this is achieved by only one point of $N_1$. Similarly, if we maximize the dot product over $N_2$, we get the maximum value $|\sin t|$, which is also obtained by a single maxima. Thus, $\gamma(t)$ will be a separating point if and only if 
    \begin{displaymath}
        |\cos t| =|\sin t| \implies t = \dfrac{\pi}{4},\dfrac{3\pi}{4}.
    \end{displaymath}
    Therefore, the separating points will be
    \begin{displaymath}
        \left\{\frac{1}{\sqrt{2}}(\cos s,\sin s,\cos \theta,\sin \theta):s,\theta\in \mathbb{R}\right\}.
    \end{displaymath}
    \vspace{0.3cm}
    \hf Note that $\mathrm{Cu}(N_1)=N_2$ and vice-versa as well as 
    \begin{displaymath}
        \mathbb{S}^3/\mathrm{Cu}(N_j)\cong (S^1\times S^2)/(S^1\times \infty)
    \end{displaymath}
    where $S^1\times S^2$ is the fibrewise compatification of the normal bundle of $N_j$. We conclude that 
    \begin{displaymath}
        \mathbb{S}^3/\cutn\cong \Big(\frac{S^1\times S^2}{S^1\times \infty}\Big)\vee \Big(\frac{S^1\times S^2}{S^1\times \infty}\Big).
    \end{displaymath}
\end{eg}

\hf There are some topological similarities between $\cutn$ and $M-N$. We recall that a topological pair $(X, A)$ is called a \textit{good pair\index{good pair}} if $A$ is closed in $X$ and there is an open subset $U\subseteq X$ with $A\subseteq U$ such that $A$ is a strong deformation retract in $U$.
\begin{lemma}\label{defretM-N}
The cut locus $\cutn$ is a strong deformation retract of $M-N$. In particular, $(M,\cutn)$ is a good pair and the number of path components of $\cutn$ equals that of $M-N$.
\end{lemma}
\begin{proof}
    Consider the map $H:(M-N)\times [0,1]\to M-N$ defined via the normal exponential map 
    \begin{align*}
        H(q,t) = \begin{cases}
            \exp_\nu\left[\left\{t \cdot \rho\paran{\frac{\exp_\nu^{-1}(q) }{\norm{\exp_\nu^{-1}(q)}}}+(1-t)\norm{\exp_\nu^{-1}(q)}\right\}\frac{\exp_\nu^{-1}(q)}{\norm{\exp_\nu^{-1}(q)}}\right]  & \text{ if } q\notin \cutn \\[1ex]
            q & \text{ if } q\in \cutn.
        \end{cases}
    \end{align*}
    If $q\in M - (\cutn\cup N)$, then let $\gamma$ be the unique $N$-geodesic joining $N$ to $q$. The path $H(q,t)$ is the image of this geodesic from $q$ to the first cut point along $\gamma$. The continuity of $\rho$ implies that $H$ is continuous. It also satisfies $H(q,0)=q$ and $H(q,1)\in\cutn$. The claims about good pair and path components are clear.
\end{proof}

\begin{rem}
    The above lemma, in particular, shows that the homotopy type of the cut locus of a submanifold is independent of the choice of the Riemannian metric.
\end{rem}

\begin{cor}\label{htpy-cut-locus}
    If two embeddings $f,g:N\to M$ are ambient isotopic, then $\mathrm{Cu}(f(N))$ and $\mathrm{Cu}(g(N))$ are homotopy equivalent.
\end{cor}
\begin{proof}
    The hypothesis implies that there is a diffeomorphism $\varphi: M\to M$ such that $\varphi(f(N))=g(N)$. Thus, $M-\mathrm{Cu}(f(N))$ is homeomorphic to $M-\mathrm{Cu}(g(N))$ and the claim follows from the lemma above. Note that in the smooth category, the notion of isotopic and ambient isotopic are equivalent (refer to \S 8.1 of the book \cite{Hir76}). Thus, the same conclusion holds if we assume that the embeddings are isotopic.
\end{proof}

\begin{rem}
    Without the assumption of $M$ being closed, the above result fails to be true. One may consider $M=S^1\times \R$ with the natural product metric and $N=S^1$. In fact, the universal cover of $M$ is $\R\times \R$ while that of $N$ is $\R$. If we choose a periodic curve in $\R^2$ which is isotopic to the $x$-axis and has non-empty cut locus in $\R^2$, then we may pass via the covering map to obtain an embedding $g$ of $N$ isotopic to the embedding $f$ identifying $N$ with $S^1\times \{0\}$. For this pair, $\mathrm{Cu}(f(N))=\varnothing$ while $\mathrm{Cu}(g(N))\neq \varnothing$.
\end{rem}

\hf Several other identifications between topological invariants can be explored. For instance, if $\iota:N^k\hookrightarrow M^d$ is as before such that $M-N$ is path connected, then 
\begin{equation}\label{isopi}
    \iota_\ast:\pi_j(\cutn)\stackrel{\cong}{\longrightarrow}\pi_j(M)
\end{equation}
if $0\leq j\leq d-k-2$ while $\iota_\ast$ is a surjection for $j=d-k-1$. The proof of this relies on a general position argument, i.e., being able to find a homotopy of the sphere that avoids $N$, followed by \Cref{defretM-N}. Surjectivity of $\iota_\ast$ if $j\leq d-k-1$ is imposed by the requirement that a sphere $S^j$ in general position must not intersect $N^k$. Injectivity of $\iota$ for $j\leq d-k-2$ is imposed by the condition that a homotopy $S^j\times [0,1]$ in general position must not intersect $N^k$. This observation \eqref{isopi} generalizes a result in \cite[Proposition 4.5 (1)]{Sak96}.

\vspace{0.2cm}
\hf The inclusion $i:\cutn\hookrightarrow M$ induces a long exact sequence in homology
\begin{displaymath}
    \cdots \to H_j(\cutn)\stackrel{i_\ast}{\longrightarrow} H_j(M)\to H_j(M,\cutn)\stackrel{\partial}{\longrightarrow} H_{j-1}(\cutn)\to \cdots
\end{displaymath}
As $(M,\cutn)$ is a good pair (cf \Cref{defretM-N}), we replace the relative homology of $(M,\cutn)$ with reduced homology of $M/\cutn\cong \textup{Th}(\nu)$. This results in the following long exact sequence
\begin{equation}\label{lesThom}
    \cdots \to H_j(\cutn)\stackrel{i_\ast}{\longrightarrow} H_j(M)\stackrel{q}{\longrightarrow} \widetilde{H}_j(\textup{Th}(\nu))\stackrel{\partial}{\longrightarrow} H_{j-1}(\cutn)\to \cdots
\end{equation}
If $N=\{p\}$ is a point, then $\textup{Th}(\nu)=S^d$ and \eqref{lesThom} imply isomorphisms
\begin{displaymath}
    i_\ast:H_j(\textup{Cu}(p))\stackrel{\cong}{\longrightarrow} H_j(M),\,\,i^\ast:H^j(M)\stackrel{\cong}{\longrightarrow} H^j(\textup{Cu}(p))
\end{displaymath}
for $j\neq d,d-1$ (cf \cite[Proposition 4.5 (2)]{Sak96}). 

\begin{rem}
    The long exact sequence \eqref{lesThom} can be interpreted as the dual to the long exact sequence in cohomology of the pair $(M,N)$. If $N=N_1\sqcup \cdots\sqcup N_l$ is a disjoint union of submanifolds of dimension $k_1, \ldots, k_l$ respectively, then the Thom isomorphism implies that 
    \begin{displaymath}
        \widetilde{H}_j(\textup{Th}(\nu))\cong \widetilde{H}_j(\textup{Th}(\nu_1))\oplus \cdots\oplus \widetilde{H}_j(\textup{Th}(\nu_l))\cong H_{j-(d-k_1)}(N_1)\oplus\cdots\oplus H_{j-(d-k_l)}(N_l),
    \end{displaymath}
    where $\nu_j$ is the normal bundle of $N_j$. Applying Poincar\'{e} duality to each $N_j$, we obtain isomorphisms
    \begin{displaymath}
        \widetilde{H}_j(\textup{Th}(\nu))\cong \oplus_{i=1}^l H^{d-j}(N_i)= H^{d-j}(N).
    \end{displaymath}
    Poincar\'{e}-Lefschetz duality applied to the pair $(M,N)$ provides isomorphisms
    \begin{equation}\label{MNcutn}
        \check{H}^j(M,N)\cong H_{d-j}(M-N).
    \end{equation}
    As $M$ and $N$ are triangulable, $\check{\textup{C}}$ech cohomology may be replaced by singular cohomology. Since $M-N$ deforms to $\cutn$ by \Cref{defretM-N}, we have isomorphisms
    \begin{equation}\label{MNcutn2}
        H^j(M,N)\cong H_{d-j}(\cutn).
    \end{equation}
    Combining all these isomorphisms, we obtain the long exact sequence in cohomology for $(M,N)$ from \eqref{lesThom}. 
\end{rem}

\begin{lemma}\label{homcutn}
    Let $N$ be a closed submanifold of $M$ with $l$ components. If $M$ has dimension $d$, then $H_{d-1}(\cutn)$ is free abelian of rank $l-1$ and $H_{d-j}(\cutn)\cong H^j(M)$ if $j-2\geq k$, where $k$ is the maximum of the dimension of the components of $N$.
\end{lemma}
\begin{proof}
    It follows from \eqref{MNcutn} that 
    \begin{displaymath}
        H_{d-1}(\cutn)\cong H^1(M,N).
    \end{displaymath}
    Consider the long exact sequence associated to the pair $(M,N)$ 
    \begin{displaymath}
        0\to H^0(M,N)\to H^0(M)\stackrel{i^\ast}{\rightarrow} H^0(N)\to H^1(M,N)\to H^1(M)\to H^1(N)\to \cdots
    \end{displaymath}
    If $N$ has $l$ components, i.e., $N=N_1\sqcup \cdots\sqcup N_l$ where $N_j$ has dimension $k_j$, then $H^1(M,N)$ is torsion-free. This follows from the fact that $i^\ast(1)=(1,\ldots,1)$ and $H^1(M)$ is free abelian. 

    \vspace{0.1cm}
    \begin{rem}
        In the above  Lemma, if in particular, $H^1(M)=0$, then $H_{d-1}(\cutn)\cong \mathbb{Z}^{l-1}$.
    \end{rem}
    \hf The long exact sequence for the pair $(M,N)$ imply that there are isomorphisms
    \begin{equation}\label{cutnM}
        H_{d-j}(\cutn)\cong H^j(M,N)\stackrel{\cong}{\longrightarrow} H^j(M)
    \end{equation}
    if $j\geq k+2$, where $k=\max\{k_1,\ldots,k_l\}$. 
\end{proof}

\begin{rem}
    Cut locus can be very hard to compute. For a general space, we have the notion of topological dimension. This notion coincides with the usual notion if the space is triangulable. However, in \cite{BaMi62}, the authors proved that the singular homology of a space may be non-zero beyond its topological dimension. \v{C}ech (co)homology is better equipped to detect topological dimension and is the reason why one may prefer it over singular homology due to the generic fractal like nature of cut loci (see the remarks following \Cref{thm:ThmC} in \Cref{ch:introduction}). Although the topological dimension of $\cutn$ is at most $d-1$, it is not apparent that $H_{d-1}(\cutn)$ is a free abelian group. 
\end{rem}

\vspace{0.3cm}
\hf There are several applications of this discussion.
\begin{thm}\label{homsph}
    Let $N$ be a smooth homology $k$-sphere embedded in a Riemannian manifold homeomorphic to $S^d$. If $d\geq k+3$, then the cut locus $\cutn$ is homotopy equivalent to $S^{d-k-1}$.
\end{thm}
\begin{proof}
    As $N$ has codimension at least $3$, its complement is path-connected. It follows from \eqref{isopi} and  \Cref{defretM-N} that $M-N$ is $(d-k-2)$-connected. In particular, $M-N$ is simply-connected and by Hurewicz isomorphism, $H_j(M-N)=0$ if $j\leq d-k-2$.  Note that $H_d(M-N)=0$ as $M-N$ is a non-compact manifold of dimension $d$. 

    \vspace{0.1cm}
    \hf If $k>0$, then by  \Cref{homcutn}, $H_{d-1}(M-N)=0$. Moreover, by Poincar\'{e}-Lefschetz duality \eqref{MNcutn}, we infer that the only non-zero higher homology of $M-N$ is $H_{d-k-1}(M-N)\cong\mathbb{Z}$. By Hurewicz Theorem there is an isomorphism, $\pi_{d-k-1}(M-N)\cong\mathbb{Z}$. Let 
    \begin{displaymath}
        \alpha:S^{d-k-1}\to M - N
    \end{displaymath}
    be a generator. The map $\alpha_\ast$ induces an isomorphism on all homology groups between two simply-connected CW complexes. It follows from Whitehead's Theorem that $\alpha$ is a homotopy equivalence. Using \Cref{defretM-N}, we obtain our homotopy equivalence $H_1\comp \alpha:S^{d-k-1}\to \cutn$.

    \vspace{0.1cm}
    \hf If $k=0$, then by \Cref{homcutn}, $H_{d-1}(M-N)\cong \mathbb{Z}$. Arguments similar to the $k>0$ case now applies to obtain a homotopy equivalence with $S^{d-1}$.
\end{proof}

\hf The above result is foreshadowed by \Cref{join} where we showed that the cut locus of $N=\mathbb{S}^k_i$ inside $M=\mathbb{S}^d$ is $\mathbb{S}^{d-k-1}_l$. It also differs from Poincar\'{e}-Lefschetz duality in that we are able to detect the exact homotopy type of the cut locus. In fact, when $M$ and $N$ are real analytic and the embedding is also real analytic, then by \Cref{Buchner} we infer that $\cutn$ is a simplicial complex of dimension at most $d-1$. Towards this direction, \Cref{homsph} can be pushed further.

\begin{prop}\label{homsph2}
    Let $N$ be a real analytic homology $k$-sphere embedded in a real analytic homology $d$-sphere $M$. If $d\geq k+3$, then the cut locus $\cutn$ is a simplicial complex of dimension at most $(d-1)$, having the homology of $(d-k-1)$-sphere with fundamental group isomorphic to that of $M$.
\end{prop}

\vspace{0.3cm}
\noindent The proof of this is a combination of ideas used in the proof of \Cref{homsph}. The homotopy type cannot be deduced here due to the presence of a non-trivial fundamental group. An intriguing example can be obtained by combining \Cref{homsph2} and Poincar\'{e} homology sphere.

\begin{eg}[Cut locus of $0$-sphere in Poincar\'{e} sphere]\label{eg: Poincare}\index{ Poincar\'{e} sphere}
    Let $\tilde{I}$ be the binary icosahedral group. It is a double cover of $I$, the icosahedral group, and can be realized a subgroup of $SU(2)$. It is known that $H_1(\tilde{I};\mathbb{Z})=H_2(\tilde{I};\mathbb{Z})=0$, i.e., it is perfect and the second homology of the classifying space $B\tilde{I}$ is zero. A presentation of $\tilde{I}$ is given by
    \begin{displaymath}
        \tilde{I}=\langle s,t\,|\,(st)^2=s^3=t^5\rangle.
    \end{displaymath}
    In fact, if we construct a cell complex $X$ of dimension $2$ using the presentation above, then $X$ has one $0$-cell, two $1$-cells and two $2$-cells. The cellular chain complex, as computed from the presentation, is given by
    \begin{equation*}
        \xymatrix@R+1pc@C+2pc{
        0\ar[r] & \mathbb{Z}^2\ar[r]^{\spmat{-1 & 2 \\ 3 & -5}} & \mathbb{Z}^2\ar[r]^-{0} & \mathbb{Z}\ar[r] & 0
        }
    \end{equation*}
    Therefore, $H_1(X)=H_2(X)=0$ while $\pi_1(X)=\tilde{I}$. 

    \hf In contrast, consider the cut locus $C$ of the $0$-sphere in $SU(2)/\tilde{I}$, the Poincar\'{e} homology sphere. As $SU(2)$ is real analytic, so is the homology sphere. By \Cref{homsph2}, $C$ is a finite, connected simplicial complex of dimension $2$ such that $\pi_1(C)\cong \tilde{I}$ and $H_\bullet(C;\mathbb{Z})\cong H_\bullet(S^2;\mathbb{Z})$. The existence of this space is interesting for the following reason: although $X\vee S^2$ has the same topological invariants as $C$, we are unable to determine whether $X\vee S^2$ is homotopy equivalent to $C$.
\end{eg}
\hf In the codimension two case, we have two results. 
\begin{thm}\label{cutlocus-surface}
    Let $\Sigma$ be a closed, orientable, real analytic surface of genus $g$ and $N$ a non-empty, finite subset. Then $\cutn$ is a connected graph, homotopy equivalent to a wedge product of $|N|+2g-1$ circles.
\end{thm}
\begin{proof}
    As $\Sigma-N$ is connected, \Cref{defretM-N} implies that $\cutn$ is connected. It follows from \Cref{Buchner} that $\cutn$ is a finite $1$-dimensional simplicial complex, i.e., a finite graph. In this case, $\textup{Th}(\nu)$ is a wedge product of $|N|$ copies of $S^2$ (cf \eqref{Thomwedge}). We consider \eqref{lesThom} with $j=2$:
    \begin{displaymath}
        0\stackrel{i_\ast}{\longrightarrow} \mathbb{Z}\stackrel{q}{\longrightarrow} \widetilde{H}_2(\vee_{|N|} S^2)\stackrel{\partial}{\longrightarrow} H_{1}(\cutn)\stackrel{i_\ast}{\longrightarrow} H_{1}(\Sigma)\to 0
    \end{displaymath}
    Note that $H_{d-1}(\Sigma)$ is torsion-free, whence all the groups appearing in the long exact sequence are free abelian groups. This implies that 
    \begin{equation*}\label{bettid-1}
        \dim_\mathbb{Z} H_{1}(\cutn)=2g+|N|-1.
    \end{equation*}
    As $\cutn$ is connected finite graph, collapsing a maximal tree $T$ results in a quotient space $\cutn/T$ which is homotopic to $\cutn$ as well being a wedge product of $|N|+2g-1$ circles.
\end{proof}

\begin{rem}
    The authors in \cite{ItVi15} proved that every finite, connected graph can be realized as the cut locus (of a point) of some surface. There remains the question of orientability of the surface. As noted in the proof of \Cref{cutlocus-surface}, if the surface is orientable and $|N|=1$, then the graph has an even number of generating cycles. If $\Sigma$ is non-orientable, then $\Sigma\cong (\mathbb{RP}^2)^{\# k}$ has non-orientable genus $k$ and the oriented double cover of $\Sigma$ has genus $g=k-1$. Recall that $H_1(\Sigma)\cong \mathbb{Z}^{k-1}\oplus\mathbb{Z}_2$ and $H_2(\Sigma)=0$. Looking at \eqref{lesThom} with $j=2$ we obtain
    \bgd
    0\to \mathbb{Z}\to H_1(\cu)\to \mathbb{Z}^{k-1}\oplus\mathbb{Z}_2\to 0.
    \edd
    Thus, $H_1(\cu)\cong\mathbb{Z}^k$ as homology of graphs are free abelian. Let $B_\ep(\cu)$ denote the $\ep$-neighbourhood of $\cu$ in $\Sigma$. For $\ep$ sufficiently small, this is a surface such that $\overline{B_\ep(\cu)}$ has one boundary component. The compact surface $B_\ep(\cu)$ is reminiscent of ribbon graphs. The surface $\Sigma$ can be obtained as the connect sum of a disk centered at $p$ and the closure of $B_\ep(\cu)$. Therefore, non-orientability of $\Sigma$ is equivalent to non-orientability of $B_\ep(\cu)$. A similar observation appears in the unpublished article \cite[Theorem 3.7]{ItVi11}.
\end{rem}

\begin{eg}[Homology spheres of codimension two]\label{codim2}
    In continuation of \Cref{homsph}, let $N\hookrightarrow S^{k+2}$ be a homology sphere of dimension $k\geq 1$. Since $N$ has codimension two, $S^{k+2}-N$ is path connected and so is $\cutn$. We are not assuming that the metric on $S^{k+2}$ is real analytic. Using \eqref{MNcutn2} and the long exact sequence in cohomology of $(S^{k+2},N)$, we infer that $H_1(\cutn)\cong\mathbb{Z}$ and all higher homology groups vanish. However, the Hurewicz Theorem cannot be used here to establish that $\pi_1(\cutn)\cong\mathbb{Z}$. 
    
    \vspace{0.1cm}
    \hf In particular cases, we may conclude that $\cutn$ is homotopic to a circle. It was proved in \cite{Plo82} that certain homology $3$-spheres $N$, obtained by a Dehn surgery of type $\frac{1}{2a}$ on a knot, smoothly embed in $S^5$ with complement a homotopy circle. Since $M-N$ deforms to $\cutn$, it follows that there is a map $\alpha:S^1\to \cutn$ inducing isomorphisms on homotopy and homology groups.

    \vspace{0.1cm}
    \hf If $k=1$, then a homology $1$-sphere is just a knot $K$ in $S^3$. Since $S^3-K$ deforms to $\mathrm{Cu}(K)$, the fundamental group of the cut locus is the knot group. Moreover, in the case of real analytic knots in $\mathbb{S}^3$, the cut locus is a finite simplicial complex of dimension at most $2$ (cf \Cref{Buchner}). Except for the unknot, the knot group is never a free group while the fundamental group of a connected, finite graph is free. This observation establishes that $\mathrm{Cu}(K)$ is always a $2$-dimensional simplicial complex, whenever $K$ is a non-trivial (real analytic) knot in $\mathbb{S}^3$.
\end{eg}

\hf Finally, we will end this chapter by proving that the complement of cut locus deforms to the submanifold.

\begin{thm}\label{thm: Morse-Bott}
    Let $N$ be a closed embedded submanifold of a complete Riemannian manifold $M$. Let $d:M\to\R$ be the distance function with respect to $N$. If $f=d^2$, then its restriction to $M-\cutn$ is a Morse-Bott function, with $N$ as the critical submanifold. Moreover, the gradient flow of $f$ deforms $M-\cutn$ to $N$.
\end{thm}
\begin{proof}
    It follows from \Cref{thm:CharacterizationOfCutLocusInTermsOfFocalPoint} the map $\exp_\nu^{-1}:M - \paran{\cutn \cup N}\to \nu - \{0\}$ is an (into) diffeomorphism and $\dist(N,q) = \norm{\exp^{-1}_\nu(q)}$ and hence the distance function is of class $C^{\infty}$ at $q\in M - \paran{\cutn\cup N}$. Using Fermi coordinates (cf \Cref{dsq-Fermi}), we have seen that the distance squared function is smooth around $N$ and therefore it is smooth on $M - \cutn$. By \Cref{dsq-MB}, the Hessian of this function at $N$ is non-degenerate in the normal direction. It is well-known \cite[Proposition 4.8]{Sak96} that $\|\nabla d(q)\|=1$ if $d$ is differentiable at $q\in M$. Thus, for $q\in M-(\cutn\cup N)$ we have
    \begin{equation}\label{graddsq}
        \|\nabla f(q)\|=2d(q)\|\nabla d(q)\|=2d(q).
    \end{equation}
    Let $\gamma$ be the unique unit speed $N$-geodesic that joins $N$ to $q$, i.e., 
    \begin{displaymath}
        \gamma:[0,d(q)]\to M,\,\,\gamma(0)=p,\,\gamma(d(q))=q,\,\|\gamma'\|=1.
    \end{displaymath}
    We may write $\nabla f(q)=\lambda \gamma'(d(q))+w$, where $w$ is orthogonal to $\gamma'(d(q))$. But 
    \begin{displaymath}
        \left\langle \nabla f\big|_q, \gamma'(d(q))\right\rangle = \frac{d}{dt}f(\gamma(d(q)+t))\Big|_{t=0}=\frac{d}{dt}(d(q)^2+2d(q)t+t^2)\Big|_{t=0}=2d(q).
    \end{displaymath}
    Thus, $\lambda=2d(q)$ and combined with \eqref{graddsq}, we conclude that $\nabla f(q)=2d(q)\gamma'(d(q))$. Therefore, the negative gradient flow line initialized at $q\in M-\cutn$ is given by
    \begin{displaymath}
        \eta(t)=\gamma(d(q)e^{-2t}).
    \end{displaymath}
    These flow lines define a flow which deform $M-\cutn$ to $N$ in infinite time. 
\end{proof}

\vspace{0.3cm}
\hf The reader may choose to revisit the example of $GL(n,\R)$ discussed in \Cref{Sec:IlluminatingExample} and treat it as a concrete illustration of the Theorem above.
	\chapter{Application to Lie groups}\label{ch:ApplicationToLieGroups}
\minitoc
\hf Due to the classical results of Cartan, Iwasawa and others, we know that any connected Lie group \index{Lie group} $G$ is diffeomorphic to the product of a maximally compact subgroup $K$ and the Euclidean space. In particular, $G$ deforms to $K$. For semisimple groups, this decomposition is stronger and is attributed to Iwasawa. The Killing form \index{Killing form} on the Lie algebra \index{Lie algebra} $\mathfrak{g}$ is non-degenerate and negative definite for compact semi-simple Lie algebras. For such a Lie group $G$, consider the Levi-Civita connection \index{Levi-Civita connection} associated to the bi-invariant metric obtained from negative of the Killing form. This connection coincides with the Cartan connection. 

\bigskip
\hf We will consider two examples, both of which are non-compact and non semisimple. We prove that these Lie groups $G$ deformation retract to maximally compact subgroups $K$ via gradient flows of appropriate Morse-Bott functions.\index{Morse-Bott functions} This requires a choice of a left-invariant metric which is right-$K$-invariant, and a careful analysis of the geodesics associated with the metric. In particular, we provide a possibly new proof of the surjectivity of the exponential map for $U(p,q)$. The results of this chapter is based on joint work with Basu \cite[\S 4]{BaPr21}. 

\section{Matrices with positive determinant}\label{Sec: GLnSOn}

\hfb Let $g$ be a left-invariant metric on $GL(n,\rbb)$, the set of all invertible matrices. Recall that a left-invariant metric $g$ on a Lie group is determined by its restriction at the identity. For $A\in GL(n,\rbb)$, consider the left multiplication map $l_A:GL(n,\rbb) \to GL(n,\rbb),~ B\mapsto AB$. This extends to a linear isomorphism from $M(n,\rbb)$ to itself. Thus, the differential $(Dl_A)_I:T_IGL(n,\rbb)\to T_AGL(n,\rbb)$ is an isomorphism and given by $l_A$ itself. For $X,Y\in T_IGL(n,\rbb)$, 
	\begin{equation*} \label{eq: left-invariant iso}
		g_I(X,Y) = g_A((Dl_A)_IX,(Dl_A)_IY)=g_A(AX,AY).
	\end{equation*}
We choose the left-invariant metric on $GL(n,\rbb)$ generated by the Euclidean metric at $I$. Therefore,
\begin{displaymath}
    g_{A^{-1}}(X,Y) = \innerprod{AX}{AY}_I := \trace{(AX)^T\!AY} = \trace{X^T\!A^T\!AY}.
\end{displaymath}
Note that this metric is right-$O(n,\rbb)$-invariant. We are interested in the distance between an invertible matrix $A$ (with $\det(A)>0$) and $SO(n,\rbb)$. Since $SO(n,\rbb)$ is compact, there exists $B\in SO(n,\rbb) $ such that $d(A,B) = \dist(A,SO(n,\rbb))$. 
\begin{lemma}\label{CartanGLn}
    If $D$ is a diagonal matrix with positive diagonal entries $\lambda_1,\cdots,\lambda_n$, then 
    \begin{displaymath}
        \dist(D,SO(n,\rbb)) = d(D,I).
    \end{displaymath}
    Moreover, $I$ is the unique minimizer and the associated minimal geodesic is given by $\gamma(t)=e^{t\log D}$.
\end{lemma}
\begin{proof}
    Let $B\in SO(n,\rbb)$ satisfying $d(D,B) = \dist(D,SO(n,\rbb))$. Since with respect to the left-invariant metric $GL^+(n,\rbb)$ is complete, there exists a minimal geodesic $\gamma:[0,1]\to GL^+(n,\rbb)$ joining $B$ to $D$, i.e., 
    \begin{displaymath}
        \gamma(0) = B,~~\gamma(1) = D, ~~\text{ and } ~~ l(\gamma) = d(D,B).
    \end{displaymath}
    The first variational principle implies that $\gamma'(0)$ is orthogonal to $T_BSO(n,\rbb)$. It follows from \cite[\S 2.1]{MaNe16} that $\eta(t)=e^{tW}$ is a geodesic if $W$ is a symmetric matrix. Moreover, $\eta'(0)=W$ is orthogonal to $T_I SO(n,\rbb)$. As left translation is an isometry and isometry preserves geodesic, it follows that $\gamma(t) = Be^{tW}$ is a geodesic with $\gamma'(0)$ orthogonal to $T_BSO(n,\rbb)$. By the defining properties of $\gamma$,  $D=\gamma(1)=Be^W$. Since $e^W$ is symmetric positive definite, we obtain two polar decompositions \index{polar decompositions} of $D$, i.e., $D = ID$ and $ D = Be^W$. By the uniqueness of the polar decomposition for invertible matrices, $B=I$ and $D=e^W$. 

    \vspace{0.2cm}
    \hf In order to compute $d(I,D)$, note that 
    \begin{displaymath}
        e^W = D = e^{\log D},
    \end{displaymath}
    where $\log D$ denotes the diagonal matrix with entries $\log \lambda_1,\cdots,\log \lambda_n$. As $W$ and $\log D$ are symmetric, and matrix exponential is injective on the space of symmetric matrices, we conclude that $W = \log D$. The geodesic is given by $\gamma(t)=e^{t\log D}$ and 
    \begin{equation}\label{eq: distance left invariant}
        \dist(D,SO(n,\rbb)) = \norm{\gamma'(0)}_I = \norm{\log D}_I = \left(\sum_{i=1}^n (\log \lambda_i)^2\right)^{\frac{1}{2}}.
    \end{equation}
    Thus, the distance squared function will be given by $\sum_{i=1}^n (\log \lambda_i)^2$. 
\end{proof}

\hf Now for any $A\in GL^+(n,\rbb)$ we can apply the SVD decomposition, \index{SVD decomposition} i.e., $A = UDV^T$ with $\sqrt{A^TA} = VDV^T$ and $\log\sqrt{A^TA} = V(\log D) V^T$. Note that $U,V\in SO(n,\rbb)$ and $D$ is a diagonal matrix with positive entries. The left-invariant metric is right-invariant with respect to orthogonal matrices. Thus,
\begin{displaymath}
    \dist(A,SO(n,\rbb)) = \dist(D,SO(n,\rbb)) = \norm{\log D}_I,
\end{displaymath}
where the last equality follows from the lemma (see \eqref{eq: distance left invariant}). As 
\begin{displaymath}
    \norm{\log D}_I = \norm{V(\log D) V^T}_I = \norm{\log\sqrt{A^TA}}_I,
\end{displaymath}
it follows from the arguments of the lemma and the metric being bi-$O(n,\rbb)$-invariant that
\begin{displaymath}
    \gamma(t)=Ue^{t \log D}V^T
\end{displaymath}
is a minimal geodesic joining $UV^T$ to $A$, realizing $\dist(A,SO(n,\rbb))$. As the minimizer $UV^T$ is unique, $\mathrm{Se}(SO(n,\rbb))$ is empty, implying that $\mathrm{Cu}(SO(n,\rbb))$ is empty as well. In fact, $UV^T=A\sqrt{A^T A}^{-1}$ and 
\begin{equation}\label{GLdefOver2}
    \gamma(t)=Ue^{t \log D}V^T=UV^T Ve^{t \log D}V^T=A\sqrt{A^T A}^{-1}e^{t\log \sqrt{A^T A}}.
\end{equation}
If we compare \eqref{GLdefOver1}, the deformation of $GL(n,\rbb)$ to $O(n,\rbb)$ inside $M(n,\rbb)$, with \eqref{GLdefOver2}, then in both of the cases, an invertible matrix $A$ deforms to $A\sqrt{A^T A}^{-1}$. Finally, observe that the normal bundle of $SO(n,\rbb)$ is diffeomorphic to $GL^+(n,\rbb)$.

\section{Indefinite unitary groups}\label{Sec: Upq}
\index{indefinite unitary groups}	
\hfb Let $n$ be a positive integer with $n=p+q$. Consider the inner product on $\C^n$ given by
\begin{displaymath}
    \lan (w_1,\ldots,w_n),(z_1,\ldots,z_n)\ran = z_1\overline{w_1}+\cdots+z_p\overline{w_p}-z_{p+1}\overline{w_{p+1}}-\cdots- z_n\overline{w_n}.
\end{displaymath}
This is given by the matrix $I_{p,q}$ in the following way:
\begin{displaymath}
    \lan \mathbf{w},\mathbf{z}\ran=\overline{\mathbf{w}}^t I_{p,q} \mathbf{z}=\left(\begin{array}{ccc}
        \overline{w}_1 & \cdots & \overline{w}_n
        \end{array}\right)\left(\begin{array}{cc}
        I_p & 0 \\
        0 & -I_q
        \end{array}\right)\left(\begin{array}{c}
        z_1\\
        \vdots\\
        z_n
     \end{array}\right)
\end{displaymath}

\hf Let $U(p,q)$ denote the subgroup of $GL(n,\C)$ preserving this indefinite form, i.e., $\mathcal{A}\in \upq$ if and only if $\mathcal{A}^\ast I_{p,q}\mathcal{A}=I_{p,q}$. In particular, $\det \mathcal{A}$ is a complex number of unit length. By convention, $I_{n,0}=I_n$ and $ I_{0,n}=-I_n$, both of  which corresponds to $U(n,0)=U(n)=U(0,n)$, the unitary group. In all other cases, the inner product is indefinite.

\bigskip
\hf The group $U(1,1)$ is given by matrices of the form
\begin{displaymath}
    \mathcal{A}=\left(
        \begin{array}{cc}
            \alpha & \beta\\
            \lambda \overline{\beta} & \lambda \overline{\alpha}
        \end{array}\right),\,\,\lambda\in S^1,\,\,|\alpha|^2-|\beta|^2=1.
\end{displaymath}
More generally, we shall use 
\begin{displaymath}
    \mathcal{A}=\left(
        \begin{array}{cc}
            A & B\\
            C & D
        \end{array}\right)
\end{displaymath}
to denote an element of $\upq$. It follows from the definition that $\mathcal{A}\in \upq$ if and only if
\begin{eqnarray*}
    A^\ast A-C^\ast C& = & I_p\\
    A^\ast B-C^\ast D& = & 0_{p\times q}\\
    B^\ast B-D^\ast D& = & -I_q.
\end{eqnarray*}
Observe that if $Av=0$, then
\begin{displaymath}
    0=A^\ast Av=C^\ast Cv+v,
\end{displaymath}
which implies that $C^\ast C$, a positive semi-definite matrix, has $-1$ as an eigenvalue unless $v=0$. Therefore, $A$ is invertible, and the same argument works for $D$.
\begin{lemma}
    The intersection of $U(p+q)$ with $\upq$ is $U(p)\times U(q)$. Moreover, if $\mathcal{A}\in \upq$, then $\mathcal{A}^\ast,\sqrt{\mathcal{A}^\ast\mathcal{A}}\in \upq$.
\end{lemma}
\begin{proof}
    If $\mathcal{A}\in U(p)\times U(q)$, then 
    \begin{eqnarray*}
        A^\ast A+C^\ast C& = & I_p\\
        B^\ast B+D^\ast D& = & I_q.
    \end{eqnarray*}
    This implies that both $B$ and $C$ are zero matrices. If $\mathcal{A}\in \upq$, then $\mathcal{A}^\ast=I_{p,q}\mathcal{A}^{-1}I_{p,q}$ and
    \begin{align*}
        (\mathcal{A}^\ast\mathcal{A})^\ast I_{p,q}(\mathcal{A}^\ast\mathcal{A}) & =(\mathcal{A}^\ast\mathcal{A}) I_{p,q}(\mathcal{A}^\ast\mathcal{A}) \\
        & =I_{p,q}\mathcal{A}^{-1}I_{p,q}\mathcal{A}I_{p,q} I_{p,q}\mathcal{A}^{-1}I_{p,q}\mathcal{A}\\
        & =I_{p,q}=\mathcal{A}^\ast I_{p,q}\mathcal{A}.
    \end{align*}
    This also implies that $\mathcal{A}  I_{p,q}\mathcal{A}^\ast=I_{p,q}$. 

    \vspace{0.1cm}
    \hf All the eigenvalues of $\mathcal{A}^\ast\mathcal{A}$ are positive. Moreover, if $\lambda$ is an eigenvalue of $\mathcal{A}^\ast\mathcal{A}$ with eigenvector $\mathbf{v}=(v_1,\ldots,v_p,v_{p+1},\ldots, v_n)$, then 
    \begin{displaymath}
        I_{p,q}\mathbf{v}=\mathcal{A}^\ast\mathcal{A}\,I_{p,q}\mathcal{A}^\ast\mathcal{A}\mathbf{v}=\lambda(\mathcal{A}^\ast\mathcal{A}\,I_{p,q}\mathbf{v}),
    \end{displaymath}
    which implies that $\lambda^{-1}$ is also an eigenvalue with eigenvector $\mathbf{v}'=(v_1,\ldots,v_p,-v_{p+1},\ldots, -v_n)$. If $\{\mathbf{v}_1, \ldots,\mathbf{v}_n\}$ is an eigenbasis of $\mathcal{A}^\ast\mathcal{A}$ with (possibly repeated) eigenvalues $\lambda_1,\ldots,\lambda_n$, then 
    \begin{displaymath}
        \sqrt{\mathcal{A}^\ast\mathcal{A}}\,I_{p,q}\sqrt{\mathcal{A}^\ast\mathcal{A}}\mathbf{v}_j=\sqrt{\mathcal{A}^\ast\mathcal{A}}\,I_{p,q}\sqrt{\lambda_j}\mathbf{v}_j=\sqrt{\lambda_j}\sqrt{\mathcal{A}^\ast\mathcal{A}}\mathbf{v}_j'=\mathbf{v}_j'=I_{p,q}\mathbf{v}_j.
    \end{displaymath}
    Thus, $\sqrt{\mathcal{A}^\ast\mathcal{A}}$ satisfies the defining relation for a matrix to be in $\upq$. 
\end{proof}

\hf We may use the polar decomposition (for matrices in $GL(n,\C)$) to write
\begin{displaymath}
    \mathcal{A}=U |\mathcal{A}|,\,\,\textup{where}\,\,U=\mathcal{A}\left(\sqrt{\mathcal{A}^\ast\mathcal{A}}\right)^{-1}, |\mathcal{A}|=\sqrt{\mathcal{A}^\ast\mathcal{A}},
\end{displaymath}
where $U,|\mathcal{A}|\in \upq$. For $\textup{U}(1,1)$ this decomposition takes the form
\begin{displaymath}
    \left(\begin{array}{cc}
        \alpha & \beta\\
        \lambda \overline{\beta} & \lambda \overline{\alpha}
        \end{array}\right)=\left(\begin{array}{cc}
        \frac{\alpha}{|\alpha|} & 0\\
        0  & \lambda \frac{\overline{\alpha}}{|\alpha|}
        \end{array}\right)\left(\begin{array}{cc}
        |\alpha| & \frac{|\alpha|\beta}{\alpha}\\
        \frac{|\alpha| \overline{\beta}}{\overline{\alpha}} & |\alpha|
    \end{array}\right)
\end{displaymath}

\hf The Lie algebra $\mathfrak{u}_{p,q}$ is given by matrices $X\in M_n(\C)$ such that
\begin{displaymath}
    X^\ast I_{p,q}+I_{p,q}X=0.
\end{displaymath}
This is real Lie subalgebra of $M_{p+q}(\C)$. It contains the subalgebras $\mathfrak{u}_p, \mathfrak{u}_q$ as Lie algebras of the subgroups $U(p)\times I_q$ and $I_p \times U(q)$. Consider the inner product 
\begin{displaymath}
    \lan \cdot,\cdot\ran:\mathfrak{u}_{p,q}\times \mathfrak{u}_{p,q}\to\R,\,\,\,\lan X,Y\ran:=\textup{trace}(X^\ast Y).
\end{displaymath}

\begin{lemma}
    The inner product is symmetric and positive-definite. 
\end{lemma}
\begin{proof}
    Note that 
    \begin{displaymath}
        \lan X,Y\ran=\textup{trace}(-I_{p,q}XI_{p,q}Y)=\textup{trace}(-I_{p,q}YI_{p,q}X)=\lan Y,X\ran.
    \end{displaymath}
    Since $\overline{\lan X,Y\ran}=\lan Y,X\ran$ due to the invariance of trace under transpose, we conclude that the inner product is real and symmetric. It is positive-definite as $\lan X,X\ran=\textup{trace}(X^\ast X)\geq 0$ and equality holds if and only if $X$ is the zero matrix. 
\end{proof}

\hf The Riemannian metric obtained by left translations of $\lan\cdot,\cdot\ran$ will also be denoted by $\lan \cdot,\cdot\ran$. We shall analyze the geodesics for this metric. The Lie algebra $\mathfrak{u}_p\oplus \mathfrak{u}_q$ of $U(p)\times U(q)$ consists of 
\begin{displaymath}
    \left(
        \begin{array}{cc}
            A & 0 \\
            0 & D
        \end{array}\right),\,\,A+A^\ast=0,\,D+D^\ast =0.
\end{displaymath}
Let $\mathfrak{n}$ denote the orthogonal complement of $\mathfrak{u}_p\oplus \mathfrak{u}_q$ inside $\mathfrak{u}_{p,q}$. As $\mathfrak{n}$ is of (complex) dimension $pq$, and 
\begin{displaymath}
    \left\{\left(
        \begin{array}{cc}
            0 & B\\
            B^\ast & 0
        \end{array}\right)\,\Big|\,B\in M_{p,q}(\C)\right\}
\end{displaymath}
is contained in $\mathfrak{n}$, this is all of it. We may verify that
\begin{eqnarray*}
    \left[ \left(\begin{array}{cc}
    A & 0\\
    0 & D
    \end{array}\right),\left(\begin{array}{cc}
    0 & B\\
    B^\ast & 0
    \end{array}\right)\right] & = & \left(\begin{array}{cc}
    0 & AB-BD\\
    DB^\ast-B^\ast A & 0
    \end{array}\right)\in\mathfrak{n}\\
    \left[ \left(\begin{array}{cc}
    0 & B\\
    B^\ast & 0
    \end{array}\right),\left(\begin{array}{cc}
    0 & C\\
    C^\ast & 0
    \end{array}\right)\right] & = & \left(\begin{array}{cc}
    BC^\ast-CB^\ast & 0\\
    0 A & B^\ast C-C^\ast B
    \end{array}\right)\in\mathfrak{u}_p\oplus \mathfrak{u}_q.
\end{eqnarray*}
\begin{lemma}
    Let $\gamma$ be the integral curve, initialized at $e$, for a left-invariant vector field $Y$. This curve is a geodesic if $Y(e)$ either belongs to $\mathfrak{n}$ or to $\mathfrak{u}_p\oplus\mathfrak{u}_q$.
\end{lemma}
\begin{proof}
    The Levi-Civita connection $\nabla$ is given by the Koszul formula
    \begin{displaymath}
        2\lan X,\nabla_Z Y\ran = Z\lan X,Y\ran+Y\lan X,Z\ran -X\lan Y,Z\ran +\lan Z,[X,Y]\ran+\lan Y,[X,Z]\ran - \lan X,[Y,Z]\ran. 
    \end{displaymath}
    Putting $Z=Y$ and $X$, two left-invariant vector fields, in the above, we obtain
    \begin{displaymath}
        \lan X, \nabla_Y Y\ran = \lan Y, [X,Y]\ran.
    \end{displaymath}
    To prove our claim, it suffices to show that $\nabla_Y Y=0$, ie, $\lan Y, [X,Y]\ran=0$ for any $X$. Let us assume that $Y(e)\in\mathfrak{n}$. If $X(e)\in\mathfrak{n}$, then $[X(e),Y(e)]\in \mathfrak{u}_p\oplus\mathfrak{u}_q$, which implies that $\lan Y(e), [X(e),Y(e)]\ran=0$. If $X(e)\in \mathfrak{u}_p\oplus\mathfrak{u}_q$, then
    \begin{eqnarray*}
        \lan Y, [X,Y]\ran & = &  \lan\left(\begin{array}{cc}
        0 & B\\
        B^\ast & 0
        \end{array}\right), \left(\begin{array}{cc}
        0 & AB-BD\\
        DB^\ast-B^\ast A & 0
        \end{array}\right)\ran\\
        & = & \textup{trace}\left(\begin{array}{cc}
        B(DB^\ast-B^\ast A) & 0\\
        0 & B^\ast(AB-BD)
        \end{array}\right)\\
        & = & \textup{trace}(BDB^\ast-BB^\ast A)+\textup{trace}(B^\ast AB-B^\ast BD)\\
        & = & 0
    \end{eqnarray*}
    by the cyclic property of trace. Thus, $\nabla_Y Y=0$ if $Y(e)\in\mathfrak{n}$; similar proof works if $Y(e)\in \mathfrak{u}_p\oplus\mathfrak{u}_q$. 
\end{proof}

\begin{rem}
    An integral curve of a left-invariant vector field (also called $1$-parameter subgroups) need not be a geodesic in $\upq$. For instance, if $X+Y$ is a left-invariant vector field given by $X(e)\in\mathfrak{u}_p\oplus\mathfrak{u}_q$ and $Y(e)\in \mathfrak{n}$, then $\nabla_{X+Y}(X+Y)=0$ if and only if $\nabla_X Y=\frac{1}{2}[X,Y]$ and $\nabla_Y X=\frac{1}{2}[Y,X]$. This happens if and only if the metric is bi-invariant, i.e.,
    \begin{displaymath}
        \lan [X,Z],Y\ran=\lan X, [Z,Y]\ran.
    \end{displaymath} 
    This is not true in general; for instance, with $X(e)\in \mathfrak{u}_p\oplus\mathfrak{u}_q$ and linearly independent $Y(e),Z(e)\in \mathfrak{n}$, we get $\lan [X,Z],Y\ran-\lan X, [Z,Y]\ran\neq 0$.
\end{rem}

\hf Consider the matrix
\begin{displaymath}
    Y=\left(\begin{array}{cc}
    0 & B\\
    B^\ast & 0
    \end{array}\right)\in\mathfrak{n}.
\end{displaymath}
Let $B=U \sqrt{B^\ast B}$ and $B^\ast =\sqrt{B^\ast B}\,U^\ast$ be the polar decompositions for the rectangular matrices. It follows from direct computation that
\begin{eqnarray*}
    e^Y & = & \left(\begin{array}{cc}
    I_p+\frac{BB^\ast}{2!} + \frac{(BB^\ast)^2}{4!}+\cdots & \frac{B}{1!}+\frac{B (B^\ast  B)}{3!} + \frac{B (B^\ast B)^2 }{5!}+\cdots \\
    \frac{B^\ast}{1!}+\frac{(B^\ast B)B^\ast}{3!} + \frac{(B^\ast B)^2 B^\ast}{5!}+\cdots  & I_q+\frac{B^\ast B}{2!} + \frac{(B^\ast B)^2}{4!}+\cdots 
    \end{array}\right)\\
    & = & \left(\begin{array}{cc}
    \cosh(\sqrt{BB^\ast}) & U\sinh(\sqrt{B^\ast B})\\
    \sinh(\sqrt{B^\ast B})U^\ast & \cosh(\sqrt{B^\ast B})
    \end{array}\right).
\end{eqnarray*}
It can be checked that 
\begin{displaymath}
    e^\mathfrak{n}\cap \left(U(p)\times U(q)\right)=\{I_n\}.
\end{displaymath}
It is known that the non-zero eigenvalues of $Y$ are the non-zero eigenvalues of $\sqrt{BB^\ast}$ and their negatives. 
\begin{thm}\label{mainthm}
    For any element $\mathcal{A}\in \upq$, the associated matrix $\sqrt{\mathcal{A}^\ast\mathcal{A}}$ can be expressed uniquely as $e^Y$ for $Y\in \mathfrak{n}$. Moreover, there is a unique way to express $\mathcal{A}$ as a product of a unitary matrix and an element of $e^\mathfrak{n}$, and it is given by the polar decomposition.
\end{thm}

\vspace{0.3cm}
\noindent In order to prove the result, we discuss some preliminaries on logarithm of complex matrices. In general, there is no unique logarithm. However, the Gregory series \index{Gregory series}
\begin{displaymath}
    \log A=-\sum_{m=0}^{\infty}\frac{2}{2m+1}\left[(I-A)(I+A)^{-1}\right]^{2m+1}
\end{displaymath}
converges if all the eigenvalues of $A\in M_n(\C)$ have positive real part, see \cite[\S 11.3, page 273]{Hig08}. In particular, $\log A$ is well-defined for Hermitian positive-definite matrix. This is often called the \textit{principal logarithm}\index{principal logarithm} of $A$. This logarithm satisfies $e^{\log A}=A$. There is an integral form of logarithm that applies to matrices without real or zero eigenvalues; it is given by
\begin{displaymath}
    \log A =(A-I)\int_0^1 \left[s(A-I)+I\right]^{-1}ds.
\end{displaymath}
\begin{lemma}\label{inv}
    The inverse of $\mathcal{A}^\ast\mathcal{A}+I_n$ for $\mathcal{A}\in \upq$ is given by
    \begin{displaymath}
        \left[\mathcal{A}^\ast\mathcal{A}+I_n\right]^{-1}=\frac{1}{2}\left(\begin{array}{cc}
        I_p & -A^{-1}B\\
        -B^\ast (A^\ast)^{-1} & I_q
        \end{array}\right).
    \end{displaymath}
\end{lemma}
\begin{proof}
    Since $\mathcal{A}^\ast\mathcal{A}$ has only positive eigenvalues, $\mathcal{A}^\ast\mathcal{A}+I_n$ has no kernel. We note that 
    \begin{displaymath}
        \mathcal{A}^\ast\mathcal{A}+I_n=\left(\begin{array}{cc}
        2C^\ast C+2I_p & 2A^\ast B\\
        2B^\ast A & 2B^\ast B+2I_q
        \end{array}\right)=\left(\begin{array}{cc}
        2A^\ast A & 2A^\ast B\\
        2B^\ast A & 2D^\ast D
        \end{array}\right).
    \end{displaymath}
    The inverse matrix satisfies 
    \begin{displaymath}
        \left(\begin{array}{cc}
        2A^\ast A & 2A^\ast B\\
        2B^\ast A & 2D^\ast D
        \end{array}\right)\left(\begin{array}{cc}
        E & F\\
        F^\ast & G
        \end{array}\right)=\left(\begin{array}{cc}
        I_p & 0\\
        0 & I_q
        \end{array}\right).
    \end{displaymath} 
    As the matrices are Hermitian, the three constraints that $E,F,G$ must satisfy (and are uniquely determined by) are
    \begin{eqnarray*}
        E & = & \textstyle{\frac{1}{2}}(A^\ast A)^{-1}-A^{-1}BF^\ast\\
        G & = & \textstyle{\frac{1}{2}}(D^\ast D)^{-1}-D^{-1}CF\\
        F & = & -A^{-1}BG.
    \end{eqnarray*}
    We note that $E=\frac{1}{2}I_p$, $G=\frac{1}{2}I_q$ and $F=-\frac{1}{2}A^{-1}B$ satisfy the above equations. For instance, 
    \begin{align*}
        \frac{1}{2}(A^\ast A)^{-1}-A^{-1}BF^\ast & = \frac{1}{2}(A^\ast A)^{-1}+\frac{1}{2}A^{-1}BB^\ast (A^{\ast})^{-1} \\
        & = \frac{1}{2}(A^\ast A)^{-1}+\frac{1}{2}A^{-1}(AA^\ast-I_p)(A^{\ast})^{-1}=\frac{1}{2}I_p,
    \end{align*}
    where $BB^\ast=AA^\ast-I_p$ is a consequence of $\mathcal{A}^\ast\in \upq$. Yet another consequence is $AC^\ast=BD^\ast$, which is equivalent to
    \begin{displaymath}
        A^{-1}B=(D^{-1}C)^\ast.
    \end{displaymath}
    In a similar vein,
    \begin{align*}
        \frac{1}{2}(D^\ast D)^{-1}-D^{-1}CF & = =\frac{1}{2}(D^\ast D)^{-1}+\frac{1}{2}D^{-1}CC^\ast (D^{\ast})^{-1} \\
        & = \frac{1}{2}(D^\ast D)^{-1}+\frac{1}{2}D^{-1}(DD^\ast-I_q)(D^{\ast})^{-1} \\
        & = \frac{1}{2}I_q,
    \end{align*}
    where $CC^\ast=DD^\ast-I_q$ is due to $\mathcal{A}^\ast\in \upq$.
\end{proof}

\begin{proof}[Proof of \Cref{mainthm}]
    We use Gregory series expansion for computing the principal logarithm of $\mathcal{A}^\ast\mathcal{A}$ along with \Cref{inv}:
    \begin{align*}
        & \log (\mathcal{A}^\ast\mathcal{A}) 
        \\
        & = \sum_{m=0}^\infty{\frac{2}{2m+1}}\left[2\left(\begin{array}{cc}
        A^\ast A-I_p & A^\ast B\\
        B^\ast A & D^\ast D-I_q
        \end{array}\right)
        \frac{1}{2}\left(\begin{array}{cc}
        I_p & -A^{-1}B\\
        -B^\ast (A^\ast)^{-1} & I_q
        \end{array}\right)\right]^{2m+1}\\
        & = \sum_{m=0}^\infty{\frac{2}{2m+1}}\left(\begin{array}{cc}
        0 & A^{-1}B\\
        B^\ast (A^\ast)^{-1} & 0
        \end{array}\right)^{2m+1}.
    \end{align*}
    We set $Y=\frac{1}{2}\log (\mathcal{A}^\ast\mathcal{A})$. It is clear that $Y\in\mathfrak{n}$ and $e^Y=\sqrt{\mathcal{A}^\ast\mathcal{A}}$. It is known that the exponential map is injective on Hermitian matrices. This implies the uniqueness of $Y$.
    
    \vspace{0.1cm}
    \hf If $U_1e^{Y_1}=U_2 e^{Y_2}$ are two decompositions of $\mathcal{A}\in\upq$ with $U_i\in \textup{U}(p)\times\textup{U}(q)$ and $Y_i\in\mathfrak{n}$, then 
    \begin{displaymath}
        e^{2Y_1}=e^{Y_1}U_1^\ast U_1 e^{Y_1}=e^{Y_2}U_2^\ast U_2 e^{Y_2}=e^{2Y_2}.
    \end{displaymath}
    By the injectivity of the exponential map (on Hermitian matrices), we obtain $Y_1=Y_2$, which implies that $U_1=U_2$.
\end{proof}

\hf We infer the following (see \cite[ Lemma 1, page 211]{YaSt75} for a different proof) result.
\begin{cor}\label{expsurj}
    The exponential map $\textup{exp}:\mathfrak{u}_{p,q}\to U(p,q)$ is surjective.
\end{cor}
\begin{proof}
    Using the polar decomposition and \Cref{mainthm}, 
    \begin{displaymath}
        \mathcal{A}=\mathcal{A}\big(\sqrt{\mathcal{A}^\ast\mathcal{A}}\big)^{-1}\sqrt{\mathcal{A}^\ast\mathcal{A}}=\mathcal{A}\big(\sqrt{\mathcal{A}^\ast\mathcal{A}}\big)^{-1}e^Y.
    \end{displaymath}
    Since the matrix exponential is surjective for $U(p)\times U(q)$, choose $Z\in \mathfrak{u}_p\oplus\mathfrak{u}_q$ such that $e^Z=\mathcal{A}(\sqrt{\mathcal{A}^\ast\mathcal{A}})^{-1}$. By Baker-Campbell-Hausdorff formula, we may express $e^Z e^Y$ as exponential of an element in $\mathfrak{u}_{p,q}$. 
\end{proof}

\hf The distance from any matrix $\mathcal{A}\in \upq$ to $U(p)\times U(q)$ is given by the length of the curve
\begin{displaymath}
    \gamma(t)=\mathcal{A}\big(\sqrt{\mathcal{A}^\ast\mathcal{A}}\big)^{-1}e^{tY},
\end{displaymath}
which can be computed (and simplified via left-invariance) as follows
\begin{displaymath}
    \ell(\gamma)=\int_0^1 \|\gamma'(t)\|_{\gamma(t)}\,dt=\int_0^1 \|Y\|\,dt=\|Y\|.
\end{displaymath}
Note that 
\begin{displaymath}
    \|Y\|^2=\textup{trace}(Y^\ast Y)=\textup{trace}\left[\textstyle{\frac{1}{4}}(\log (\mathcal{A}^\ast\mathcal{A}))^2\right].
\end{displaymath}
Thus, the distance squared function is given by 
\begin{displaymath}
    d^2:\upq\to\R,\,\,\mathcal{A}\mapsto \textstyle{\frac{1}{4}}\textup{trace}\left[\left(\log (\mathcal{A}^\ast\mathcal{A})\right)^2\right].
\end{displaymath}
	\chapter[Equivariant cut locus]{Equivariant cut locus}\label{ch:equivariantCutLocusTheorem}
\minitoc

\hfb Let $M$ be a smooth manifold on which a compact Lie group $G$ acts freely. It is known that $M/G$ is a smooth manifold. Moreover, if $M$ has a Riemannian metric and $G$ acts isometrically, then there is an induced Riemannian metric on $M/G$. Let $N$ be a $G$-invariant  submanifold of $M$. We want to find the cut locus of $N/G$ in $M/G$. In this chapter we will discuss the equality between $\cutn/G$ and $\cutn[N/G]$.  We will start the chapter by motivating with an example, then state the theorem and then recall Riemannian submersion which will be useful for the proof of the theorem. At the end we will discuss an application of this result to complex projective hypersurfaces.

\section{Statement of the theorem}
\hfb Let us discuss an example which will be helpful to arrive at the statement of the main theorem. Let $\mathbb{RP}^n$ denotes the $n$-dimensional real projective space which is obtained from $\mathbb{S}^n$ by identifying points $p$ and $-p$. Equivalently, this space can be obtained from the action of $\mathbb{Z}_2$ on $\mathbb{S}^n$. We know from \Cref{eg:CutLocusofPointRPn} that for a point $p\in \mathbb{RP}^n$, the cut locus is $\mathbb{RP}^{n-1}$. In \Cref{join}, we showed that $\mathrm{Cu}\left(\mathbb{S}_i^k\right)=\mathbb{S}_l^{n-k-1}$, where $\mathbb{S}_i^k \hookrightarrow \mathbb{S}^n$ denote the embedding of the $k$-sphere in the first $k+1$ coordinates while $\mathbb{S}^{n-k-1}_l$ denote the embedding of the $(n-k-1)$-sphere in the last $n-k$ coordinates. So we have
\begin{displaymath}
	\begin{tikzcd}
		\mathbb{S}^n \supset\left\{p,-p\right\}\arrow[d,  "\mathbb{Z}_2"] \arrow[rr, "\mathrm{Cu}"] &  &  \mathbb{S}^{n-1} \arrow[d, "\mathbb{Z}_2"]\\
		\mathbb{RP}^n \supset\{p\} \arrow[rr, "\mathrm{Cu}"]& & \mathbb{RP}^{n-1}
	\end{tikzcd}
\end{displaymath}

\noindent Similarly, one can see a similar diagram for the complex projective space $\mathbb{CP}^n$ which is obtained by taking $\mathbb{S}^1$ action on $\mathbb{S}^{{2n+1}}$. We have 
\begin{displaymath}
	\begin{tikzcd}
		\mathbb{S}^{2n+1} \supset \mathbb{S}_i^1 \arrow[d,  "\mathbb{S}^1"] \arrow[rr, "\mathrm{Cu}"] &  &  \mathbb{S}^{2n-1}_f \arrow[d, "\mathbb{S}^1"]\\
		\mathbb{CP}^n \supset\{p\} \arrow[rr, "\mathrm{Cu}"]& & \mathbb{CP}^{n-1}
	\end{tikzcd}
\end{displaymath}

\noindent Thus, it is natural to ask whether the following diagram commutes for a compact Lie group $G$. 
\begin{displaymath}
	\begin{tikzcd}
		M \supset N \arrow[d,  "G"] \arrow[rr, "\mathrm{Cu}"] &  &  \mathrm{Cu}(N) \arrow[d, "G"]\\
		M/G \supset N/G \arrow[rr, "\mathrm{Cu}"]& & \mathrm{Cu}(N)/G
	\end{tikzcd}
\end{displaymath}

\noindent The above diagram make sense if $N$ and $\cutn$ are $G$-invariant subsets. Note that if the action is free and isometric, i.e., length of the curves $\gamma$ and $g\cdot \gamma$ are the same,  then $\mathrm{Se}(N)$ is $G$-invariant. 
\begin{figure}[!htbp]
	\centering
	\begin{subfigure}{.45\textwidth}
    \def\svgwidth{0.8\columnwidth}
    %% Creator: Inkscape 1.2 (1:1.2.1+202207142221+cd75a1ee6d), www.inkscape.org
%% Accompanies image file 'separatingSetIsG-Invariant-1.pdf' (pdf, eps, ps)
\begingroup%
  \makeatletter%
  \providecommand\color[2][]{%
    \errmessage{(Inkscape) Color is used for the text in Inkscape, but the package 'color.sty' is not loaded}%
    \renewcommand\color[2][]{}%
  }%
  \providecommand\transparent[1]{%
    \errmessage{(Inkscape) Transparency is used (non-zero) for the text in Inkscape, but the package 'transparent.sty' is not loaded}%
    \renewcommand\transparent[1]{}%
  }%
  \providecommand\rotatebox[2]{#2}%
  \newcommand*\fsize{\dimexpr\f@size pt\relax}%
  \newcommand*\lineheight[1]{\fontsize{\fsize}{#1\fsize}\selectfont}%
  \ifx\svgwidth\undefined%
    \setlength{\unitlength}{282.56388266bp}%
    \ifx\svgscale\undefined%
      \relax%
    \else%
      \setlength{\unitlength}{\unitlength * \real{\svgscale}}%
    \fi%
  \else%
    \setlength{\unitlength}{\svgwidth}%
  \fi%
  \global\let\svgwidth\undefined%
  \global\let\svgscale\undefined%
  \makeatother%
  \begin{picture}(1,0.7268979)%
    \lineheight{1}%
    \setlength\tabcolsep{0pt}%
    \put(0,0){\includegraphics[width=\unitlength,page=1]{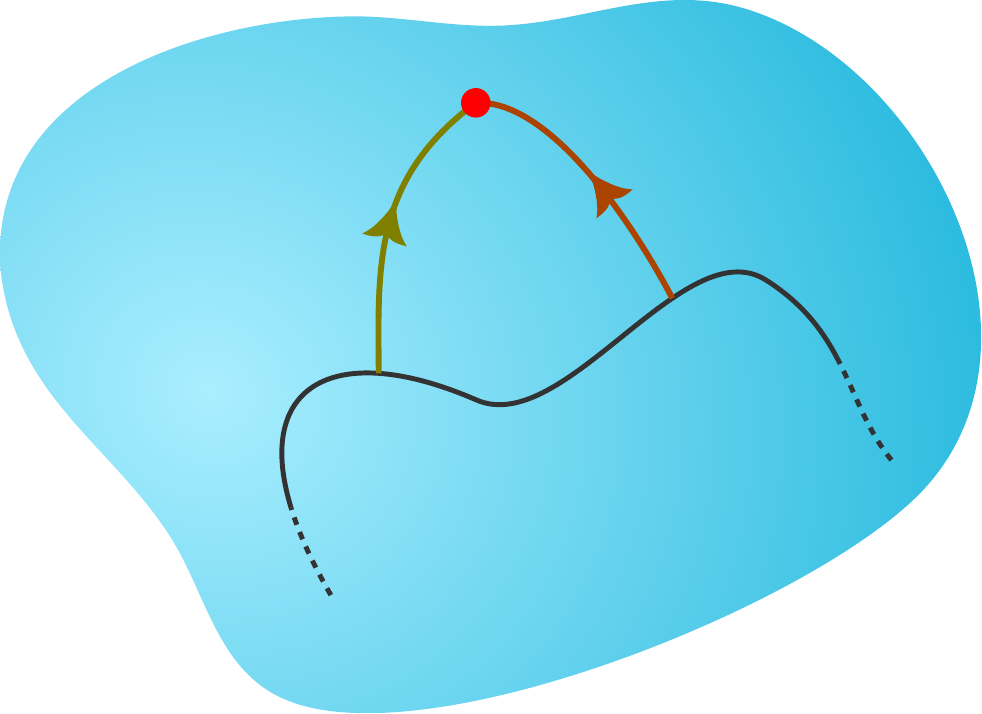}}%
    \put(0.54063113,0.25653855){\makebox(0,0)[lt]{\lineheight{1.25}\smash{\begin{tabular}[t]{l}$N$\end{tabular}}}}%
    \put(0.06171333,0.51225937){\makebox(0,0)[lt]{\lineheight{1.25}\smash{\begin{tabular}[t]{l}$M$\end{tabular}}}}%
    \put(0.32726021,0.51224345){\color[rgb]{0.50196078,0.50196078,0}\makebox(0,0)[lt]{\lineheight{1.25}\smash{\begin{tabular}[t]{l}$\gamma$\end{tabular}}}}%
    \put(0.63381953,0.54439601){\color[rgb]{0.66666667,0.26666667,0}\makebox(0,0)[lt]{\lineheight{1.25}\smash{\begin{tabular}[t]{l}$\eta$\end{tabular}}}}%
    \put(0.4497778,0.6475709){\makebox(0,0)[lt]{\lineheight{1.25}\smash{\begin{tabular}[t]{l}$p$\end{tabular}}}}%
  \end{picture}%
\endgroup%

	\end{subfigure}
	\begin{subfigure}{.45\textwidth}
    \def\svgwidth{0.8\columnwidth}
    %% Creator: Inkscape 1.2 (1:1.2.1+202207142221+cd75a1ee6d), www.inkscape.org
%% Accompanies image file 'separatingSetIsG-Invariant-2.pdf' (pdf, eps, ps)
\begingroup%
  \makeatletter%
  \providecommand\color[2][]{%
    \errmessage{(Inkscape) Color is used for the text in Inkscape, but the package 'color.sty' is not loaded}%
    \renewcommand\color[2][]{}%
  }%
  \providecommand\transparent[1]{%
    \errmessage{(Inkscape) Transparency is used (non-zero) for the text in Inkscape, but the package 'transparent.sty' is not loaded}%
    \renewcommand\transparent[1]{}%
  }%
  \providecommand\rotatebox[2]{#2}%
  \newcommand*\fsize{\dimexpr\f@size pt\relax}%
  \newcommand*\lineheight[1]{\fontsize{\fsize}{#1\fsize}\selectfont}%
  \ifx\svgwidth\undefined%
    \setlength{\unitlength}{282.56386104bp}%
    \ifx\svgscale\undefined%
      \relax%
    \else%
      \setlength{\unitlength}{\unitlength * \real{\svgscale}}%
    \fi%
  \else%
    \setlength{\unitlength}{\svgwidth}%
  \fi%
  \global\let\svgwidth\undefined%
  \global\let\svgscale\undefined%
  \makeatother%
  \begin{picture}(1,0.72689795)%
    \lineheight{1}%
    \setlength\tabcolsep{0pt}%
    \put(0,0){\includegraphics[width=\unitlength,page=1]{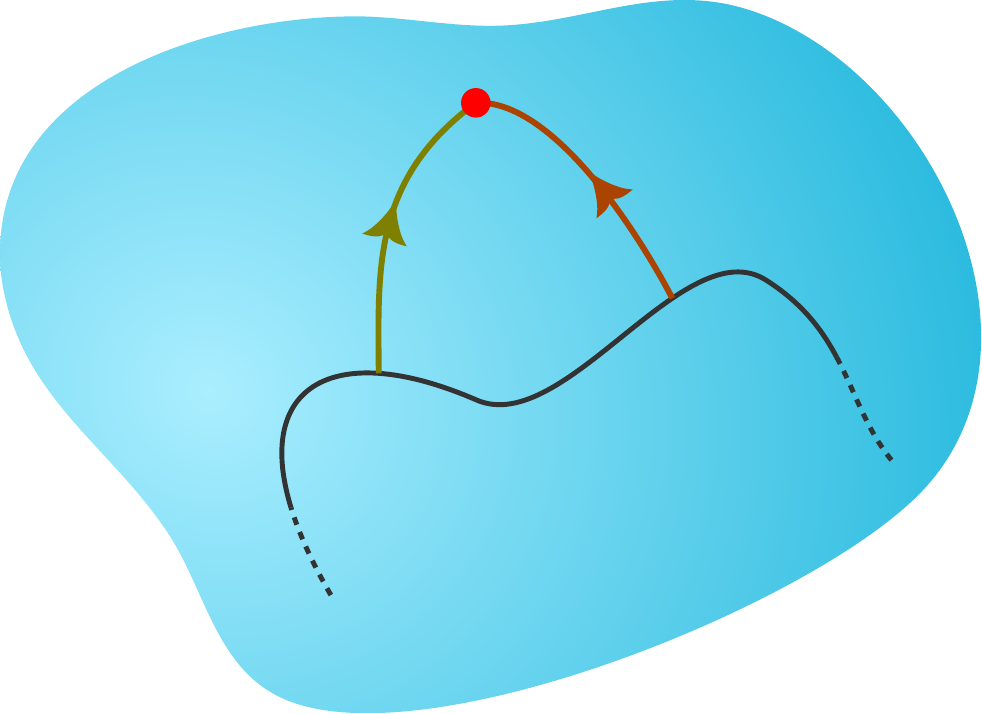}}%
    \put(0.54063117,0.25653858){\makebox(0,0)[lt]{\lineheight{1.25}\smash{\begin{tabular}[t]{l}$N$\end{tabular}}}}%
    \put(0.06171333,0.51225958){\makebox(0,0)[lt]{\lineheight{1.25}\smash{\begin{tabular}[t]{l}$M$\end{tabular}}}}%
    \put(0.2385612,0.51224351){\color[rgb]{0.50196078,0.50196078,0}\makebox(0,0)[lt]{\lineheight{1.25}\smash{\begin{tabular}[t]{l}$g\cdot\gamma$\end{tabular}}}}%
    \put(0.64778401,0.54439592){\color[rgb]{0.66666667,0.26666667,0}\makebox(0,0)[lt]{\lineheight{1.25}\smash{\begin{tabular}[t]{l}$g\cdot\eta$\end{tabular}}}}%
    \put(0.41257635,0.65137356){\makebox(0,0)[lt]{\lineheight{1.25}\smash{\begin{tabular}[t]{l}$g\cdot p$\end{tabular}}}}%
  \end{picture}%
\endgroup%

	\end{subfigure}
	\caption{$\sen$ is $G$-invariant}
\end{figure}
\noindent Now to show that the cut locus of $N$, which is the closure of $\mathrm{Se}(N)$,  is $G$-invariant we take $x\in \overline{\mathrm{Se}(N)}$. So there exists a sequence $\left(x_n\right) \subset \mathrm{Se}(N)$ such that $x_n\to x$ (this convergence is with respect to the Riemannian metric). This implies $g\cdot x_n \to g\cdot x$ as the action is continuous. Hence, $g\cdot x\in \mathrm{Se}(N).$  

\hf The following theorem tells us that the above diagram commutes if $G$ is a compact Lie group and the action is free and isometric.

\begin{thm}[Equivariant cut locus theorem]\label{thm:equivariant-cut-locus}\index{equivariant cut locus theorem}
    Let $M$ be a closed and connected Riemannian manifold  and $G$ be any compact Lie group which acts on $M$ freely and isometrically. Let $N$ be any $G$-invariant closed submanifold of $M$, then we have an equality
    \begin{displaymath}
        \mathrm{Cu}(N)/G = \mathrm{Cu}(N/G).
    \end{displaymath}
\end{thm}

\begin{rem}
	\begin{enumerate}[(i)]
		\item If the action of $G$ is not isometric, then we can construct a $G$-invariant metric on $M$ by averaging any metric on $M$ over $G$. In fact, for any $p\in M$ and any vectors $\mathbf{v}_1,\mathbf{v}_2\in T_pM$, we can define
		\begin{displaymath}
			\left\langle \mathbf{v}_1,\mathbf{v}_2 \right\rangle \defeq \int_G \left\langle \mathbf{v}_1,\mathbf{v}_2 \right\rangle_{g\cdot p}~dg,
		\end{displaymath}
		where the integral is taken with respect to the Haar measure. Then the theorem is valid with respect to the new metric.
		\item Recall that, in \Cref{thm: Morse-Bott}, we proved that the gradient flow of distance squared function from the submanifold $N$ deforms $M-\cutn$ to $N$. By construction, the flow lines are $G$-invariant. So we also have that $M/G-\cutn/G$ deforms to $N/G$.
	\end{enumerate}  
\end{rem}

\vspace{0.3cm}    
\noindent To prove the theorem we have to encounter mainly two problems.
\begin{enumerate}[a.]
	\item Whether a distance minimal geodesic in $M$ projects down to a distance minimal geodesic in $M/G$?
	\item Whether a distance minimal geodesic in $M/G$ lifts to a distance minimal geodesic in $M$?
\end{enumerate}

\section{Proof of the equivariant cut locus theorem} \label{sec:Riemannian-submersion}\index{Riemannian submersion}

\hfb In this section we will recall some results on Riemannian submersion and most of the results can be found in the book \cite[Chapter V, section 26]{Mic08}.

\begin{defn}\label{defn:Ehresmann-connection}
	Let $\pi:E\to B$ be a smooth principal $G$-bundle. An \textit{Ehresmann connection}\index{Ehresmann connection} on $E$ is a smooth subbundle $\mathcal{H}$ of $TE$, called the \textit{horizontal bundle}\index{horizontal bundle} of the connection, such that $TE=\mathcal{H}\directsum \mathcal{V}$, where $\mathcal{V}_p=\ker \left(d\pi_p:T_pE\to T_{\pi(p)}B\right)$.
\end{defn}

\vspace{0.3cm}
\hf The bundle $\mathcal{V}$ is called the \textit{vertical bundle}\index{vertical bundle}, and it is independent of the connection chosen. It follows from the definition that $\mathcal{H}_p$ depends smoothly on $p$ and $\mathcal{H}_p\cap \mathcal{V}_p=\{0\}$. Moreover, the map $d\pi_p$ restricts to $\mathcal{H}_p$ is an isomorphism on $T_{\pi(p)}B$. If $E$ is a Riemannian manifold with metric $g$, then by choosing a horizontal bundle $\mathcal{H}$ we have $\mathcal{H}_p=\mathcal{V}_p^\perp$.

\begin{defn}\label{defn:Riemannian-submersion}
	 A smooth submersion $\pi:(E,g)\to (B,g')$ is called a \textit{Riemannian submersion}\index{Riemannian submersion} if the linear map $d\pi_p$ preserves the length of the horizontal vectors for each point $p\in E$. Equivalently, $d\pi_p$ is a linear isometry between $\mathcal{H}_p$ and $T_{\pi(p)}B$.
\end{defn}

\vspace{0.3cm}
\noindent Using the above definitions we can define the following type of vectors.

\begin{defn}
	Let $\pi:E\to B$ be a Riemannian submersion. A vector field $X\in \mathfrak{X}(E)$ is called
	\begin{itemize}
		\item \textit{vertical} \index{vertical vector field} if for any $p\in E$, $X_p\in \mathcal{V}_p$, denoted by $X^\text{ver}$, and
		\item \textit{horizontal} \index{horizontal vector field} if for any $p\in E$, $X_p\in \mathcal{H}_p$, denoted by $X^\text{hor}$.
	\end{itemize}		
\end{defn}

\vspace{0.3cm}
\noindent We can uniquely decompose any vector field $X\in \mathfrak{X}(E)$ as
\begin{displaymath}
	X = X^\text{ver}+X^{\text{hor}}
\end{displaymath}
into its horizontal and vertical components.

\vspace{0.3cm}
\hf Once we have a connection, we now have a preferred way of lifting vectors from $TB$ to $TE$. Recall that
a vector $\tilde{X}\in T_eE$ is a \textit{lift} of $X\in T_{\pi(e)B}$ if $T_{e}\pi(\tilde{X})=X$. In absence of a connection, there are many different choices of lifts of a vector, and any two choices differ by a vertical vector. That is, if $\tilde{X},\tilde{X}'$ are lifts of $X$,  then $\tilde{X}-\tilde{X}'$ is vertical. Once we have a connection, we can define the \textit{horizontal lift} (with respect to a connection $\mathcal{H}$) of $X$ as the horizontal component of any lift of $X$. This definition
is, of course, independent of the choice of lift, since any two differ by a vertical vector, whose horizontal
component vanishes. Similarly, we can lift vector fields by lifting them in a pointwise fashion.
\begin{defn}[Horizontal lift of vector fields]\index{horizontal lift of a vector field}
	Let $X\in \mathfrak{X}(B)$ be a vector field and $\mathcal{H}\subset TE$ an Ehresmann connection on $E$. We define the \textit{horizontal lift of $X$} as the vector field $\tilde{X}\in \mathfrak{X}(E)$ which satisfies $d\pi(\tilde{X})=X$ and $\tilde{X}_e\in \mathcal{H}_e$ for all $e\in E$.
\end{defn}

\hf Suppose that we have a curve $\gamma:[0,1]\to B$. At each point over the curve, we have a vector $\gamma'(t)\in T_{\gamma(t)}B$, which we can lift to the fiber above $\gamma(t)$. So if we choose a starting point $e_0\in \pi^{-1}(\gamma(0))$, we can find an integral curve along all these lifted vectors on the fibers over the curve $\gamma$. In the end we obtain a curve $\tilde{\gamma}:[0,1]\to E$ satisfying $\pi\circ \tilde{\gamma}=\gamma,~\tilde{\gamma}(0)=e_0$, and $\tilde{\gamma}'(t)\in \mathcal{H}_{\tilde{\gamma}(t)}$ for all $t$. We call it a horizontal lift of $\gamma$.

\begin{defn}[Horizontal lift of a curve]\label{defn:horizontal-lift}\index{horizontal lift of a curve}
	Let $\pi:E\to B$ be a fiber bundle with a connection $\mathcal{H}$. Let $\gamma$ be a smooth curve in $B$ through $\gamma(0)=b$. Let $e\in E$ be such that $\pi(e)=b$. A \textit{horizontal lift}\index{horizontal lift} of $\gamma$ through $e$ is a curve $\tilde{\gamma}$ in $E$ such that $\pi\circ \tilde{\gamma}=\gamma,~\tilde{\gamma}(0)=e$, and $\tilde{\gamma}\kern 0.02cm'(t)\in \mathcal{H}_{\tilde{\gamma}(t)}$.
\end{defn}

\vspace{0.3cm}
\hf For every point $t_0\in (0,1)$, we can find $\epsilon>0$ such that the vector field $\gamma'$ can be extended to a vector field over $\gamma\big|_{(t_0-\epsilon,t_0+\epsilon)}$. Then we look at the horizontal vector field $\tilde{X}$ defined on the bundle $E\big|_{U}$, where $U\supseteq \gamma((t_0-\epsilon,t_0+\epsilon))$. Then $\tilde{\gamma}$ is the integral curve of $\tilde{X}$ starting at the prescribed point $e_0\in \pi^{-1}(\gamma(0))$. Since $[0, 1]$ is compact, a usual gluing argument will help us to construct a horizontal lift of $\gamma$. Hence, we have the following proposition. For a detailed proof of the proposition we refer the reader to \cite[Chapter XIV, Proposition 3.5(i)]{Lang99}.
\begin{prop}\label{prop:uniqueness-of-horizontal-lift}
	Given a smooth path $\gamma:[0,1]\to B$ such that $\gamma(0)=b$ and $e_0\in \pi^{-1}(b)$, there is a unique horizontal lift $\tilde{\gamma}$ of $\gamma$ through $e_0\in E$.
\end{prop}

\vspace{0.3cm}
\noindent Recall the Quotient manifold theorem \cite[Theorem 21.10]{Lee13}.

\begin{thm}[Quotient Manifold Theorem]
	Suppose a Lie group $G$  acting smoothly, freely, and properly on a smooth manifold $M$. Then the orbit space $M/G$ is a topological manifold of dimension equal to $\dim M-\dim G$, and has a unique smooth structure with the property that the quotient map $\pi:M\to M/G$  is a smooth submersion.
\end{thm}

\vspace{0.3cm}
\noindent Using the above theorem, we can define a unique metric on $M/G$ such that $\pi$ is a Riemannian submersion. In fact, for any $x\in M/G$ and $p\in \pi^{-1}(x)$, we take the vertical space $\mathcal{V}_p \defeq \ker d\pi_p$ and the horizontal space $\mathcal{H}_p\defeq \mathcal{V}_p^\perp$ so that $T_pM = \mathcal{V}_p \oplus \mathcal{H}_p$. Since $d\pi_p$ is surjective, then $d\pi_p\big|_{\mathcal{H}_p}:\mathcal{H}_p\to T_{\pi(p)}M/G$ is a bijection. Define 
\begin{displaymath}
	h_x( v,w) \defeq \left\langle d\pi^{-1}\big|_{\mathcal{H}_p}(v),d\pi^{-1}\big|_{\mathcal{H}_p}(w) \right\rangle.
\end{displaymath}
Since the action is isometric, the metric is independent of the choice of point $p$. Thus, $h$ defines a well-defined metric on $M/G$ and $\pi$ is a Riemannian submersion.

\vspace{0.3cm}
\hf The following is the key lemma for proving \Cref{thm:equivariant-cut-locus} and the proof of the same can be found in \cite[Lemma 26.11]{Mic08}.

\begin{lemma}\label{lemma:horizontal-lift-is-a-geodesic}
	Let $\left(E,g_E\right)$ and $\left(B,g_B\right)$ be two Riemannian manifolds and $\pi:E\to B$ be a Riemannian submersion. Let $\gamma$ be a geodesic in $B$ and $\tilde{\gamma}$ be the horizontal lift of $\gamma$. Then we have:
	\begin{enumerate}
		\item The length of $\gamma$ and $\tilde{\gamma}$ are same.
		\item $\tilde{\gamma}'(t)$ is perpendicular to each fiber $E_{\tilde{\gamma}(t)}$.
		\item $\tilde{\gamma}$ is a geodesic in $E$.
	\end{enumerate}
\end{lemma}

\vspace{0.3cm}
\noindent We also have the following result.
\begin{thm}[O'Neill]\label{thm:O'Neill-theorem}
	Let $\pi:E\to B$ be a Riemannian submersion. If $\tilde{\gamma}$ is a geodesic in $E$ and $\tilde{\gamma}'(0)\in \mathcal{H}_{\tilde{\gamma}(0)}$, then $\tilde{\gamma}'(t)\in \mathcal{H}_{\tilde{\gamma}(t)}$ for all $t$. Moreover, $\pi\circ \tilde{\gamma}$ is a geodesic in $B$ and the length is preserved.
\end{thm}

\vspace{0.3cm}
\noindent The above theorem can also be proved using \Cref{lemma:horizontal-lift-is-a-geodesic}, see \cite[Corollary 26.12]{Mic08}.

\vspace{.3cm}

\noindent Combining \Cref{lemma:horizontal-lift-is-a-geodesic} and \Cref{thm:O'Neill-theorem}, we have the following correspondence. 

\begin{thm}\label{thm:geodesic-correspondence}
	There is a one-to-one correspondence between the geodesics on $M/G$ and geodesics on $M$ which are horizontal.
\end{thm}

\vspace{.3cm}
\noindent We now ready to prove \Cref{thm:equivariant-cut-locus}.

\begin{proof}[Proof of \Cref{thm:equivariant-cut-locus}]
	Note that if $\tilde{\gamma}$ is an $N$-geodesic, then $\tilde{\gamma}'(1)\in \mathcal{H}_{\tilde{\gamma}(1)}$ and from \Cref{thm:O'Neill-theorem} $\tilde{\gamma}'(t)\in \mathcal{H}_{\tilde{\gamma}(t)}$ and hence $\tilde{\gamma}$ is a horizontal geodesic which implies $\gamma=\pi\circ \tilde{\gamma}$ is a geodesic in $M/G$.

	\hf Let $\tilde{p}\in \mathrm{Se}(N)$. This implies there exists at least two $N$-geodesic, say $\tilde{\gamma}$ and $\tilde{\eta}$ such that $l(\tilde{\gamma})=l(\tilde{\eta})=d(\tilde{p},N)$. Let us denote $\gamma=\pi\circ \tilde{\gamma}$ and $\eta = \pi\circ\tilde{\eta}$. Due to uniqueness of horizontal lift (\Cref{prop:uniqueness-of-horizontal-lift}), $\gamma$ and $\eta$ will be two different geodesics and the lengths will be same. Also, note that as $\tilde{\gamma}$ is an $N$-geodesic, $\gamma$ will be an $N/G$ geodesic. Otherwise, there exists an $N/G$-geodesic $\delta$ joining $p$ to $N/G$ which gives a horizontal lift $\tilde{\delta}$ whose length is strictly less than $\tilde{\gamma}$, a contradiction (see \Cref{fig:Se(N)/G-subset-Se(N/G)}). Hence, $\pi(\tilde{p})\in \mathrm{Se} (N/G)$. 
	\begin{figure}[!htb]
		\centering
		\includegraphics[width=0.5\textwidth]{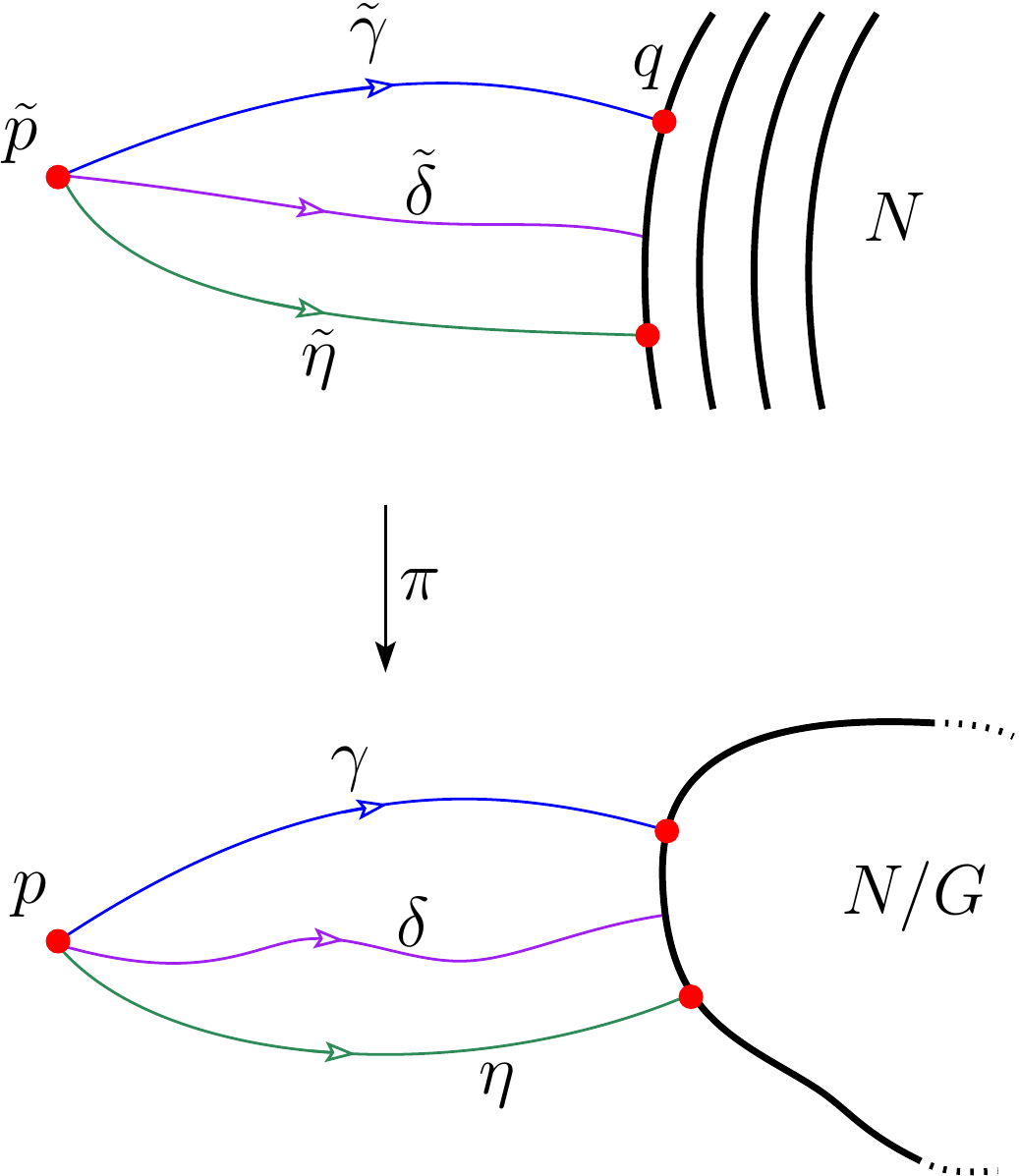}
		\caption{$N$-geodesics maps to $N/G$-geodesics}
		\label{fig:Se(N)/G-subset-Se(N/G)}
	\end{figure}
	On the other hand, if $\gamma$ is an $N/G$ geodesic starting from $p$, then its horizontal lift $\tilde{\gamma}$ will be a geodesic. In fact, it will be an $N$-geodesic. If not, let $\tilde{\eta}$ be such that $l \left(\tilde{\eta}\right)=d \left(\tilde{p},N\right)$ which implies $\tilde{\eta}$ is horizontal. Hence, $\eta$ will be a geodesic and 
	\begin{align*}
		l(\gamma) & =  d \left(p,N/G\right) = l(\eta) \\ 
		& = l \left(\tilde{\eta}\right) < l \left(\tilde{\gamma}\right) = l(\gamma),
	\end{align*}
	a contradiction. 
	Thus $\tilde{p}\in \mathrm{Se}(N)$. This proves that $\mathrm{Se}(N)/G=\mathrm{Se}(N/G)$. In order to prove the theorem, note that we have the following relation. 
	\begin{displaymath}
		\mathrm{Se}(N/G) \subseteq \overline{\sen}/G \subseteq \overline{\mathrm{Se}(N/G)}.
	\end{displaymath}
	As $\overline{\sen}$ is a closed set and $G$ is a compact Lie group, so $\overline{\sen}/G$ is closed. Thus, we have 
	\begin{displaymath}
		\overline{\sen}/G = \overline{\mathrm{Se}(N/G)} \implies \cutn/G=\cutn[N/G].
	\end{displaymath}
\end{proof}

\hf We will discuss some examples based on the above result. Recall from \Cref{join}, 
\begin{displaymath}
	 \mathrm{Cu} \left(\sbb^k_i\right)=\sbb^{n-k-1}_l.
\end{displaymath}
\begin{eg}[Real projective spaces]
	Take $M=\sbb^{n},~N=\sbb^{k}_i$ and $G=\sbb^0\cong \mathbb{Z}_2$. Applying \Cref{thm:equivariant-cut-locus}, we get
	\begin{displaymath}
		\mathrm{Cu}\left(\mathbb{RP}^k_i\right) \cong \mathbb{RP}^{n-k-1}_l.
	\end{displaymath}
\end{eg}

\begin{eg}
	Take $M=\sbb^{2n+1},~N=\sbb^{2k+1}_i$ and $G=\sbb^1$. Applying \Cref{thm:equivariant-cut-locus}, we get
	\begin{displaymath}
		\mathrm{Cu}\left(\mathbb{CP}^k_i\right) \cong \mathbb{CP}^{n-k-1}_l.
	\end{displaymath}
\end{eg}

\begin{eg}[Lens space]
	Consider $\sbb^{2n+1}\subset \cbb^{n+1}$. Let $p$ be a prime number and $\xi=e^{\frac{2\iota\pi}{p}}$ be a primitive $p^{\text{th}}$ root of unity and let $q_1,\cdots,q_{n+1}$ be integers coprime to $p$. Consider $\zbb_p=\left\{1,\xi,\xi^2,\cdots, \xi^{p-1}\right\}$ and let it acts on $\sbb^{2n+1}$ by $\xi \left(z_1,\cdots,z_{n+1}\right)\defeq \left(\xi^{q_1}z_1,\cdots,\xi^{q_{n+1}}z_{n+1}\right)$. The orbit space is denoted by $L(p;q_1,\ldots,q_{n+1})$ and is called a \emph{lens space}.\index{lens spaces} Since $\mathrm{Cu} \left(\sbb^{2n-1}\subset \sbb^{2n+1}\right)=\sbb^1$, so taking the $\zbb_p$ action, gives that the cut locus of $L\left(p;q_1,\ldots,q_n\right)$ in $L \left(p;q_1,\ldots,q_{n+1}\right)$ is $\sbb^1$. In general, we have
	\begin{displaymath}
		\mathrm{Cu}\left(L \left(p;q_1,\ldots,q_{k+1}\right)\right) \cong L \left(p;q_{k+2},\ldots,q_{n+1}\right).
	\end{displaymath}
\end{eg}

\section{Cut locus of hypersurface in complex projective space}\label{subsec:cut-locus-for-X(d)} 

\hfb Let $\zbf=\left(z_0,\cdots,z_{n}\right)\in \mathbb{C}^{n+1}$ and $[\zbf]\in \mathbb{CP}^{n}$, then the Fermat hypersurface of degree $d$ is given by the polynomial $f(\zbf) = z_0^d+\cdots+z_{n}^d$, 
\begin{displaymath}
	X(d) \defeq \{[\zbf]:f(\zbf)=0\}.
\end{displaymath} 

\noindent The homotopy type of the complement of the above hypersurface is well studied in the article \cite{KuWo80}. Since the partial derivatives $\frac{\partial f}{\partial z_j}$ do not vanish simultaneously on $\mathbb{C}^{n+1}-\{0\}$, the hypersurface is nonsingular. We wish to find its cut locus and want to compare our result with the result in \cite[Proposition 3.1]{KuWo80}. In that paper, the authors found the homotopy type of the complement of the hypersurface.
\begin{prop}[\cite{KuWo80}]
	The complement of $X(d)$ is homotopic to the base space of the $n$ universal principal $\mathbb{Z}_d$-bundle constructed by Milnor from the join of $n+1$ copies of $\mathbb{Z}_d$.  
\end{prop}

\vspace{0.3cm}
\noindent Using \Cref{thm: Morse-Bott}, the same can be studied by looking at the cut locus of $X(d)$ in $\mathbb{CP}^n$. We will use \Cref{thm:equivariant-cut-locus} to find the cut locus of $X(d)$. Recall that there is a principal $\mathbb{S}^1$-bundle given by
\begin{displaymath}
	\pi:\mathbb{S}^{2n+1}\to \mathbb{CP}^n,~ (z_0,\cdots,z_n)\mapsto [z_0:\cdots:z_n].
\end{displaymath} 
Therefore, using \Cref{thm:equivariant-cut-locus} it is enough to find the cut locus of $\pi^{-1}(X(d))\defeq \tilde{X}(d)$. We propose the following conjecture.

\begin{conj}\label{conj:cut-locus-of_X(d)}
	The cut locus of $\tilde{X}(d)\subseteq \sbb^{2n+1}$ is $\zbb_d^{\star(n+1)}\times_{\zbb_d}\sbb^1$, where $\times_{\mathbb{Z}_d}$ is the diagonal action of $\mathbb{Z}_d$ and $\star$ denotes the topological join of spaces.  
\end{conj}

\vspace{0.3cm}
\noindent Using \Cref{thm:equivariant-cut-locus}, the cut locus of $X(d)$ will be $\mathbb{Z}_d^{\star(n+1)}$ and hence we are recovering the $n^{\text{th}}$-stage of Milnor join construction of classifying space for $\mathbb{Z}_d$. Note that we can see $\zbb_d\subset \sbb^1$ and $\mathbb{S}^{2n+1}$ is the join of $n$ circles. Hence, $\zbb_d^{\star(n+1)}\subset \sbb^{2n+1}$ and also note that 
\begin{displaymath}
	 \zbb_d^{\star(n+1)} \times \mathbb{S}^1 \hookrightarrow \sbb^{2n+1},~ \left(\vbf,e^{\iota\theta}\right) \mapsto \vbf e^{-\iota\theta}
\end{displaymath}
gives a well-defined map from $\zbb_d^{\star(n+1)} \times_{\zbb_d}\sbb^1$ to $\sbb^{2n+1}$. We will prove some particular cases of the above conjecture. More precisely, we will prove the above conjecture holds for $d=2$ and arbitrary $n$ (\Cref{thm:cut-locus-of_Xn_2}) and $n=1$ and  arbitrary $d$ (\Cref{thm:cut-locus-for_X1_d}). Let us denote 
\begin{displaymath}
	\widetilde{\mathrm{Cu}}= \mathrm{Cu} (\tilde{X}(d)),\text{ and } \widetilde{\mathrm{Se}}= \mathrm{Se} (\tilde{X}(d)).
\end{displaymath}

\begin{thm}[Cut locus of $X(2)$] \label{thm:cut-locus-of_Xn_2}
	The cut locus of $\tilde{X}(2)\subseteq \sbb^{2n+1}$ is $\sbb^{n}\times_{\zbb_2}\sbb^1\cong \{(\vbf\cos \theta,\vbf\sin\theta):\vbf\in \sbb^n,\theta\in [0,2\pi]\}$. Hence, the cut locus of $X(2)$ in $\CP^{n}$ will be $\mathbb{RP}^{n+1}$.
\end{thm}

\begin{proof}
	We will show that $\{(\vbf\cos \theta,\vbf \sin\theta):\vbf\in \sbb^n,\theta\in \rbb\} = \mathrm{Se}(\tilde{X}(2))$. Let $\vbf\in \sbb^n$ and $\theta\in \rbb$. Let us write $z_j=x_j+\iota y_j$. We can write $\tilde{X}(2)$ as
	\begin{align*}
		\tilde{X}(2) & = \left\{\left(z_0,\cdots,z_n\right)\in \cbb^{n+1}:\sum_{i=0}^n z_i^2=0,\text{ and } \sum_{i=0}^n\left|z_i\right|^2=1\right\} \\ 
		& = \left\{ \left(x_0,y_0,\cdots,x_n,y_n\right): \sum_{i=0}^nx_i^2=\frac{1}{2}=\sum_{i=0}^{n} y_i^2,\text{ and } \sum_{i=0}^{n} x_iy_i=0	\right\} \\
		& = \left\{ \left(x_0,x_1,\cdots,x_n,y_0,y_1,\cdots,y_n\right): \sum_{i=0}^nx_i^2=\frac{1}{2}=\sum_{i=0}^{n} y_i^2,\text{ and } \sum_{i=0}^{n} x_iy_i=0	\right\}.
	\end{align*}
	If $A\in O(n+1)$, then $\tilde{A}=\left(\begin{array}[]{c|c}
		A & \mathbf{0}   \\ \hline 
		\mathbf{0}  & A
	\end{array}\right)\in SO(2n+2)$. Note that
	\begin{enumerate}
		\item[i] $\tilde{A}\in \mathrm{Iso}\left(\sbb^{2n+1}\right)$, where $\mathrm{Iso}(M)$ denotes the set of all isometries of $M$. 
		\item[ii] $\tilde{A}$ maps $\tilde{X}(2)$  to itself, and $\widetilde{\mathrm{Cu}}$ to itself. 
	\end{enumerate}
	Thus, if $p\in \widetilde{\mathrm{Se}}\subseteq \widetilde{\mathrm{Cu}}$ and let $\gamma$ and $\eta$ be two distance minimal geodesics joining $p$ to $\tilde{X}(2)$, then $\tilde{A}\gamma$ and $\tilde{A}\eta$ will be two minimal geodesics joining $\tilde{A}p$ to $\tilde{X}(2)$. As the action of $O(n+1)$ on $\sbb^n$ is transitive, it suffices to check if $\ebf _1\cos\theta+\ebf_{n+2}\sin \theta\in \widetilde{\mathrm{Se}}$. We can further reduce our work by looking at the matrix 
	\begin{displaymath}
		B = 
		\begin{pmatrix}
			\begin{array}{cccc|cccc}
				\cos \theta  & 0 & \cdots & 0 & \sin \theta & 0 & \cdots & 0 
				\\
				0 &  &&   & 0 &  && 
				\\
				\vdots &  &I_n&   & \vdots &  &\bigzero_n& 
				\\
				0 &  &&   & 0 &  &&  
				\\ \hline
				-\sin \theta  & 0 & \cdots & 0 & \cos \theta & 0 & \cdots & 0 
				\\
				0 &  &&   & 0 &  && 
				\\
				\vdots &  &\bigzero_n&   & \vdots &  &I_n& 
				\\
				0 &  &&   & 0 &  &&  
				\\
			\end{array}
		\end{pmatrix},
	\end{displaymath}
	which is again an isometry of $\sbb^{2n+1}$ and sends $\ebf _1\cos\theta+\ebf_{n+2}\sin \theta$ to $\ebf_1$. Hence, it is enough to prove that $\ebf _1\in \widetilde{\mathrm{Se}}$.
	Note that 
	\begin{displaymath}
		\dist \left(\ebf_1,\tilde{X}(2)\right)=\frac{\pi}{4}.
	\end{displaymath}
	Let
	\begin{displaymath}
		 \vbf_1 = \ebf_{2n+2},~\vbf_2=-\ebf_{2n+2}.
	\end{displaymath}
	Consider the geodesics 
	\begin{displaymath}
		 \gamma(t)=\ebf_1\cos t+\ebf_{2n+2}\sin t,\text{ and } \eta(t) = \ebf_1\cos t-\ebf_{2n+2}\sin t,~t\in \mathbb{R}. 
	\end{displaymath}
	Note that $\gamma$ and $\eta$ intersect the set $\tilde{X(2)}$ at $t=\frac{\pi}{4}$, so their lengths are same, and it is equal to the distance between $\ebf_1$ and the submanifold. This proves that $\ebf_1\in \widetilde{\mathrm{Se}}$. 

	\hf Conversely, if $p = (\vbf,\wbf)\in \widetilde{\mathrm{Se}}$, then we will show that there exists $\ubf \in \sbb^n$ and $\theta\in \rbb$ such that $\vbf = \ubf \cos \theta$ and $\wbf=\ubf\sin \theta$. Note that if $\{\vbf,\wbf\}$ is linearly dependent, then there exists $\theta\in \rbb$ such that 
	\begin{displaymath}
		 \vbf = \hat{\vbf}\cos \theta,~~ \wbf = \hat{\vbf}\sin \theta
	\end{displaymath}
	and hence $p\in \widetilde{\mathrm{Se}}$. So we assume that $\vbf$ and $\wbf$ are linearly independent. Suppose $p\notin \sbb^n\times_{\zbb_2} \sbb^1$. We need to show that $p\notin \widetilde{\mathrm{Se}}$. To the contrary, let us assume that $p\in \widetilde{\mathrm{Se}}$. Consider a unit speed geodesic $\gamma(t)$ in the direction of $\vbb=(\vbf_1,\vbf_2)\in T_{(\vbf,\wbf)}\sbb^{2n+1}$ which implies 
	\begin{align}
		\left\|\vbf_1\right\|^2+\left\|\vbf_2\right\|^2=1 \label{eq:tangent_space_of_sphere_cond-1} \\ 
		\left\langle \vbf_1, \vbf \right\rangle + \left\langle \vbf_2, \wbf \right\rangle=0 \label{eq:tangent_space_of_sphere_cond-2}
	\end{align}
	Consider the curve $\gamma(t)=\left(\vbf\cos t+\vbf_1\sin t, \wbf\cos t+\vbf_2\sin t\right)$ for $t\in \mathbb{R}$. Note that $2\le \operatorname{rank}[\vbf,\wbf,\vbf_1,\vbf_2] \le 4$. We will prove that none of the cases is possible. Since $p\in \widetilde{\mathrm{Se}}$, so $p\in \widetilde{\mathrm{Cu}}$ which implies there exists $t\in \rbb$ such that $\gamma(t)\in \tilde{X}(2),\gamma'(t)\in \left(T_{\gamma(t)}\tilde{X}(2)\right)^\perp$ and $t$ will be minimum among all such values. So we have
	\begin{align}
		& \left\|\vbf\right\|^2\cos^2t+\left\|\vbf_1\right\|^2\sin^2t+\left\langle \vbf, \vbf_1 \right\rangle \sin 2t = \frac{1}{2} \label{eq:geodesic-on-X(2)-1},
		\\
		& \left\|\wbf\right\|^2\cos^2t+\left\|\vbf_2\right\|^2\sin^2t+\left\langle \wbf, \vbf_2 \right\rangle \sin 2t = \frac{1}{2}, \notag
		\\
		& \left\langle \vbf, \wbf \right\rangle\cos^2t + \left\langle \vbf_1, \vbf_2 \right\rangle\sin^2t + \frac{1}{2} \left(\left\langle \vbf, \vbf_2 \right\rangle + \left\langle \vbf_1, \wbf \right\rangle\right)\sin 2t=0\label{eq:geodesic-on-X(2)-3}. 
	\end{align}
	Now we will make use of the other condition $\gamma'(t)\in \left(T_{\gamma(t)}\tilde{X}(2)\right)^\perp$. Consider the vector $\ubb= \left(-\wbf\cos t+\vbf_2\sin t,\vbf\cos t+\vbf_1\sin t\right)$. We claim that $\ubb\in T_{\gamma(t)}\tilde{X}(2)$. Note that $(\vbf,\wbf)\in T_{(\mathbf{p},\mathbf{q})}\tilde{X}(2)$ implies $\left\langle \mathbf{p} ,\vbf  \right\rangle=0, \left\langle \mathbf{q} ,\wbf  \right\rangle=0$ and $\left\langle \mathbf{p} ,\wbf  \right\rangle+ \left\langle \mathbf{q} , \vbf  \right\rangle=0$. 
	\begin{align*}
		\left\langle \mathbf{p}, \vbf \right\rangle & = \left\langle \left(\vbf \cos t+\vbf_1\sin t\right), \left(-\wbf \cos t-\vbf_2\sin t\right) \right\rangle 
		\\
		& = - \left\langle \vbf, \wbf \right\rangle\cos^2t-\left\langle \left(\vbf,\vbf_2 \right\rangle + \left\langle \vbf_1, \wbf \right\rangle \right) \cos t \sin t - \left\langle \vbf_1, \vbf_2 \right\rangle\sin^2t 
		\\
		& = 0 \qquad (\text{from \eqref{eq:geodesic-on-X(2)-3}}).
	\end{align*}
	Similarly, $\left\langle \mathbf{q}, \wbf \right\rangle=0$, and 
	\begin{align*}
		& \kern 0.5cm \left\langle \mathbf{p}, \wbf \right\rangle + \left\langle \mathbf{q}, \vbf \right\rangle \\
        & = \left\langle \vbf \cos t+\vbf_1\sin t, \vbf\cos t+\vbf_1\sin t  \right\rangle + \left\langle \wbf\cos t+\vbf_2\sin t, -\wbf\cos t-\vbf_2\sin t \right\rangle 
		\\
		& = \left\|\vbf\right\|^2\cos^2t + \left\|\vbf_1\right\|^2\sin^2t + 2 \left\langle \vbf,\vbf_1\right\rangle\cos t\sin t - \left\|\wbf\right\|^2\cos^2t-\left\|\vbf_2\right\|^2\sin^2t 
		\\ 
		& \kern 2cm - 2 \left\langle \wbf, \vbf_2 \right\rangle\cos t \sin t  
		\\ 
		& = \cos^2t \left(\left\|\vbf\right\|^2-\left\|\wbf\right\|^2\right)+ \sin^2t \left(\left\|\vbf_1\right\|^2-\left\|\vbf_2\right\|^2\right) + \sin 2t \left(\left\langle \vbf,\vbf_1 \right\rangle - \left\langle \wbf,\vbf_2 \right\rangle\right) 
		\\ 
		& = 0 \qquad (\text{from \eqref{eq:tangent_space_of_sphere_cond-1} and \eqref{eq:tangent_space_of_sphere_cond-2}}). 
	\end{align*}
	Therefore, $\ubb\in T_{\gamma(t)}\tilde{X}(2)$ and hence $\left\langle \ubb, \gamma'(t) \right\rangle = 0$ which implies 
	\begin{equation}
		\left\langle \vbf,\vbf_2 \right\rangle - \left\langle \vbf_1,\wbf \right\rangle = 0 \label{eq:perp-cond}.
	\end{equation}
	Define 
	\begin{displaymath}
		\tilde{\ubf}_1 = \sqrt{2} \left(\vbf\cos t+\vbf_1\sin t\right),\text{ and } \tilde{\ubf}_2 = \sqrt{2} \left(\wbf\cos t+\vbf_2\sin t\right).
	\end{displaymath}
	Note that $\tilde{\ubf}_1\perp \tilde{\ubf}_2$ and both are vectors in $\rbb^{n+1}$. We extend $\left\{\tilde{\ubf}_1,\tilde{\ubf}_2\right\}$ to an orthonormal basis of $\rbb^{n+1}$, say $\left\{\tilde{\ubf}_1,\tilde{\ubf}_2, \tilde{\ubf}_3,\ldots,\tilde{\ubf}_{n+1}\right\}$. If $\left(\wbf_1,\wbf_2\right)\in T_{\gamma(t)}\tilde{X}(2)$, then $\left\langle \wbf_1,\tilde{\ubf}_1 \right\rangle=0$, $\left\langle \wbf_2,\tilde{\ubf}_2 \right\rangle=0$ and $\left\langle \wbf_1,\tilde{\ubf}_2 \right\rangle+\left\langle \wbf_2,\tilde{\ubf}_1 \right\rangle=0$ as $\gamma(t) = \frac{1}{\sqrt{2}}(\tilde{\ubf}_1,\tilde{\ubf}_2)$. This implies $\wbf_1,\wbf_2\in \operatorname{Span}\left\{\tilde{\ubf}_3,\ldots,\tilde{\ubf}_{n+1}\right\}$ or 
	\begin{displaymath}
		 \wbf_1=-\tilde{\ubf}_2+\sum_{j\ge 3}c_j\tilde{\ubf}_j \text{ and } \wbf_2=\tilde{\ubf}_1+\sum_{j\ge 3}d_j\tilde{\ubf}_j.
	\end{displaymath}
	Since $\wbf_i\in T_{\gamma(t)}\tilde{X}(2)$, 
	\begin{align*}
		\left\langle \gamma'(t), \left(\wbf_1,\wbf_2\right) \right\rangle = 0 & \implies \text{ for } j\ge 3,~\left\langle \gamma'(t),\left(\tilde{\ubf}_j,0\right) \right\rangle = 0, \text{ and }  \left\langle \gamma'(t),\left(0,\tilde{\ubf}_j\right) \right\rangle = 0.
	\end{align*}
	The above implies 
	\begin{equation}
		 -\vbf\sin t+\vbf_1\cos t, -\wbf\sin t+\vbf_2\cos t\in \operatorname{Span} \left\{\tilde{\ubf}_1,\tilde{\ubf}_2\right\} \label{eq:spanning-cond}.
	\end{equation}
	Since $-\vbf\sin t+\vbf_1\cos t\in \operatorname{Span} \left\{\tilde{\ubf}_1,\tilde{\ubf}_2\right\}$, 
	\begin{align}
		 & -\vbf\sin t+\vbf_1\cos t = \alpha \sqrt{2} \left(\vbf\cos t+\vbf_1\sin t\right) + \beta\sqrt{2} \left(\wbf\cos t+\vbf_2\sin t\right)  \nonumber
		 \\ 
		 \implies & \vbf(\sin t-\alpha\sqrt{2}\cos t)+ \wbf(-\beta \sqrt{2}\cos t) + \nonumber \\
        & \kern 1.5cm \vbf_1 (\cos t-\alpha\sqrt{2}\sin t)  + \vbf_2(\beta\sqrt{2}\sin t) = 0. \label{eq:rk-4}
	\end{align}
	If $\operatorname{rank} [\vbf,\wbf,\vbf_1,\vbf_2]=4$, then from  equation \eqref{eq:rk-4}
	\begin{align*}
		\sin t-\alpha \sqrt{2}\cos t =0 = \cos t-\alpha\sqrt{2}\sin t,~\text{and }  -\beta\cos t=0=\beta \sin t.
	\end{align*}
	This implies $2\alpha^2=-1$, which is absurd. Thus, all four vectors can not be linearly independent. Now we are remaining with two cases: $\operatorname{rank} [\vbf,\wbf,\vbf_1,\vbf_2]=2$ or $\operatorname{rank} [\vbf,\wbf,\vbf_1,\vbf_2]=3$.

    \vspace{0.3cm}
	\noindent \textcolor{blue}{\textbf{Case 1:}} $\operatorname{rank} [\vbf,\wbf,\vbf_1,\vbf_2]=2:$ Let $\vbf_1=\alpha \vbf+\beta\wbf$ and $\vbf_2=\gamma\vbf+\delta\wbf$ for some $\alpha,\beta,\gamma,\delta\in \rbb$. Observe that 
	\begin{align}
		\eqref{eq:perp-cond} \implies & \gamma \left\|\vbf\right\|^2-\beta \left\|\wbf\right\|^2= (\alpha-\delta) \left\langle \vbf,\wbf \right\rangle 
		\nonumber \\
		\implies & (\gamma+\beta) \left\|\vbf\right\|^2-(\alpha-\delta)\left\langle \vbf,\wbf \right\rangle = \beta \label{eq:rk-2_cond-1}
		\\
		\eqref{eq:tangent_space_of_sphere_cond-2} \implies & \alpha \left\|\vbf\right\|^2+\delta \left\|\wbf\right\|^2=-(\beta+\gamma)\left\langle \vbf,\wbf \right\rangle \nonumber 
		\\
		\implies & (\alpha-\delta)\left\|\vbf\right\|^2+(\beta+\gamma)\left\langle \vbf,\wbf \right\rangle = -\delta \label{eq:rk-2_cond-2}.
		% \eqref{eq:tangent_space_of_sphere_cond-1} \implies & \left(\alpha^2+\gamma^2\right)\left\|\vbf\right\|^2+2(\alpha \beta +\gamma \delta) \left\langle \vbf,\wbf \right\rangle+\left(\beta^2+\delta^2\right)\left\|\wbf\right\|^2=1. \label{eq:rk-2_cond-3}
	\end{align}

	Using equations \eqref{eq:rk-2_cond-1} and \eqref{eq:rk-2_cond-2}, we obtain
	\begin{displaymath}
		 \left\|\vbf\right\|^2 \left((\alpha-\delta)^2+(\beta+\gamma)^2\right)= \beta \gamma + \beta^2 - \alpha \delta + \delta^2. 
	\end{displaymath}
	Note that $(\alpha-\delta)^2+(\beta+\gamma)^2=0$ implies $\alpha=\delta$ and $\beta=-\gamma$. But from equations \eqref{eq:rk-2_cond-1} and \eqref{eq:rk-2_cond-2} implies that $\alpha=\beta=\delta=\gamma=0$, which is not possible. Thus, 
	\begin{equation*}
		\begin{split}
			& \left\|\vbf\right\|^2 = \frac{\beta \gamma+\beta^2-\alpha \delta+\delta^2}{(\alpha-\delta)^2+(\beta+\gamma)^2},
			\\
			& \left\|\wbf\right\|^2 = \frac{\beta \gamma+\delta^2-\alpha \delta+\beta^2}{(\alpha-\delta)^2+(\beta+\gamma)^2}, \text{ and }
			\\ 
			& \left\langle \vbf,\wbf \right\rangle = \frac{-(\alpha \beta+\gamma \delta)}{(\alpha-\delta)^2+(\beta+\gamma)^2}.\label{eq:v_w_and_their_innerprod}
		\end{split}
	\end{equation*}
	From equation \eqref{eq:geodesic-on-X(2)-1}, we have
	\begin{align*}
		& \left(\alpha^2+\gamma^2\right)\left\|\vbf\right\|^2 + 2 \left(\alpha \beta+\gamma \delta\right) \left\langle \vbf,\wbf \right\rangle + \left(\beta^2+\delta^2\right)\left\|\wbf\right\|^2=1 
		\\
		\begin{split}
			\implies & \left(\alpha^2+\gamma^2\right)\left(\beta \gamma+\beta^2-\alpha \delta+\delta^2\right)-2(\alpha \beta+\gamma \delta)^2 \\
			& \kern 1cm +\left(\beta^2+\delta^2\right)\left(\beta \gamma +\delta^2-\alpha \delta+\beta^2\right)- (\alpha-\delta)^2-(\beta+\gamma)^2= 0
		\end{split}
		\\ 
		\begin{split}
			\implies & (\beta \gamma-\alpha \delta)\left(\sum \alpha^2\right)+ 2 \left(\alpha^2+\gamma^2\right)\left(\beta^2+\delta^2\right)-2(\alpha \beta+\gamma \delta)^2 
			\\
			& \kern 2cm -\left(\sum \alpha^2\right)+2(\alpha \delta-\beta \gamma)=0 
		\end{split}
		\\ 
		\implies & (\beta \gamma-\alpha \delta-1)\left(\sum \alpha^2\right)+2(\beta \gamma-\alpha \delta)^2-2(\beta \gamma-\alpha \delta)=0 
		\\
		\implies & (\beta \gamma-\alpha \delta-1)\left(\sum \alpha^2\right)+ 2(\beta \gamma-\alpha \delta)(\beta \gamma-\alpha \delta-1)=0 
		\\
		\implies & (\beta \gamma-\alpha \delta-1)\left((\alpha-\delta)^2+(\beta+\delta)^2\right)=0,
	\end{align*}
	which implies 
	\begin{equation}
		 \beta \gamma-\alpha \delta=1. \label{eq:ad-bc_cond}
	\end{equation}
	Now we will use the condition that $\gamma(t)\in \tilde{X}(2)$ which was given by equations \eqref{eq:geodesic-on-X(2)-1} and \eqref{eq:geodesic-on-X(2)-3}. From \eqref{eq:geodesic-on-X(2)-1} we have 
	\begin{align*}
		& \left\|\vbf\right\|^2\cos^2+ \left(\alpha^2 \left\|\vbf\right\|^2+\beta^2 \left\|\wbf\right\|^2+2 \alpha \beta \left\langle \vbf,\wbf \right\rangle\right)\sin^2t \\
        & \kern 2cm + \left(\alpha \left\|\vbf\right\|^2+\beta \left\langle \vbf,\wbf \right\rangle\right)\sin 2t = \frac{1}{2} 
		\\
		\begin{split}
			\implies & \sin^2t \left(\alpha^2 \left(1+\beta^2+\delta^2\right)+\beta^2 \left(1+\alpha^2+\gamma^2\right)-2 \alpha \beta(\alpha \beta +\gamma \delta)\right) \\ 
			& \kern 1.5cm 
			+ \cos^2t \left(1+\beta^2+\delta^2\right) \\
			& \kern 1.5cm  + \sin 2t \left(\alpha \left(1+\beta^2+\delta^2\right) -\beta(\alpha \beta+\gamma \delta)\right)= \frac{2+\sum \alpha^2}{2}
		\end{split}
		\\ 
		\implies & \left(\alpha^2+\beta^2+1\right)\sin^2t + \left(\beta^2+\delta^2+1\right)\cos^2t + (\alpha-\delta)\sin 2t= \frac{2+\sum \alpha^2}{2} 
		\\
		\implies & \alpha^2+\beta^2 + \cos^2t \left(\beta^2+\delta^2+1-\alpha^2-\beta^2+1\right)+ (\alpha-\delta)\sin 2t = \frac{\sum \alpha^2}{2} 
		\\
		\implies & \left(\delta^2-\alpha^2\right)\cos^2t + (\alpha-\delta)\sin 2t - \frac{\gamma^2+\delta^2-\alpha^2-b^2}{2}=0,
	\end{align*}
	which simplifies to 
	\begin{equation*}
		 \left(\delta^2-\alpha^2\right)\cos 2t + 2(\alpha-\delta)\sin 2t = \gamma^2-\beta^2 \label{eq:sin-cosine-rel-1}. 
	\end{equation*}
	Similarly, using \eqref{eq:geodesic-on-X(2)-3} we have 
	\begin{equation*}
		 -(\alpha+\delta)(\beta+\gamma)\cos 2t+2(\beta+\gamma)\sin 2t = (\alpha-\delta)(\beta-\gamma) \label{eq:sin-cosine-rel-2}.
	\end{equation*}
	Writing the last two relation into matrix form we have 
	\begin{align}
		\begin{bmatrix}
		-(\alpha-\delta)(\alpha+\delta) & 2(\alpha-\delta) \\ 
		-(\alpha+\delta)(\beta+\gamma) & 2(\beta+\gamma)
		\end{bmatrix}
		\begin{bmatrix}
		\cos 2t \\ \sin 2t
		\end{bmatrix}
		= 
		\begin{bmatrix}
			\gamma^2-\beta^2 \\ (\alpha-\delta)(\beta-\gamma)
		\end{bmatrix}. \label{eq:system-of-eq}
	\end{align}
	Note that the rank of the coefficient matrix is $1$ and this can occur if one row is linear multiple of the other. 

	\begin{itemize}
		\item Let $\alpha=\delta$ and $\beta\neq -\gamma$. The system \eqref{eq:system-of-eq} has a solution, so $\beta-\gamma=0$. Note that 
		\begin{align*}
			\eqref{eq:rk-2_cond-1}  \implies \beta \left(\left\|\vbf\right\|^2-\left\|\wbf\right\|^2\right)= 0  \implies \beta=0 \text{ or } \left\|\vbf\right\|^2=\left\|\wbf\right\|^2. 
		\end{align*}
		But, note that $\beta=0$ can not be possible because if $\beta=0=\gamma$, then using \eqref{eq:ad-bc_cond} $\alpha^2=-1$, a contradiction. Therefore, $\left\|\vbf\right\|^2=\frac{1}{2}=\left\|\wbf\right\|^2$. Again using \eqref{eq:rk-2_cond-2}, 
		\begin{align*}
			\left\langle \vbf,\wbf \right\rangle= \frac{-\alpha}{2\beta}. 
		\end{align*}
		Since
		\begin{align*}
			\left\|\vbf_1\right\|^2+\left\|\vbf_2\right\|^2 = 1 & \implies  \frac{\left(\alpha^2+\beta^2\right)}{2} +4\alpha \beta \left(\frac{-\alpha}{2\beta}\right)+ \frac{\left(\alpha^2+\beta^2\right)}{2}=1 \\ 
			& \implies \beta^2-\alpha^2=1.
		\end{align*}
		Consider 
		\begin{align*}
			\left\langle \vbf,\wbf \right\rangle = -\frac{\alpha}{2 \beta} & \implies 4\beta^2\left\langle \vbf,\wbf \right\rangle^2 = \alpha^2 = \beta^2-1 
			\\
			& \implies \beta^2 = \frac{1}{1-4 \left\langle \vbf,\wbf \right\rangle^2},~\alpha^2 = \frac{4 \left\langle \vbf,\wbf \right\rangle^2}{1-4 \left\langle \vbf,\wbf \right\rangle^2}. 
		\end{align*}
		Note that the above expression is valid as $\vbf$ and $\wbf$ are linearly independent and each has norm $\frac{1}{\sqrt{2}}$. Thus, if $p\in \widetilde{\mathrm{Se}}$, then the only two possible directions are $\mathcal{v}_1=(\vbf_1,\vbf_2)$ and $-\mathcal{v}=(-\vbf_1,-\vbf_2)$. If $\gamma$ and $\eta$ be two geodesics in $\mathcal{v}_1$ and $-\mathcal{v}_1$ direction respectively, then they intersect $\tilde{X}(2)$ at $t$ and $\pi-t$ respectively. As their lengths are same, and they are unit speed geodesics, so $t=\pi-t$ which implies $t=\frac{\pi}{2}$. This implies $\left(\vbf_1,\vbf_2\right)\in \tilde{X}(2)$ which means 
		\begin{align*}
			\left\langle \vbf_1,\vbf_2 \right\rangle= 0 & \implies \frac{\alpha \beta}{2}  + \frac{\alpha \beta}{2} + 2\alpha \beta \left\langle \vbf,\wbf \right\rangle = 0 \\ 
			& \implies \alpha \beta(2+\left\langle \vbf,\wbf \right\rangle) = 0 \\
			& \implies \alpha = 0. 
		\end{align*}
	If $\alpha=0=\delta$, then $\beta^2=1$ which implies $\beta=\pm 1$, and hence $(\vbf,\wbf)\in \tilde{X}(2)$, which is a contradiction.

	\item Let $\beta+\gamma=0$ or $\alpha-\delta\neq 0$. But $\alpha-\delta\neq 0$, so $\beta=\gamma=0$ which implies $\vbf_1=\alpha \vbf$ and $\vbf_2=\delta\wbf$. Observe that 
	\begin{align*}
		\eqref{eq:rk-2_cond-1} & \implies \left\langle \vbf,\wbf \right\rangle=0 \implies \left\langle \vbf_1,\vbf_2 \right\rangle=0. \\
		\eqref{eq:tangent_space_of_sphere_cond-1} & \implies \alpha^2 \left\|\vbf\right\|^2 + \delta^2 \left\|\wbf\right\|^2 = 1 \\ 
		& \implies \left\|\wbf\right\|^2 \left(\alpha^2+\delta^2\right)=1-\alpha^2\\ 
		& \implies \left\|\wbf\right\|^2 = \frac{1-\alpha^2}{\alpha^2+\delta^2}, \text{ and } \left\|\vbf\right\|^2 =  \frac{\delta^2-1}{\alpha^2+\delta^2}. 
	\end{align*}
	Now we use the condition that $\alpha \delta=-1$ to obtain, 
	\begin{displaymath}
		 \alpha^2=\frac{\left\|\wbf\right\|^2}{\left\|\vbf\right\|^2}, \text{ and } \delta^2 = \frac{\left\|\vbf\right\|^2}{\left\|\wbf\right\|^2}, 
	\end{displaymath}
	which is again fixed and there are only two possible directions $\mathcal{v}_1$ and $-\mathcal{v}_1$ and hence this case is also not possible. 

	\item Finally, both row are non-zero and let $\alpha+\delta=\lambda(\beta+\gamma)$. So \eqref{eq:system-of-eq} become 
	\begin{align*}
		& \begin{bmatrix}
			-(\alpha+\delta)\lambda(\beta+\gamma) & 2\lambda(\beta+\gamma) \\ 
			-(\alpha+\delta)(\beta+\gamma) & 2(\beta+\gamma) 
		\end{bmatrix}
		\begin{bmatrix}
			\cos 2t \\ \sin 2t
		\end{bmatrix}
		= \begin{bmatrix}
			-\left(\beta^2-\gamma^2\right) \\ \lambda \left(\beta^2-\gamma^2\right)
		\end{bmatrix} 
		\\
		\implies & 
		\begin{bmatrix}
			-(\alpha+\delta)\lambda & 2\lambda \\ 
			-(\alpha+\delta) & 2
		\end{bmatrix} 
		\begin{bmatrix}
			\cos 2t \\ \sin 2t
		\end{bmatrix} 
		= \begin{bmatrix}
			-(\beta+\gamma) \\ \lambda (\beta+\gamma)
		\end{bmatrix},
	\end{align*}
	which implies $\lambda^2=-1$, a contradiction. 
	\end{itemize} 
	Thus, we have proved that $\operatorname{rank}[\vbf,\wbf,\vbf_1,\vbf_2]\neq 2$.
	
	\vspace{0.3cm}
	\noindent \textcolor{blue}{\textbf{Case 2:}} $\operatorname{rank}[\vbf,\wbf,\vbf_1,\vbf_2]=3:$ Since rank is $3$, without loss of generality we assume that $\vbf_2\in  \operatorname{Span}\left\{\vbf,\wbf,\vbf_1\right\}$. Let us write $\vbf_2= a \vbf+b\wbf+c\vbf_1$ for some $a,b,c\in \rbb$. Since the rank is three we can assume that all the vectors are in $\rbb^3$ and let $\times$ denote the vector cross product. Let 
	\begin{align*}
		\vbb_1 & = \left(\vbf\cos t+\vbf_1\sin t)\right) \times \left(\wbf\cos t+\vbf_2\sin t\right) \\ 
		& = (\vbf\times \wbf)\cos^2 t+ \left(\vbf\times \vbf_2\right)\cos t\sin t+ \left(\vbf_1\times \wbf\cos t\sin t\right)+ \left(
		\vbf_1\times \vbf_2\right)\sin^2 t. 
	\end{align*}
	Using \eqref{eq:spanning-cond}, we have 
	\begin{align}
		\vbb_1 \cdot \left(-\vbf\sin t+\vbf_1\cos t\right)=0 & \implies \left[\vbf,\wbf,\vbf_1\right]\cos t+\left[\vbf_1,\vbf,\vbf_2\right]\sin t=0 \label{eq:rk-3-cond-1} \\
		\vbb_1 \cdot \left(-\vbf\sin t+\vbf_2\cos t\right)=0 & \implies \left[\vbf,\wbf,\vbf_2\right]\cos t+\left[\vbf_1,\wbf,\vbf_2\right]\sin t=0 \label{eq:rk-3-cond-2}.
	\end{align}
	Note that 
	\begin{align*}
		\left[\vbf_1,\vbf,\vbf_2\right] = \left[\vbf_1,\vbf,a\vbf+b\wbf+c\vbf_1\right] = b \left(\vbf,\wbf,\vbf_1\right).
	\end{align*}
	So from \eqref{eq:rk-3-cond-1},
	\begin{equation}
		 \left[\vbf,\wbf,\vbf_1\right]\cos t+ b \left[\vbf,\wbf,\vbf_1\right]\sin t = 0  \implies \cos t+ b\sin t=0 \label{eq:cos+bsin=0}. 
	\end{equation}
	Above implies, 
	\begin{displaymath}
		 \cos t= \frac{-b}{1+b^2},\text{ and } \sin t=\frac{1}{1+b^2}.
	\end{displaymath}
	Similarly, using \eqref{eq:rk-3-cond-2},
	\begin{align}
		& \left[\vbf,\wbf,a\vbf+b\wbf+c\vbf_1\right]\cos t+ \left[\vbf_1,\wbf,a\vbf+b\wbf+c\vbf_1\right]\sin t = 0 \nonumber \\
		\implies & c \left(\vbf,\wbf,\vbf_1\right)\cos t+ a \left[\vbf_1,\wbf,\vbf\right]\sin t= 0 \nonumber\\ 
		\implies & \left[\vbf,\wbf,\vbf_1\right](c \cos t-a\sin t)=0 \nonumber \\
		\implies & c\cos t = a\sin t\label{eq:c_cost=asint}.
	\end{align}
	Using \eqref{eq:cos+bsin=0} and \eqref{eq:c_cost=asint} we obtain 
	\begin{equation*}
		 a+bc=0 \label{eq:a+bc=0}.
	\end{equation*} 
	Now we collect some more conditions using previous conditions. 
	\begin{align}
		\eqref{eq:perp-cond} & \implies a \left\|\vbf\right\|^2 + b \left\langle \vbf,\wbf \right\rangle + c \left\langle \vbf,\vbf_1 \right\rangle - \left\langle \vbf_1,\wbf \right\rangle =0\label{eq:rk-3-perp-cond-1}
		\\
		& \implies -bc \left\|\vbf\right\|^2+b \left\langle \vbf,\wbf \right\rangle+c \left\langle \vbf_1,\vbf \right\rangle - \left\langle \vbf_1,\wbf \right\rangle=0\nonumber
		\\ 
		& \implies \left(\vbf_1-b\vbf\right)\cdot (\wbf-c\vbf)=0 \nonumber
		\\
		\eqref{eq:tangent_space_of_sphere_cond-2} & \implies \left\langle \vbf,\vbf_1 \right\rangle+ a \left\langle \vbf,\wbf \right\rangle+ b \left\|\wbf\right\|^2+ c \left\langle \vbf_1,\wbf \right\rangle=0 \label{eq:rk-3-tangent-cond-1}
		\\
		& \implies \left\langle \vbf,\vbf_1 \right\rangle-bc \left\langle \vbf,\wbf \right\rangle+b \left\|w\right\|^2+ c \left\langle \vbf,\vbf_1 \right\rangle=0 \nonumber 
		\\ 
		& \implies \vbf_1\cdot (\vbf-c\wbf)=b\wbf\cdot (c\vbf-\wbf)\nonumber. 
	\end{align}
	Multiply \eqref{eq:rk-3-perp-cond-1} by $c$ and add to \eqref{eq:rk-3-tangent-cond-1} to obtain 
	\begin{equation}
		 \left\langle \vbf_1,\vbf \right\rangle = \frac{bc^2}{1+c^2}\left\|\vbf\right\|^2- \frac{b}{1+c^2}\left\|\wbf\right\|^2 ]\label{eq:v1.v}. 
	\end{equation}
	Similarly, 
	\begin{equation}
		 \left\langle \vbf_1,\wbf \right\rangle = \frac{-bc}{1+c^2}+b \left\langle \vbf,\wbf \right\rangle \label{eq:v1.w}.
	\end{equation}
	We use \eqref{eq:geodesic-on-X(2)-1} and substitute the value of $\cos t,\sin t$ and use \eqref{eq:v1.w} and \eqref{eq:v1.v} to obtain 
	\begin{equation*}
		 \left\|\vbf_1\right\|	^2= \frac{b^2 \left(c^2-1\right)}{1+c^2} \left\|\vbf\right\|^2- \frac{2b^2}{1+c^2} \left\|\wbf\right\|^2+ \frac{1+b^2}{2} \label{eq:v_1-norm}. 
	\end{equation*}
	Now use $\left\|\gamma(t)\right\|^2=1$,
	\begin{align*}
		& \frac{b^2\left(c^2-1\right)}{1+c^2} \left\|\vbf\right\|^2 + \frac{b^2 \left(c^2-1\right)}{1+c^2}\left\|\wbf\right\|^2 + \frac{b^2 \left(c^2-1\right)^2 + \left(c^2+1\right)^2}{2 \left(1+c^2\right)}=1 \\
		\implies & \frac{b^2\left(c^2-1\right)}{1+c^2}\left(\left\|\vbf\right\|^2+\left\|\wbf\right\|^2\right)+ \frac{b^2 \left(c^2-1\right)^2 + \left(c^2+1\right)^2}{2 \left(1+c^2\right)}=1 \\ 
		\implies & 2b^2 \left(c^2-1\right) + b^2 \left(c^2-1\right)^2 = 2 \left(1+c^2\right)-\left(1+c^2\right)^2 \\ 
		\implies & \left(c^2-1\right)\left(2b^2+b^2c^2-b^2\right)= \left(1+c^2\right)\left(1-c^2\right) \\ 
		\implies & \left(c^2-1\right)\left(b^2+b^2c^2+1+c^2\right)=0 \\
		\implies & c= \pm 1.
	\end{align*}
	We also have 
	\begin{align*}
		& \left(\vbf\cos t+\vbf_1\sin t\right)\cdot \left(\wbf \cos t+ \left(a\vbf+b\wbf+c\vbf_1\right)\sin t\right)=0 \\
		\implies & \left(-b\vbf+\vbf_1\right)\cdot \left(-b \wbf+a\vbf+b\wbf+c\vbf_1\right)=0 \\
		\implies & b^2c \left\|\vbf\right\|^2-2bc \left\langle \vbf,\vbf_1 \right\rangle+ c \left\|\vbf_1\right\|^2 = 0 \\ 
		\implies & \frac{c \left(1+b^2\right)}{2}=0 \\ 
		\implies & c = 0,
	\end{align*}
	which is a contradiction. Hence, the rank can not be $3$. Therefore, $\vbf$ and $\wbf$ are linearly dependent and hence, $(\vbf,\wbf)\in \sbb^n\times_{\zbb_2}\sbb^1$. 
\end{proof}

We now prove that the cut locus of $\tilde{X}(d)$ when $n=1$ will be $\left(\zbb_d\star \zbb_d\right) \times_{\zbb_d} \sbb^1$. 

\begin{note}[Cut locus of $X(2)$ for $n=1$] \label{note:cut-locus-of_X(2)-for_n=1}
	For $n=1$, the cut locus of $\tilde{X}(2)$ is
	\begin{equation*}
		\mathrm{Cu}(\tilde{X}(2))= \left\{\frac{1}{2} \left(\cos s+\sin t, \sin s+\cos t,\sin s-\cos t,-\cos s+\sin t\right):s,t\in \rbb\right\}.
	\end{equation*}
	Moreover, the cut locus of $X(2)$ will be $\mathbb{RP}^1$. 
\end{note}

\begin{proof}
	Let us write $z_j=x_j+\iota y_j$. Note that 
	\begin{align*}
		 \tilde{X}(2) & = \left\{\left(z_0,z_1\right)\in \cbb^{2}:z_0^2+z_1^2=0,\text{ and } \left|z_0\right|^2+\left|z_1\right|^1=1 \right\} \\
		 & = \left\{\left(x_0,y_0,x_1,y_1\right):x_0^2+x_1^2=\frac{1}{2}=y_0^2+y_1^2,\text{ and } x_0y_0+x_1y_1=0\right\}.
	\end{align*}

	\noindent Since $\left(x_0,x_1\right)$ and $\left(y_0,y_1\right)$ can not be zero vectors, so without loss of generality, we assume that $x_0y_1\neq 0$. Since 
	\begin{align*}
		x_0y_0+x_1y_1=0 \implies \frac{y_0}{y_1}+\frac{x_1}{x_0} = 0 \implies y_0 = -y_1 \left(\frac{x_1}{x_0}\right).
	\end{align*} 
	Now as 
	\begin{align*}
		y_0^2+y_1^2=\frac{1}{2} \implies y_1^2 \left(\frac{x_1}{x_0}\right)^2+y_1^2=\frac{1}{2} \implies y_1^2 \left(\frac{x_1^2}{x_0^2}+1\right) = \frac{1}{2} \implies y_1=\pm x_0.
	\end{align*}
	Similarly, we have 
	\begin{displaymath}
		x_1=\pm y_0.
	\end{displaymath}
	Therefore, 
	\begin{align*}
		\tilde{X}(2) & = \left\{\left(x_0,y_0,y_0,-x_0\right):x_0^2+y_0^2=\frac{1}{2}\right\}\sqcup \left\{\left(x_0,-y_0,y_0,x_0\right):x_0^2+y_0^2=\frac{1}{2}\right\} \\
        & = S^1_1\sqcup S^1_2.
	\end{align*}
	Define a linear transformation 
	\begin{displaymath}
		T:\rbb^4\to \rbb^4, ~(a,b,c,d) \mapsto \frac{1}{\sqrt{2}}(a-d,b+c,b-c,a+d).
	\end{displaymath}
	Note that $T$ maps $\sbb^3$ onto $\sbb^3$ and is an isometry hence it preserves the cut locus. So 
	\begin{align*}
		\mathrm{Cu}(T(\tilde{X}(2))) & = \mathrm{Cu} \left(T \left(S^1_1\right)\sqcup T \left(S^1_2\right)\right) \\ 
		& = \mathrm{Cu} \left(\left\{(a,b,0,0):a^2+b^2=1\right\} \sqcup \left\{(0,0,c,d):c^2+d^2=1\right\}\right) \\
		& = \left\{\frac{1}{\sqrt{2}}(\cos s,\sin s, \cos t,\sin t):s,t\in\rbb \right\}.
	\end{align*}
	Therefore, the cut locus of $X(2)$ can be found by the inverse transformation which is 
	\begin{displaymath}
		 T^{-1}(x,y,z,w) = \frac{1}{2} (x+2,y+z,y-z,w-x).
	\end{displaymath}
	Hence, 
	\begin{align*}
		\mathrm{Cu} (\tilde{X}(2)) & = \left\{\frac{1}{2} \left(\cos s+\sin t, \sin s+\cos t,\sin s-\cos t,-\cos s+\sin t\right):s,t\in \rbb\right\}\\ & \cong \sbb^1\times \sbb^1.
	\end{align*}
	Quotient with $\sbb^1$ give the required cut locus.
\end{proof}

\begin{thm}\label{thm:cut-locus-for_X1_d}
	For $n=1$, we have 
	\begin{displaymath}
		\widetilde{\mathrm{Cu}} = \left(\zbb_d\star\zbb_d\right) \times_{\zbb_d}\sbb^1.
	\end{displaymath}
\end{thm}

\begin{proof}
	Note that 
	 \begin{displaymath}
		\zbb_d\star\zbb_d = \left\{\vbf=(\xi^k\cos \phi,\xi^l\sin\phi):0\le k,l\le d-1,~\text{and } 0\le \phi\le \frac{\pi}{2}\right\},
	 \end{displaymath}
	 where $\xi^k$ is a $d^\text{th}$ root of unity. Let $\vbf\in \zbb_d\star \zbb_d$ and $\theta\in (0,2\pi)$. We will show that $\vbf e^{\iota \theta}\in \widetilde{\mathrm{Se}}$. Due to $\sbb^1$-symmetry, it is enough to show that $\vbf\in \widetilde{\mathrm{Se}}$. Now consider the matrix 
	 \begin{displaymath}
		  A = 
		  \begin{pmatrix}
			\lambda_1 & 0 \\ 0 & \lambda_2	  
		  \end{pmatrix} \in U(2),
		  \text{ such that } \lambda_1^d=1=\lambda_2^d.
	 \end{displaymath}
	 We observe that $A$ maps $\tilde{X}(d)$ to itself and hence it is enough to show that $p=(\cos \phi,0,\sin\phi,0)\in \widetilde{\mathrm{Se}}$. Note that 
	\begin{align*}
		\tilde{X}(d) & = \left\{\left(z_1,z_2\right)\in \sbb^3: z_1^d+z_2^d=0\right\} \\
			& = \left\{\left(z_1,z_2\right)\in \sbb^3: z_1^d=-z_2^d,~ \left|z_1\right|=\frac{1}{\sqrt{2}}=\left|z_2\right|\right\} \\ 
			& = \left\{\left(z,\xi z\right)\in \sbb^3: \xi^d = -1, \text{ and } \left|z\right|=\frac{1}{\sqrt{2}}\right\} \\
			& = \bigsqcup_{k=0}^{d-1}\left\{(z,\xi z): \xi = e^{\frac{(2k+1)\pi}{d}},~|z|=\frac{1}{\sqrt{2}} \right\} = \bigsqcup_{k=0}^{d-1}X_k
	\end{align*}
	 
	\noindent We now compute the distance of $p$ from $X_k$ and will show that the distance is same with at least two components, which will show that $p\in \widetilde{\mathrm{Se}}$. Note that for any point $Q=(x_1,y_1,x_2,y_2)\in \tilde{X}(d)$, the distance between $p$ and $Q$ is given by (look at the \Cref{fig:distance_p_to_Xd})
	\begin{figure}[!htb]
		\centering
    \def\svgwidth{0.3\columnwidth}
    %% Creator: Inkscape 1.2.1 (1:1.2.1+202210291244+9c6d41e410), www.inkscape.org
%% Accompanies image file 'distance_p_to_Xd.pdf' (pdf, eps, ps)
\begingroup%
  \makeatletter%
  \providecommand\color[2][]{%
    \errmessage{(Inkscape) Color is used for the text in Inkscape, but the package 'color.sty' is not loaded}%
    \renewcommand\color[2][]{}%
  }%
  \providecommand\transparent[1]{%
    \errmessage{(Inkscape) Transparency is used (non-zero) for the text in Inkscape, but the package 'transparent.sty' is not loaded}%
    \renewcommand\transparent[1]{}%
  }%
  \providecommand\rotatebox[2]{#2}%
  \newcommand*\fsize{\dimexpr\f@size pt\relax}%
  \newcommand*\lineheight[1]{\fontsize{\fsize}{#1\fsize}\selectfont}%
  \ifx\svgwidth\undefined%
    \setlength{\unitlength}{163.61612996bp}%
    \ifx\svgscale\undefined%
      \relax%
    \else%
      \setlength{\unitlength}{\unitlength * \real{\svgscale}}%
    \fi%
  \else%
    \setlength{\unitlength}{\svgwidth}%
  \fi%
  \global\let\svgwidth\undefined%
  \global\let\svgscale\undefined%
  \makeatother%
  \begin{picture}(1,1.68697905)%
    \lineheight{1}%
    \setlength\tabcolsep{0pt}%
    \put(0,0){\includegraphics[width=\unitlength,page=1]{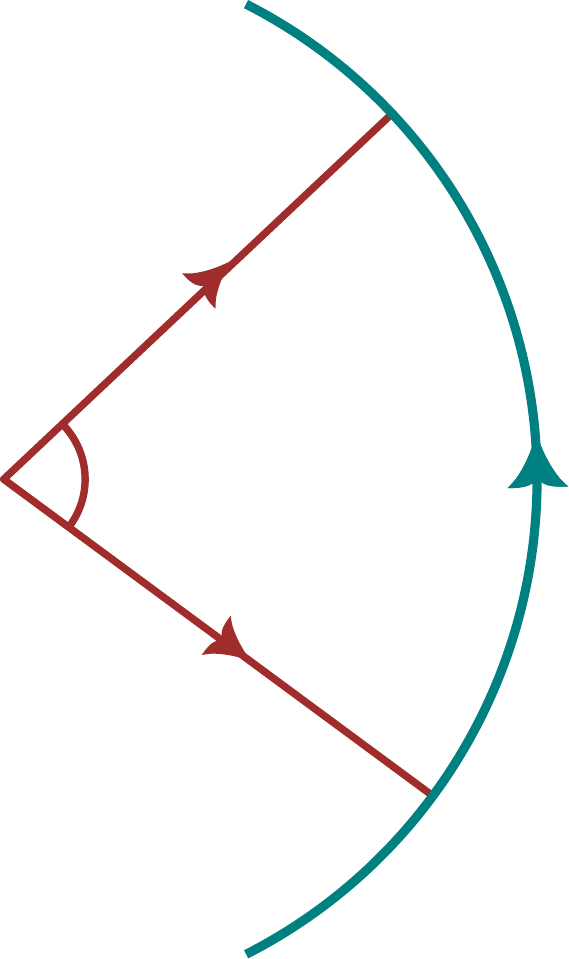}}%
    \put(0.17609985,0.82622238){\makebox(0,0)[lt]{\lineheight{1.25}\smash{\begin{tabular}[t]{l}$\theta$\end{tabular}}}}%
    \put(0,0){\includegraphics[width=\unitlength,page=2]{distance_p_to_Xd.pdf}}%
    \put(0.80295215,0.23136971){\makebox(0,0)[lt]{\lineheight{1.25}\smash{\begin{tabular}[t]{l}$p$\end{tabular}}}}%
    \put(0.73226542,1.50520188){\makebox(0,0)[lt]{\lineheight{1.25}\smash{\begin{tabular}[t]{l}$Q=(x_1,y_1,x_2,y_2)$\end{tabular}}}}%
  \end{picture}%
\endgroup%

		\caption{distance of $p$ to $\tilde{X}(d)$ \label{fig:distance_p_to_Xd}}
	\end{figure}
	
	\begin{displaymath}
		\mathrm{dist}(p,Q) = \cos^{-1}(p\cdot Q) = \cos^{-1} \left(x_1\cos \phi + x_2\sin \phi\right).
	\end{displaymath} 

	\noindent Therefore, the distance between $p$ and the set $\tilde{X}(d)$ is given by
	\begin{align*}
		\mathrm{dist}(p,\tilde{X}(d)) & = \inf \left\{\mathrm{dist}(p,Q):Q\in \tilde{X}(d)\right\}\\
		& = \inf \left\{\cos^{-1}\left(x_1\cos \phi+x_2\sin\phi\right):\left(x_1,y_1,x_2,y_2\right)\in \tilde{X}(d) \right\}. 
	\end{align*}
	As $\cos^{-1}$ is a decreasing function, it is equivalent to maximize $x_1\cos \phi + x_2\sin \phi$ such that $x_2=x_1\cos \left(\frac{(2k+1)\pi}{d}\right)-y_1\sin\left(\frac{(2k+1)\pi}{d}\right)$, and $x_1^2+y_1^2=\frac{1}{2}$. This maximum value will be
	\begin{displaymath}
		\sqrt{\frac{1+\sin 2\phi\cos \left(\frac{(2k+1)\pi}{d}\right)}{2}}.
	\end{displaymath}
	Therefore, the distance from $X_k$ will be
	\begin{displaymath}
		\cos^{-1} \left(\sqrt{\frac{1+\sin 2\phi\cos \left(\frac{(2k+1)\pi}{d}\right)}{2}}\right).
	\end{displaymath}
	Note that if $d$ is even, then the above distance is same from $X_k$ and $X_{d-1-k}$, therefore, the point is a separating point and hence is a cut point. If $d$ is odd, then the above still holds. The only thing to make sure that the distance from each component is smaller than the distance from $X_{\frac{d-1}{2}}$, but this is true as $0\le \phi \le \frac{\pi}{2}$. Therefore, we proved that $p\in \widetilde{\mathrm{Se}}$ and hence it is in $\widetilde{\mathrm{Cu}}$. The inverse inclusion follows with a similar argument. In fact, if we take any other point, then there will be only one distance minimal geodesic, that will occur from one component. Let $p=\left(r_1e^{\iota \theta_1},r_2 e^{\iota \theta_2}\right)$ be any point in $\mathbb{S}^3$ not in the above form. Note that if $\theta_1=\theta_2$, then due to $\mathbb{S}^1$-equivariant, $p\in \widetilde{\mathrm{Se}}$ if and only if $(r_1,0,r_2,0)\in \widetilde{\mathrm{Se}}$, which is a point in the above form. So we assume that $\theta_1\neq \theta_2$ and $\theta_1,\theta_2 \in [0,2\pi/d]$. Then we have to minimize the distance from $p$ to $\tilde{X}(d)$, that is
	\begin{displaymath}
		\dist(p,\tilde{X}(d)) = \inf \left\{\cos^{-1}(p\cdot (z_1,z_2)):(z_1,z_2)\in \tilde{X}(d)\right\}.
	\end{displaymath} 
	This is equivalent to maximizing the dot product $p\cdot (z_1,z_2)$, and this maximum will be achieved from one component, say $X_l$, of $\tilde{X}(d)$ as $0<|\theta_1-\theta_2|<2\pi/d$. A similar computation shows that there is only one distance minimal geodesic from the component $X_l$, and this shows that $p\notin \widetilde{\mathrm{Se}}$. Hence, the theorem is proved.  
\end{proof}

\begin{rem}
	In \cite[\S 2.2]{Aud05}, the author has shown that the function
	\begin{displaymath}
		f:\mathbb{CP}^n\to \mathbb{R},~[x+\iota y]\mapsto \dfrac{\left\|y\right\|^2}{\left(1+\left\|y\right\|^2\right)^2} 
	\end{displaymath}
	is a Morse-Bott function with two critical submanifolds $-$ $X(2)$ and $\mathbb{RP}^n$. We know that $\mathbb{CP}^n\setminus X(2)$ deformation retracts to $\mathbb{RP}^n$ via Morse-Bott flow, whereas \Cref{defretM-N} implies that $\mathbb{CP}^n\setminus X(2)$ deformation retracts to $\cutn[X(2)]$. Hence, $\cutn[X(2)]$ and $\mathbb{RP}^n$  have the same homotopy type. However, it's not a priori clear whether they are equal. Moreover, our calculation shows a computation of $\cutn[T_1\mathbb{S}^n] \subseteq \mathbb{S}^{2n+1}$, where $T_1\mathbb{S}^n$ is the unit tangent bundle of $\mathbb{S}^n$.   We deduce that the cut locus is the unique non-trivial $\mathbb{S}^1$-bundle of $\mathbb{RP}^n$ and this  is a new computation.
\end{rem}

	%The endmatter
	\backmatter
	\addcontentsline{toc}{chapter}{Bibliography}
	\bibliographystyle{apalike}

	\newpage\hbox{}\thispagestyle{empty}\newpage
	
	\addcontentsline{toc}{chapter}{Index}
	\printindex 
	\thispagestyle{empty}
\end{document}